\documentclass{memo-l}

\usepackage[table]{xcolor}
\usepackage{multirow}
\usepackage{amssymb}
\usepackage{thmtools}
\usepackage{units}
\usepackage{graphicx}
\usepackage[all]{xy}
\usepackage{pdflscape}
\usepackage{makecell}
\usepackage{soul}
\setul{1pt}{.4pt}

\usepackage{tikz}
\usetikzlibrary{arrows,calc,positioning,decorations.pathreplacing}
\usepackage{tikz-cd}
\usetikzlibrary{decorations.markings}
\tikzset{double line with arrow/.style args={#1,#2}{decorate,decoration={markings,%
mark=at position 0 with {\coordinate (ta-base-1) at (0,1pt);
\coordinate (ta-base-2) at (0,-1pt);},
mark=at position 1 with {\draw[#1] (ta-base-1) -- (0,1pt);
\draw[#2] (ta-base-2) -- (0,-1pt);
}}}} 

\usepackage{relsize}
\usepackage{mhsetup}
\usepackage{mathtools}
\usepackage{stmaryrd}
\usepackage{caption,subcaption}
\usepackage[mathscr]{euscript}
\usepackage{accents}
\usepackage{xparse}
\usepackage{combelow}
\usepackage{enumerate}
\usepackage{enumitem}
\usepackage[colorlinks,citecolor=blue,breaklinks]{hyperref}
\usepackage[export]{adjustbox} 

\renewcommand{\S}{\ensuremath{\mathbb{S}}}

\newcommand{\cU}{\ensuremath{\mathcal{U}}}
\newcommand{\bD}{\ensuremath{\mathbb{D}}}
\newcommand{\bR}{\ensuremath{\mathbb{R}}}
\newcommand{\Z}{\ensuremath{\mathbb{Z}}}
\newcommand{\R}{\ensuremath{\mathbb{R}}}

\newcommand{\D}{\ensuremath{\mathbb{D}}}
\newcommand{\Tor}{\ensuremath{\mathbb{T}^2}}

\newcommand{\Proj}{\ensuremath{\mathbb{P}^2}}
\newcommand{\Moeb}{\ensuremath{\mathbb{M}^2}}
\newcommand{\Sym}{\ensuremath{\mathfrak{S}}}
\newcommand{\ab}{\ensuremath{\mathrm{ab}}}

\newcommand{\Lie}{\ensuremath{\mathcal{L}}}
\newcommand{\LCS}{\ensuremath{\varGamma}}
\newcommand{\B}{\ensuremath{\mathbf{B}}}
\newcommand{\F}{\ensuremath{\mathbf{F}}}
\newcommand{\PB}{\ensuremath{\mathbf{P}}}
\newcommand{\vB}{\ensuremath{\mathbf{vB}}}
\newcommand{\vP}{\ensuremath{\mathbf{vP}}}
\newcommand{\wB}{\ensuremath{\mathbf{wB}}}
\newcommand{\wP}{\ensuremath{\mathbf{wP}}}
\newcommand{\exwB}{\ensuremath{\widetilde{\mathbf{w}}\mathbf{B}}}
\newcommand{\exwP}{\ensuremath{\widetilde{\mathbf{w}}\mathbf{P}}}
\newcommand{\ev}{\mathrm{ev}}

\DeclareMathOperator{\lcm}{lcm}
\DeclareMathOperator{\Supp}{Supp}
\DeclareMathOperator{\ima}{Im}
\DeclareMathOperator{\Aut}{Aut}
\DeclareMathOperator{\Inn}{Inn}
\DeclareMathOperator{\Int}{Int}

\DeclareRobustCommand{\Spar}{\ifmmode\mathsection\else\textsection\fi}

\makeatletter
\DeclareFontFamily{OMX}{MnSymbolE}{}
\DeclareSymbolFont{MnLargeSymbols}{OMX}{MnSymbolE}{m}{n}
\SetSymbolFont{MnLargeSymbols}{bold}{OMX}{MnSymbolE}{b}{n}
\DeclareFontShape{OMX}{MnSymbolE}{m}{n}{
    <-6>  MnSymbolE5
   <6-7>  MnSymbolE6
   <7-8>  MnSymbolE7
   <8-9>  MnSymbolE8
   <9-10> MnSymbolE9
  <10-12> MnSymbolE10
  <12->   MnSymbolE12
}{}
\DeclareFontShape{OMX}{MnSymbolE}{b}{n}{
    <-6>  MnSymbolE-Bold5
   <6-7>  MnSymbolE-Bold6
   <7-8>  MnSymbolE-Bold7
   <8-9>  MnSymbolE-Bold8
   <9-10> MnSymbolE-Bold9
  <10-12> MnSymbolE-Bold10
  <12->   MnSymbolE-Bold12
}{}

\let\llangle\@undefined
\let\rrangle\@undefined
\DeclareMathDelimiter{\llangle}{\mathopen}%
                     {MnLargeSymbols}{'164}{MnLargeSymbols}{'164}
\DeclareMathDelimiter{\rrangle}{\mathclose}%
                     {MnLargeSymbols}{'171}{MnLargeSymbols}{'171}
\makeatother

\newtheorem{theorem}{Theorem}[chapter]
\newtheorem{lemma}[theorem]{Lemma}
\newtheorem{proposition}[theorem]{Proposition}
\newtheorem{corollary}[theorem]{Corollary}
\newtheorem{conjecture}[theorem]{Conjecture}

\theoremstyle{definition}
\newtheorem{definition}[theorem]{Definition}
\newtheorem{example}[theorem]{Example}

\newtheorem{convention}[theorem]{Convention}
\newtheorem{fact}[theorem]{Fact}
\newtheorem{notation}[theorem]{Notation}

\theoremstyle{remark}
\newtheorem{remark}[theorem]{Remark}

\numberwithin{section}{chapter}
\numberwithin{equation}{chapter}
\numberwithin{figure}{chapter}

\makeindex

\begin{document}

\frontmatter

\title{When the Lower Central Series Stops: \\ A Comprehensive Study for Braid Groups and Their Relatives}

\author{Jacques Darné}
\address{Institut de Recherche en Mathématique et en Physique, Chemin du Cyclotron 2, 1348 Ottignies-Louvain-la-Neuve, Belgique}
\email{jacques.darne@normalesup.org}
\thanks{This paper was written while the first author was employed by the FNRS as a \emph{Charg\'e de recherches}. The first and third authors were partially supported by the ANR Projects ChroK ANR-16-CE40-0003 and AlMaRe ANR-19-CE40-0001-01.}

\author{Martin Palmer}
\address{Institutul de Matematic\u{a} Simion Stoilow al Academiei Rom{\^a}ne, 21 Calea Grivi\cb{t}ei, 010702 Bucure\cb{s}ti, Romania}
\email{mpanghel@imar.ro}
\thanks{The second author was partially supported by a grant of the Romanian Ministry of Education and Research, CNCS - UEFISCDI, project number PN-III-P4-ID-PCE-2020-2798, within PNCDI~III}

\author{Arthur Soulié}
\address{Center for Geometry and Physics, Institute for Basic Science (IBS), 77 Cheongamro, Nam-gu, Pohang-si, Gyeongsangbuk-do, Korea 790-784 \& POSTECH, Gyeongsangbukdo, Korea}
\email{artsou@hotmail.fr, arthur.soulie@ibs.re.kr}
\thanks{The third author was partially supported by a Rankin-Sneddon Research Fellowship of the University of Glasgow and by the Institute for Basic Science IBS-R003-D1.}

\thanks{All three authors would like to thank the anonymous referee for very helpful comments and suggestions on an earlier version of this work.}

\date{February 16, 2022}



\begin{abstract}
Understanding the lower central series of a group is, in general, a difficult task. It is, however, a rewarding one: computing the lower central series and the associated Lie algebras of a group or of some of its subgroups can lead to a deep understanding of the underlying structure of that group. Our goal here is to showcase several techniques aimed at carrying out part of this task. In particular, we seek to answer the following question: when does the lower central series stop? We introduce a number of tools that we then apply to various groups related to braid groups: the braid groups themselves, surface braid groups, groups of virtual and welded braids, and \emph{partitioned} versions of all of these groups. The path from our general techniques to their application is far from being a straight one, and some astuteness and tenacity is required to deal with all of the cases encountered along the way. Nevertheless, we arrive at an answer to our question for each and every one of these groups, save for one family of partitioned braid groups on the projective plane. In several cases, we even compute completely the lower central series. Some results about the lower central series of Artin groups are also included.
\end{abstract}

\maketitle

\tableofcontents

\chapter*{Introduction}

One of the most basic objects one needs to understand when studying the structure of a group $G$ is its \emph{lower central series} (shortened to ``LCS") $G = \LCS_1(G) \supseteq \LCS_2(G) \supseteq \cdots$. Its behaviour varies greatly from one group to another. For instance, if $G$ is perfect (i.e.~all its elements can be written as products of commutators), its LCS is completely trivial; this holds for instance for mapping class groups of closed surfaces of genus $g\geq3$ \cite[Thm~5.1]{Korkmaz_low}. On the contrary, if $G$ is nilpotent, or residually nilpotent, its LCS contains deep information about the structure of~$G$; examples of residually nilpotent groups include free groups~\cite[Chap.~5]{MagnusKarrassSolitar}, pure braid groups~\cite{FalkRandell1,FalkRandell2}, pure braid groups on surfaces~\cite{BellingeriBardakov2, BellingeriGervais}, pure welded braid groups~\cite[\Spar 5.5]{BerceanuPapadima}, and conjecturally pure virtual braid groups~\cite{BardakovMikhailovVershininWu}.
The LCS is also deeply connected to the structure of the group ring of $G$. In particular, Quillen \cite{Quillen} proved that if we consider the filtration of the group ring $\mathbb Q G$ by the powers of its augmentation ideal, then the associated graded algebra is isomorphic to the universal enveloping algebra of the Lie algebra $\Lie(G) \otimes \mathbb Q$, where $\Lie(G)$ is the graded Lie ring obtained from $\LCS_*(G)$.

The amount of information one can hope to extract from the study of a LCS depends in the first place on whether or not it \emph{stops} in the following sense:
\begin{definition}
The LCS of a group $G$ is said to \emph{stop} if there exists an integer $i \geq 1$ such that $\LCS_i(G)=\LCS_{i+1}(G)$. We say that it \emph{stops at $\LCS_i$} if $i$ is the smallest integer for which this holds. Otherwise, we say the LCS \emph{does not stop} or else that it \emph{stops at $\infty$}.
\end{definition}
It follows from the definition of the LCS (recalled in~\Spar\ref{sec_def_LCS} below) that if $\LCS_i(G)=\LCS_{i+1}(G)$ for some $i \geq 1$, then $\LCS_k(G)=\LCS_{k+1}(G)$ for all $k \geq i$, whence our choice of terminology.

\subsection*{Partitioned braid groups.}
In this memoir, we study the LCS of the following families of groups, and their \emph{partitioned} versions, in the sense described below:
\begin{itemize}
     \item The Artin braid group $\B_n$,
     \item The virtual braid group $\vB_n$,
     \item The welded braid group $\wB_n$,
     \item The extended welded braid group $\exwB_n$,
     \item The group $\B_n(S)$ of braids on any surface $S$. 
\end{itemize} 

Let $G_n$ denote one of the above groups. In each case, there is a notion of the underlying permutation of an element of $G_n$, corresponding to a canonical surjection $\pi \colon G_n \twoheadrightarrow \Sym_n$ to the symmetric group, from which we can define \emph{partitioned} versions of $G_n$. Let us first fix our conventions concerning partitions of integers:
\begin{definition}\label{def_partitions}
Let $n \geq 1$ be an integer. A \emph{partition of $n$} is an $l$-tuple $\lambda = (n_1, \ldots , n_l)$ of integers $n_i \geq 1$, for some $l \geq 1$ called the \emph{length} of $\lambda$, such that $n$ is the sum of the $n_i$. Given such a $\lambda$, for $j \leq l$, let us define $t_j := \sum_{i \leq j} n_i$, including $t_0 = 0$. Then the set $b_j(\lambda) := \{t_{j-1}+1, \ldots , t_j\}$ is referred to as the \emph{$j$-th block} of $\lambda$, and $n_i$ is called the \emph{size} of the $i$-th block.
\end{definition}

For $\lambda = (n_1, \ldots , n_l)$ a partition of $n$, we consider the preimage
\[G_\lambda := \pi^{-1}(\Sym_{\lambda}) = \pi^{-1}\left(\Sym_{n_1} \times \cdots \times \Sym_{n_l}\right),\]
which is called the $\lambda$-\emph{partitioned} version of $G_n$. There are two extremal situations: the trivial partition $\lambda = (n)$ simply gives the group $G_n$, whereas the discrete partition $\lambda = (1,1, \ldots ,1)$ corresponds to the subgroup of \emph{pure} braids in $G_n$. 

As we will see later on, for all the families of groups described above, the LCS of $G_n$ stops at $\LCS_2$ or $\LCS_3$ (when $n$ is at least $3$ or $4$, depending on the context), whereas the LCS of the subgroup of pure braids is a very complex object (in particular, it does not stop, when $n$ is at least $2$ or $3$). We can thus expect the partitioned braid groups $G_\lambda$ to display a range of intermediate behaviours when $\lambda$ varies, and this is indeed what we observe.

\subsection*{Methods.} A fundamental tool in the study of LCS is the graded Lie ring structure on the associated graded $\Lie(G) := {\bigoplus}_{i\geq1} \LCS_i(G)/ \LCS_{i+1}(G)$. Namely, this is a graded abelian group endowed with a Lie bracket induced by commutators in $G$. It is always generated, as a Lie algebra over $\Z$, by its degree one piece, which is the abelianisation $G^{\ab} = G/ \LCS_2(G)$. This often allows one to use \emph{disjoint support arguments} to show that the LCS stops, when it does. Precisely, if one can show that pairs of generators of $G^{\ab}$ have commuting representatives in $G$ (which is the case if they have representatives whose \emph{supports} are disjoint, for a certain notion of support), then, by definition of the Lie bracket, they commute in $\Lie(G)$. In this case, $\Lie(G)$ is abelian, and it is generated by $G^{\ab}$, which means that it is reduced to $G^{\ab}$. In turn, that means that $\LCS_i(G) = \LCS_{i+1}(G)$ whenever $i \geq 2$. This kind of argument is used throughout the memoir. In particular, one can apply it readily to each of the $G_n$ above to show that its LCS stops at $\LCS_2$ whenever $n$ is at least $3$ or $4$ (depending on the amount of space needed to have representatives with disjoint support for pairs of generators of $G^{\ab}$). 

Depending on the groups considered, \emph{having disjoint support} will have a different meaning, but it will always imply that the elements considered commute. The most obvious definition is for mapping classes, which have disjoint support if they can be realised by homeomorphisms having disjoint support in the usual sense. This can be applied to usual braids, seen as mapping classes of the punctured disc. For braids on surfaces, one can generalise this by using a very similar idea: two such braids have disjoint support if they can be realised as geometric braids that do not move the same strands, and move them in disjoint regions of the surface. Equivalently, we view the surface braid group as a subgroup of the mapping class group of the punctured surface. The same definition can be used for welded braids, which can actually be seen as mapping classes (see~\cite{Damiani}). However, for the latter, it will be convenient to use another notion of support relying on their interpretation as automorphisms of free groups; with this point of view, the support is the list of generators involved in the definition of the automorphism (see Definition~\ref{d:support-autFn}). Finally, for virtual braids, we will need a diagrammatic definition (Definition~\ref{d:support-diagram}), listing the strands that really interact with the others in a diagram.

One other main line of argumentation for studying LCS is given by \emph{looking for quotients whose LCS is well-understood}. Namely, if we can find a quotient of $G$ whose LCS does not stop, then neither does the LCS of $G$; see Lemma~\ref{lem:stationary_quotient}. Typically, we look for a quotient that is a semi-direct product of an abelian group with $\Z$ or $\Z/2$, a free product of abelian groups or a wreath product of an abelian group with some $\Sym_\lambda$, whose LCS can be computed completely; see Appendix~\ref{sec:appendix_2}.

Finally, a very important tool in our analysis is the study of the \emph{quotient by the residue}. Precisely, if we denote by $\LCS_\infty(G)$ (abbreviated $\LCS_\infty$ when the context is clear) the intersection of the $\LCS_i(G)$, the LCS of $G$ is ``the same'' as the LCS of $G/\LCS_\infty$: each $\LCS_i(G)$ is the preimage of $\LCS_i(G/\LCS_\infty)$ by the canonical projection, and this projection induces an isomorphism between $\Lie(G)$ and $\Lie(G/\LCS_\infty)$. In particular, one of $\LCS_*(G)$ and $\LCS_*(G / \LCS_\infty)$ stops if and only if the other does, which happens exactly when $G/\LCS_\infty$ is nilpotent. Considering $G/\LCS_\infty$ instead of $G$ can lead to very important simplifications. Let us illustrate this by an example, variations of which are used throughout the memoir. We know that $\LCS_\infty = \LCS_2$ for $\B_n$, and that $\LCS_2$ contains the elements $\sigma_i \sigma_j^{-1}$ of $\B_n$. Thus whenever we have a morphism $\B_n \to G$, the subgroup $\LCS_\infty(G)$ must contain the image of $\LCS_\infty(\B_n)$, which contains the images of $\sigma_i \sigma_j^{-1}$, so all the $\sigma_i$ have the same image in $G/\LCS_\infty$.

\subsection*{Results.} Does the LCS stop? We give a complete answer to this question for all of the families of groups listed above, with the single exception of $\B_{2,m}(\Proj)$ with $m \geq 3$ (see Conjecture~\ref{conj:B2m}). We also obtain some amount of information on the associated Lie rings. In particular, we compute completely the Lie rings of partitioned braid groups on surfaces in the stable case. Besides their intrinsic value, these results have several applications, notably to the representation theory of braid groups and their relatives. See for instance \cite{BellingeriGodelleGuaschi} and the work of the second and third authors \cite{PSI}, where the knowledge of the structure of such LCS is key in the construction and the study of representations of these groups using homological approaches. Furthermore, let us mention that one can see surface braid groups and their LCS as invariants of the surfaces themselves; as an application of this point of view, we recover the Riemann-Hurwitz formula for coverings of closed surfaces in Remark~\ref{rmk:Riemann-Hurwitz_formula}.  

In addition to the families of groups introduced above, we also find out when the LCS stops for what we call the \emph{tripartite} welded braid groups $\wB(n_P, n_{S_+}, n_S)$, which are the fundamental groups of the configuration spaces of $n_P$ points, $n_{S_+}$ oriented circles and $n_S$ unoriented circles in $3$-space, where all the circles are unlinked and unknotted. This is a generalisation of both $\wB_n = \wB(0, n, 0)$ and $\exwB_n = \wB(0,0,n)$. We define partitioned versions $\wB(\lambda_P, \lambda_{S_+}, \lambda_S)$ of these groups in the obvious way; see Definition~\ref{def_partitioned_tripartite}.
These groups are not only natural geometrical generalisations (obtained by mixing oriented circles, non-oriented circles and moving points) of welded-type braid groups, they also have key applications for the homological constructions of representations of the groups $\wB_n$ and $\exwB_n$ introduced by \cite{PSI}. In particular, the extension of the Burau representation to welded braid groups introduced by Vershinin~\cite{Vershinin} (by assigning explicit matrices to generators) may alternatively be defined by the methods of \cite{PSI} considering the abelianisation of $\wB(1,n,0)$; see \cite{PS0}.

\medskip

We summarise our results in three tables on pages \pageref{table:stable}--\pageref{table:unstable-non-stopping}. These are organised as follows:

\begin{itemize}
\item In Table \ref{table:stable} are gathered the \emph{stable} cases, which are those where the blocks of the partitions are large enough for the \emph{disjoint support argument} described above to be applied readily. 
\item In Table \ref{table:unstable-stopping} are gathered the cases where there are blocks in the partitions which are too small for the disjoint support argument to be applied readily, but not too many of them, so that the LCS still stops. 
\item In Table \ref{table:unstable-non-stopping} are gathered the cases where the LCS does not stop.
\end{itemize}

Some of our results have already been obtained in the literature, with different methods. Namely, the question of whether or not the lower central series stops has already been studied:
\begin{itemize}
    \item by Gorin and Lin \cite{GorinLin} for $\B_{n}$ and Kohno \cite{Kohno} for the pure braid group $\PB_n = \B_{1,\ldots,1}$, which is moreover known to be residually nilpotent by Falk and Randell \cite{FalkRandell1,FalkRandell2}.
    \item by Bellingeri, Gervais and Guaschi \cite{BellingeriGervaisGuaschi} for $\B_{n}(S)$ where $S$ is a compact, connected, orientable surface with or without boundary.
    \item by Bellingeri and Gervais \cite{BellingeriGervais} for the pure surface braid group $\PB_n(S)$ where $S$ is a compact, connected, non-orientable surface with or without boundary and different from the projective plane $\Proj$.
    \item by Gonçalves and Guaschi \cite{DacibergGuaschiII, DacibergGuaschiI} for $\B_{n}(\mathbb{S}^{2})$ and $\B_{n}(\mathbb{S}^{2}-\mathscr{P})$ where $\mathscr{P}$ is a finite set of points in $\mathbb{S}^{2}$.
    \item by Guaschi and de Miranda e Pereiro \cite{GuaschiPereiro} for $\B_{n}(S)$ where $S$ is a compact, connected, non-orientable surface without boundary.
    \item by van Buskirk \cite{vanBuskirk1966} and by Gonçalves and Guaschi \cite{GoncalvesGuaschiprojective,DacibergGuaschiIII,GoncalvesGuaschisphereFDprojective}, both for the braid group on the projective plane $\B_n(\Proj)$.
    \item by Bardakov and Bellingeri \cite{BardakovBellingeri2009} for the virtual braid group $\vB_n$. The question of whether the pure virtual braid group $\vP_{n}:=\vB_{1,\ldots,1}$ is residually (torsion-free) nilpotent or not for $n\geq4$ remains an open problem, while this property is proven for $\vP_{3}$ by Bardakov, Mikhailov, Vershinin and Wu \cite{BardakovMikhailovVershininWu}. That the pure welded braid group $\wP_{n}:=\wB_{1,\ldots,1}$ is residually (torsion free) nilpotent is proven by Berceanu and Papadima \cite[\Spar 5.5]{BerceanuPapadima}. Results on the lower central series of the classical, virtual and welded pure braid groups and their relatives are collected in the survey paper of Suciu and Wang \cite{SuciuWang}.
\end{itemize}

\subsection*{Notation in the tables.}
The letter $\lambda = (n_1, \ldots , n_l)$ denotes a partition of $n$ of length $l \geq 1$. The letter $\mu$ denotes a partition that is either empty or whose blocks have size at least $3$ for classical braids and surface braids, and at least $4$ for virtual and welded braids. On the other hand, $\nu$ denotes any partition (possibly empty, unless stated otherwise).

In Table \ref{table:unstable-stopping}, the function $f$ is defined by $f(m) = \mathrm{max}\{ v_2(m) , 1 \}$, where $v_2$ is the $2$-adic valuation. The number $\epsilon$ is either $0$ or $1$ (the precise value may depend on the case, in particular on $m$, and is unknown, although we conjecture that it is always $1$ for $m$ even and $0$ for $m$ odd, so that $f(m) + \epsilon = v_2(m) + 1$ in all cases).

The letter $S$ denotes any connected surface (not necessarily compact nor orientable, and possibly with boundary). Six exceptional surfaces are mentioned in the tables, denoted by $\D$ (the disc), $\D-pt$ (the disc minus an interior point), $\Tor$ (the torus), $\Moeb$ (the M{\"o}bius strip), $\S^2$ (the $2$-sphere) and $\Proj$ (the projective plane). A surface $S$ is called \emph{generic} if it is not one of these six exceptional surfaces.

The symbol $(\dagger)$ in front of a family of groups indicates that the result concerned is already partly known in the literature quoted above.

\begin{table}[h]
\centering
\bgroup
\def\arraystretch{1.4}
\begin{tabular}{!{\vrule width 1.5pt}c|c|c||c|c|c!{\vrule width 1.5pt}}
\noalign{\hrule height 1.5pt}

\multicolumn{6}{!{\vrule width 1.5pt}c!{\vrule width 1.5pt}}{\rule{0pt}{1.5em}{\Large The stable cases}} \\
\noalign{\hrule height 1.5pt}

\multicolumn{2}{!{\vrule width 1.5pt}c|}{{\large Family of groups}} & {\large Partition} & {\makecell{\large Stops \\ at $\LCS_k$}} & {\large Ref.} & {\makecell{\large Lie \\ Alg.}} \\
\noalign{\hrule height 1.5pt}

\multicolumn{2}{!{\vrule width 1.5pt}c|}{Classical braids $\B_\lambda$} & $n_i \geq 3$ $(\dagger)$ & $k=2$ & \ref{thm:partitioned-braids} {\small (\ref{LCS_stable_partitioned_braids})} & \ref{partitioned_B^ab} \\
\noalign{\hrule height 1.5pt}

\multirow{4}{*}{\makecell{Surface \\ braids \\ $\B_\lambda(S)$}} & $S \subseteq \S^2$ & $n_i \geq 3$ $(\dagger)$ & $k=2$ & \multirow{4}{*}{\ref{Lie_ring_partitioned_B(S)}} & \ref{partitioned_B(S)^ab} \\
\cline{2-4} \cline{6-6}

& $S \nsubseteq \S^2$, orientable & $n_i \geq 3$ $(\dagger)$ & $k = 3$ && \Spar\ref{sec:Lie_ring_partitioned_braid} \\
\cline{2-4} \cline{6-6}

& \multirow{2}{*}{$S$ non-orientable} & $l=1,\ n_1 \geq 3$ & $k=2$ && \ref{Bn(S)^ab} \\
\cline{3-4} \cline{6-6}

&& $l\geq 2,\ n_i \geq 3$ & $k=3$ && \Spar\ref{sec:Lie_ring_partitioned_braid} \\
\noalign{\hrule height 1.5pt}

\multicolumn{2}{!{\vrule width 1.5pt}c|}{Virtual braids $\vB_\lambda$} & \multirow{3}{*}{$n_i \geq 4$} & \multirow{3}{*}{$k=2$} & \multirow{3}{*}{\ref{LCS_of_partitioned_v(w)B}} & \multirow{3}{*}{\ref{partitioned_v(w)B^ab}} \\
\cline{1-2}

\multicolumn{2}{!{\vrule width 1.5pt}c|}{Welded braids $\wB_\lambda$} &&&& \\
\cline{1-2}

\multicolumn{2}{!{\vrule width 1.5pt}c|}{Ext.\ welded braids $\exwB_\lambda$} &&&& \\
\cline{1-6}

\multicolumn{2}{!{\vrule width 1.5pt}c|}{\makecell{Tripartite welded braids \\ $\wB(\lambda_P, \lambda_{S_+}, \lambda_S)$}} & \rule{0pt}{21pt} $\begin{cases} n_{i,P} \geq 3 \\ n_{i,S_{+}}, n_{i,S} \geq 4 \end{cases}$ \rule[-15pt]{0pt}{0pt} & $k=2$ & \ref{LCS_of_partitioned_exwBn_2} {\small (\ref{stable_LCS_of_partitioned_exwBn_2})} & \ref{tripartite_wB^ab} \\
\noalign{\hrule height 1.5pt}

\end{tabular}
\egroup
\caption{The stable cases.}
\label{table:stable}
\addcontentsline{toc}{section}{Table \ref{table:stable}: The stable cases}
\end{table}

\begin{landscape}
\begin{table}[p]
\centering
\bgroup
\def\arraystretch{1.4}
\begin{tabular}{!{\vrule width 1.5pt}c|c|c||c|c|c!{\vrule width 1.5pt}}
\noalign{\hrule height 1.5pt}

\multicolumn{6}{!{\vrule width 1.5pt}c!{\vrule width 1.5pt}}{\rule{0pt}{1.5em} {\Large The unstable cases for which the LCS stops}} \\
\noalign{\hrule height 1.5pt}

\multicolumn{2}{!{\vrule width 1.5pt}c|}{{\large Family of groups}} & {\large Partition} & {\large Stops at $\LCS_k$} & {\large Ref.} & {\large Lie Alg.} \\
\noalign{\hrule height 1.5pt}

\multicolumn{2}{!{\vrule width 1.5pt}c|}{\multirow{2}{*}{Classical braids $\B_\lambda$}} & $(2)$ & \multirow{2}{*}{$k=2$} & \multicolumn{2}{c!{\vrule width 1.5pt}}{$\B_2 \cong \Z$} \\
\cline{3-3} \cline{5-6}

\multicolumn{2}{!{\vrule width 1.5pt}c|}{} & $(1, \mu)$, $(1, 1, \mu)$ && \ref{thm:partitioned-braids} {\small (\ref{LCS_B1mu}, \ref{LCS_B11mu})} & \ref{partitioned_B^ab} \\
\noalign{\hrule height 1.5pt}

\multirow{11}{*}{\makecell{Surface braids \\ $\B_\lambda(S)$}} & $S = \D-pt$ & $(1, \mu)$ & $k=2$ & \ref{LCS_B1nu(D-pt)} & \ref{partitioned_B^ab} \\
\cline{2-6}

& \multirow{2}{*}{$S = \Tor$} & $(1)$ & $k=2$ & \multicolumn{2}{c!{\vrule width 1.5pt}}{$\B_1(\Tor) \cong \Z^2$} \\
\cline{3-6}

&& $(1, \mu)$, $\mu \neq \varnothing$ & $k=3$ & \ref{thm:braid_torus} & \ref{partitioned_B(S)^ab} and \ref{partitioned_Lie(B(Sg)} \\
\cline{2-6}

& $S = \Moeb$ & $(1)$ & $k=2$ & \multicolumn{2}{c!{\vrule width 1.5pt}}{$\B_1(\Moeb) \cong \Z$} \\
\cline{2-6}

& \multirow{3}{*}{$S = \S^2$} & $(2)$ or $(2,1)$ & \multirow{2}{*}{$k=2$} & \multicolumn{2}{c!{\vrule width 1.5pt}}{$\B_2(\S^2) \cong \Z/2$, $\B_{2,1}(\S^2) \cong \Z/4$} \\
\cline{3-3} \cline{5-6}

&& $(1, \mu)$, $(1, 1, \mu)$, $(1, 1, 1, \mu)$ && \ref{LCS_B_lambda(S2)} & \ref{partitioned_B(S)^ab} \\
\cline{3-6}

&& $(2, m)$, $m \geq 3$ & $k = f(m)+ 1 + \epsilon$ & \ref{LCS_B2m(S2)} & -- \\
\cline{2-6}

& \multirow{4}{*}{$S = \Proj$} & $(1, m)$, $m \geq 3$ & $k = f(m) + 2 +\epsilon$ & \ref{LCS_B1m(P2)} & -- \\
\cline{3-6}

&& $(1)$ & $k=2$ & \multicolumn{2}{c!{\vrule width 1.5pt}}{$\B_1(\Proj) \cong \Z/2$} \\
\cline{3-6}

&& $(1,1)$ & $k=3$ & \multicolumn{2}{c!{\vrule width 1.5pt}}{$\B_{1,1}(\Proj) \cong Q_8$ (\ref{B2(Proj)})} \\
\cline{3-6}

&& $(2)$ & $k=4$ & \multicolumn{2}{c!{\vrule width 1.5pt}}{$\B_2(\Proj) \cong Dic_{16}$ (\ref{B2(Proj)})} \\
\noalign{\hrule height 1.5pt}

\multicolumn{2}{!{\vrule width 1.5pt}c|}{Virtual braids $\vB_\lambda$} & \multirow{2}{*}{$(1, \mu)$} & \multirow{2}{*}{$k=2$} & \multirow{2}{*}{\ref{LCS_of_partitioned_v(w)B} {\small (\ref{LCS_vB_1mu})}} & \multirow{2}{*}{\ref{partitioned_v(w)B^ab}} \\
\cline{1-2}

\multicolumn{2}{!{\vrule width 1.5pt}c|}{Welded braids $\wB_\lambda$} &&&& \\
\cline{1-6}

\multicolumn{2}{!{\vrule width 1.5pt}c|}{\multirow{3}{*}{\makecell{Tripartite welded braids \\ $\wB(\lambda_P, \lambda_{S_+}, \lambda_S)$}}} &$((1, \ldots , 1, \mu_P), \mu_{S_+}, \mu_{S})$ & \multirow{2}{*}{$k=2$} & \ref{LCS_of_partitioned_exwBn_2} {\small (\ref{tripartite_wB_isolated_P})} & \multirow{2}{*}{\ref{tripartite_wB^ab}} \\
\cline{3-3} \cline{5-5}

\multicolumn{2}{!{\vrule width 1.5pt}c|}{} & $((1, \ldots, 1, \mu_P), (1, \mu_{S_+}), \mu_S)$ && \ref{LCS_of_partitioned_exwBn_2} {\small (\ref{tripartite_wB_isolated_S+}, \ref{tripartite_wB_isolated_P_and_S+})} & \\
\cline{3-6}

\multicolumn{2}{!{\vrule width 1.5pt}c|}{} & $((2, \nu_P), \varnothing, \mu_S)$, $\mu_S \neq \varnothing$ & $k=3$ & \ref{LCS_of_partitioned_exwBn_2} {\small (\ref{prop:welded_points_block_size_two})} & \ref{tripartite_wB^ab} and \ref{prop:welded_points_block_size_two} \\
\noalign{\hrule height 1.5pt}

\end{tabular}
\egroup
\caption{The unstable cases for which the LCS stops.}
\label{table:unstable-stopping}
\addcontentsline{toc}{section}{Table \ref{table:unstable-stopping}: The unstable cases for which the LCS stops}
\end{table}
\end{landscape}

\begin{table}[b]
\centering
\bgroup
\def\arraystretch{1.4}
\begin{tabular}{!{\vrule width 1.5pt}c|c|c||c!{\vrule width 1.5pt}}
\noalign{\hrule height 1.5pt}

\multicolumn{4}{!{\vrule width 1.5pt}c!{\vrule width 1.5pt}}{\rule{0pt}{1.5em} {\Large The unstable cases for which the LCS does not stop}} \\
\noalign{\hrule height 1.5pt}

\multicolumn{2}{!{\vrule width 1.5pt}c|}{{\large Family of groups}} & {\large Partition} & {\large Ref.} \\
\noalign{\hrule height 1.5pt}

\multicolumn{2}{!{\vrule width 1.5pt}c|}{\multirow{2}{*}{Classical braids $\B_\lambda$}} & $(1, 1, 1, \nu)$ $(\dagger)$ & \ref{thm:partitioned-braids} {\small (\ref{LCS_B111})} \\
\cline{3-4}

\multicolumn{2}{!{\vrule width 1.5pt}c|}{} & $(2, \nu)$, $l \geq 2$ & \ref{thm:partitioned-braids} {\small (\ref{LCS_B2mu}, \ref{LCS_B22mu}, \ref{LCS_B12mu})} \\
\noalign{\hrule height 1.5pt}

\multirow{11}{*}{\makecell{Surface \\ braids \\ $\B_\lambda(S)$}} & $S$ generic & $(1, \nu)$, $(2, \nu)$ $(\dagger)$ & \ref{LCS_unstable_partitioned_surface_braids} \\
\cline{2-4}

& \multirow{2}{*}{$S = \D-pt$} & $(1,1, \nu)$ & \ref{LCS_B1nu(D-pt)} \\
\cline{3-4}

&& $(2,\nu)$ & \ref{LCS_unstable_partitioned_surface_braids} \\
\cline{2-4}

& \multirow{2}{*}{$S = \Tor$} & $(1,1, \nu)$ & \ref{thm:braid_torus} \\
\cline{3-4}

&& $(2,\nu)$ $(\dagger)$ & \ref{thm:braid_torus} {\small (\ref{LCS_unstable_partitioned_surface_braids})} \\
\cline{2-4}

& \multirow{2}{*}{$S = \Moeb$} & $(1, \nu)$, $l \geq 2$ & \ref{LCS_unstable_partitioned_surface_braids}, \ref{LCS_B11(M)}, \ref{LCS_B1m(Moeb)} \\
\cline{3-4}

&& $(2,\nu)$ & \ref{LCS_unstable_partitioned_surface_braids} \\
\cline{2-4}

& \multirow{2}{*}{$S = \S^2$} & $(1,1,1,1, \nu)$ & \ref{LCS_B_lambda(S2)} {\small (\ref{LCS_P4(S2)})} \\
\cline{3-4}

&& $(2, \nu)$, $l \geq 3$ or $(2,2)$ & \ref{LCS_B_lambda(S2)} {\small (\ref{LCS_B2mu(S2)}, \ref{LCS_B22(S2)})} \\
\cline{2-4}

& \multirow{2}{*}{$S = \Proj$} & $(1, \nu)$, $l \geq 3$ & \ref{LCS_B1mu(P2)} \\
\cline{3-4}

&& $(2, \nu)$, $l \geq 3$ or $(2,2)$ or $(2,1)$ & \ref{LCS_B2mu(P2)}, \ref{LCS_B22(P2)}, \ref{LCS_B12(P2)} \\
\noalign{\hrule height 1.5pt}

\multicolumn{2}{!{\vrule width 1.5pt}c|}{Virtual braids $\vB_\lambda$} & \multirow{2}{*}{$(1, 1, \nu)$, $(2, \nu)$, $(3, \nu)$} & \multirow{3}{*}{\ref{LCS_of_partitioned_v(w)B}} \\
\cline{1-2}

\multicolumn{2}{!{\vrule width 1.5pt}c|}{Welded braids $\wB_\lambda$} && \\
\cline{1-3}

\multicolumn{2}{!{\vrule width 1.5pt}c|}{Ext.\ w.\ braids $\exwB_\lambda$} & [$(1, \nu)$ with $\nu\neq\varnothing$], $(2, \nu)$, $(3, \nu)$ & \\
\cline{1-4}

\multicolumn{2}{!{\vrule width 1.5pt}c|}{\multirow{5}{*}{\makecell{Tripartite \\ welded braids \\ $\wB(\lambda_P, \lambda_{S_+}, \lambda_S)$}}} & $(\nu_P, \nu_{S_+}, (2,\nu_S)), (\nu_P, \nu_{S_+}, (3,\nu_S))$ & \multirow{3}{*}{\ref{LCS_of_partitioned_exwBn_2}} \\
\cline{3-3}

\multicolumn{2}{!{\vrule width 1.5pt}c|}{} & $(\nu_P, (2,\nu_{S_+}), \nu_S), (\nu_P, (3,\nu_{S_+}), \nu_S)$ & \\
\cline{3-3}

\multicolumn{2}{!{\vrule width 1.5pt}c|}{} & $(\nu_P, (1,1,\nu_{S_+}), \nu_S)$ & \\
\cline{3-4}

\multicolumn{2}{!{\vrule width 1.5pt}c|}{} & $(\nu_P, \nu_{S_+}, (1,\nu_S))$, $l\geq 2$ & \ref{LCS_of_partitioned_exwBn_2} {\small (\ref{tripartite_wB_isolated_S})} \\
\cline{3-4}

\multicolumn{2}{!{\vrule width 1.5pt}c|}{} & $((2,\nu_P), \nu_{S_+}, \nu_S)$, $\nu_{S_+} \neq \varnothing$ & \ref{LCS_of_partitioned_exwBn_2} {\small (\ref{tripartite_wB_P2_and_S+})} \\
\noalign{\hrule height 1.5pt}

\end{tabular}
\egroup
\caption{The unstable cases for which the LCS does not stop.}
\label{table:unstable-non-stopping}
\addcontentsline{toc}{section}{Table \ref{table:unstable-non-stopping}: The unstable cases for which the LCS does not stop}
\end{table}

\mainmatter

\chapter{General recollections}
\label{sec:general_recollections}

In this chapter, we recall some classical notions and tools to study the lower central series of groups. These will be used throughout the memoir.

\section{Commutator calculus and lower central series}\label{sec_def_LCS}

Let $G$ be a group. Recall that the \emph{lower central series} (LCS) of $G$ is the descending sequence of normal subgroups $G = \LCS_1(G) \supseteq \LCS_2(G) \supseteq \cdots$ of $G$, also denoted by $\LCS_*(G)$, defined by
\[\LCS_i(G):=\begin{cases}
G & \textrm{if \ensuremath{i=1},}\\
\left[G,\LCS_{i-1}(G)\right] & \textrm{if \ensuremath{i\geq2},}
\end{cases}\]
where $\left[G,\LCS_i(G)\right]$ is the subgroup of $G$ generated by all commutators $\left[x,y\right]:=xyx^{-1}y^{-1}$ with $x$ in $G$ and $y$ in $\LCS_{i-1}(G)$. The subgroups $\LCS_i(G)$ are fully invariant, and in particular normal in $G$. As a consequence, one can also think of the LCS as an ascending chain of quotients $G/\LCS_i(G)$ of $G$. Recall that the abelianisation $G^\mathrm{\ab}$ of $G$ is the first of these quotients, namely $G/\LCS_{2}(G)$. In general, $G/\LCS_{c+1}(G)$ is the universal $c$-nilpotent quotient of $G$ (recall that $G$ is called $c$-nilpotent if $\Gamma_{c+1}(G) = \{1\}$). The group $G$ is called \emph{residually nilpotent} if its \emph{residue} $\LCS_\infty(G) := \bigcap \LCS_i(G)$ is equal to $\{1\}$. The quotient $G/\LCS_\infty(G)$ is the universal (and, in particular, the largest) residually nilpotent quotient of $G$.

The following lemma is folklore and will be used very often in the sequel. In particular, we will often use its contrapositive: if a group $G$ has a quotient whose LCS does not stop, then the LCS of $G$ does not stop either.

\begin{lemma}\label{lem:stationary_quotient}
Let $H$ be a quotient of $G$. If $\LCS_i(G) = \LCS_{i+1}(G)$ for some $i$, then $\LCS_i(H) = \LCS_{i+1}(H)$.
\end{lemma}

\begin{proof}
For all $k \geq 1$, it follows from the definition of the LCS that $\LCS_k H = \pi(\LCS_k G)$. As a consequence, $\LCS_{i+1} H = \LCS_i H$ whenever $\LCS_{i+1} G = \LCS_i G$.
\end{proof}

We now give two partial converses to Lemma~\ref{lem:stationary_quotient}. The first one (also the most obvious one) is the case of a quotient by a subgroup having some finiteness properties. Since most of the extensions that we consider in the sequel will \emph{not} satisfy the required hypothesis, we will use it only once, in the very simple case when the kernel is cyclic of order two. Nevertheless, we state it in a general framework:
\begin{lemma}\label{Quotient_by_finite_subgroup}
Let $K \hookrightarrow G \twoheadrightarrow H$ be a short exact sequence of groups. Suppose that there exists $l \geq 0$ such that every strictly decreasing central filtration of $K$ stops after at most $l$ steps, that is,  if $K = K_1 \supsetneq  K_2 \supsetneq  \cdots  \supsetneq K_m$ is a nested sequence of subgroups satisfying $[K, K_i] \subset K_{i+1}$, then $m \leq l$. Suppose moreover that for some $i \geq 1$, we have $\LCS_{i+1} H = \LCS_i H$. Then $\LCS_{i+l+1} G = \LCS_{i+l} G$.
\end{lemma}

\begin{proof}
The filtration $K \cap \LCS_*(G)$ is a central filtration of $K$, so it can strictly decrease only $l$ times. If $\LCS_{i+1} H = \LCS_i H$, then $\LCS_i(H) = \LCS_{i+k}(H)$ is the image of $\LCS_{i+k}(G)$ in $H$, for all $k \geq 0$. Recall that if $L$ and $M$ are subgroups of $G$ such that $L \subseteq M$, then $L$ and $M$ are equal if and only if their image in $H$ and their intersection with $K$ are equal. As a consequence, for $\LCS_{i+k}(G)$ to decrease when $k$ grows, its intersection with the kernel must decrease, which can happen at most $l$ times. So the LCS of $G$ must stop at most at $\LCS_{i+l}$.
\end{proof}

\begin{example}
If $K$ is finite, such an $l$ clearly exists. In fact, for any central filtration $K_*$ on $K$, since the cardinal of $K/K_m$ is the product of the cardinals of the $K_i/K_{i+1}$, one can take $l$ to be the number of prime factors in the cardinal of $K$ (a bound that is optimal if $K$ is abelian). In particular, we will apply this with $K \cong \Z/2$ and $l = 1$ in the proof of Proposition~\ref{LCS_B1m(P2)}.
\end{example}

\begin{remark}
Finite groups are not the only ones for which the hypothesis holds. We could for instance apply the Lemma  with $K$ simple, or more generally with $K$ perfect (with $l = 0$). Also, the class of groups $K$ satisfying this hypothesis is stable by extensions. In fact, an equivalent way of stating it is to ask that the maximal residually nilpotent quotient of $K$ (that is, $K/\LCS_\infty(K)$) is finite.
\end{remark}

The other partial converse to Lemma~\ref{lem:stationary_quotient} that we will use concerns quotients by central subgroups. This case is a bit more subtle, and requires the following result, which can be useful when calculating quotients by residues.

\begin{proposition}\label{Extensions_and_residues}
Let $G$ be a group, and let $N$ be a normal subgroup of $G$. Suppose that for some $i \geq 2$, $N \cap \LCS_i(G) = \{1\}$. Then the canonical morphism $\LCS_\infty(G) \rightarrow  \LCS_\infty(G/N)$ is an isomorphism. In particular, $G$ is residually nilpotent if and only if $G/N$ is. Moreover, we have a short exact sequence:
\[N \hookrightarrow G/\LCS_\infty \twoheadrightarrow (G/N)/\LCS_\infty.\]
\end{proposition}

\begin{proof}
Let $\pi \colon G \twoheadrightarrow G/N$ be the canonical projection. Since $N \cap \LCS_\infty(G) = \{1\}$, the induced morphism $\LCS_\infty(G) \rightarrow  \LCS_\infty(G/N)$ is injective. Let us show that it is surjective. Let $y \in \LCS_\infty(G/N)$. Since $\LCS_k (G/N) = \pi(\LCS_k G)$ by definition of the LCS, there is, for each $k \geq 1$, some $x_k \in \LCS_k (G)$ such that $\pi(x_k) = y$. Then $x_k x_{k+1}^{-1} \in N \cap \LCS_k (G)$, which implies that $x_k x_{k+1}^{-1} = 1$ whenever $k \geq i$. Thus the sequence $(x_k)$ is stationary at $x := x_i$, which must be in $\LCS_\infty(G)$ by definition of the $x_k$, and is sent to $y$ by $\pi$. This proves the first part of the Proposition.
Now, let us consider the commutative diagram of groups:
\[\begin{tikzcd}
1 \ar[r, hook] \ar[d, hook] & \LCS_\infty(G) \ar[r, "\cong"] \ar[d, hook] & \LCS_\infty(G/N) \ar[d, hook] \\
N \ar[r, hook] \ar[d, "="] & G \ar[r, two heads] \ar[d, two heads] & G/N \ar[d, two heads] \\
N \ar[r, dashed] & G/\LCS_\infty  \ar[r, dashed] & (G/N)/\LCS_\infty. \\
\end{tikzcd}\]
By the Nine Lemma, the bottom row must be a short exact sequence.
\end{proof}

The following corollary provides the promised partial converse to Lemma~\ref{lem:stationary_quotient}:
\begin{corollary}\label{Quotient_by_central_subgroup}
Let $G$ be a group and $A$ be a central subgroup of $G$. Suppose that for some $i \geq 2$, the canonical map $A \rightarrow G/\LCS_i(G)$ is injective. If the LCS of $G/A$ stops at $\LCS_k$, then the LCS of $G$ stops at $\LCS_k$ or at $\LCS_{k+1}$. 
\end{corollary}

\begin{proof}
The extension $A \hookrightarrow G/\LCS_\infty \twoheadrightarrow (G/A)/\LCS_\infty$ from Proposition~\ref{Extensions_and_residues} is a central one. As a consequence, if  $(G/A)/\LCS_\infty$ is $k$-nilpotent, then $G/\LCS_\infty$ is nilpotent of class $k$ or $k+1$.
\end{proof}

\begin{remark}
In explicit examples, the hypothesis of Corollary~\ref{Quotient_by_central_subgroup} is most easily checked when $i=2$, since we typically have a complete understanding of~$G^{\ab}$; see for example the proof of Proposition \ref{LCS_B2m(S2)}. 
\end{remark}

\section{Lie rings of lower central series}\label{Lie_rings}

We now recall the definition and basic properties of a key tool for studying the LCS of a group, namely its associated Lie ring. We refer the reader to \cite[Chap.~1]{Lazard} for further details.

Note that, for all $i\geq 1$, $[\LCS_i(G),\LCS_i(G)] \subseteq [G,\LCS_{i}(G)]\subseteq \LCS_{i+1}(G)\subseteq \LCS_{i}(G)$. Thus, $\LCS_{i}(G)$ is a normal subgroup of $G$, and the quotient $\Lie_i(G):=\LCS_{i}(G)/\LCS_{i+1}(G)$ is an abelian group. Moreover, one can show that $[\LCS_i(G),\LCS_j(G)] \subseteq \LCS_{i+j}(G)$ for all $i,j \geq 1$, which is the crucial property allowing us to define the Lie ring associated with $\LCS_*(G)$:
\begin{proposition}[{\cite[Th.~2.1]{Lazard}}]\label{prop:lazard}
The graded abelian group defined by
$\Lie(G):=\bigoplus_{i \geq 1}\Lie_{i}(G)$
is a Lie ring, with the Lie bracket induced by the commutator map of $G$.
\end{proposition}

\begin{convention}
Let $g$ be an element of $G$. If there is an integer $d$ such that $g \in \LCS_d (G) - \LCS_{d+1}(G)$, it is obviously unique. We then call $d$ the \emph{degree} of $g$ with respect to $\LCS_{*}(G)$. The notation $\overline g$ denotes the class of $g$ in some quotient $\Lie_i(G)$. If the integer $i$ is not specified, it is assumed that $i = d$, which means that $\overline g$ denotes the only non-trivial class induced by $g$ in $\Lie(G)$. If such a $d$ does not exist (that is, if $g \in \bigcap \LCS_i(G)$), we say that $g$ has degree $\infty$ and we put $\overline g = 0$.
\end{convention}

With this convention, the Lie bracket $[-,-]$ of $\Lie(G)$ is given by the collection of bilinear maps $\Lie_i(G) \times \Lie_j(G) \rightarrow \Lie_{i+j}(G)$ defined by:
\[ \forall x \in \Lie_i(G),\ \forall y \in \Lie_j(G),\ [\overline x, \overline y]:= \overline{[x,y]} \in \Lie_{i+j}(G).\]

The following lemma, which will be used several times in the sequel to identify $G/\LCS_\infty$ for some group $G$, is one illustration of the use of Lie rings in studying the LCS: 

\begin{lemma}\label{Quotient_by_residue}
Let $p \colon G \twoheadrightarrow Q$ be a surjective group morphism. If $Q$ is a residually nilpotent group, then the following conditions are equivalent:
\begin{itemize}
\item $\Lie(p) \colon \Lie(G) \twoheadrightarrow \Lie(Q)$ is an isomorphism.
\item $p$ induces an isomorphism $G/\LCS_\infty \cong Q$.
\end{itemize}
\end{lemma}

\begin{proof}
If $Q$ is residually nilpotent, $p$ induces a map $G/\LCS_\infty \twoheadrightarrow Q$ between two residually nilpotent groups. Since $G \twoheadrightarrow G/\LCS_\infty$ induces an isomorphism between the associated Lie rings, the statement for $G$ can be deduced from the statement for $G/\LCS_\infty$. Thus, we can assume that $G$ is residually nilpotent and, under this hypothesis, we need to show that $p$ is an isomorphism if and only if $\Lie(p)$ is. Clearly, if $p$ is an isomorphism, then $\Lie(p)$ is too. Conversely, if $p$, which is surjective, is not an isomorphism, then there is some non-trivial element $x$ in its kernel. Since $G$ is residually nilpotent, $x$ induces a non-trivial class $\overline x$ in $\Lie(G)$, which is sent to $0$ by $\Lie(p)$. This implies that $\Lie(p)$ is not injective, which concludes our proof. 
\end{proof}

\section{Computing abelianisations from decompositions}

Let us recall some classical tools for computing the abelianisation from some decomposition of a given group. The abelianisation functor $G \mapsto G^{\ab}$ is a left adjoint, hence right exact. In order to compute the abelianisation of an extension, one can say more. Given a short exact sequence of groups $H \hookrightarrow G \twoheadrightarrow K$, let us denote by $(H^{\ab})_K$ the coinvariants of $H^{\ab}$ with respect to the action of $K$ on $H^{\ab}$ induced by conjugation in $G$.
\begin{lemma}\label{abelianization_from_coinv}
The short exact sequence $H \hookrightarrow G \twoheadrightarrow K$ induces the following exact sequence of abelian groups:
\[(H^{\ab})_K \rightarrow G^{\ab} \rightarrow K^{\ab} \rightarrow 0.\]
\end{lemma}

\begin{proof}
The conjugation action of $G$ on $H^{\ab}$ factors through $G/H = K$, hence $(H^{\ab})_G = (H^{\ab})_K$. Since we have an exact sequence $H^{\ab} \rightarrow G^{\ab} \rightarrow K^{\ab} \rightarrow 0$, it suffices to show that the morphism $H^{\ab} \rightarrow G^{\ab}$ factors through $(H^{\ab})_K$. It is equivariant with respect to the action of $G$ induced by conjugation (which is obviously trivial on $G^{\ab}$), whence the result.
\end{proof}

For split exact sequences, we can say even more:

\begin{lemma}\label{lem:abelianization semidirect}
The abelianisation of a semidirect product $H\rtimes K$ is isomorphic to the product $(H^{\ab})_K \times K^{\ab}$.
\end{lemma}

\begin{proof}
As a consequence of the usual formula $[x, yz] = [x, y] (x [x,z] x^{-1})$, one sees that the commutator subgroup $[H \rtimes K, H \rtimes K]$ is normally generated by $[H,H]$, $[H,K]$ and $[K,K]$. We can take the quotient by these three sets of relations successively: $(H \rtimes K)/[H,H]$ is isomorphic to $H^{\ab} \rtimes K$, then killing $[H,K]$ gives $(H^{\ab})_K \times K$ and finally, $((H^{\ab})_K \times K)/[K,K] \cong (H^{\ab})_K \times K^{\ab}$.
\end{proof}

\chapter{Strategy and first examples}
\label{sec:strategy}

In this chapter, we present some general ideas used to decide whether the LCS stops or not. As a first example, we then apply these ideas to Artin groups.

\section{Generation in degree one -- first consequences}\label{subsec:gen_deg_one_first_consequences}

The Lie ring associated to the LCS has the following fundamental property:
\begin{proposition}\label{generationproperty}
The Lie ring $\Lie(G)$ is generated in degree one. That is, it is generated by the abelianisation $\Lie_1(G) = G^{\ab}$ as a Lie algebra over $\Z$.
\end{proposition}

\begin{proof}
It is a direct consequence of the definitions: the equality $\Lie_k(G) = [\Lie_1(G), \Lie_{k-1}(G)]$ is obtained directly from $\LCS_k(G) = [\LCS_1(G), \LCS_{k-1}(G)]$, by passing to the appropriate quotients.
\end{proof}

A first consequence of this is the following:
\begin{corollary}\label{cyclic_abelianization}
Let $G$ be a group. If $G^{\ab}$ is cyclic, then $\LCS_2 G = \LCS_3 G$.
\end{corollary}

\begin{proof}
Proposition~\ref{generationproperty} implies that the Lie ring $\mathcal L(G)$ is a quotient of the free Lie ring on $G^{\ab}$. Since $G^{\ab}$ is cyclic, the latter is the abelian Lie ring consisting only of $G^{\ab}$. As a consequence,  $\LCS_2 G / \LCS_3 G = \Lie_2(G) = \{0\}$.
\end{proof}

\begin{example}[Braids]
\label{eg:braids}
Directly from their usual presentations, one computes the abelianisation of the braid groups: $\B_n^{\ab} \cong \Z$ for $n \geq 2$. Thus $\LCS_2 (\B_n) = \LCS_3 (\B_n)$. This fact is originally due to Gorin and Lin \cite{GorinLin}, who proved it by different methods. This property of $\B_n$, which is also true for any quotient of $\B_n$ (such as the symmetric group $\Sym_n$), may also be seen as a particular case of the computations for Artin groups below (Proposition~\ref{LCS_Artin_stops}).
\end{example}

\begin{example}[Knot groups]
For any knot, the knot group has (infinite) cyclic abelianisation, thus its LCS stops at $\LCS_2$. This generalises readily to the enveloping groups of any connected quandle; see for instance~\cite[Prop.~3.3]{BardakovNasybullovSingh}.
\end{example}

\begin{example}[Automorphisms of free groups]\label{eg_Aut(Fn)}
Consider the automorphism group $\Aut(\F_n)$ of the free group on $n$ letters. The kernel $IA_n$ of the projection from $\Aut(\F_n)$ onto $\Aut(\F_n^{\ab}) \cong \mathrm{GL}_n(\Z)$ is generated by the usual $K_{ij}$ and $K_{ijk}$ from \cite[\Spar3.5]{MagnusKarrassSolitar}, which are easily seen to be commutators of automorphisms. Thus $\Aut(\F_n)^{\ab} \cong \mathrm{GL}_n(\Z)^{\ab}$. Whenever $n \geq 3$, this group is cyclic of order two, so the LCS of $\Aut(\F_n)$ stops at $\LCS_2$, and so does the one of $\mathrm{GL}_n(\Z)$. 
\end{example}

An easy generalisation of Corollary~\ref{cyclic_abelianization} is:
\begin{corollary}\label{commuting_representatives}
Let $G$ be a group. Let $S$ be a generating set of $G^{\ab}$. Suppose that, for each pair $(s,t) \in S^2$, we can find representatives $\tilde s, \tilde t \in G$ of $s$ and $t$ such that $\tilde s$ and $\tilde t$ commute. Then $\LCS_2 G = \LCS_3 G$.
\end{corollary}

\begin{proof}
The Lie ring $\Lie(G)$ is generated by $S$. Moreover, the fact that $\left[\tilde s,\tilde t \right] = 1$ in $G$ readily implies that $[s, t] = 0$ in $\Lie(G)$. Since the brackets $[s, t]$ for $(s,t) \in S^2$ generate $\Lie_2(G) = [\Lie_1(G), \Lie_1(G)]$, we see that $\LCS_2 G /\LCS_3 G = \Lie_2(G)= \{0\}$. In fact, $\Lie(G)$ is an abelian Lie ring, reduced only to $\Lie_1(G) = G^{\ab}$.
\end{proof}

We have not made any effort to make the above corollary as general as possible. In particular, $\tilde s$ and $\tilde t$ may commute only up to an element of $\LCS_3 G$, and the conclusion still holds. Also, one may think of similar statements showing that $\LCS_3 G =\LCS_4 G$, and so on. Weak as it may seem, our statement is already very useful. In particular, when applied to groups whose elements have a geometrical interpretation, it will often happen that $\tilde s$ and $\tilde t$ can be chosen \emph{with disjoint support} (whatever this means, depending on the context -- see for instance \Spar\ref{subsec_support} for precise definitions in certain cases), which readily implies that they commute. We will sometimes need a more refined version of the above, but we will discuss it in each particular situation.

\begin{example}[Automorphisms of $\F_2$]
As an example of a case where Corollary~\ref{commuting_representatives} does not work, but the same kind of technique does apply, let us consider $\Aut(\F_2)$. We have mentioned that $\Aut(\F_n)^{\ab} \cong \mathrm{GL}_n(\Z)^{\ab}$; see Example~\ref{eg_Aut(Fn)}. For $n = 2$, this is no longer cyclic, but isomorphic to $(\Z/2)^2$, generated by the (equivalences classes of the) automorphisms $\sigma$ and $\tau$ acting as follows (fixing free generators $x$ and $y$ of $\F_2$):
\[
\sigma(x) = y, \qquad \sigma(y) = x, \qquad\qquad \tau(x) = x^{-1}, \qquad \tau(y) = y.
\]
It follows that $\Lie_2(\Aut(\F_2))$ is generated by the (equivalence class of the) automorphism $\iota = [\sigma,\tau]$ acting by $\iota(x) = x^{-1}$ and $\iota(y) = y^{-1}$. It is easy to check that $\iota$ commutes with both $\sigma$ and $\tau$, so $\Lie_3(\Aut(\F_2)) = 0$. Thus, the LCS of $\Aut(\F_2)$ stops at $\LCS_3$, as does that of $\mathrm{GL}_2(\Z)$.
\end{example}

Let us spell out another useful consequence of Proposition~\ref{generationproperty}:
\begin{corollary}\label{torsion}
Let $G$ be a group and $d \geq 1$ be an integer. If $\Lie_k(G)= \LCS_k G / \LCS_{k+1} G$ is a $d$-torsion abelian group for some $k \geq 1$, then $\Lie_l(G) = \LCS_l G / \LCS_{l+1} G$ is too for all integers $l \geq k$.
\end{corollary}

\begin{proof}
If an element $x \in \Lie(G)$ is of $d$-torsion, then for all $y \in \Lie(G)$, the bracket $[x,y]$ is too, because $d \cdot [x,y] = [d \cdot x,y] = 0$. Since $\Lie_{l+1}(G) = [\Lie_1(G), \Lie_l(G)]$, we get our result by induction on $l$.
\end{proof}

\begin{example}[Virtual and welded braids]\label{G^ab_equals_Z+Z/d}
Let $G$ be a group such that $G^{\ab} \cong \Z \times (\Z/d)$, where the factors are generated, respectively, by $u$ and $v$. Then $\Lie_2(G)$ is generated, as an abelian group, by $[u,v]$. Since $d \cdot [u,v] = [u, d \cdot v] = 0$, $\Lie_2(G)$ is of $d$-torsion, and then all the $\Lie_k(G)$, for $k \geq 2$, must be too. This applies, for instance, to the groups $\vB_2$ and $\wB_2$ (which are both isomorphic to $\Z*\Z/2$), with $d = 2$; see Chapter~\ref{sec:virtual_welded_braids}. 
\end{example}

\section{Artin groups}\label{subsec:Artin_groups}

Let $S$ be a set, and $M = (m_{s,t})_{s \in S}$ be a Coxeter matrix, i.e. a symmetric matrix with coefficients in $\mathbb N \cup \{\infty\}$, with $m_{s,s} = 1$, and $m_{s,t} \geq 2$ if $s \neq t$. Let $A_M$ be the associated Artin group, defined by the presentation:
\[A_M = \left\langle S\ \middle| \underbrace{ststs \cdots}_{m_{s,t}} = \underbrace{tstst \cdots}_{m_{s,t}}\ \ (s,t \in S,\ m_{s,t} \neq \infty) \right\rangle.\]

Let us consider the graph $\mathcal G$ whose incidence matrix is $M$ modulo $2$. Namely, $\mathcal G$ is obtained by taking $S$ as its set of vertices and by drawing an edge between $s$ and $t$ whenever $m_{s,t}$ is an odd integer.

\begin{lemma}
The group $A_M^{\ab}$ is free abelian on the set $\pi_0(\mathcal G)$ of connected components of $\mathcal G$.
\end{lemma}

\begin{proof}
This is clear from the presentation: in $A_M^{\ab}$, the relation between $s$ and $t$ becomes $\bar s = \bar t$ if $m_{s,t}$ is odd, and it becomes trivial if $m_{s,t}$ is even or $m_{s,t} = \infty$.
\end{proof}

Let us now study the LCS of $A_M$. Suppose first that $\mathcal G$ is connected. Then  $A_M^{\ab}$ is cyclic, and Corollary~\ref{cyclic_abelianization} applies: the LCS of $A_M$ stops at $\LCS_2$. This holds in particular for the classical braid group; see Example~\ref{eg:braids}. Now, if $\mathcal G$ has several connected components, we need to study the interactions between the corresponding generators of the Lie ring of $A_M$. The simplest case happens when all the even $m_{s,t}$ are equal to $2$. Then $A_M$ splits into the (restricted) direct product of the Artin groups corresponding to the connected components of our graph, thus $\mathcal L(A_M)$ is a direct sum of copies of $\Z$ (concentrated in degree one). Thus, we obtain a first result (which recovers \cite[Prop.~1]{BellingeriGervaisGuaschi} for spherical Artin groups):

\begin{proposition}\label{LCS_Artin_stops}
If all the coefficients of $M$ are finite, and are either odd or equal to $2$, then the LCS of $A_M$ stops at $\LCS_2$.
\end{proposition}

In order to get a step further, we need to study more closely the interactions between the generators of the Lie ring of $A_M$ corresponding to the connected components of $\mathcal G$.
\begin{lemma}
Let $s,t \in S$. If $m_{s,t} = 2k$ for some integer $k$, then $k \cdot [\bar s, \bar t] = 0$ in $\Lie(A_M)$.
\end{lemma}

\begin{proof}
If $m_{s,t} = 2k$, then $(st)^k = (ts)^k$ in $A_M$. We can write this relation as $(st)^k(ts)^{-k} = 1$. We recall that modulo $\LCS_3(A_M)$, commutators commute with any other element. Hence, modulo $\LCS_3(A_M)$, we have:
\[1 = (st)^k(ts)^{-k} = (st)^{k-1}[s,t](ts)^{-(k-1)} = [s,t](st)^{k-1}(ts)^{-(k-1)} = \cdots = [s,t]^k.\]
Thus, in $\mathcal L_2(A_M)$, we have $0 = k \cdot \overline{[s,t]} = k \cdot [\bar s,\bar t]$, as claimed.
\end{proof}

From this, we deduce that for $x,y \in \pi_0(\mathcal G)$, we have $d_{x,y} \cdot [x,y] = 0$ in $\mathcal L_2(A_M)$, where $d_{x,x} = 1$ and $d_{x,y} = \gcd \left\{\frac{m_{s,t}}{2}\ \middle|\ s,t \in S,\ \bar s = x \text{ and } \bar t = y \right\}$ if $x \neq y$. In particular, this proves:
\begin{proposition}\label{LCS_Artin_torsion}
If all the coefficients of $M$ are finite, then $\LCS_2(A_M)/\LCS_3(A_M)$ is of $d$-torsion, where $d = \lcm \left\{d_{x,y}\ |\ x,y \in \pi_0(\mathcal G)\right\}$.
\end{proposition}

\begin{remark}
Proposition~\ref{LCS_Artin_torsion} is more general than Proposition~\ref{LCS_Artin_stops}, which is recovered as a particular case where $d = 1$. In fact, for $d$ to be equal to $1$ (which implies that the LCS stops at $\LCS_2$), we do not need the even $m_{s,t}$ to be equal to $2$, but only all the $d_{x,y}$ to be equal to $1$.
\end{remark}

\begin{remark}
Proposition~\ref{LCS_Artin_torsion} does not say anything when at least one of the $d_{x,y}$ is infinite. For instance, when all the $m_{s,t}$ are infinite, $A_M$ is the free group on $S$, whose Lie ring is without torsion. More generally, Right-Angled Artin Groups (where all the $m_{s,t}$ are infinite or equal to $2$) have torsion-free Lie rings~\cite{Duchamp}. Another example is the subgroup $KB_n$ of the virtual braid group $\vB_n$ given by the normal closure of $\B_n \subset \vB_n$ (denoted by $H_n$ in~\cite{BardakovBellingeri2009}). This is an Artin group with all $m_{s,t}$ infinite or equal to $2$ or $3$ \cite[Prop.~17]{BardakovBellingeri2009}; its abelianisation is free abelian and its LCS stops at $\LCS_2$ for $n\geq 3$ \cite[Prop.~19]{BardakovBellingeri2009}.
\end{remark}

\chapter{Partitioned braids}
\label{sec:partitioned_braids}

This chapter is devoted to the study of the LCS of the group of \emph{partitioned braids}; see Definition~\ref{def_partitioned_braids} below. Our main results are summarised in Theorem~\ref{thm:partitioned-braids}. The group of partitioned braids is a subgroup of $\B_n$, which has already been studied notably by Manfredini \cite{Manfredini} and by Bellingeri, Godelle and Guaschi \cite{BellingeriGodelleGuaschi}. The former gave a presentation of this group, using the Reidemeister-Schreier method, which may be applied to this finite-index subgroup of Artin's braid group. 

\begin{remark}
One could use the aforementioned presentation of the partitioned braid group \cite{Manfredini} to get generators, a calculation of the abelianisation, and (with a little bit of work) the stable case in the study of the LCS. However, we will avoid using this presentation altogether, for several reasons. Firstly, we want the present memoir to be as self-contained as possible. Secondly, even if a presentation is of some help in the study of the LCS of a group, there is only so much that one can deduce directly from it; in fact, not much more than a computation of the abelianisation, and that $\LCS_2 = \LCS_3$ when it holds. In particular, most of our non-stable results would not be simplified by using the presentation. Thirdly, and perhaps most importantly, to the best of our knowledge, nobody has written down a presentation of the other partitioned groups that we study later on. This could certainly be done using the Reidemeister-Schreier technique (at least when we have a presentation of the non-partitioned group), but this would require a fair amount of work, which we intend to avoid. We will do so precisely by generalising proofs that do not depend on the use of Manfredini's presentation.
\end{remark}

\section{Reminders: braids}

We recall that the standard generators $\sigma_i$ and $A_{ij}$ of the braid group $\B_n$ and of the pure braid group $\PB_n$ respectively, are the braids drawn in Figure~\ref{fig:partitioned-braid-generators}. In Figure~\ref{fig:partitioned-braids}, they also appear in a ``bird's eye view'' as paths of configurations.

\begin{figure}[ht]
\centering
\includegraphics[scale=0.67]{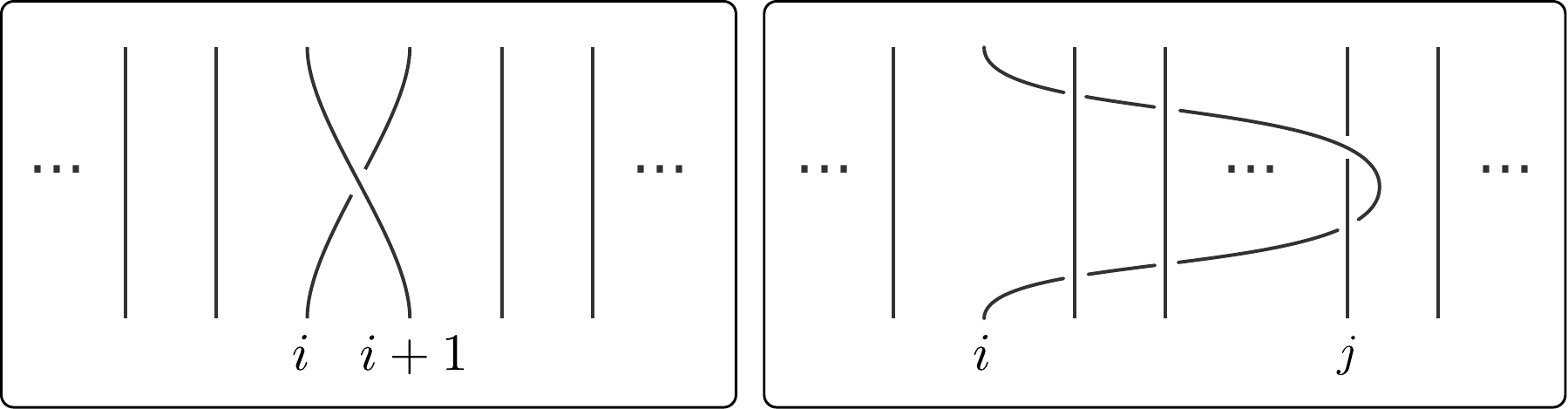}
\caption{The standard generators $\sigma_i$ and $A_{ij}$ of, respectively, the braid group and the pure braid group.}
\label{fig:partitioned-braid-generators}
\end{figure}

\section{Basic theory of partitioned braids}

For a partition $\lambda = (n_1, \ldots , n_l)$ of an integer $n$, we denote by $\Sym_\lambda$ the subgroup $\Sym_{n_1} \times \cdots \times \Sym_{n_l}$ of the symmetric group $\Sym_n$. Let us now consider the braid group $\B_n$ on $n$ strands, and the usual projection $\pi \colon \B_n \twoheadrightarrow \Sym_n$.

\begin{definition}\label{def_partitioned_braids}
Let $n \geq 1$ be an integer, and let $\lambda = (n_1, \ldots , n_l)$ be a partition of $n$. The corresponding \emph{partitioned braid group} (also referred to as the group of \emph{$\lambda$-partitioned braids}) is:
\[\B_{\lambda} := \pi^{-1}(\Sym_\lambda) = \pi^{-1}\left(\Sym_{n_1} \times \cdots \times \Sym_{n_l}\right)\ \subseteq \B_n.\]
\end{definition}

\begin{lemma}\label{generators_of_partitioned_braids}
Let $\lambda = (n_1, \ldots , n_l)$ be a partition of an integer $n \geq 1$. Then $\B_\lambda$ is the subgroup of $\B_n$ generated by:
\begin{itemize}
\item The $\sigma_\alpha$ for $1 \leq \alpha \leq n$ such that $\alpha$ and $\alpha + 1$ are in the same block of $\lambda$.
\item The $A_{\alpha\beta}$ for $1 \leq \alpha < \beta \leq n$ such that $\alpha$ and $\beta$ do not belong to the same block of $\lambda$.
\end{itemize}
\end{lemma}

\begin{proof}
Consider the subgroup $G$ of $\B_n$ generated by these elements. Clearly, $G \subseteq \B_\lambda$, and we need to show that $G$ contains $\B_\lambda$. First, we see that $G$ contains $\B_{n_1} \times \cdots \times \B_{n_l}$, which is generated by the chosen $\sigma_\alpha$. As a consequence, $\pi(G) = \Sym_{n_1} \times \cdots \times \Sym_{n_l}$. Then, $G$ also contains the $A_{\alpha\beta}$, for all $1 \leq \alpha < \beta \leq n$. Indeed, the $A_{\alpha\beta}$ missing in the list of the statement are the ones with $\alpha$ and $\beta$ in the same block of $\lambda$, which are exactly the ones belonging to $\B_{n_1} \times \cdots \times \B_{n_l}$. Thus, the pure braid group $\PB_n$ is contained in $G$. Now, if $\beta$ is any $\lambda$-partitioned braid, then $\pi(\beta) \in \Sym_{n_1} \times \cdots \times \Sym_{n_l}$, hence we can choose $g \in G$ such that $\pi(g) = \pi(\beta)$. Then $g^{-1}\beta$ is a pure braid, thus $g^{-1}\beta \in G$, which implies that $\beta \in G$.
\end{proof}

\begin{remark}
The generating set of $\B_\lambda$ described in Lemma~\ref{generators_of_partitioned_braids} is clearly redundant, since $A_{\alpha\beta}$ is conjugate to $A_{\gamma\delta}$ if $\alpha$ (resp.~$\beta$) is in the same block as $\gamma$ (resp.~$\delta$). For instance, in $\B_{2,2}$, we have $\sigma_{1}A_{13}\sigma_{1}^{-1}=A_{23}$. Since it is more convenient for our purpose, we prefer to keep it that way. Notice however that it is easy to extract a minimal set of generators: using the computation of $\B_\lambda^{\ab}$ below, it is easy to show that keeping all the $\sigma_\alpha$ and choosing only one lift $A_{\alpha\beta}$ of each $a_{ij}$ (that is, one $A_{\alpha\beta}$ for each pair of blocks) does give a minimal set of generators of $\B_\lambda$. A similar remark (with a similar method for choosing minimal sets of generators) applies to the generating sets described later in the memoir for variants of partitioned braid groups.
\end{remark}

We now compute $\B_\lambda^{\ab}$, using the above generating set, together with the split projections corresponding to forgetting blocks of strands. These projections can be seen as particular cases of the projections from Proposition~\ref{Fadell-Neuwirth} below, applied to braids on the disc.

\begin{proposition}\label{partitioned_B^ab}
Let $\lambda = (n_1, \ldots , n_l)$ be a partition of an integer $n \geq 1$. The abelianisation  $\B_\lambda^{\ab}$ is free abelian on the following basis:
\begin{itemize}
\item For each $i \in \{1, \ldots , l\}$ such that $n_i \geq 2$, one generator $s_i$: this is the common class of the  $\sigma_\alpha$ for $\alpha$  and $\alpha + 1$ in the $i$-th block of $\lambda$.
\item For each $i, j \in \{1, \ldots , l\}$ such that $i < j$, one generator $a_{ij}$: this is the common class of the  $A_{\alpha\beta}$ for $\alpha$ in the $i$-th block of $\lambda$ and $\beta$ in the $j$-th one.
\end{itemize}
\end{proposition}

\begin{proof}
The abelianisation $\B_\lambda^{\ab}$ is generated by the classes of the generators from Lemma~\ref{generators_of_partitioned_braids}. Moreover, we note that for any $i$, all the $\sigma_\alpha$ with $\alpha$ and $\alpha + 1$ in the $i$-th block of $\lambda$ (which exist only if $n_i \geq 2$) are conjugate to one another: for instance, if $n_1 \geq 3$, then $\sigma_2 = (\sigma_2 \sigma_1)^{-1} \sigma_1 (\sigma_2 \sigma_1)$ is a conjugation relation inside $\B_\lambda$. Similarly, for each choice of $i < j$, all the $A_{\alpha\beta}$ with $\alpha$ in the $i$-th block and $\beta$ in the $j$-th one are conjugate to one another. Thus, the family described in the statement is well-defined and generates $\B_\lambda^{\ab}$. Let us show that it is linearly independent, by using the projections obtained by forgetting strands. 

Suppose that $\sum k_i s_i + \sum k_{ij} a_{ij} = 0$ for some integers $k_i$ ($i \leq l$) and $k_{ij}$ ($i < j \leq l$). Let us fix $i$ such that $n_i \geq 2$, and let us apply the canonical projection $\B_\lambda^{\ab} \twoheadrightarrow \B_{n_i}^{\ab} \cong \Z$ to the above relation. Since this projection kills all the generators save $s_i$, we get $k_i = 0$. This holds for all $i$, so we are left with the relation $\sum k_{ij} a_{ij} = 0$. To which, for any choice of $i < j$, we can apply the morphism $\B_\lambda^{\ab} \twoheadrightarrow \B_{n_i, n_j}^{\ab} \rightarrow \B_{n_i + n_j}^{\ab} \cong \Z$. This kills all the $a_{kl}$, save $a_{ij}$ (which is sent to $2$), hence $k_{ij} = 0$, whence the result.
\end{proof}

\section{The lower central series}

This section is devoted to the proof of the following result, which states exactly when the LCS of the partitioned braid group stops. The proof requires different techniques for a number of different cases. We consider each case in turn and then combine the individual results into a proof of Theorem \ref{thm:partitioned-braids} at the very end of this section.

\begin{theorem}
\label{thm:partitioned-braids}
Let $n \geq 1$ be an integer, let $\lambda = (n_1, \ldots , n_l)$ be a partition of $n$. The LCS of the partitioned braid group $\B_{\lambda}$:
\begin{itemize}
\item \emph{stops at $\LCS_2$} if $n_i \geq 3$ for all $i$, save at most two indices for which $n_i = 1$. 
\item \emph{does not stop} in all the other cases, except for $\B_2 \cong \Z$.
\end{itemize}
\end{theorem}

\subsection{The stable case: a disjoint support argument}

\begin{proposition}\label{LCS_stable_partitioned_braids}
Let $n \geq 1$ be an integer, let $\lambda = (n_1, \ldots , n_l)$ be a partition of~$n$. Suppose that all the $n_i$ are at least $3$. Then the LCS of $\B_{\lambda}$ stops at $\LCS_2$.
\end{proposition}

\begin{proof}
Consider the generating set for $\B_{\lambda}^{\ab}$ described in Corollary~\ref{partitioned_B^ab}. For any pair of such generators, it is possible to find lifts in $\B_{\lambda}$ having disjoint support (as mapping classes of the punctured disc), and thus commuting -- see Figure~\ref{fig:partitioned-braids} for the three cases that may arise. Then, we can apply Corollary~\ref{commuting_representatives} to prove our claim.
\end{proof}

\begin{figure}[ht]
\centering
\includegraphics[scale=0.9]{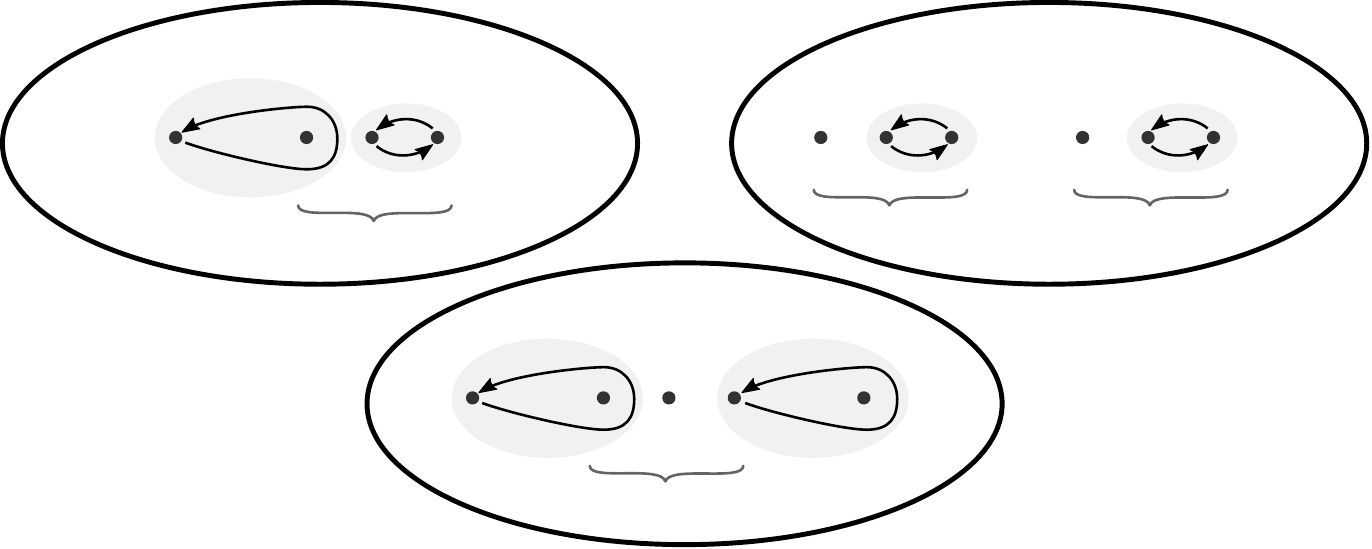}
\caption{Choosing representatives with disjoint support for pairs of generators of the abelianisation of the partitioned braid group. Braces indicate points that lie in the same block of the partition.}
\label{fig:partitioned-braids}
\end{figure}

\subsection{Blocks of size $1$}

\begin{lemma}\label{LCS_B111}
If at least three blocks of the partition $\lambda$ are of size $1$, then the LCS of $\B_{\lambda}$ does not stop.
\end{lemma}

\begin{proof}
Under the hypothesis, there is a surjection $\B_{\lambda} \twoheadrightarrow \B_{1,1,1} \cong \mathbf{P}_3$ obtained by forgetting all the blocks save three blocks of size $1$. The LCS of $\PB_3 \cong\F_2 \rtimes \Z$ does not stop, since it is an almost-direct product, and the LCS of $\F_2$ does not stop; see for instance~\cite{FalkRandell1}. As a consequence, the one of $\B_{\lambda}$ does not either by Lemma~\ref{lem:stationary_quotient}.
\end{proof}

The cases with one or two blocks of size $1$ (and still no blocks of size $2$) are more difficult to handle. In both cases, we will have to use the following observation:

\begin{lemma}\label{P3^ab}
The quotient of $\PB_3$ by the relation $[A_{13}, A_{23}] = 1$ is $\PB_3^{\ab} \cong \Z^3$.
\end{lemma}

\begin{proof}
Let $N$ be the normal closure of $[A_{13}, A_{23}]$ in $\PB_3$. We want to show that $N = \LCS_2(\PB_3)$. Clearly, $N \subseteq \LCS_2(\PB_3)$. To show the converse inclusion, we need to show that $\PB_3/N$ is abelian. We check that the relations $A_{12}A_{13}A_{12}^{-1} = A_{23}^{-1}A_{13}A_{23}$ and $A_{12}^{-1}A_{23}A_{12} = A_{13}A_{23}A_{13}^{-1}$ hold in $\PB_3$. As a consequence, modulo the relation $[A_{13}, A_{23}] = 1$ (that is, modulo $N$), we get $A_{12}A_{13}A_{12}^{-1} \equiv A_{13}$ and $A_{12}^{-1}A_{23}A_{12} \equiv A_{23}$. Thus $A_{12}$, $A_{13}$ and $A_{23}$ commute modulo $N$ and therefore $\PB_3/N$ is abelian.
\end{proof} 

\begin{proposition}\label{LCS_B1mu}
Let $n \geq 1$ be an integer, let $\lambda = (1, n_2, \ldots , n_l)$ be a partition of $n$, with $n_i \geq 3$ for $i \geq 2$. Then the LCS of $\B_{\lambda}$ stops at $\LCS_2$.
\end{proposition}

\begin{proof}
The case $l = 2$ (i.e.~$\lambda = (1,m)$ for some integer $m \geq 3$) works exactly like the stable case of Proposition~\ref{LCS_stable_partitioned_braids}. Namely, the two generators of the abelianisation do have lifts with disjoint supports, so that $\LCS_2(\B_{1,m}) = \LCS_3(\B_{1,m})$ by Corollary~\ref{commuting_representatives}. This however does not work for $l > 2$.

Let us denote by $\mu$ the partition $(n_2, \ldots , n_l)$ of $n-1$, so that  $\lambda = (1, \mu)$. We are going to show that $\B_{\lambda}/\LCS_\infty$ is abelian, which implies that $\LCS_\infty = \LCS_2$ for $\B_{\lambda}$.

It follows from Proposition~\ref{LCS_stable_partitioned_braids} that $\LCS_\infty = \LCS_2$ for $\B_{\mu}$. As a consequence, the obvious morphism $\B_{\mu} \rightarrow \B_\lambda/\LCS_\infty$ factors through $\B_{\mu}/\LCS_2 \cong \B_{\mu}^{\ab}$. Given the description of $\B_{\mu}^{\ab}$ from Proposition~\ref{partitioned_B^ab}, we have that modulo $\LCS_\infty$, $\sigma_\alpha \equiv \sigma_\alpha'$ and $A_{\alpha\beta} \equiv A_{\alpha'\beta'}$ if $\alpha$ and $\alpha'$ (resp.~$\beta$ and $\beta'$) are in the same block of $\mu$. Moreover, the corresponding classes $s_i$ and $a_{ij}$ commute with one another. In the same way, we can deduce from the case of $\B_{1,m}$ for $m \geq 3$ treated above that the class of a generator $A_{1\alpha}$ depends only on the block of $\alpha$. More precisely, we use the fact that the obvious morphism  $\B_{1,n_i} \rightarrow \B_{1, \mu}/\LCS_\infty$ factors through $\B_{1,n_i}^{\ab}$, where all the $A_{1\alpha}$ are identified. Moreover, the corresponding $a_{1i}$ commutes with all the $s_p$ and $a_{pq}$ coming from $\B_\mu$, since they have lifts in $\B_\lambda$ with disjoint support.

We are left with showing that $a_{1i}$ commutes with $a_{1j}$ when $i < j$. This uses Lemma~\ref{P3^ab}. Let us choose any $\alpha$ in the $i$-th block and any $\beta$ in the $j$-th one, and consider the morphism $\PB_3 \rightarrow \B_\lambda$ induced by $1 \mapsto 1$, $2 \mapsto \alpha$ and $3 \mapsto \beta$. If we compose it with the projection onto $\B_\lambda/\LCS_\infty$, we get a morphism $f$ sending $A_{12}$ to $a_{1i}$, $A_{13}$ to $a_{1j}$, and $A_{23}$ to $a_{ij}$. Since $a_{ij}$ commutes with $a_{1j}$, $f$ sends $[A_{13}, A_{23}]$ to $1$. Lemma~\ref{P3^ab} then implies that it factors through $\PB_3^{\ab}$, so that all the elements in its image commute, including $a_{1i}$ and $a_{1j}$. This finishes the proof that $\B_\lambda/\LCS_\infty$ is abelian, whence equal to $\B_\lambda/\LCS_2$.
\end{proof}

\begin{proposition}\label{LCS_B11mu}
Let $n \geq 1$ be an integer, let $\lambda = (1, 1, n_3, \ldots , n_l)$ be a partition of $n$, with $n_i \geq 3$ for $i \geq 3$. Then the LCS of $\B_{\lambda}$ stops at $\LCS_2$.
\end{proposition}

\begin{proof}
Let $\mu$ denote the partition $(n_3, \ldots , n_l)$ of $n-2$, so that  $\lambda = (1,1, \mu)$. It follows from Proposition~\ref{LCS_B1mu} that $\LCS_\infty = \LCS_2$ for $\B_{1, \mu}$, hence both the obvious maps $\B_{1, \mu} \rightarrow \B_\lambda/\LCS_\infty$ factor through $\B_{1, \mu}^{\ab}$. From this, we deduce that $\B_\lambda/\LCS_\infty$ is generated by elements $s_i$ ($3 \leq i \leq l$) and $a_{ij}$ ($1 \leq i < j \leq l$), and that the $a_{1j}$ (resp.~the $a_{2j}$), the $s_p$ ($p \geq 1$) and the $a_{pq}$ ($q > p \geq 3$) commute with one another. Moreover, for all $i,j \geq 3$,  $a_{1i}$ and $a_{2j}$ have lifts with disjoint support in $\B_\lambda$, so they commute (even when $i = j$).

We are left with showing that the class $a_{12}$ of $A_{12}$ commutes with all the other generators. A disjoint support argument shows that it commutes with the $s_k$ and $a_{kl}$ for $l > k  \geq 3$. Now, for $i \geq 3$, let us choose $\alpha$ in the $i$-th block, and let us consider the morphism $\PB_3 \rightarrow \B_\lambda$ induced by $1 \mapsto 1$, $2 \mapsto 2$ and $3 \mapsto \alpha$. If we compose it with the projection onto $\B_\lambda/\LCS_\infty$, we get a morphism $f$ sending $A_{12}$ to $a_{12}$, $A_{13}$ to $a_{1i}$, and $A_{23}$ to $a_{2i}$. Since $a_{1i}$ commutes with $a_{2i}$, $f$ sends $[A_{13}, A_{23}]$ to $1$. Lemma~\ref{P3^ab} then implies that it factors through $\PB_3^{\ab}$, so that all the elements in its image commute, showing that $a_{12}$ commutes with $a_{1i}$ and with $a_{2i}$. Thus, we have proved that $\B_\lambda/\LCS_\infty$ is abelian, whence $\LCS_\infty(\B_\lambda) = \LCS_2(\B_\lambda)$.
\end{proof}

\subsection{Blocks of size $2$}

When there is exactly one block of size $2$ and no block of size $1$, we get a complete description of the quotient of $\B_{\lambda}$ by its residue:

\begin{proposition}\label{LCS_B2mu}
Let $n \geq 1$ be an integer, let $\lambda = (2, n_2, \ldots , n_l)$ be a partition of $n$, with $n_i \geq 3$ if $i \geq 2$. Then $\B_{\lambda}/\LCS_\infty$ decomposes as a direct product of $l(l-1)/2$ copies of $\Z$ with $\Z^{2(l-1)} \rtimes \Z$, where $\Z$ acts via the involution exchanging the elements $e_{2i}$ and $e_{2i+1}$ of a basis of $\Z^{2(l-1)}$. In particular, if $l \geq 2$, then the LCS of $\B_{\lambda}$ does not stop. 
\end{proposition}

\begin{proof}
Let $\mu$ denote the partition $(n_2, \ldots , n_l)$ of $n-2$, so that  $\lambda = (2, \mu)$. Then the canonical projection $\B_{2, \mu} \twoheadrightarrow \Sym_{2, \mu} \twoheadrightarrow \Sym_2$ has $\B_{1,1, \mu}$ as its kernel. Moreover, $\LCS_\infty(\B_{2, \mu})$ contains $\LCS_\infty(\B_{1,1, \mu})$, which is equal to $\LCS_2(\B_{1,1, \mu})$ by Proposition~\ref{LCS_B11mu}. We show that these are in fact equal. In order to do this, it is enough to show that $\B_{2, \mu}/\LCS_2(\B_{1,1, \mu})$ is residually nilpotent. In fact, we are going to compute it completely.

First, let us remark that it makes sense to consider this quotient: $\LCS_2(\B_{1,1, \mu})$ is a characteristic subgroup of $\B_{1,1, \mu}$, which is normal (of index $2$) in $\B_{2, \mu}$, hence it is a normal subgroup of $\B_{2, \mu}$. Next, we can write $\B_{2, \mu}/\LCS_2(\B_{1,1, \mu})$ as an extension:
\begin{equation}\label{ext1}
\begin{tikzcd}
\B_{1,1, \mu}^{\ab} \ar[r, hook]
&\B_{2, \mu}/\LCS_2(\B_{1,1, \mu})\ar[r, two heads] 
&\Sym_2. 
\end{tikzcd}
\end{equation}
We can use the method from Appendix~\ref{section_pstation_of_ext} to get a presentation of the group $G := \B_{2, \mu}/\LCS_2(\B_{1,1, \mu})$. Namely, we have a presentation of the kernel: $\B_{1,1, \mu}^{\ab}$ is free abelian on the $s_i$ and the $a_{ij}$ indexed by the blocks of $(1,1,\mu)$. We also know the action of $\Sym_2$ on $\B_{1,1, \mu}^{\ab}$ induced by conjugation by $\sigma_1$ in $\B_{2, \mu}$: it exchanges the $a_{1j}$ with the $a_{2j}$ with $j \geq 3$ and it acts trivially on all the other generators. Finally, we can lift the only relation defining $\Sym_2$ to $\sigma_1^2 = A_{12}$ in $\B_{2, \mu}$. As a consequence, we get the presentation:
\begin{equation*}
\resizebox{\hsize}{!}{$
G = \sbox0{$
\begin{array}{l|ll}
&\forall i,j,p,q,u,v,\ &[s_i, s_j] = [s_i, a_{pq}] = [a_{pq}, a_{uv}] = 1, \\
s,&\forall i \geq 1,\ &[s, s_i] = 1, \\
s_i\ (3 \leq i \leq l+1), &\forall j > i \geq 3,\ &[s, a_{ij}] = 1, \\
a_{ij}\ (1 \leq i < j \leq l+1). &\forall j \geq 3,\ &s a_{1j} s^{-1} =  a_{2j} \text{ and } s a_{2j} s^{-1} =  a_{1j}, \\
&&s^2 = a_{12}.
\end{array}$
}
\mathopen{\resizebox{1.2\width}{1.1\ht0}{$\Bigg\langle$}}
\usebox{0}
\mathclose{\resizebox{1.2\width}{1.1\ht0}{$\Bigg\rangle$}}
$}
\end{equation*}
One can deduce from this presentation that $G$ decomposes as $\Z^N \times (\Z^{2(l-1)} \rtimes \Z)$ where the first factor is free abelian on the $s_i$ ($i \geq 3$) and the $a_{ij}$ ($j > i \geq 3$) (hence $N = l(l-1)/2$), and the action of $\Z$ (free on $s$) on $\Z^{2(l-1)}$ (free abelian on the  $a_{1j}$ and the $a_{2j}$) is given by $s$ exchanging the $a_{1j}$ and the $a_{2j}$. Checking that this holds is a matter of writing the appropriate split projections from the presentations of $G_l$ and its factors.

Finally, the decomposition of $G$ allows us to apply Proposition~\ref{LCS_Klein_tau} to compute its LCS. Namely, we apply it to $A = \langle a_{1j}, a_{2j} \rangle_{j \geq 3}$ (which is free abelian on these generators) endowed with the involution $\tau$ exchanging $a_{1j}$ with $a_{2j}$ for all $j$. Then $V = \ima(\tau - 1)$ is the free abelian subgroup generated by the $a_{1j} - a_{2j}$, and for $k \geq 2$, we have $\LCS_k(G_l) = 2^{k-1}V$. In particular, this LCS does not stop. However its intersection is trivial: the group $G_l$ is residually nilpotent, which implies that $\LCS_\infty(\B_{2, \mu}) = \LCS_\infty(\B_{1,1, \mu})$, and finishes our proof.
\end{proof}

\begin{remark}
The Lie ring of $G$ (which identifies with the Lie ring of $\B_{2, \mu}$) can be completely computed, using Corollary~\ref{Lie_Klein_tau}. Namely, it identifies with $\Z^N \times (L \rtimes \Z)$, where $L = \Z^{l-1} \oplus (\Z/2)^{l-1} \oplus (\Z/2)^{l-1} \oplus \cdots$ and the action of the generator $t$ of $\Z$ on $L$ is via the degree-one map $\Z^{l-1} \twoheadrightarrow (\Z/2)^{l-1} \cong (\Z/2)^{l-1} \cong \cdots$. In other words, as a Lie ring, $\Lie(\B_{2, \mu})$ admits the presentation via generators $t, X_1, \ldots , X_{l-1}, Y_1, \ldots , Y_N$ and relations:
\[\begin{cases}
[Y_i, Y_j] = [Y_i, X_k] = [Y_i, t] = [X_k, X_l] = 0, \\
2[t, X_i] = 0.
\end{cases}\] 
\end{remark}

Let us now turn our attention to the case when there are two blocks of size $2$ in the partition.

\begin{proposition}\label{LCS_B22}
The LCS of $\B_{2,2}$ does not stop.
\end{proposition}

\begin{proof}
Since $\B_{2,2}$ surjects onto $\B_{2,2}(\S^2)$, this is a direct consequence of Proposition~\ref{LCS_B22(S2)} below, by an application of Lemma~\ref{lem:stationary_quotient}. Alternatively, one can adapt the proof of  Proposition~\ref{LCS_B22(S2)} to this case, getting that:
\[\B_{2,2}/\langle A_{12}, A_{34}, \LCS_2(\PB_4)\rangle \cong (\Z^2)^{\otimes 2} \rtimes (\Sym_2)^2,\]
where $\Sym_2$ acts on $\Z^2$ by permutation of the factors. Then the methods of Appendix~\ref{section_pstation_of_ext} can be used to compute completely the LCS of this group.
\end{proof}

\begin{corollary}\label{LCS_B22mu}
If at least two blocks of the partition $\lambda$ are of size $2$, then the LCS of $\B_{\lambda}$ does not stop.
\end{corollary}

\begin{proof}
Under the hypothesis, there is a surjection $\B_{\lambda} \twoheadrightarrow \B_{2,2}$. Thus, this corollary is obtained from a direct application of Lemma~\ref{lem:stationary_quotient}.
\end{proof}

\subsection{Blocks of size $1$ and $2$: study of $\B_{1,2}$}

We use that $\B_{1,2}$ is isomorphic to the Artin group of type $\B_2$, a classical fact of which we give a proof, for the sake of completeness. 

\begin{lemma}\label{LCS_B12}
The group $\B_{1,2}$ is isomorphic to the Artin group of type $B_2$, that is, to $G= \langle \sigma, x\ |\ (\sigma x)^2 = (x \sigma)^2 \rangle$. As a consequence, it is residually nilpotent, but not nilpotent. In particular, its LCS does not stop.
\end{lemma}

\begin{proof}
On the one hand, we can re-write the presentation of $G$ as:
\[G= \langle \sigma, x, y\ |\ (\sigma x)^2 = (x \sigma)^2,\ y = \sigma x \sigma^{-1} \rangle.\]
Then, modulo the second relation (which we conveniently rewrite $y \sigma = \sigma x$), the first one is equivalent to $(y \sigma)^2 = x (y \sigma) \sigma$, and in turn to $\sigma y \sigma^{-1} = y^{-1} x y$. Thus:
\[G= \langle \sigma, x, y\ |\ \sigma x \sigma^{-1} = y,\ \sigma y \sigma^{-1} = y^{-1} x y \rangle.\]

On the other hand, the projection $\B_{1,2} \twoheadrightarrow \B_2 \cong \Z$ splits, and its kernel identifies with the fundamental group of the disc minus two points, which is free on two generators $x$ and $y$. In fact, the corresponding action of $\B_2 \cong \langle \sigma \rangle$ on $\F_2$ is the usual Artin action. Which means exactly that the above relations are true in $\B_{1,2}$. Whence a surjection of $G$ onto $\B_{1,2}$, which induces a diagram with (split) short exact rows:
\[\begin{tikzcd} 
\langle x,y \rangle \ar[r, hook] \ar[d, two heads]
&G \ar[r, two heads] \ar[d]
&\Z \ar[d, two heads] \\
\F_2 \ar[r, hook]
&\B_{1,2} \ar[r, two heads]
&\B_2 \cong \Z.
\end{tikzcd}\]
Using the freeness of their target, one sees that the left and right vertical maps must be isomorphisms, hence so is the middle map, by the Five Lemma. The rest of the statement is then a reformulation of Proposition~\ref{Artin_B2}. 
\end{proof}

\begin{remark}
A more precise result is given by Proposition~\ref{Lie(Artin_B2)}, which describes the Lie ring of the group $G$ (hence of $\B_{1,2}$). Notice, however, that this difficult calculation is not needed if one only  wants to see that its LCS does not stop; see Remarks~\ref{LCS_Z/2*Z_weak} and~\ref{LCS_Artin_B2_weak}.
\end{remark}

\begin{corollary}\label{LCS_B12mu}
If the partition $\lambda$ has both a block of size $1$ and a block of size $2$, then the LCS of $\B_{\lambda}$ does not stop.
\end{corollary}

\begin{proof}
Apply Lemma~\ref{lem:stationary_quotient} to a surjection $\B_{\lambda} \twoheadrightarrow \B_{1,2}$.
\end{proof}

\begin{proof}[Proof of Theorem~\ref{thm:partitioned-braids}]
The first statement consists of Propositions \ref{LCS_stable_partitioned_braids}, \ref{LCS_B1mu} and \ref{LCS_B11mu}. The second one consists of the cases where $\lambda$ has at least three blocks of size $1$ (Lemma \ref{LCS_B111}), exactly one block of size $2$ together with blocks of size at least $3$ (Proposition \ref{LCS_B2mu}), at least two blocks of size $2$ (Corollary \ref{LCS_B22mu}) or at least one block of size $1$ and one block of size $2$ (Corollary \ref{LCS_B12mu}).
\end{proof}

\chapter{Virtual and welded braids}
\label{sec:virtual_welded_braids}

In this chapter, we study the LCS of the \emph{virtual}, \emph{welded} and \emph{extended welded} braid groups, which are generalisations of the classical braid groups. We first review the definitions and properties of these three groups. Then we discuss the notion of support for their elements. Finally, we come to the study of the LCS of each of these groups, whose description is given, respectively, in Propositions~\ref{LCS_of_vBn}, \ref{LCS_of_wBn} and \ref{LCS_of_exwBn}. We also study the LCS for the \emph{partitioned} versions of each of these groups; the corresponding results are summed up as Theorem~\ref{LCS_of_partitioned_v(w)B}. 

\section{Recollections}

Before diving into the study of their LCS, we review the definitions of the groups involved and recall their basic properties.

\begin{notation}\label{notation_commutation}
For clarity, in group presentations, we write $a\ \rightleftarrows\ b$ for the relation saying that $a$ and $b$ commute.
\end{notation}

\subsection{Virtual braids}

Virtual knots were introduced by Kauffman at the end of the 1990s \cite{Kauffman1}. Of course, the definition of virtual braids came with it, and the defining relations for the virtual braid group first appeared in \cite{Kauffman2}.

\subsubsection{Presentations by generators and relations}\label{par_pstation_vB} Virtual braids are best defined in the spirit of knot theory, using diagrams. However, the diagrammatic approach is nicely encoded by a group presentation, which makes for a convenient definition, which is both concise and well-suited to our purposes:

\begin{definition}[{\cite[\Spar 2]{Vershinin}}]
\label{def:presentation_virtual}
The virtual braid group $\vB_n$ is defined by generators $\sigma_{1},\ldots,\sigma_{n-1}, \tau_{1},\ldots,\tau_{n-1}$, and relations:
\begin{equation*}
\begin{cases}
\text{relations of the classical braid group $\B_n$ for the $\sigma_i$},\\
\text{relations of the symmetric group $\Sym_n$ for the $\tau_i$},\\
\text{mixed relations $\mathrm{(R1)}$ and $\mathrm{(R2)}$,}
\end{cases}
\end{equation*}
where the latter are:
\begin{equation}\label{relations virtual braid groups}
\begin{cases}
\mathrm{(R1)}\textrm{ }\sigma_{i}\ \rightleftarrows\ \tau_{k} & \text{if } \lvert i-k \rvert \geq 2;\\
\mathrm{(R2)}\textrm{ }\tau_{i}\tau_{i+1}\sigma_{i}=\sigma_{i+1}\tau_{i}\tau_{i+1} & \text{if } 1 \leq i \leq n-2.
\end{cases}
\end{equation}
\end{definition}

From the presentation, one sees immediately that there is a well-defined projection $\pi \colon \vB_n \twoheadrightarrow \Sym_n$ sending both $\sigma_i$ and $\tau_i$ to the transposition $(i, i+1) \in \Sym_n$. The kernel of this projection is denoted by $\vP_n$ and called the \emph{pure virtual braid group}. Furthermore, this projection is split, via $(i, i+1) \mapsto \tau_i$. Thus, the $\tau_i$ generate a copy of $\Sym_n$ inside $\vB_n$, and we get a semi-direct product decomposition:
\[\vB_n \cong \vP_n \rtimes \Sym_n.\]
Using the Reidemeister-Schreier method, Bardakov gave the following presentation of the finite-index subgroup $\vP_n$ of $\vB_n$:
\begin{proposition}[{\cite{Bardakov2004}}]
\label{vP_presentation}
The pure virtual braid group $\vP_n$ admits the presentation given by generators $\chi_{ij}$, for $1 \leq i \neq j \leq n$, subject to the following relations, for $i,j,k,l$ pairwise distinct:
\begin{equation}\label{relations pure virtual braid groups}
\begin{cases}
\mathrm{(vPR1)}\textrm{ }\chi_{ij}\ \rightleftarrows\ \chi_{kl} ;\\
\mathrm{(vPR2)}\textrm{ }\chi_{ij} \chi_{ik} \chi_{jk} = \chi_{jk} \chi_{ik} \chi_{ij}.
\end{cases}
\end{equation}
\end{proposition}
Here, the generator $\chi_{12}$ is $\tau_1 \sigma_1$, and $\chi_{ij} = \sigma \chi_{12} \sigma^{-1}$, where $\sigma \in \Sym_n \subset \vB_n$ is any permutation sending $i$ to $1$ and $j$ to $n$. In particular, it is easy to deduce from the relations that $\chi_{i, i+1} = \tau_i \sigma_i$.
\begin{corollary}\label{vPn^ab}
$\vP_n^{\ab}$ is free on the classes of the $\chi_{ij}$ for $1 \leq i \neq  j\leq n$.
\end{corollary}

\begin{corollary}\label{vBn_in_vBn+1}
The morphism $\iota \colon \vB_n \rightarrow \vB_{n+1}$ sending $\sigma_i$ to $\sigma_i$ and $\tau_i$ to $\tau_i$ is injective.
\end{corollary}

\begin{proof}
The morphism $\iota$ is clearly well-defined. To show that it is injective, we notice that it restricts to the morphism from $\vP_n$ to $\vP_{n+1}$ sending $\chi_{ij}$ to $\chi_{ij}$. Proposition~\ref{vP_presentation} ensures that this morphism has a well-defined retraction, which kills the $\chi_{i, n+1}$ and the $\chi_{n+1, i}$ and sends $\chi_{ij}$ to $\chi_{ij}$ if $i, j \leq n$. From this, one infers easily that $\iota$ itself is an injection from $\vB_n = \vP_n \rtimes \Sym_n$ into  $\vB_{n+1} = \vP_{n+1} \rtimes \Sym_{n+1}$.
\end{proof}

\begin{remark}\label{vBn_in_vBn+1_diag}
With the diagrammatic interpretation introduced just below, $\iota$ adds a strand away from the other ones, whilst the retraction appearing in the proof corresponds to forgetting the last strand of pure (virtual) braid diagrams.
\end{remark}

\subsubsection{A diagrammatic approach}\label{par_diag_vB} Let us now review how the above group presentation relates to the usual virtual braid diagrams.

A \emph{virtual braid diagram} is made of $n$ curves in the plane, which are monotonous with respect to one direction (which we choose to be the vertical one), with fixed endpoints and only a finite number of singularities, which are transverse double points, plus some crossing information at the double points. Namely, the different types of crossings are depicted in Figure \ref{fig:diagrams-crossings}: at the \emph{classical crossings}, one gives over/under information, and no crossing information is associated with \emph{virtual crossings}.

\begin{definition}
A \emph{diagrammatic virtual braid} is a virtual braid diagram up to planar isotopies and:
\begin{itemize}
\item the classical and virtual Reidemeister-II and Reidemeister-III moves from Figures~\ref{R-II_moves} and~\ref{R-III_moves};
\item the mixed Reidemeister-III move $\mathrm{(V)}$ from Figure~\ref{Virtual_move}.
\end{itemize}    
\end{definition}

The product of two diagrammatic virtual braids is induced by stacking diagrams on top of one another. This makes the set of diagrammatic virtual braids into a monoid. This monoid is easily seen to be generated by the elements $\sigma_i^{\pm 1}$ and $\tau_i$ depicted in Figure~\ref{Gen_of_vBn}: every virtual braid diagram  is isotopic to a stacking of diagrams with only one crossing. Then the Reidemeister-II moves (Figure~\ref{R-II_moves}) say that $\sigma_i^{-1}$ is indeed inverse to $\sigma_i$ and that $\tau_i^2 = 1$, so that the monoid is in fact a group. Finally, the relations of the presentation of $\vB_n$ correspond exactly to the Reidemeister-III moves (Figures~\ref{R-III_moves} and \ref{Virtual_move}) and to planar isotopy of diagrams. It is in fact not very hard to turn this argument into a real proof that the group of diagrammatic virtual braids identifies with $\vB_n$. 

\begin{figure}[tb]
	\centering
	\includegraphics[scale=0.17]{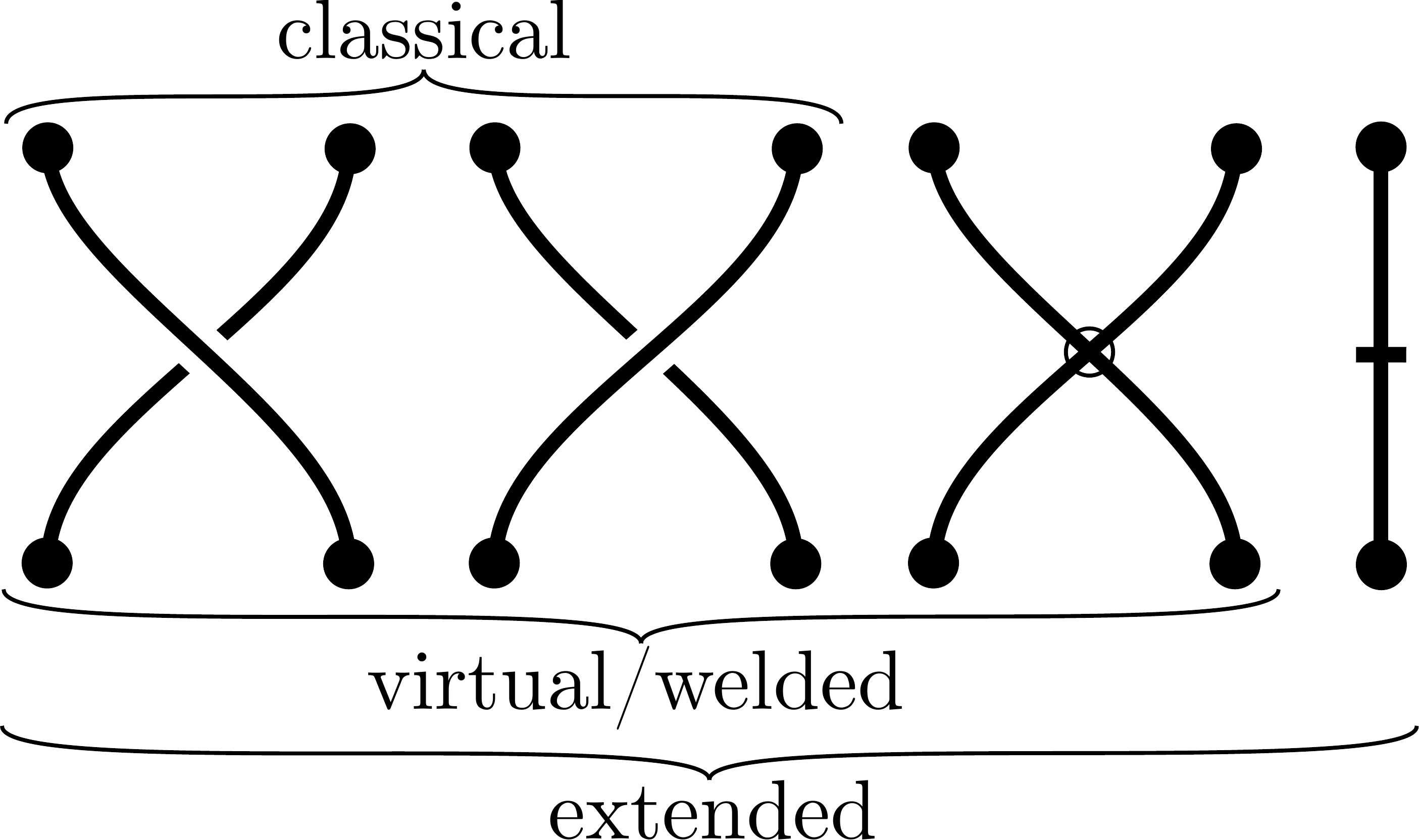}
	\caption{Crossings and decorations in braid diagrams.}
	\label{fig:diagrams-crossings}
\end{figure}

\begin{convention}
We fix the convention that reading a braid word from left to right corresponds to reading a braid diagram from top to bottom. 
\end{convention}

\begin{figure}[ht]
\centering
\begin{subfigure}[b]{.5\textwidth}
  \centering
  \includegraphics[width=.7\linewidth]{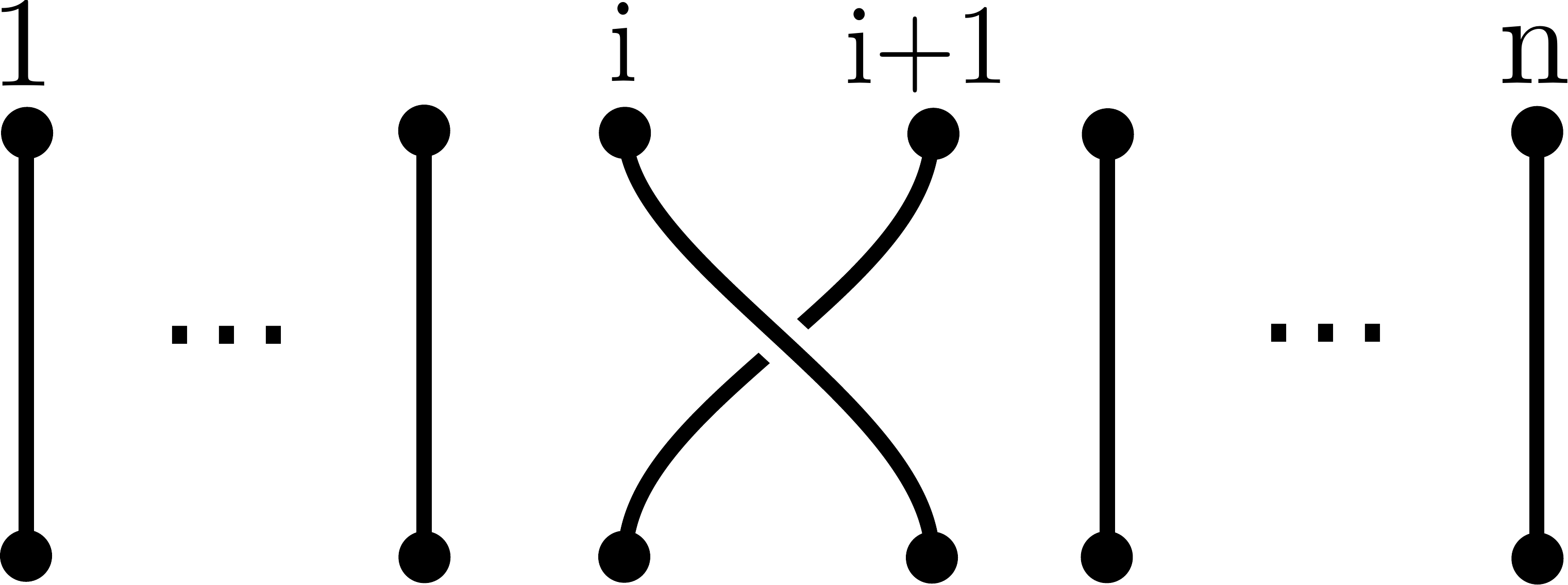}
  \caption{The generator $\sigma_i$}
\end{subfigure}%
\begin{subfigure}[b]{.5\textwidth}
  \centering
  \includegraphics[width=.7\linewidth]{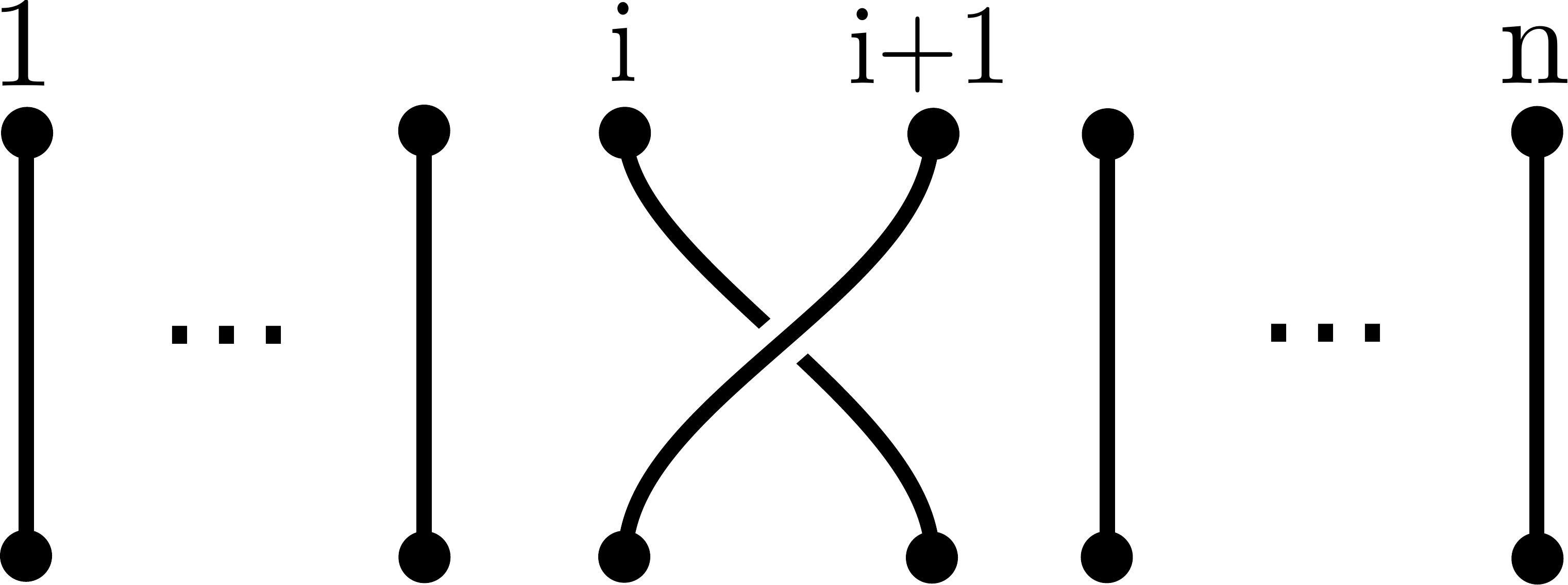}
  \caption{The generator $\sigma_i^{-1}$}
\end{subfigure}
\\[2ex]
\begin{subfigure}[b]{1\textwidth}
  \centering
  \includegraphics[width=.35\linewidth]{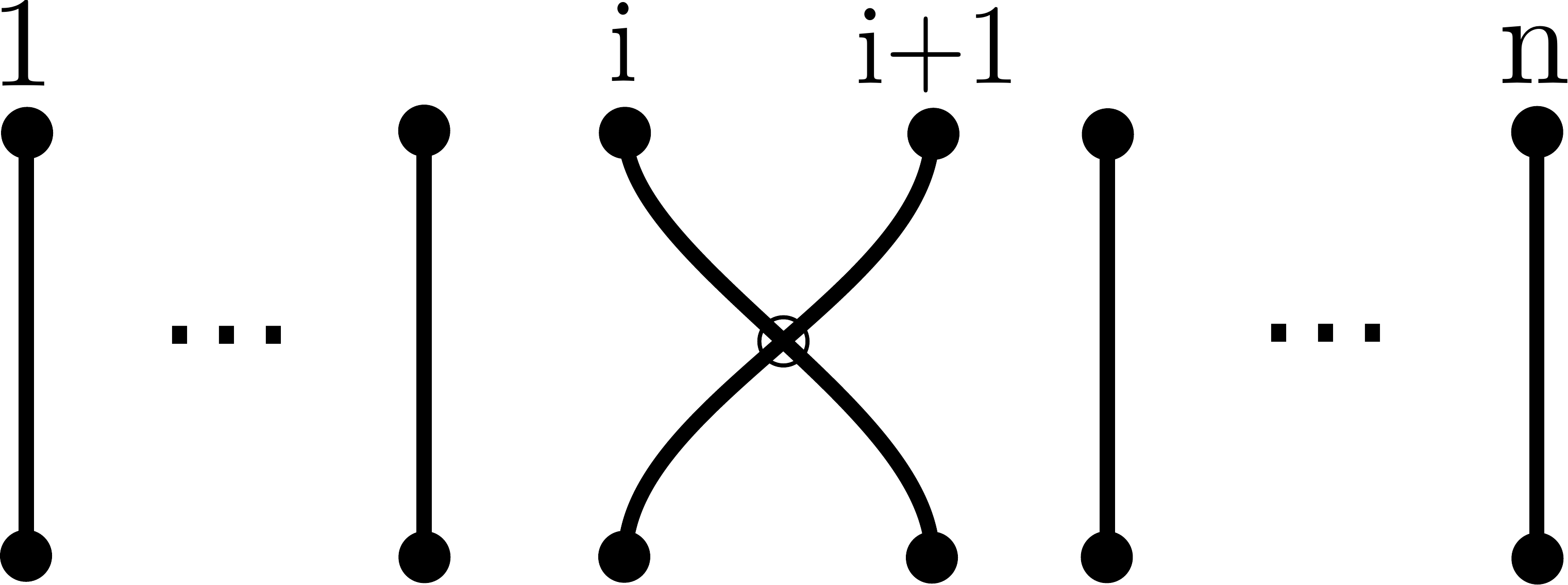}
  \caption{The generator $\tau_i$}
\end{subfigure}
\caption{Diagrams for generators of $\vB_n$.}
\label{Gen_of_vBn}
\end{figure}

\begin{figure}[ht]
\centering
\renewcommand{\thesubfigure}{$\mathrm{R_{II}}$}
\begin{subfigure}[b]{.5\textwidth}
  \centering
  \includegraphics[scale=.14]{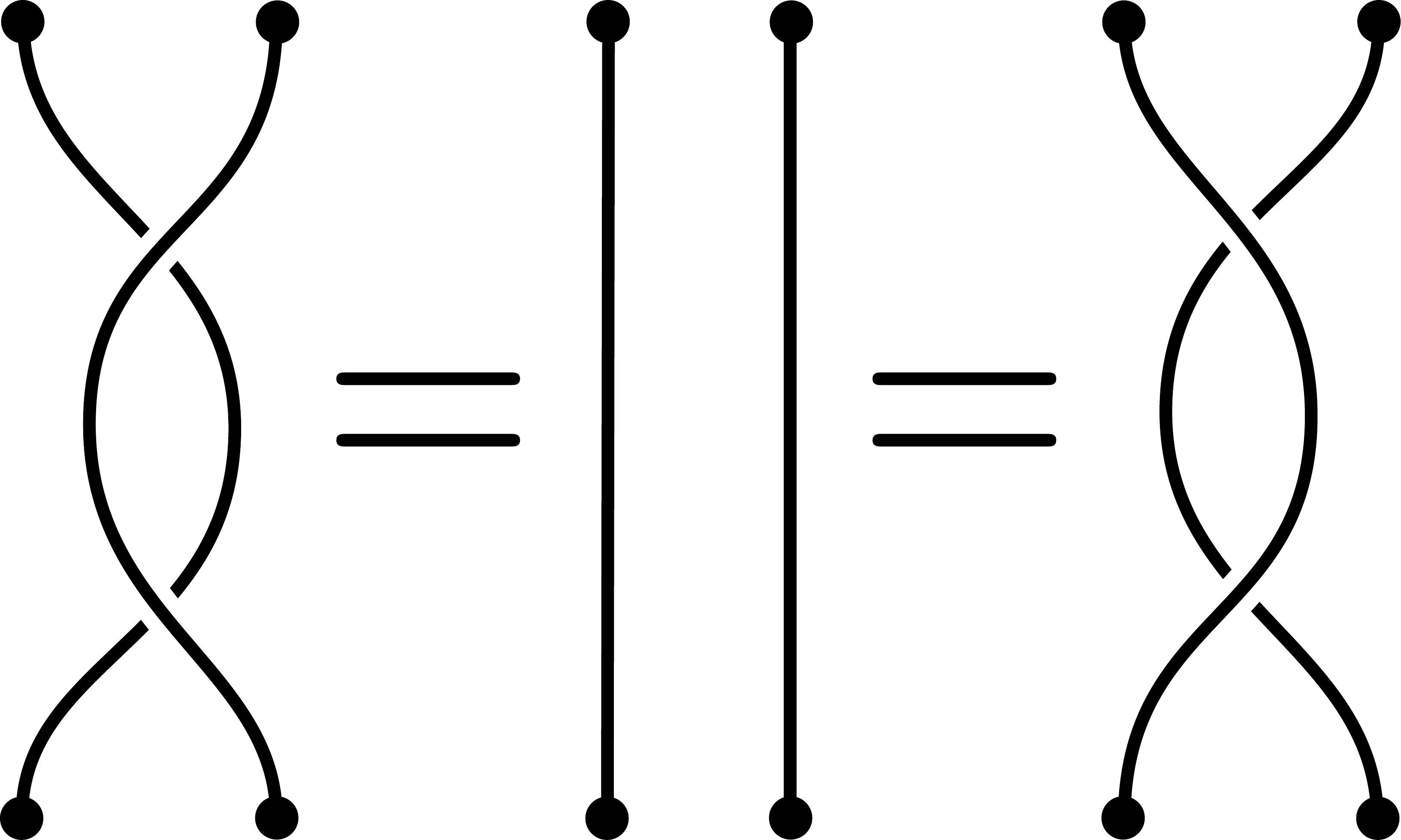}
  \caption{$\sigma_i^{-1} \sigma_i = 1 = \sigma_i \sigma_i^{-1}$}
\end{subfigure}%
\renewcommand{\thesubfigure}{$\mathrm{vR_{II}}$}%
\begin{subfigure}[b]{.5\textwidth}
  \centering
  \includegraphics[scale=.14]{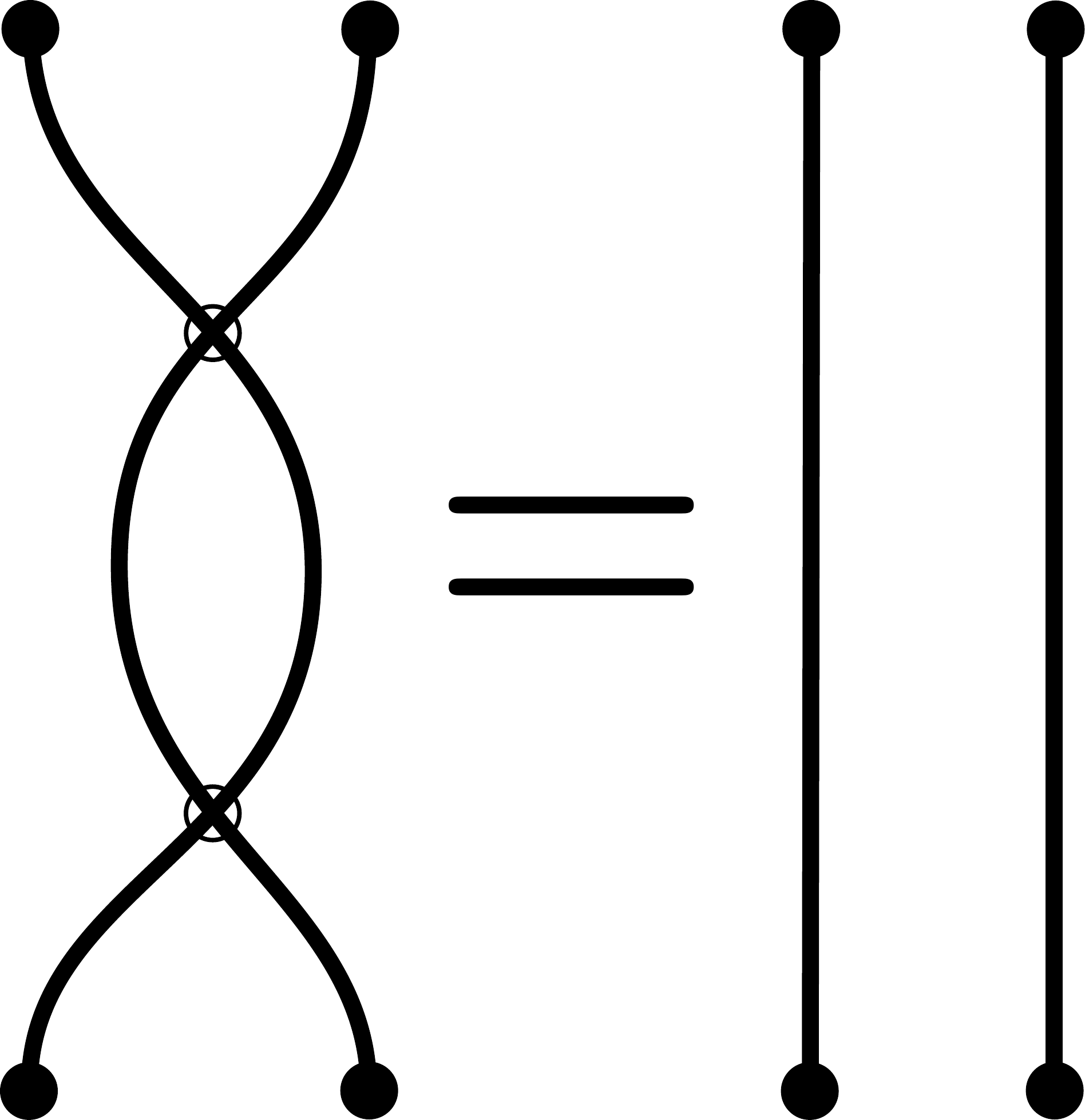}
  \caption{$\tau_i^2 = 1$}
\end{subfigure}
\caption{Reidemeister-II moves.}
\label{R-II_moves}
\end{figure}

\begin{figure}[ht]
\centering
\renewcommand{\thesubfigure}{$\mathrm{R_{III}}$}
\begin{subfigure}[b]{.5\textwidth}
  \centering
  \includegraphics[scale=.14]{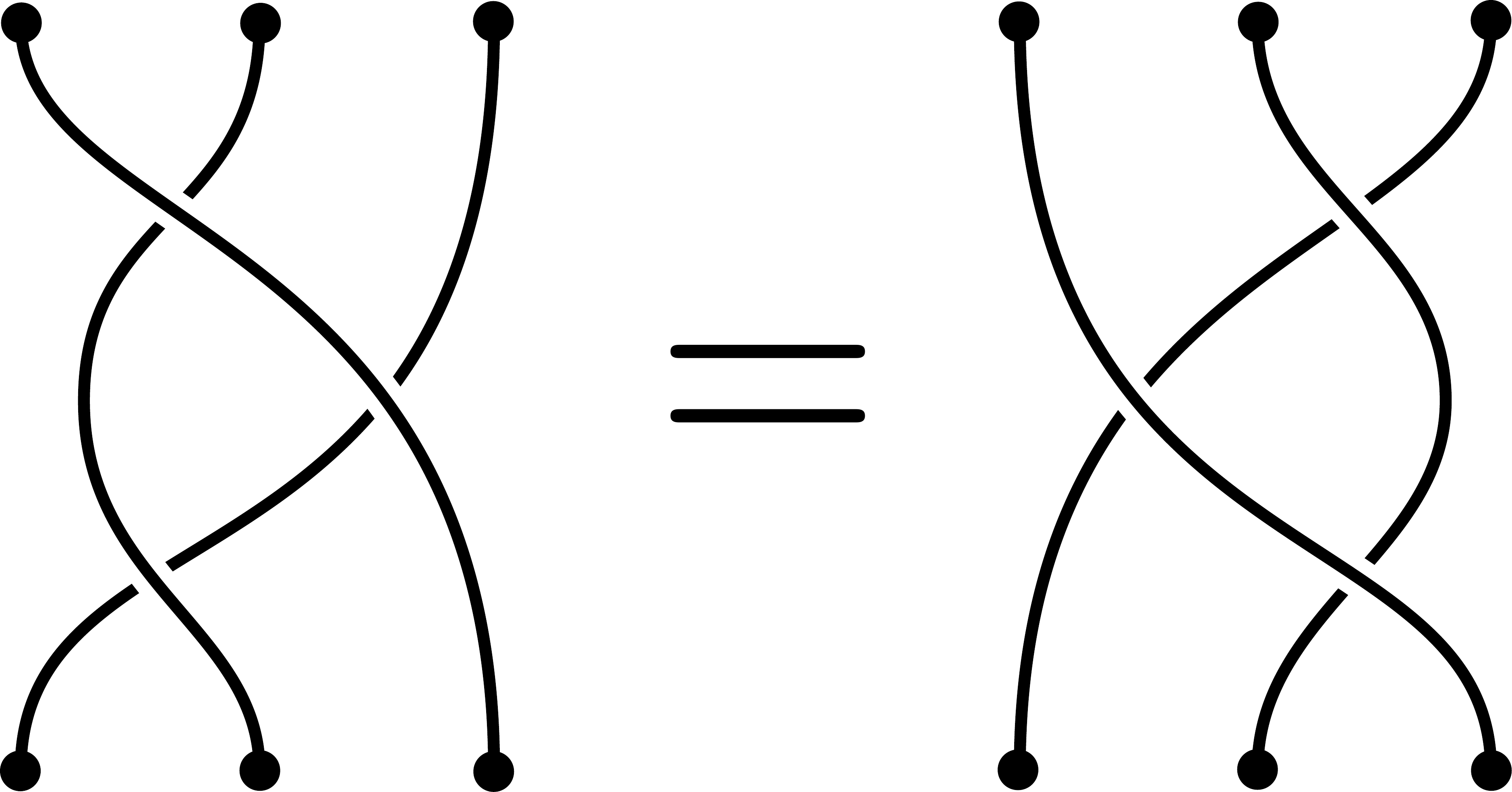}
  \caption{$\sigma_i \sigma_{i+1} \sigma_i = \sigma_{i+1} \sigma_i \sigma_{i+1}$}
\end{subfigure}%
\renewcommand{\thesubfigure}{$\mathrm{vR_{III}}$}%
\begin{subfigure}[b]{.5\textwidth}
  \centering
  \includegraphics[scale=.14]{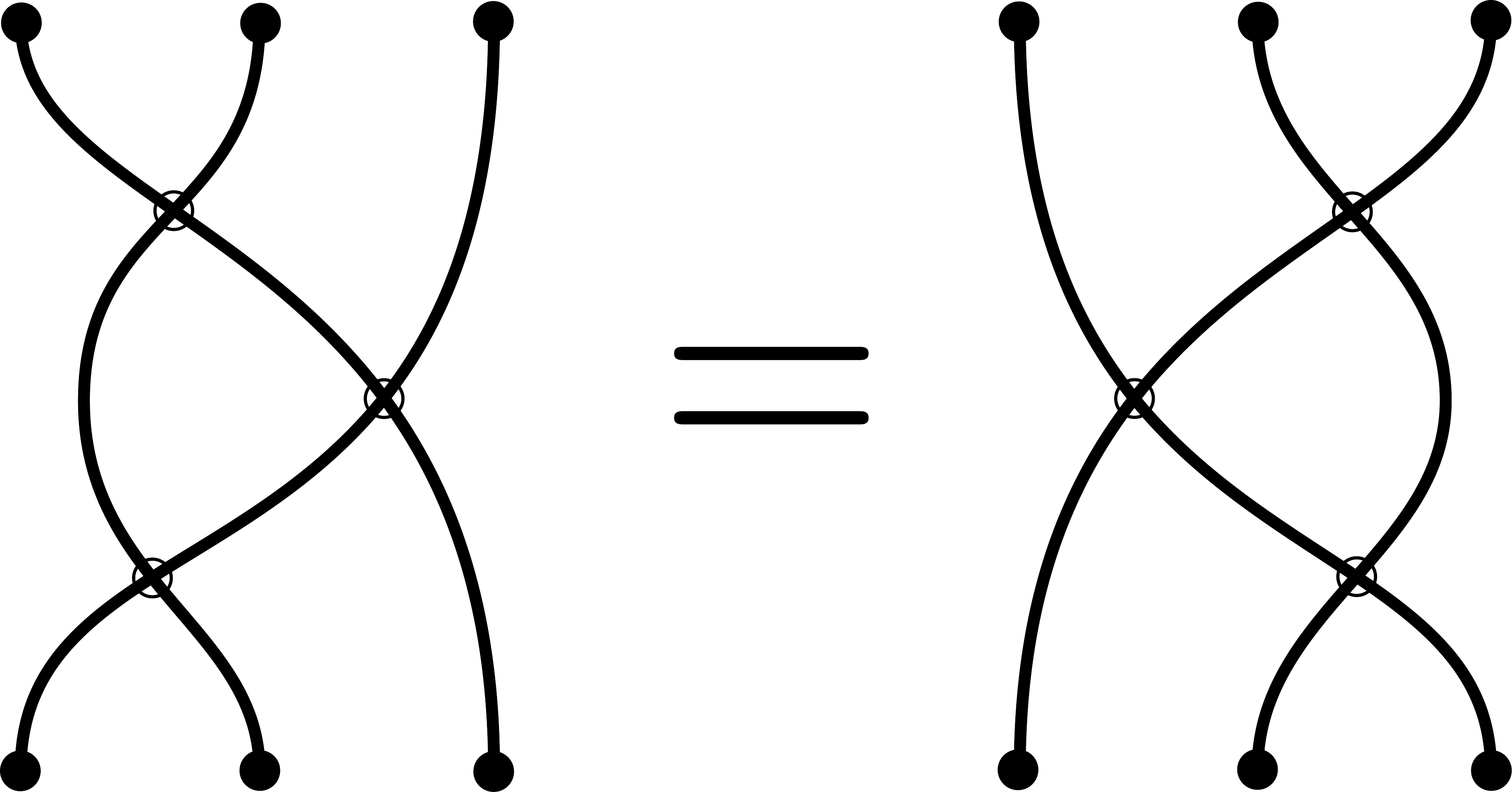}
  \caption{$\tau_i \tau_{i+1} \tau_i = \tau_{i+1} \tau_i \tau_{i+1}$}
\end{subfigure}
\caption{Reidemeister-III moves.}
\label{R-III_moves}
\end{figure}

\begin{figure}[ht]
\centering
\renewcommand{\thesubfigure}{$\mathrm{V}$}
\begin{subfigure}[b]{.5\textwidth}
  \centering
  \includegraphics[scale=.14]{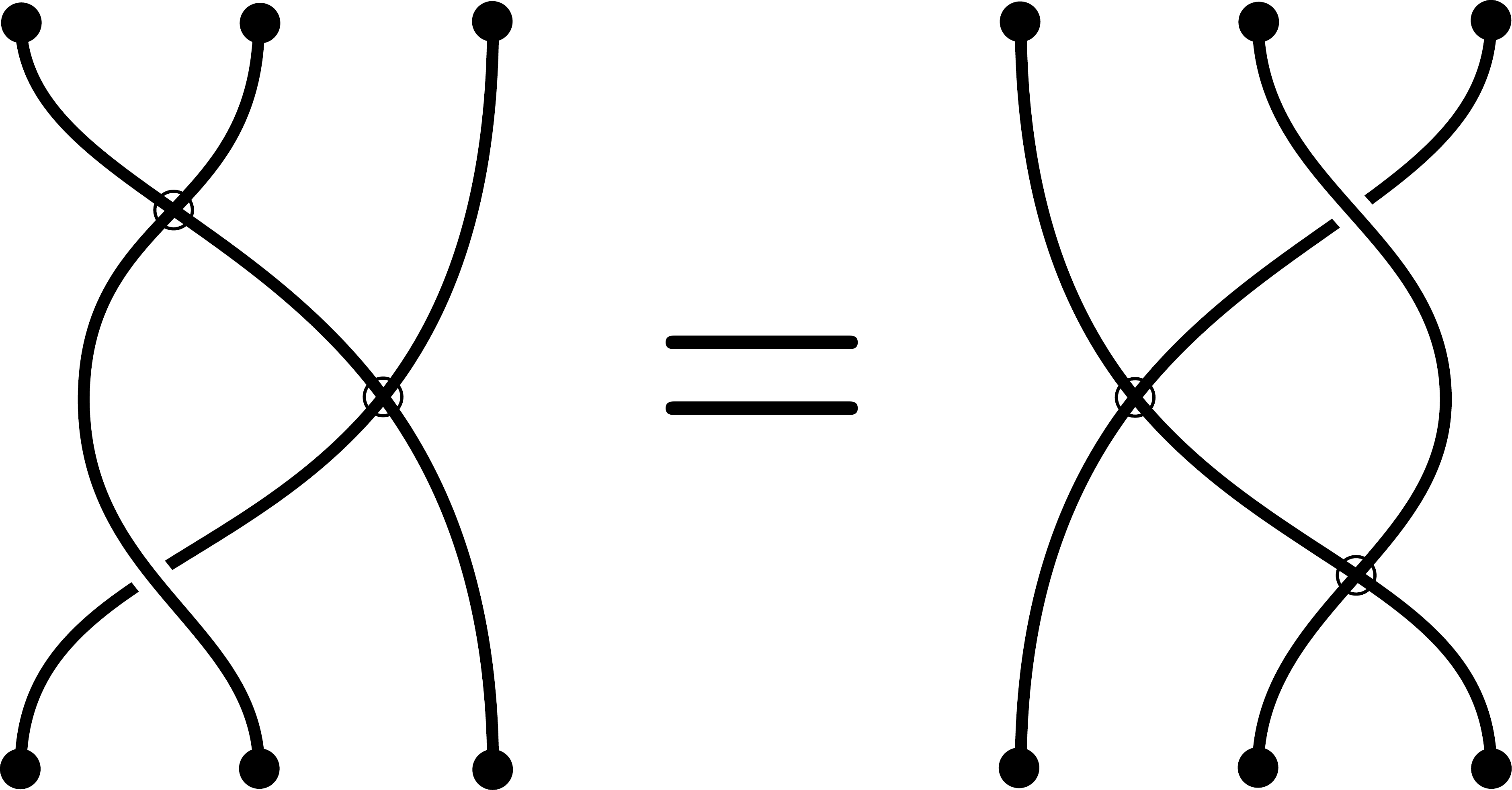}
  \caption{$\tau_i \tau_{i+1} \sigma_i = \sigma_{i+1} \tau_i \tau_{i+1}$}
\end{subfigure}
\caption{The mixed Reidemeister-III move.}
\label{Virtual_move}
\end{figure}

\begin{remark}\label{symmetries_of_v_moves}
There are variants of the above Reidemeister moves, which are consequences of the ones that we wrote. For example, the variant of $\mathrm{(V)}$ where the classical crossing is reversed follows from $\mathrm{(V)}$ and two instances of $\mathrm{(R_{II})}$. Note that such deductions can also be seen as implications between relations in the group $\vB_n$. In fact, the set of moves that do not change the virtual braid is stable under some symmetries that correspond to transformations of the group. Namely, it is stable under top-bottom mirror image (which corresponds to passing to the inverses in $\vB_n$), under left-right mirror image (which corresponds to the automorphism of $\vB_n$ defined by $\sigma_i \mapsto \sigma_{n-i}$ and $\tau_i \mapsto \tau_{n-i}$) and under flipping classical crossings (which corresponds to the automorphism sending $\sigma_i$ to $\sigma_i^{-1}$ and fixing the $\tau_i$).
\end{remark}

\begin{remark}\label{rmk:detour-move}
The move $\mathrm{(V)}$ from Figure~\ref{Virtual_move} and its variants alluded to in the previous remark can be encompassed in a general~\emph{detour move} that says that a piece of strand with only virtual crossings can be changed into any other such piece of strand with the same endpoints.
\end{remark}

\begin{figure}[ht]
\centering
\begin{subfigure}[b]{.5\textwidth}
  \centering
  \includegraphics[scale=.13, left]{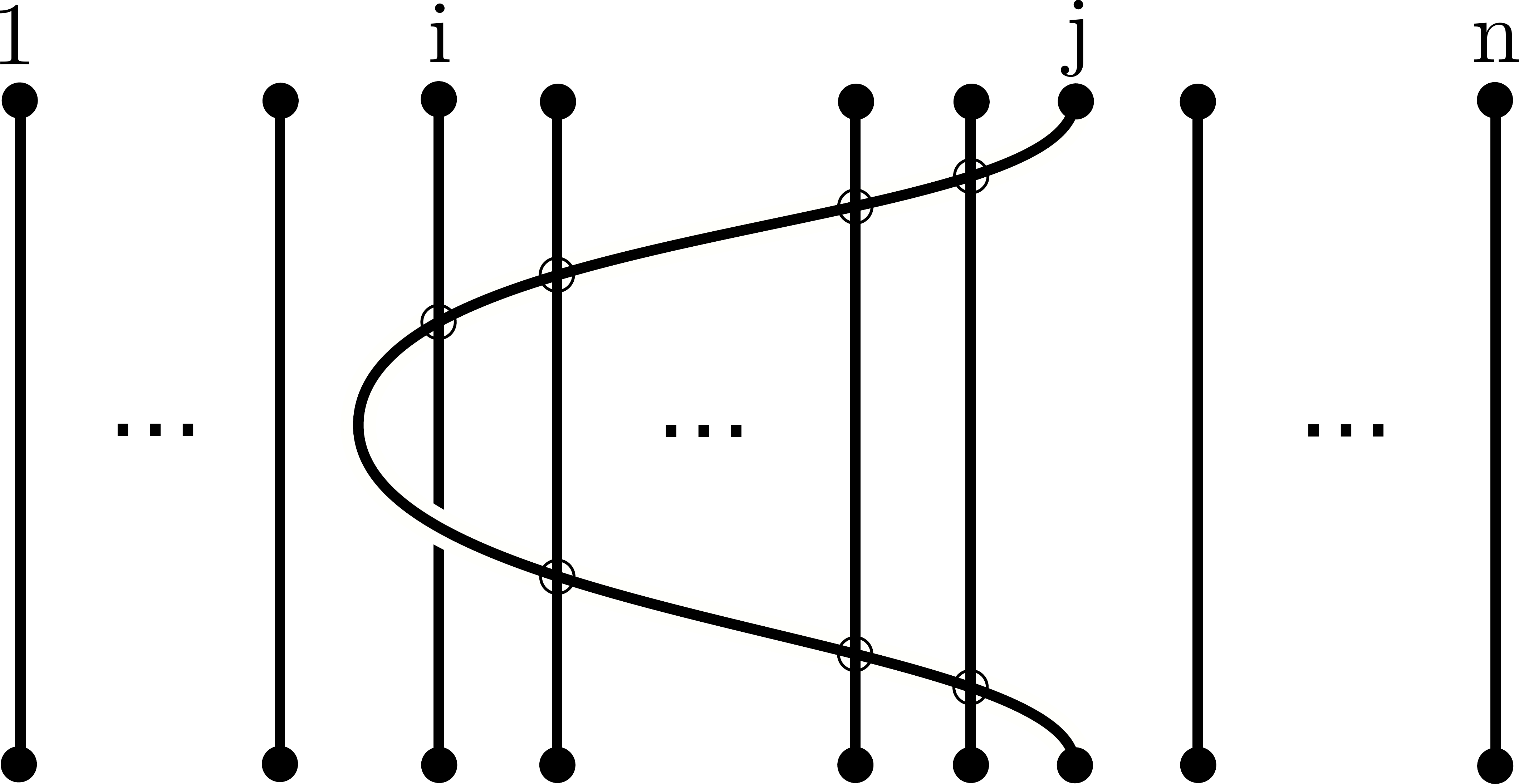}
  \caption{$\chi_{ij}$ if $i < j$}
\end{subfigure}%
\begin{subfigure}[b]{.5\textwidth}
  \centering
  \includegraphics[scale=.13, right]{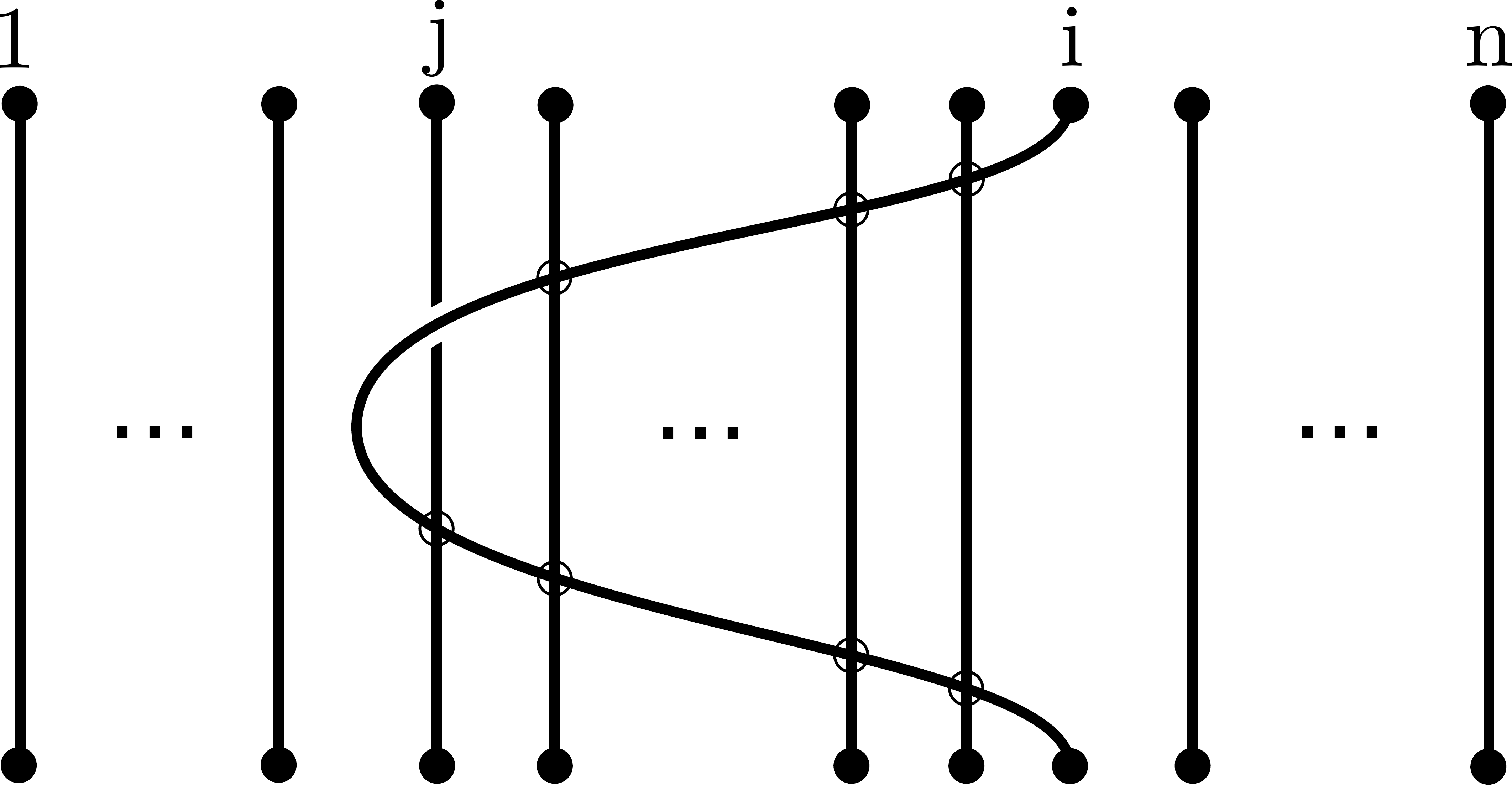}
  \caption{$\chi_{ij}$ if $i > j$}
\end{subfigure}
\caption{The standard generators $\chi_{ij}$ of pure virtual and welded braid groups. Notice that there are many other possible diagrams for them. For example, when $i<j$, any diagram with only one classical crossing where the $j$-th strand passes over the $i$-th strand from the left is a diagram of $\chi_{ij}$; see Remark \ref{rmk:detour-move}.}
\label{fig:pure-virtual-welded-braid-generators}
\end{figure}

\subsection{Welded braids} Welded braids were introduced at roughly the same time as virtual braids, by Fenn, Rim\'{a}nyi and Rourke in \cite{fenn_rimanyi_rourke}, where it was shown that they could encode basis-conjugating automorphisms of the free group. However, they had appeared before under different guises. Their first appearance can be traced back to~\cite{Goldsmith},  where they appeared as motions of unknotted circles in space and the link with basis-conjugating automorphisms of the free group was made. As basis-conjugating automorphisms of the free group, they were also studied by McCool in~\cite{McCool}. More recent important contributions to their study include~\cite{BaezWiseCrans}, \cite{BrendleHatcher2013Configurationspacesrings} and the survey~\cite{Damiani} by Damiani, which gives a nice overview of the different interpretations of this group. 

\subsubsection{Presentations by generators and relations}\label{par_pstation_wB} We can give a definition of the welded braid group similar to Definition~\ref{def:presentation_virtual} above. In fact, we can see it as a quotient of the virtual braid group by one additional relation:
\begin{definition}[{\cite{fenn_rimanyi_rourke}}]\label{def:presentation_welded}
The welded braid group $\wB_n$ is the quotient of the virtual braid group $\vB_n$ by the additional mixed relation:
\begin{equation}\label{relations welded braid groups}
\mathrm{(R3)}\textrm{ }\sigma_{i}\sigma_{i+1}\tau_{i}=\tau_{i+1}\sigma_{i}\sigma_{i+1} \ \ \text{if } 1 \leq i \leq n-2.
\end{equation}
\end{definition}

The projection $\pi \colon \vB_n \twoheadrightarrow \Sym_n$ from \Spar\ref{par_pstation_vB} factorises through the quotient by this relation. In other words, the assignment $\sigma_i, \tau_i \mapsto (i, i+1) \in \Sym_n$ also defines a projection from $\wB_n$ onto $\Sym_n$. The kernel of this projection is denoted by $\wP_n$ and called the \emph{pure welded braid group}. As in the case of virtual braids, $(i, i+1) \mapsto \tau_i$ defines a section of this projection, so that we get a semi-direct product decomposition:
\[\wB_n \cong \wP_n \rtimes \Sym_n.\]
By a slight adaptation of the Reidemeister-Schreier method from~\cite{Bardakov2004} (or an argument very similar to the proof of~\cite[Th.~2.7]{BBD}), one gets a presentation of $\wP_n$. In fact, the additional relation $\mathrm{(R3)}$ adds one relation to Bardakov's presentation of $\vP_n$ (Proposition~\ref{vP_presentation}), saying that $\chi_{ik}$ and $\chi_{jk}$ commute with each other. Modulo this relation, one can re-write the relation $\mathrm{(vPR2)}$ as a commutation relation $\mathrm{(wPR3)}$, to get:
\begin{proposition}
\label{McCool_presentation}
The pure welded braid group $\wP_n$ admits the presentation given by generators $\chi_{ij}$, for $1 \leq i \neq j \leq n$, subject to the following relations, for $i,j,k,l$ pairwise distinct:
\begin{equation}\label{relations pure welded braid groups}
\begin{cases}
\mathrm{(wPR1)}\textrm{ }\chi_{ij}\ \rightleftarrows\ \chi_{kl} ;\\
\mathrm{(wPR2)}\textrm{ }\chi_{ik}\ \rightleftarrows\ \chi_{jk} ;\\
\mathrm{(wPR3)}\textrm{ }\chi_{ij} \ \rightleftarrows\  \chi_{ik}\chi_{jk}.
\end{cases}
\end{equation}
\end{proposition}
This presentation first appeared in the work of McCool \cite{McCool}, which uses geometric group theory to show that it defines the subgroup of basis-conjugating automorphisms of the free group (see also~\Spar\ref{section_Aut(Fn)} below). 

We can deduce directly from this presentation that the abelianisation of $\wP_n$ is the same as the abelianisation of $\vP_n$ from Corollary~\ref{vPn^ab}:
\begin{corollary}\label{v(w)Pn^ab}
The canonical projection $\vP_n \twoheadrightarrow \wP_n$ induces an isomorphism $\vP_n^{\ab} \cong \wP_n^{\ab}$; both are free abelian on the classes $\overline \chi_{ij}$ for $1 \leq i \neq  j\leq n$.
\end{corollary}

\begin{remark}
Anticipating a bit the sequel, Corollary~\ref{v(w)Pn^ab} can also be proved directly, without using the presentations from Theorems~\ref{vP_presentation} and~\ref{McCool_presentation}: showing that the $\chi_{ij}$ generate $\vP_n$ (hence also $\wP_n$) is not difficult. Then, the linear independence of the classes of the $\chi_{ij}$ in $\wP_n^{\ab}$ (hence also in $\vP_n^{\ab}$) readily follows from the following fact: for any choice of $k \neq l$, there is a projection $\pi_{k,l} \colon \wP_n^{\ab} \twoheadrightarrow \wP_2^{\ab} \cong \Z^2$ (using $\wP_2 \cong \Inn(\F_2) \cong\F_2$ below) killing all the $\chi_{ij}$, save $\chi_{kl}$ and $\chi_{lk}$, which are sent to a basis of $\Z^2$. Such a projection is induced by the map $\wP_n \twoheadrightarrow \wP_2$ obtained by forgetting all the strands except the $j$-th and the $k$-th ones.
\end{remark}

\begin{remark}
Proposition~\ref{McCool_presentation} can be used to show that the morphism from $\wB_n$ to $\wB_{n+1}$ sending $\sigma_i$ to $\sigma_i$ and $\tau_i$ to $\tau_i$ is injective, by checking that its restriction to welded pure braids is a split injection, as in the proof of Corollary~\ref{vBn_in_vBn+1}. This can also be deduced from the interpretation of welded braids as automorphisms of free groups (Corollary~\ref{wBn_in_wBn+1}).
\end{remark}

\subsubsection{A diagrammatic approach}\label{par_diag_wB} Welded braids can be represented as diagrams in exactly the same way as virtual braids above. The only thing that changes is that we add one Reidemeister move, called the \emph{welded move} $\mathrm{(W)}$, saying that a strand may pass \emph{above} a virtual crossing (Figure~\ref{Welded_move}). But strands may not pass below virtual crossings, which should be thought of as \emph{welded} to the plane below the diagram, whence the name of these braids.

\begin{remark}
It is well-known \cite{Nelson, Kanenobu} that adding the move $(F)$ to welded knot theory makes the resulting knot theory trivial. However, the resulting link theory (the theory of \emph{fused links} from~\cite{KL}) is not quite trivial; this is connected to the fact that the quotient of $\wB_n$ by the relation $(F)$ is not trivial either. It is however much less rich (and much easier to understand) than $\wB_n$. This quotient is studied in~\cite{BBD}, where its structure is completely determined.
\end{remark}

\begin{remark}\label{symmetries_of_w_moves}
The variants of the move $\mathrm{(W)}$ from Figure~\ref{Welded_move} obtained by mirror images also hold, with the same arguments as in Remark~\ref{symmetries_of_v_moves}. However, the variant obtained by changing the crossings does not hold, because its vertical mirror image would then hold too, and this would be the forbidden move $\mathrm{(F)}$, which does not hold (see Remark~\ref{rk_forbidden_move} below).
\end{remark}

\begin{figure}[ht]
\centering
\renewcommand{\thesubfigure}{$\mathrm{W}$}
\begin{subfigure}[b]{.5\textwidth}
  \centering
  \includegraphics[scale=.14]{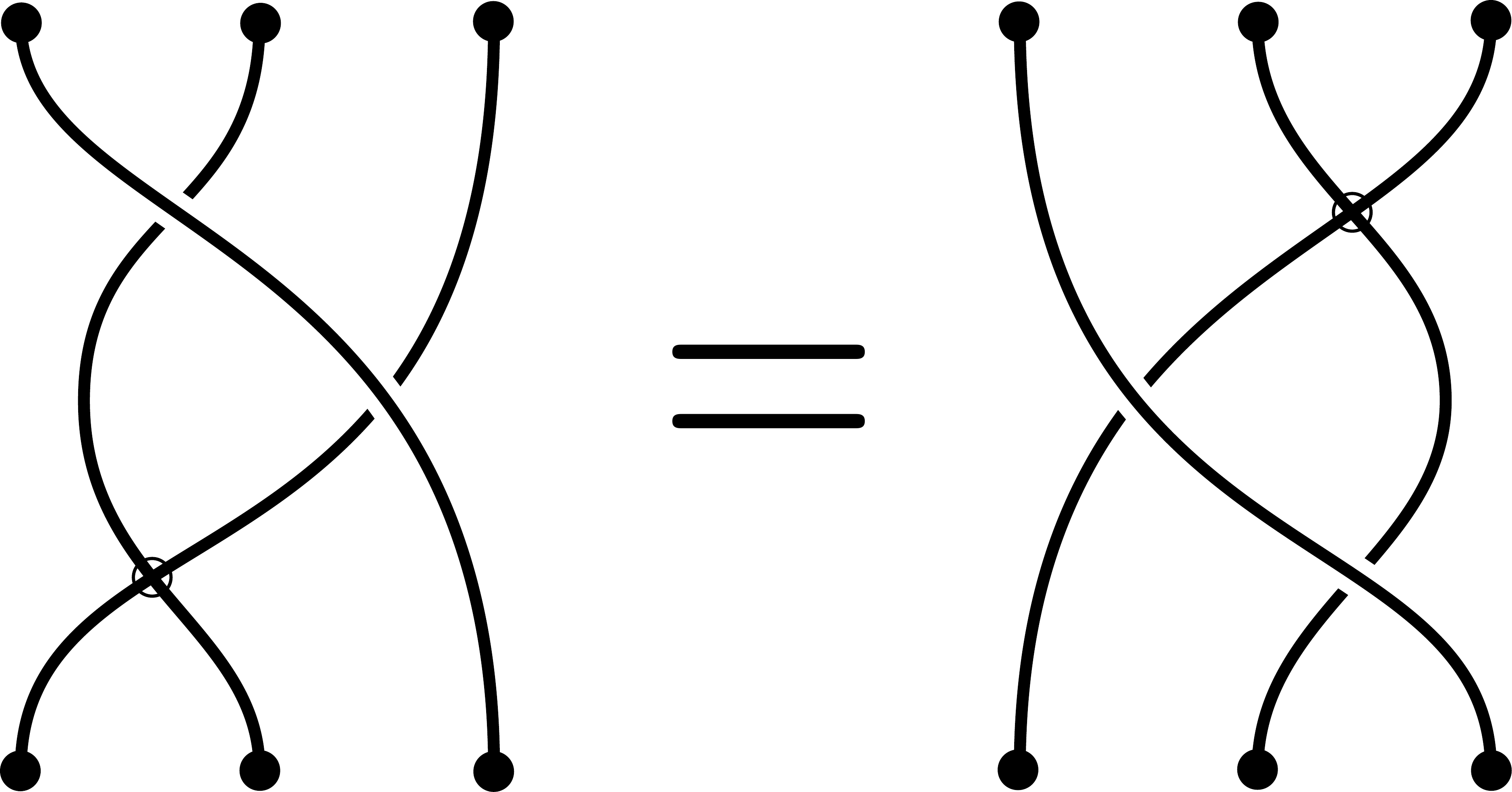}
  \caption{$\sigma_{i}\sigma_{i+1}\tau_{i} = \tau_{i+1}\sigma_{i}\sigma_{i+1}$}
\end{subfigure}%
\renewcommand{\thesubfigure}{$\mathrm{F}$}%
\begin{subfigure}[b]{.5\textwidth}
  \centering
  \includegraphics[scale=.14]{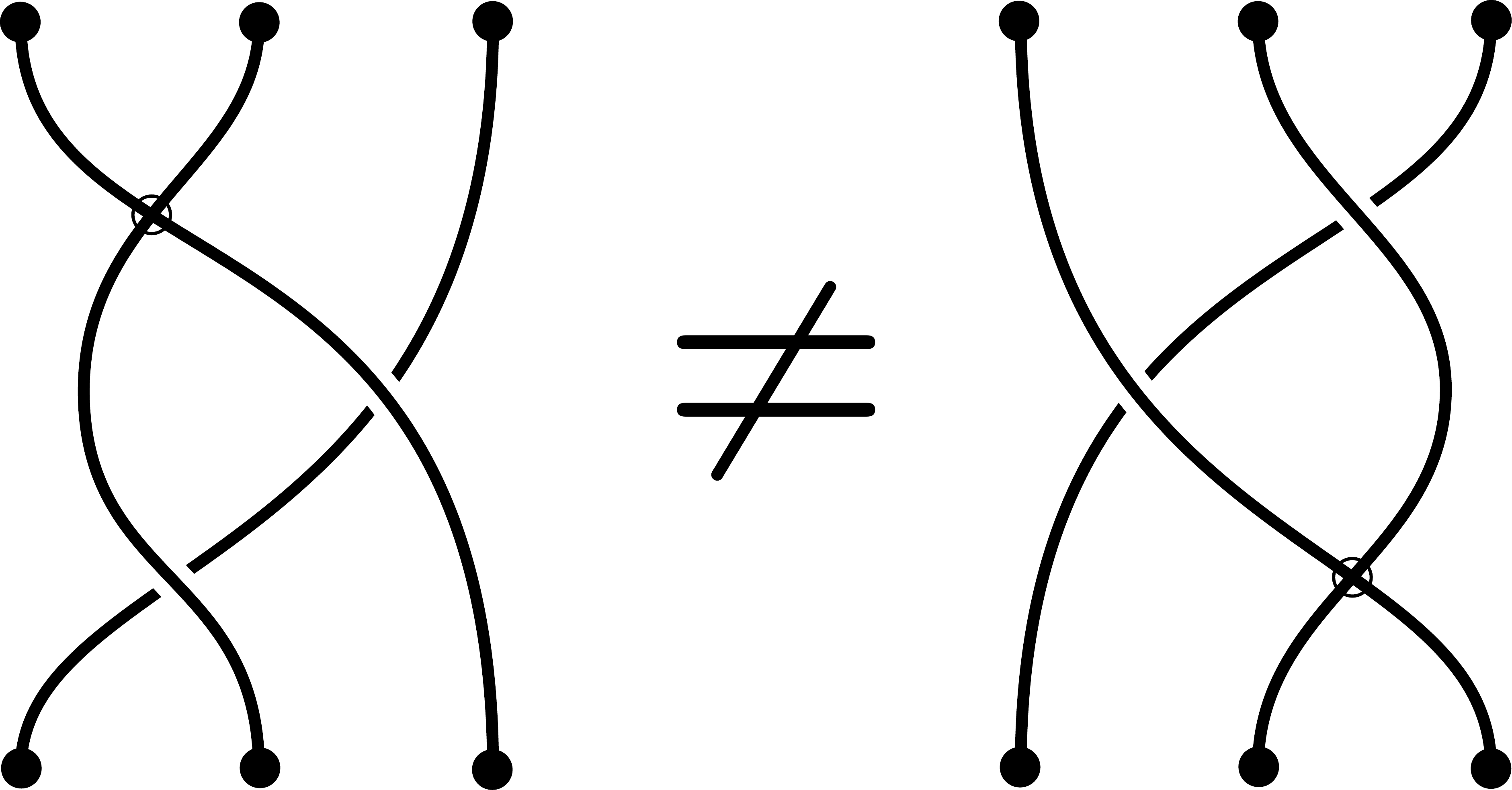}
  \caption{$\tau_i\sigma_{i+1}\sigma_i \neq  \sigma_{i+1}\sigma_i\tau_{i+1}$}
\end{subfigure}
\caption{The welded move (W) and the forbidden move (F).}
\label{Welded_move}
\end{figure}

\subsubsection{Welded braids as automorphisms of the free group}\label{section_Aut(Fn)} One of the important features of welded braids, which we will use a lot in the sequel, is their interpretation as basis-conjugating automorphisms of the free group. 

We use the following notations for automorphisms of the free group $\F_n$ on $n$ generators $x_1, \ldots , x_n$:
\[\tau_i \colon
\left\{
\begin{array}{lll}
x_i &\longmapsto &x_{i+1} \\
x_{i+1} &\longmapsto &x_i \\
\end{array}
\right. \ \ \ \text{and}\ \ \ \ 
\sigma_i \colon
\left\{
\begin{array}{lll}
x_i &\longmapsto &x_i x_{i+1} x_i^{-1} \\
x_{i+1} &\longmapsto &x_i \\
\end{array}
\right.
\]
where it is implicit that the $x_k$ are fixed, for $k \neq i, i+1$.

\begin{remark}
Our definition of $\sigma_i$ disagrees with the one in \cite{Damiani}; instead, it corresponds to the \emph{inverse} of the element called $\sigma_i$ by Damiani. This is because our $\mathrm{Aut}(\F_n)$ is the opposite group of the group that she calls $\mathrm{Aut}(\F_n)$, since we view composition of automorphisms of $\F_n$ from right to left, whereas she takes composition of automorphisms to go from left to right; see \cite[Rk.~6.3]{Damiani}. On the other hand, these considerations about the direction of composition do not affect the elements $\tau_i$ and $\rho_i$, since they are self-inverse. (We warn the reader, however, that our $\tau_i$ and $\rho_i$ correspond, respectively, to the elements denoted by $\rho_i$ and $\tau_i$ in \cite{Damiani}.)
\end{remark}

It is easy to check that these elements satisfy the defining relations for $\wB_n$. The morphism from $\wB_n$ to $\Aut(\F_n)$ that we obtain is in fact injective, as was shown by Fenn, Rim\'{a}nyi and Rourke towards the end of the century:
\begin{theorem}[{\cite{fenn_rimanyi_rourke}}]\label{wB_in_Aut(Fn)}
The morphism $\wB_n \rightarrow \Aut(\F_n)$ sending $\tau_i$ and $\sigma_i$ to the corresponding automorphisms is a well-defined isomorphism onto its image, which consists of all automorphisms of the form:
\[x_i \longmapsto w_i x_{\sigma(i)}w_i^{-1}, \ \ \ \text{with}\ w_1,\ldots, w_n \in\F_n\ \text{and}\ \sigma \in \Sym_n.\]
\end{theorem}

\begin{corollary}\label{wBn_in_wBn+1}
The morphism from $\wB_n$ to $\wB_{n+1}$ sending $\sigma_i$ to $\sigma_i$ and $\tau_i$ to $\tau_i$ is injective, and so is the morphism from $\B_n$ to $\wB_n$ (resp.~to $\vB_n$) sending $\sigma_i$ to $\sigma_i$.
\end{corollary}

\begin{proof}
All these morphisms are clearly well-defined. Theorem~\ref{wB_in_Aut(Fn)} identifies the first morphism with a restriction of the injection $\Aut(\F_n) \hookrightarrow \Aut(\F_{n+1})$ that extends automorphisms by making them fix $x_{n+1}$, so this morphism is injective. Then the classical Artin representation identifies $\B_n$ with the stabiliser of $x_1 \cdots x_n \in \F_n$ in $\wB_n \subset \Aut(\F_n)$ (see for instance \cite[Cor.~1.8.3]{Birman}), identifying the second morphism with an inclusion of subgroups of $\Aut(\F_n)$, so the second one is also injective. Finally, the morphism from $\B_n$ to $\vB_n$ can be composed with the quotient onto $\wB_n$ to give the previous morphism, so it must be injective too.
\end{proof}

\begin{remark}
Like in Remark~\ref{vBn_in_vBn+1_diag}, all the morphisms of Corollary~\ref{wBn_in_wBn+1} and of its proof admit obvious diagrammatic interpretations. 
\end{remark}

\begin{remark}\label{rk_forbidden_move}
The above morphism from $\wB_n$ to $\Aut(\F_n)$ can also be used to show that the forbidden move $\mathrm{(F)}$ from Figure~\ref{Welded_move} (or, equivalently, the associated relation) does not hold in $\wB_n$: it suffices to compute the images of both sides in $\Aut(\F_n)$ and to check that they are distinct. This implies that the welded move $\mathrm{(W)}$ (also from Figure~\ref{Welded_move}) does not hold in $\vB_n$; in other words, $\wB_n$ is different from $\vB_n$. Indeed, one can see that relations $\mathrm{(W)}$ and $\mathrm{(F)}$ are exchanged by the last automorphism of $\vB_n$ defined in Remark~\ref{symmetries_of_v_moves} (flipping classical crossings), followed by the previous one (left-right mirror image). Thus $\mathrm{(W)}$ holds in $\vB_n$ if and only if $\mathrm{(F)}$ does. And the latter does not, since it does not hold in the quotient $\wB_n$.
\end{remark}

One can check easily that the morphism from Theorem~\ref{wB_in_Aut(Fn)} identifies the generators $\chi_{ij}$ of the pure welded braid group $\wP_n$ with the automorphisms:
\[\chi_{ij} \colon x_k \mapsto 
\begin{cases}
x_j x_i x_j^{-1} &\text{if } k = i \\
x_k &\text{else.}
\end{cases}\]
Given the presentation of $\wP_n$ from Proposition~\ref{McCool_presentation}, one can see that Theorem~\ref{wB_in_Aut(Fn)} is in fact quite directly equivalent to McCool's theorem~\cite{McCool} saying that the subgroup of automorphisms of the form $x_i \mapsto w_i x_iw_i^{-1}$ admits the presentation from Proposition~\ref{McCool_presentation}. The equivalence between the two statements boils down to the fact that  $\wB_n/\wP_n \cong \Sym_n$ is isomorphic to the quotient between the two subgroups of $\Aut(\F_n)$, the latter being isomorphic to the group of permutation matrices via the canonical morphism $\Aut(\F_n) \rightarrow \Aut(\F_n^{\ab}) \cong GL_n(\Z)$. The same kind of argument will be used later to deduce from the above statement a similar one for extended welded braids (see the proof of Theorem~\ref{exwB_in_Aut(Fn)}).

\subsubsection{A geometric interpretation}\label{par_geom_wB} Another very important feature of welded braids is their geometric interpretation as \emph{loop braids}, which we describe now. These are motions of circles in $3$-dimensional space, and can also be seen as tube braids in $4$-dimensional space. For details, the reader in referred to~\cite{Damiani} and to the references therein (see in particular~\cite{Satoh, BrendleHatcher2013Configurationspacesrings}).

Let us consider the space $F_{n\S^1}(\D^3)$ of configurations of $n$ unlinked, oriented circles in the unit $3$-disc $\D^3$. By \emph{oriented} circles, we mean that a configuration is defined up to \emph{positive} re-parametrisation of the circles. The group $\Sym_n$ acts on it by permuting the copies of $\S^1$. Then the fundamental groups of $F_{n\S^1}(\D^3)$ and $F_{n\S^1}(\D^3)/\Sym_n$ identify respectively with $\wP_n$ and $\wB_n$. Precisely, the correspondence between loops in $F_{n\S^1}(\D^3)/\Sym_n$ and the above generators of $\wB_n$ is given as follows:
\begin{itemize}
\item $\tau_i$ corresponds to swapping the circles $i$ and $i+1$ without either of them passing through the other;
\item $\sigma_i$ corresponds to swapping the circles $i$ and $i+1$ while the $(i+1)$-th one goes through the $i$-th one in the direction determined by the orientation of the latter.
\end{itemize}

In fact, this identification is classically proven with automorphisms of free groups as an intermediate object. Namely, on the one hand, Theorem~\ref{wB_in_Aut(Fn)} identifies the combinatorial $\wB_n$ (with our definition via generators and relations, or via diagrams) to a subgroup of $\Aut(\F_n)$. On the other hand, there is a morphism from the geometric $\wB_n$ to $\Aut(\F_n)$, which is the Artin representation, adapted to this context. Precisely, it is a natural action of $\pi_1 \left(F_{n\S^1}(\D^3)/\Sym_n\right)$ on the fundamental group of $\R^3$ with $n$ unlinked circles removed. This morphism is in fact an isomorphism onto its image \cite[Th.~5.3]{Goldsmith}, which is the same subgroup of $\Aut(\F_n)$, whence the result.

\begin{remark}\label{precise_def_geom_wB}
We have been a little imprecise about what we mean by configurations of $n$ unlinked, oriented circles. One natural meaning is the space of all embeddings of $n$ disjoint copies of $S^1$ into the interior of $\bD^3$ whose image is an unlink, equipped with the Whitney topology, quotiented by the action of orientation-preserving diffeomorphisms of the circles. Another natural meaning is the subspace of this space consisting of all such embedded unlinks where the embedded circles are \emph{rigid}, in the sense that they are rotations, dilations and translations of the equator $S^1 \subset S^2 = \partial \bD^3$. The latter space has the advantage of being a finite-dimensional manifold, and it makes sense to speak of the unit vector normal to a circle of a configuration in this context. We will freely move between these two viewpoints in light of the theorem of Brendle and Hatcher \cite[Th.~4.1]{BrendleHatcher2013Configurationspacesrings}, which says that these two models are homotopy equivalent.
\end{remark}

\subsection{Extended welded braids} The group $\exwB_n$ of \emph{extended welded braids} is slightly bigger than $\wB_n$, which it contains as a finite-index subgroup. 

\subsubsection{Presentation by generators and relations}\label{par_pstation_exwB} Again, the definition of extended welded braids is motivated by both geometry and algebra, but we choose to define them using the group presentation. Precisely, it is a quotient of $\wB_n * (\Z/2)^n$ by some relations.

\begin{definition}
The extended welded braid group $\exwB_n$ admits a presentation with generators and relations of $\wB_n$ plus additional generators $\{\rho_{1},\ldots,\rho_n\}$ satisfying the relations of the abelian group $(\Z/2\Z)^n$ and five extra mixed relations:
\begin{equation}\label{relations extended welded braid groups}
\begin{cases}
\mathrm{(R4)}\textrm{ }\sigma_i\ \rightleftarrows\ \rho_k & \text{if } k \not\in \left\{i, i+1\right\} ;\\
\mathrm{(R5)}\textrm{ }\tau_i\ \rightleftarrows\ \rho_k & \text{if } k \not\in \left\{i, i+1\right\} ;\\
\mathrm{(R6)}\textrm{ }\sigma_i\rho_i=\rho_{i+1}\sigma_i & \text{if } i\in\left\{ 1,\ldots,n-1\right\};\\
\mathrm{(R7)}\textrm{ }\tau_i\rho_i=\rho_{i+1}\tau_i & \text{if } i\in\left\{ 1,\ldots,n-1\right\};\\
\mathrm{(R8)}\textrm{ }\sigma_i\rho_{i+1}=\rho_i\tau_i\sigma_i^{-1}\tau_i & \text{if } i\in\left\{ 1,\ldots,n-1\right\} .
\end{cases}
\end{equation}
\end{definition}

The assignments $\sigma_i, \tau_i \mapsto (0, (i, i+1))$ and $\rho_i \mapsto (e_i, id)$ define a projection from $\wB_n$ onto the hyperoctahedral group $W_n := (\Z/2) \wr \Sym_n = (\Z/2)^n \rtimes \Sym_n$. One can adapt again the use of the Reidemeister-Schreier method from \cite{Bardakov2004} to get a presentation of the kernel. In fact, no additional generator or relation appears in the process, compared to the case of welded braids: one gets exactly the presentation of Theorem~\ref{relations pure welded braid groups}, so this kernel is exactly $\wP_n$, and we get a decomposition:
\[\exwB_n \cong \wP_n \rtimes W_n.\]

\begin{remark}\label{wBn_in_exwBn}
This is one way to show that the obvious morphism from $\wB_n$ to $\exwB_n$ is injective using the group presentations (recall that $\wB_n \cong \wP_n \rtimes \Sym_n$). Alternatively, one could apply the Reidemeister-Schreier method directly to the image $G$ of $\wB_n$ in $\exwB_n$: using that relations $\mathrm{(R4)}$ to $\mathrm{(R8)}$ allow one to ``push the $\rho_i$ to the left" in words, and using the above projection to $W_n$ to distinguish them, one can see that the words $\rho_1^{\epsilon_1} \cdots \rho_n^{\epsilon_n}$ (with $\epsilon_i \in \{0,1\}$) are a set of representatives for the quotient $\exwB_n/G$ (which is not a group). Moreover, this set of representatives satisfies the Schreier condition, and by a very easy application of the Reidemeister-Schreier method, one gets that $G$ has the same presentation as $\wB_n$.
\end{remark}

\subsubsection{A diagrammatic approach}\label{par_diag_exwB} Extended welded braids can be represented by the same diagrams as $\wB_n$, with additional decorations (marked points) on the strands (see Figure~\ref{fig:diagrams-crossings}). To the Reidemeister moves listed above for $ \wB_n$, one adds a list of moves saying how the decorations interact with other decorations and with the different types of crossings: the four \emph{extended moves} depicted in Figures~\ref{Extended_moves} and~\ref{Extended_move_4}. Exactly the same reasoning as above shows that the group of such diagrams up to planar isotopy and this list of moves is isomorphic to $\exwB_n$, where the generator~$\rho_i$ corresponds to the diagram from Figure~\ref{Gen_exwB}.

\begin{remark}
As above, the set of moves that do not change the underlying extended welded braids is stable by taking mirror images, for the same reasons as in Remark~\ref{symmetries_of_w_moves} (the automorphism encoding left-right mirror image being extended by $\rho_i \mapsto \rho_{n + 1 -i}$). 
\end{remark}

\begin{figure}[tb]
	\centering
	\includegraphics[scale=0.15]{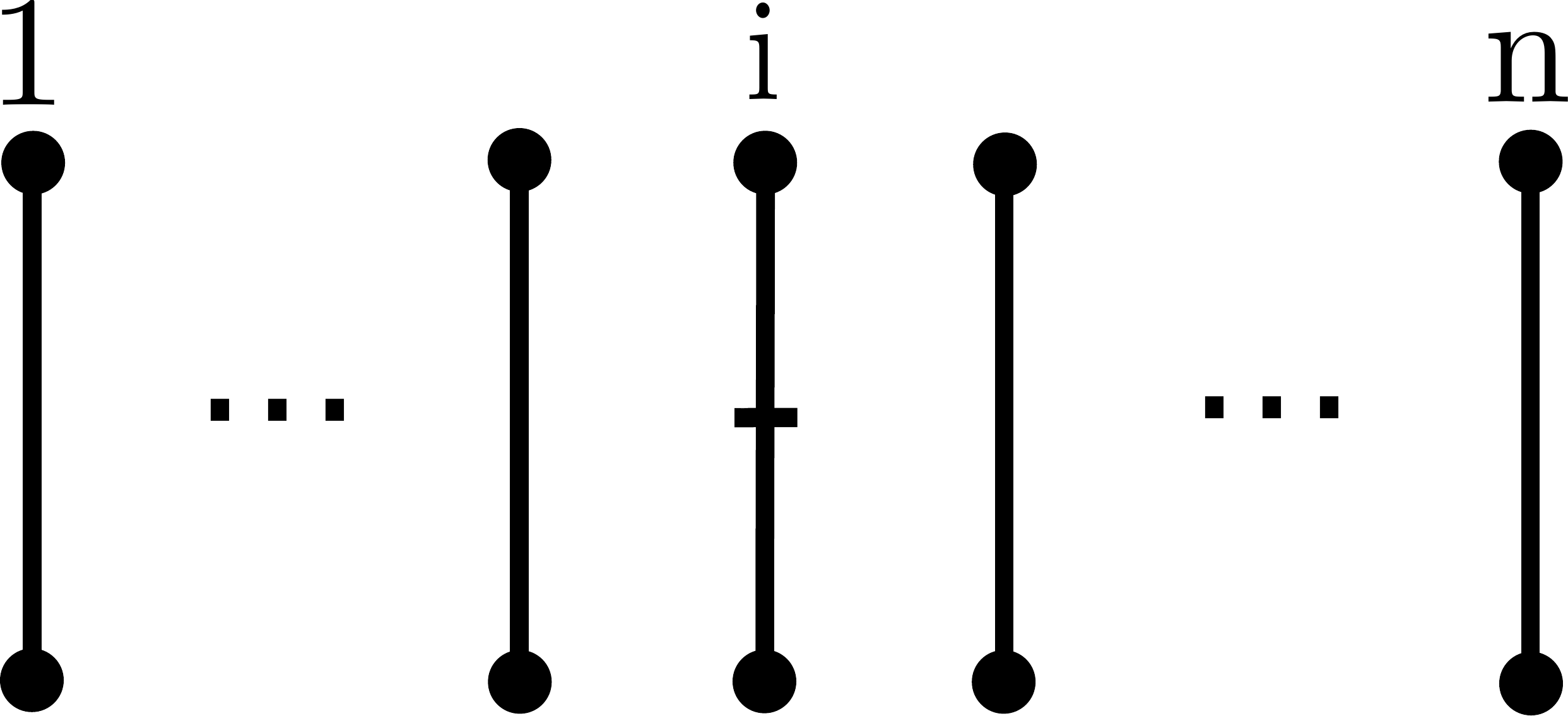}
	\caption{The generator $\rho_i$}
	\label{Gen_exwB}
\end{figure}

\begin{figure}[ht]
\centering
\renewcommand{\thesubfigure}{$\mathrm{E_I}$}
\begin{subfigure}[b]{.3\textwidth}
  \centering
  \includegraphics[scale=.14]{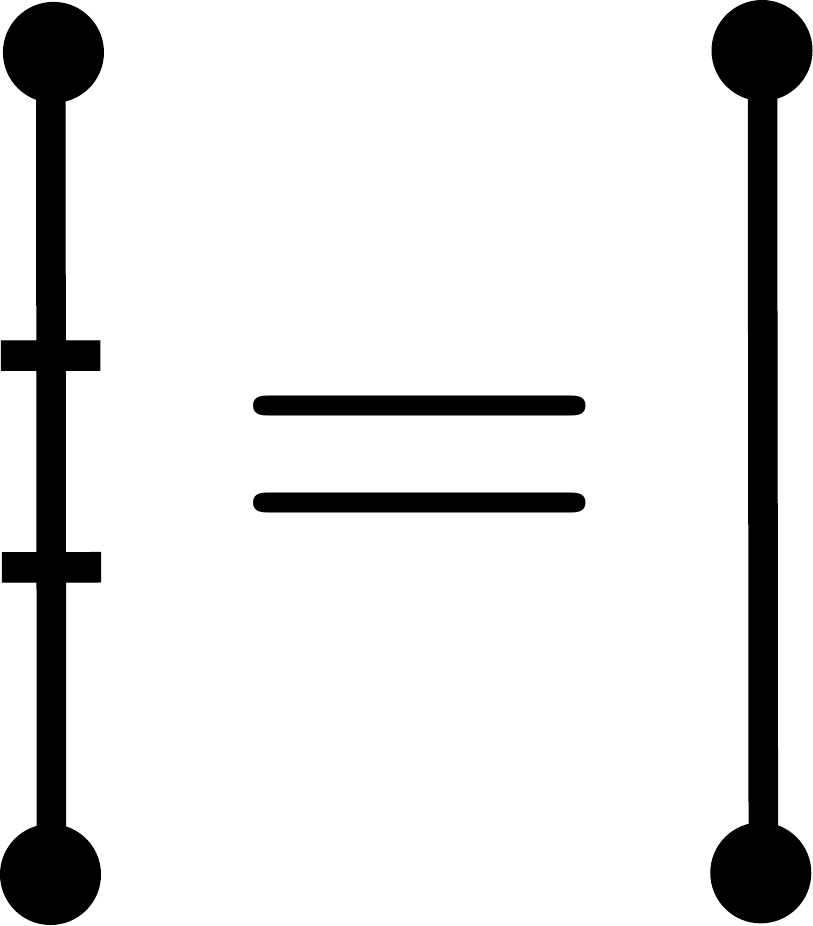}
  \caption{$\rho_i^2 = 1$}
\end{subfigure}%
\renewcommand{\thesubfigure}{$\mathrm{E_{II}}$}%
\begin{subfigure}[b]{.3\textwidth}
  \centering
  \includegraphics[scale=.14]{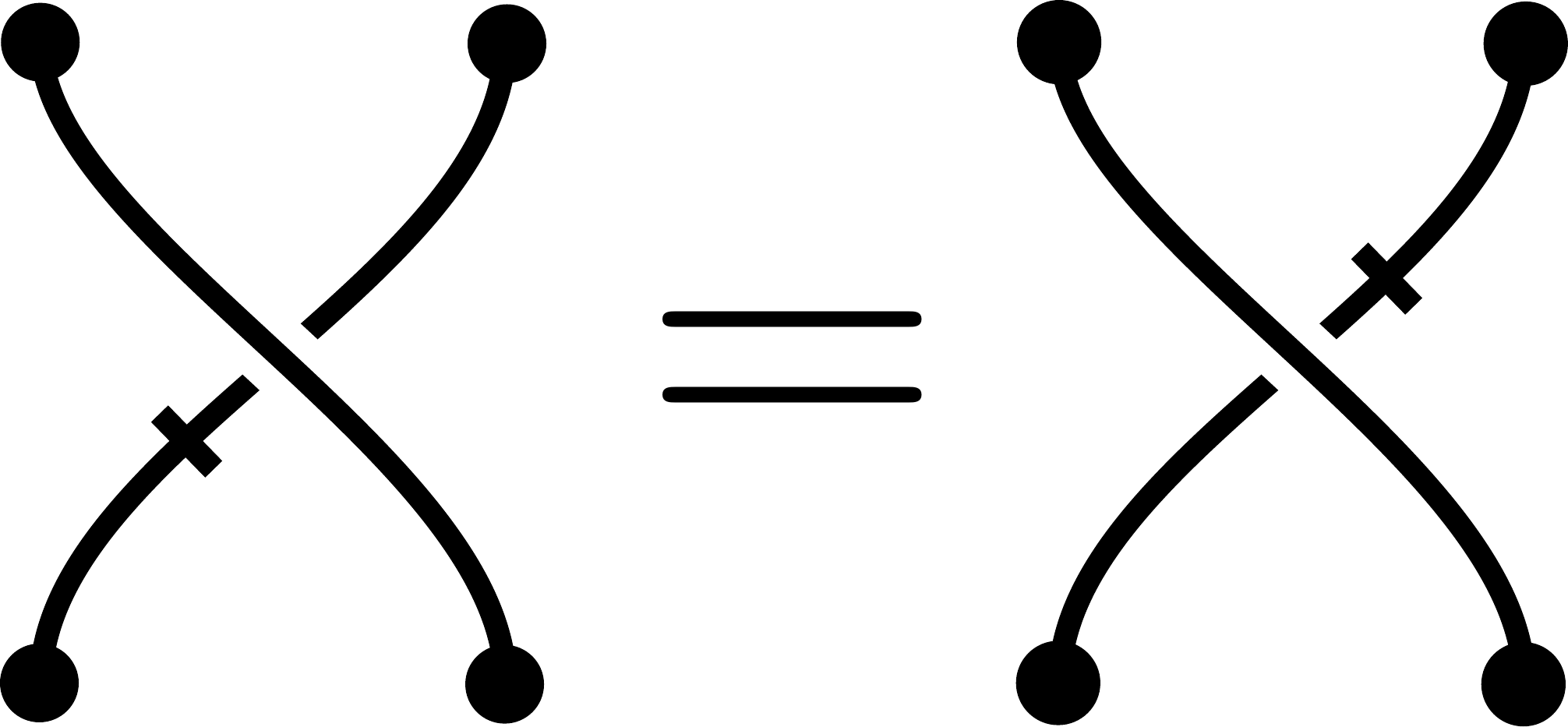}
  \caption{$\sigma_i\rho_i=\rho_{i+1}\sigma_i$}
\end{subfigure}%
\renewcommand{\thesubfigure}{$\mathrm{E_{III}}$}%
\begin{subfigure}[b]{.4\textwidth}
  \centering
  \includegraphics[scale=.14]{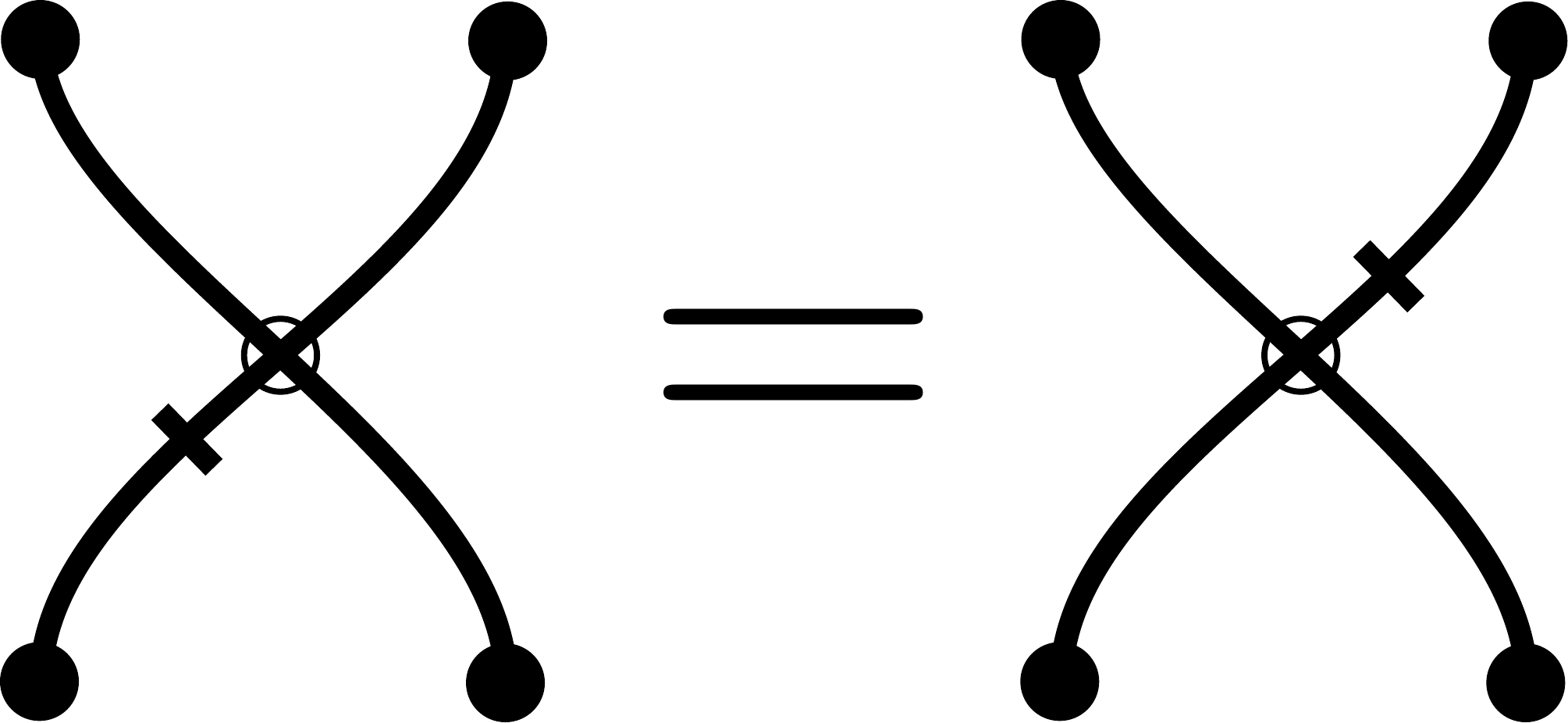}
  \caption{$\tau_i\rho_i=\rho_{i+1}\tau_i$}
\end{subfigure}
\caption{Three of the four extended welded moves}
\label{Extended_moves}
\end{figure}

\begin{figure}[tb]
\renewcommand{\thesubfigure}{$\mathrm{E_{IV}}$}
\begin{subfigure}[b]{.7\textwidth}
  \centering  
  \includegraphics[scale=.14]{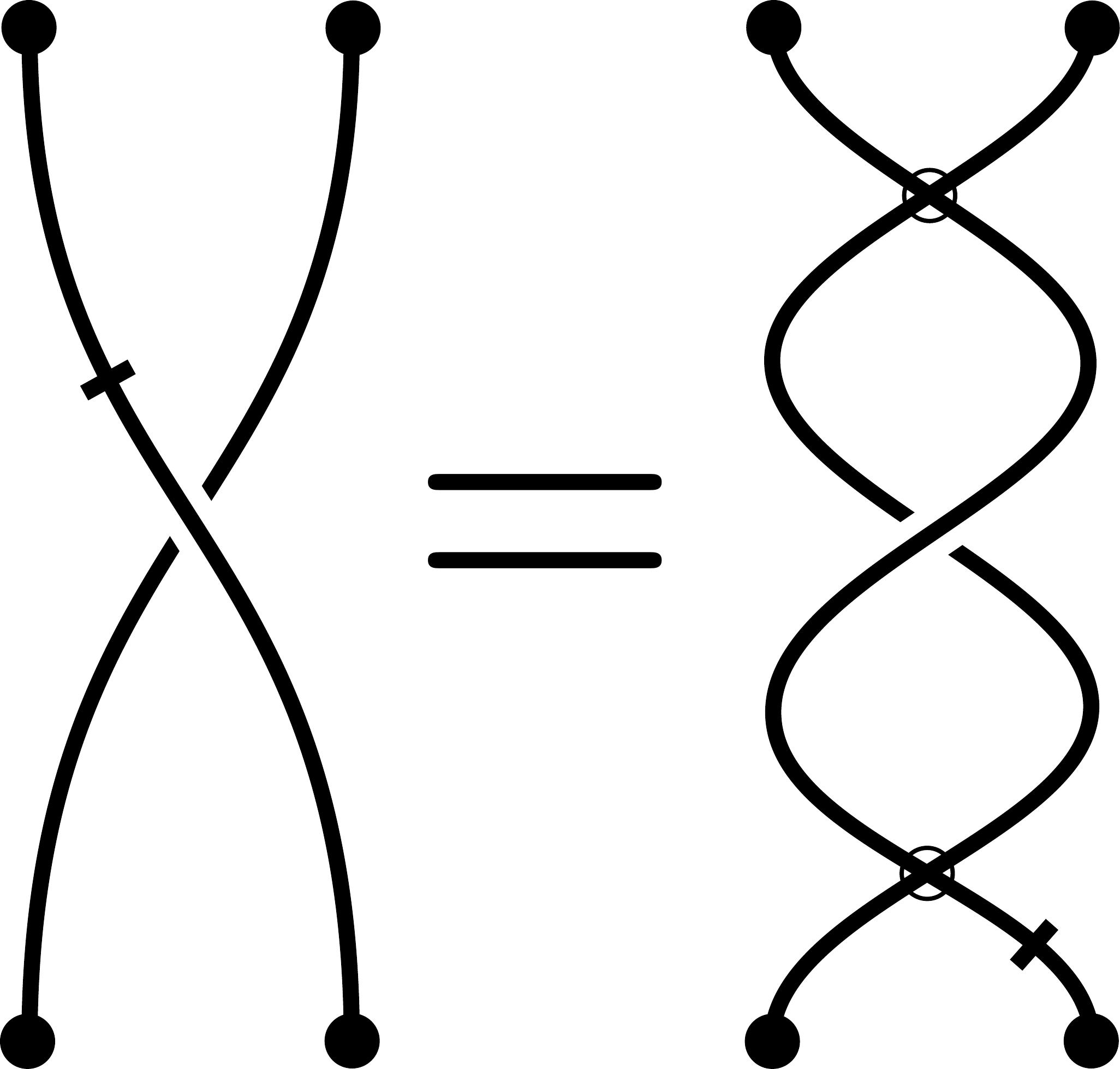}
  \caption{$\sigma_i\rho_{i+1} = \rho_i\tau_i\sigma_i^{-1}\tau_i$}
\end{subfigure}
\caption{The fourth extended welded move}
\label{Extended_move_4}
\end{figure}

\subsubsection{Extended welded braids as automorphisms of the free group}\label{section_ext_Aut(Fn)}

The above morphism from $\wB_n$ to $\Aut(\F_n)$ can be extended to $\exwB_n$ by sending $\rho_i \in \exwB_n$ to the automorphism $\rho_i \colon x_i \mapsto x_i^{-1}$ of $\F_n$ (fixing the $x_k$ if $k \neq i$). The following result (see~\cite[Th.~4.1]{Damiani}) is easily deduced from Theorem~\ref{wB_in_Aut(Fn)}:

\begin{theorem}\label{exwB_in_Aut(Fn)}
The morphism $\exwB_n \rightarrow \Aut(\F_n)$ sending $\tau_i$, $\sigma_i$ and $\rho_i$ to the corresponding automorphisms (as defined above) is a well-defined isomorphism onto its image, which consists of all automorphisms of the form:
\[x_i \longmapsto w_i x_{\sigma(i)}^{\pm 1} w_i^{-1}, \ \ \ \text{with}\ w_1,\ldots, w_n \in\F_n\ \text{and}\ \sigma \in \Sym_n.\]
\end{theorem}

\begin{proof}
Let $G$ denote the subgroup of $\Aut(\F_n)$ consisting of automorphisms of the form described in the statement. It is easy to see that the morphism $f$ under scrutiny is well-defined and takes values in $G$. The canonical morphism $\Aut(\F_n) \rightarrow \Aut(\F_n^{\ab}) \cong GL_n(\Z)$ sends $G$ onto the group of signed permutation matrices, which is isomorphic to $W_n$, and this projection is split, $\langle \tau_i, \rho_i \rangle$ being the image of the obvious splitting. Moreover, the kernel $K$ of this projection is the subgroup of automorphisms of the form $x_i \mapsto w_i x_i w_i^{-1}$, and the morphism $f$ restricts to the isomorphism between $\wP_n$ and $K$ from Theorem~\ref{wB_in_Aut(Fn)}. Since $f$ also induces an isomorphism $\exwB_n/\wP_n \cong W_n \cong G/K$, it is an isomorphism between  $G = K \rtimes W_n$ and $\exwB_n = \wP_n \rtimes W_n$.
\end{proof}

\subsubsection{A geometric interpretation}\label{par_geom_exwB} The geometric interpretation of welded braids from~\Spar\ref{par_geom_wB} holds for extended welded braids as well: they also correspond to motions of unknotted circles in space, if we allow the orientation of the circles to be flipped at the end of the motion. Precisely, the group $(\Z/2)^n$ acts on $F_{n\S^1}(\D^3)$, by changing the orientations of the different copies of $\S^1$. This action combines with the action of $\Sym_n$ from~\Spar\ref{par_geom_wB} into an action of the hyperoctahedral group $W_n = (\Z/2) \wr \Sym_n$. Then the fundamental group of $F_{n\S^1}(\D^3)/W_n$ identifies with $\exwB_n$, where the correspondence between loops in $F_{n\S^1}(\D^3)/W_n$ and the above generators of $\exwB_n$ is given as before for $\tau_i$ and $\sigma_i$, and $\rho_i$ corresponds to a loop that rotates the $i$-th circle by 180 degrees. 

As in~\Spar\ref{par_geom_wB}, this identification is classically proven with automorphisms of free group as an intermediate object. In fact, it can be seen as a slight enhancement of the identification of $\pi_1 \left(F_{n\S^1}(\D^3)/\Sym_n \right)$ with the combinatorial $\wB_n$, in the same way as Theorem~\ref{exwB_in_Aut(Fn)} above is an enhancement of Theorem~\ref{wB_in_Aut(Fn)}. Precisely, the Artin representation extends naturally to an action of $\pi_1 \left(F_{n\S^1}(\D^3)/W_n \right) =: G$ on $\F_n$. This gives a morphism $\alpha$ from $G$ to $\exwB_n$, the latter being identified to a subgroup of $\Aut(\F_n)$ via Theorem~\ref{exwB_in_Aut(Fn)}. The action of $W_n$ on $F_{n\S^1}(\D^3)$ is clearly free, so the quotient by $W_n$ is a covering, whence $\pi_1 \left(F_{n\S^1}(\D^3)\right) =: H$ is a subgroup of $G$, and $G/H \cong W_n$. The restriction of $\alpha$ to $H$ is the isomorphism from  \cite[Th.~5.3]{Goldsmith} (see~\Spar\ref{par_geom_wB}) between $H$ and $\wP_n \subset \Aut(\F_n)$, and $\alpha$ is easily seen to induce an isomorphism between $G/H \cong W_n$ and $\exwB_n/\wP_n \cong W_n$ (see the proof of Theorem~\ref{exwB_in_Aut(Fn)} for the last isomorphism). Thus $\alpha$ is an isomorphism, as required. 

Notice that the covering mentioned above is part of a commutative square of coverings:
\begin{equation}\label{square_of_coverings}
\begin{tikzcd}
F_{n\S^1}(\D^3) \ar[r, two heads] \ar[d, two heads] 
&F_{n\S^1}(\D^3)/\Sym_n  \ar[d, two heads]  \\
F_{n\S^1}(\D^3)/(\Z/2)^n \ar[r, two heads]
&F_{n\S^1}(\D^3)/W_n.
\end{tikzcd}
\end{equation}
The vertical maps in the diagram \eqref{square_of_coverings} can be thought of as ``forgetting the orientation of the circles''. If $\exwP_n$ denotes the fundamental group of $F_{n\S^1}(\D^3)/(\Z/2)^n$, (the group of \emph{pure} extended welded braids), by taking fundamental groups of \eqref{square_of_coverings}, we get, directly from the geometric interpretation of these groups, a square of \emph{injections}:
\begin{equation}\label{square_of_pi1}
\begin{tikzcd}
\wP_n \ar[r, hook] \ar[d, hook] 
&\wB_n  \ar[d, hook]  \\
\exwP_n \ar[r, hook]
&\exwB_n.
\end{tikzcd}
\end{equation}

\begin{remark}
Recalling that $\Z/2 \cong \pi_1(\Proj)$ we can give a geometric interpretation of the projection $\exwB_n \twoheadrightarrow W_n = (\Z/2) \wr \Sym_n$ corresponding to the quotient by $\wP_n$: a loop braid is sent to the collection of paths in $\Proj$ described by the unit vector normal to each circle (see Remark~\ref{precise_def_geom_wB}), together with the permutation it induces. 
\end{remark}

\section{Notion of support of an element}\label{subsec_support}

In order to apply disjoint support arguments, we first need to discuss the notion of support for elements of these groups. We will use the following notion of \emph{support} of an automorphism of $\F_n$.

\begin{definition}
\label{d:support-autFn}
Let $\F_n$ be the free group on $n$ (fixed) generators $x_1, \ldots , x_n$, and let $\varphi \in \Aut(\F_n)$. The \emph{support} $\Supp(\varphi)$ of $\varphi$ is:
\[\Supp(\varphi) := \{i \in \{1,\ldots,n\} \mid \varphi(x_i) \neq x_i \ \text{or}\ x_i\ \text{appears in some}\ \varphi(x_j) \text{ for } j \neq i\},\]
where $x_i$ is said to \emph{appear} in $w \in\F_n$ if $x_i^{\pm 1}$ is a letter in the reduced expression of $w$.
\end{definition}

\begin{example}
Using notations from \Spar\ref{section_Aut(Fn)} and \Spar\ref{section_ext_Aut(Fn)}, we have:
\begin{itemize}
\item $\Supp(\rho_\alpha) = \{\alpha\}$
\item $\Supp(\sigma_\alpha) = \Supp(\tau_\alpha) = \{\alpha, \alpha + 1\}$
\item $\Supp(\chi_{\alpha \beta}) = \{\alpha, \beta\}$
\end{itemize}
\end{example}

\begin{fact}
\label{fact_support_1}
If $\varphi, \psi \in \Aut(\F_n)$ have disjoint support, then they commute.
\end{fact}
\begin{proof}
Observing that $\varphi(x_i)=x_i$ for $i \notin \Supp(\varphi)$, and that $\varphi(x_i)$ is a word in $\{x_j \mid j \in \Supp(\varphi)\}$ for $i \in \Supp(\varphi)$, it is easy to see that
\[
\psi(\varphi(x_i)) = \varphi(\psi(x_i)) = \begin{cases}
\varphi(x_i) &\text{if }\ i \in \Supp(\varphi) \\
\psi(x_i)    &\text{if }\ i \in \Supp(\psi) \\
x_i          &\text{if }\ i \notin \Supp(\varphi) \cup \Supp(\psi),
\end{cases}
\]
for all $i \in \{1,\ldots,n\}$.
\end{proof}

The groups $\wB_n$ and $\exwB_n$ are naturally subgroups of $\Aut(\F_n)$ (see \Spar\ref{section_Aut(Fn)} and \Spar\ref{section_ext_Aut(Fn)}), so the above definition of support can be applied to (extended) welded braids. On the contrary, $\vB_n$ does not naturally embed into $\Aut(\F_n)$, so we need to use another definition of support for virtual braids. We do so using the diagrammatic description of virtual braids as equivalence classes of braid diagrams with both classical and virtual crossings (\Spar\ref{par_diag_vB}) . 

\begin{definition}
\label{d:support-diagram}
Let $\varphi \in \vB_n$. The \emph{support} $\Supp(\varphi) \subseteq \{1,\ldots,n\}$ of $\varphi$ is the set of $i \in \{1,\ldots,n\}$ such that either
\begin{itemize}
\item $i$ is not fixed by the permutation of $\{1,\ldots,n\}$ induced by $\varphi$, or
\item in every $n$-strand diagram representing $\varphi$, there is a classical crossing on the $i$-th strand.
\end{itemize}
\end{definition}

Welded braids also admit a similar diagrammatic description (see \Spar\ref{par_diag_wB}), so one may define $\Supp(\varphi)$, for $\varphi \in \wB_n \subset \Aut(\F_n)$, either as in Definition \ref{d:support-autFn} or as in Definition \ref{d:support-diagram}.

\begin{lemma}
\label{lem:two-def-support}
These two definitions of $\Supp(\varphi)$ for $\varphi \in \wB_n \subset \Aut(\F_n)$ agree.
\end{lemma}
\begin{proof}
According to Definition \ref{d:support-diagram}, we have $i \notin \Supp(\varphi)$ if and only if $i$ is fixed by the permutation of $\{1,\ldots,n\}$ induced by $\varphi$ and there is an $n$-strand diagram representing $\varphi$ with no classical crossings on the $i$-th strand. The first condition is equivalent to saying that $\varphi(x_i)$ is conjugate to $x_i$. The second condition splits into two statements: there is an $n$-strand diagram of $\varphi$ with no classical crossings involving the $i$-th strand as the \emph{upper} strand and no classical crossings involving the $i$-th strand as the \emph{lower} strand. These two statements corresponds to saying that no $x_j$ (for $j\neq i$) appears in $\varphi(x_i)$ and that $x_i$ does not appear in any $\varphi(x_j)$ (for $j\neq i$). Combining the first of these statements with the condition that $\varphi(x_i)$ is conjugate to $x_i$, we obtain the condition that $\varphi(x_i) = x_i$. Hence $i \notin \Supp(\varphi)$ if and only if $\varphi(x_i) = x_i$ and $x_i$ does not appear in any $\varphi(x_j)$ (for $j\neq i$). This is precisely Definition \ref{d:support-autFn}.
\end{proof}

\begin{remark}
Similarly, \emph{extended} welded braids admit a diagrammatic description (\Spar\ref{par_diag_exwB}), and so in this case one also has two possible definitions of $\mathrm{Supp}(\varphi)$ using either Definition \ref{d:support-diagram} or Definition \ref{d:support-autFn}. As in Lemma~\ref{lem:two-def-support} for the non-extended case, these two definitions agree.
\end{remark}

\begin{fact}
\label{fact_support_2}
If two virtual braids have disjoint support, then they commute.
\end{fact}
\begin{proof}
Denote the given two virtual braids by $\varphi$ and $\psi$. It will suffice to show that they commute after conjugating them both by the same element of $\vB_n$. By conjugating with an appropriate element (involving only virtual crossings), we may therefore arrange that $\Supp(\varphi) = \{1,\ldots,r\}$ and $\Supp(\psi) = \{r+1,\ldots,r+s\}$ for some $r+s\leq n$. By the virtual Reidemeister moves (specifically the \emph{detour moves}), we may then arrange that we have diagrams $D_\varphi$, $D_\psi$ for $\varphi$, $\psi$ such that
\[
D_\varphi = D'_\varphi \otimes \mathbb{I}_{n-r} \qquad\text{and}\qquad D_\psi = \mathbb{I}_r \otimes D'_\psi \otimes \mathbb{I}_{n-r-s},
\]
where $\mathbb{I}_k$ denotes the trivial $k$-strand diagram and $\otimes$ denotes placing virtual braid diagrams side-by-side. By ``vertically sliding the subdiagrams $D'_\varphi$ and $D'_\psi$'', we see that
\begin{align*}
D_\varphi \circ D_\psi
&= \bigl( D'_\varphi \otimes \mathbb{I}_{s} \otimes \mathbb{I}_{n-r-s} \bigr) \circ \bigl( \mathbb{I}_r \otimes D'_\psi \otimes \mathbb{I}_{n-r-s} \bigr) \\
&= \bigl( \mathbb{I}_r \otimes D'_\psi \otimes \mathbb{I}_{n-r-s} \bigr) \circ \bigl( D'_\varphi \otimes \mathbb{I}_{s} \otimes \mathbb{I}_{n-r-s} \bigr) \\
&= D_\psi \circ D_\varphi . \qedhere
\end{align*}
\end{proof}

There is another interpretation of the notion of \emph{support} in the case of \emph{welded} and \emph{extended welded} braids, where we view elements as loops of configurations of unlinked, unknotted circles. This interpretation is closer to a usual geometric notion of support. In fact, it is exactly the same if one interprets these groups as mapping class groups, as in~\cite{Damiani}. Notice that we will often talk of \emph{the} support of a mapping class, even if it is in fact not well-defined: only \emph{a} support is, as the support of a homeomorphism in the mapping class. However, \emph{having disjoint support} is well-defined for mapping classes (meaning that they have representatives with disjoint support), and we will often focus on the supports of particular representatives (the ones used to show that the mapping classes under scrutiny have disjoint support). For example, Figures \ref{fig:welded-braids-commuting} and \ref{fig:virtual-braids-commuting-2} depict (representatives of) the same pair of elements, $\chi_{\beta,\alpha'}$ and $\sigma_\alpha$, representing generators $x_{ji}$ and $s_i$ of the abelianisation of the welded braid group. In each case the supports are shaded dark and light grey respectively: in the former case these are (disjoint) compact subsets of $\bR^3$, in the latter case they are (disjoint) subcollections of strands.

\begin{figure}[ht]
\centering
\includegraphics[scale=0.6]{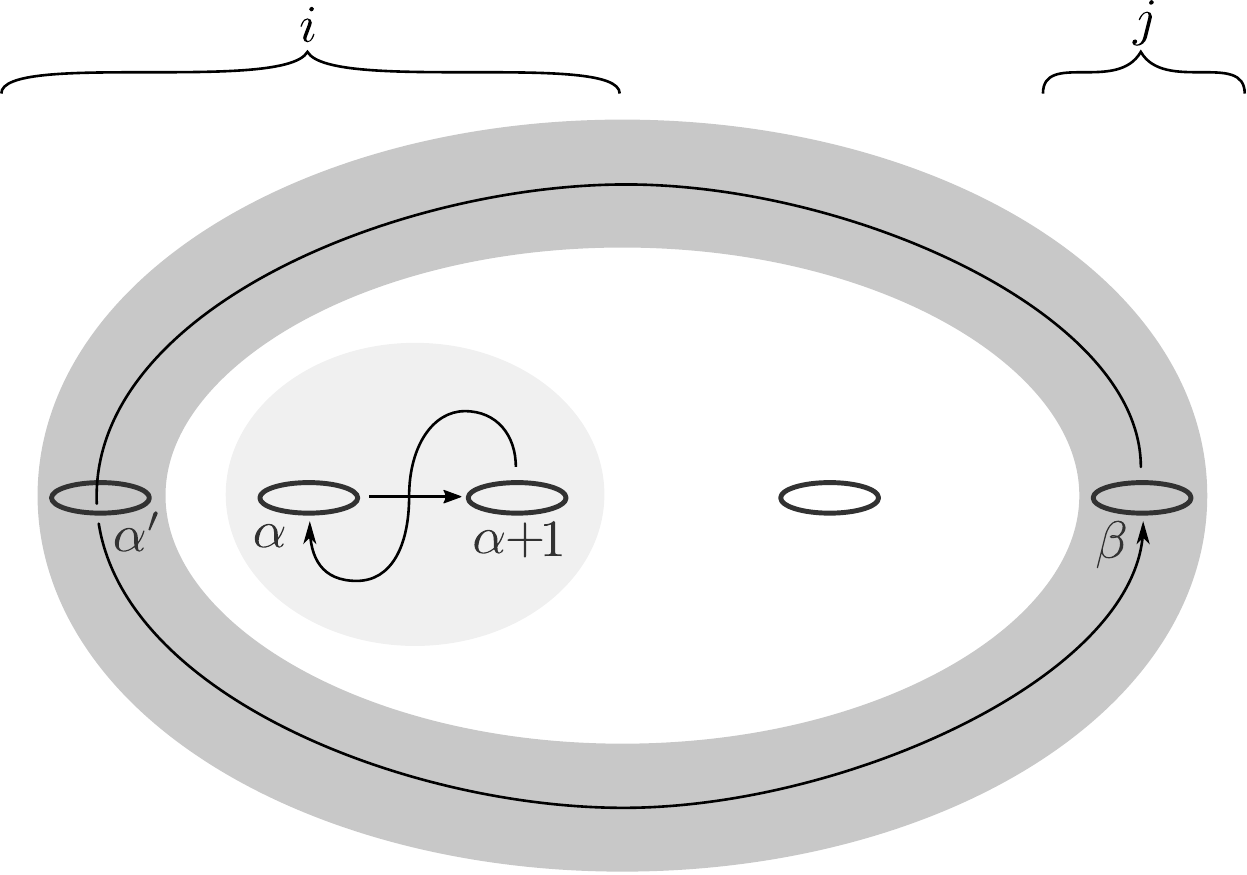}
\caption{The supports of $\sigma_\alpha$ and $\chi_{\beta,\alpha'}$ interpreted as loops of configurations of unlinks.}
\label{fig:welded-braids-commuting}

\includegraphics[scale=0.6]{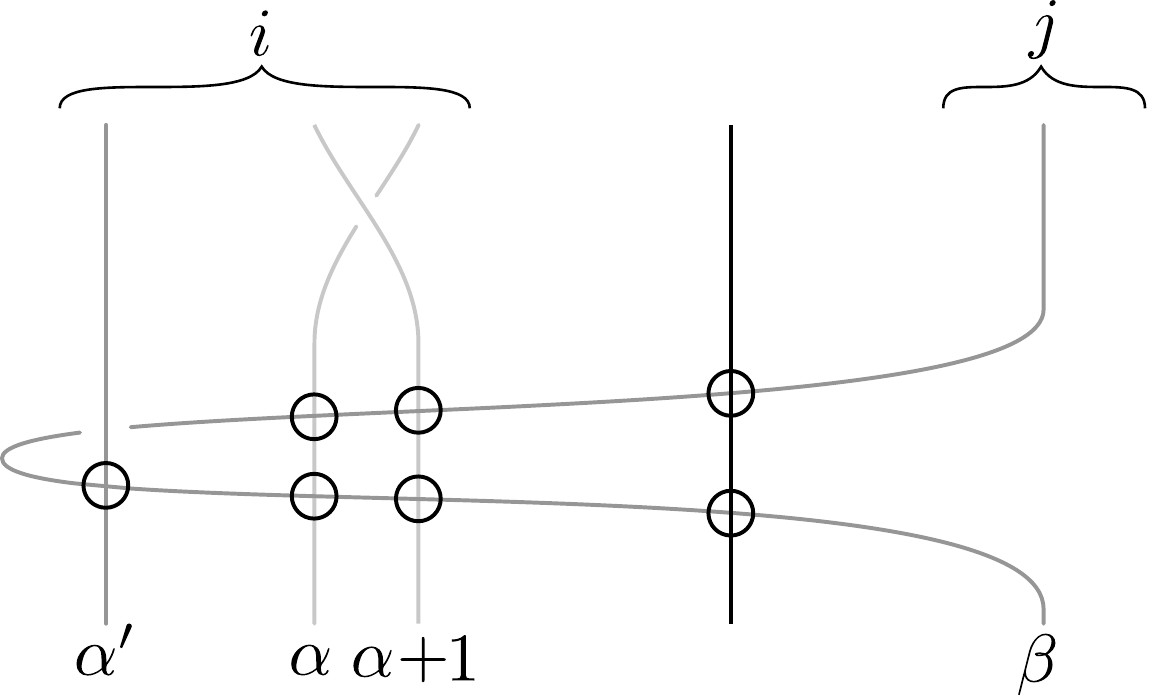}
\caption{The supports of $\sigma_\alpha$ and $\chi_{\beta,\alpha'}$ interpreted diagrammatically.}
\label{fig:virtual-braids-commuting-2}
\end{figure}

\section{Virtual braids}

We now set out to study the lower central series of $\vB_n$. Let us first remark that we can compute $\vB_n^{\ab}$ easily, using the presentation of Definition~\ref{def:presentation_virtual}. Namely, for $n\geq 2$, $\vB_n^{\ab} \cong\Z\oplus(\Z/2)$, where the first factor is generated by the common class $\sigma$ of the braid generators $\sigma_i$, and the second one is generated by the common class $\tau$ of the symmetric generators $\tau_i$.

\begin{proposition}\label{LCS_of_vBn}
The LCS of the virtual braid group $\vB_n$:
\begin{itemize}
\item \emph{stops at $\LCS_2$} if $n \geq 4$. 
\item \emph{does not stop} if $n = 2,3$.
\end{itemize}
Moreover, $\vB_2$ is residually nilpotent, whereas $\vB_3$ is not.
\end{proposition}

\begin{proof}
From the presentation of the virtual braid group $\vB_n$ (Definition~\ref{def:presentation_virtual}), we deduce that in $\vB_n^{\ab}$ all the $\sigma_i$ have the same class $\sigma$ and all the $\tau_i$ have the same class $\tau$.

\ul{Case 1: $n \geq 4$.} Then $\sigma_1$ and $\tau_3$ represent respectively $\sigma$ and $\tau$, and they commute, since they have disjoint support ; see Figure~\ref{fig:virtual-braids} for an illustration. Hence we can apply Corollary~\ref{commuting_representatives} to conclude that $\LCS_2(\vB_n) = \LCS_3(\vB_n)$.

\ul{Case 2: $n = 2$.} The group $\vB_2 \cong \Z * \Z/2$ is residually nilpotent, but not nilpotent (Proposition~\ref{Lie(Z*Z/2)}). In particular, its LCS does not stop (see also Remark~\ref{LCS_Z/2*Z_weak}).

\ul{Case 3: $n = 3$.} One can see directly from the presentation of $\vB_3$ that there is a well-defined surjection $\pi \colon \vB_3 \twoheadrightarrow \vB_2$, sending $\sigma_1$ and $\sigma_2$ to $\sigma_1$, and $\tau_1$ and $\tau_2$ to $\tau_1$. This induces a surjection $\Lie(\pi) \colon \Lie(\vB_3) \twoheadrightarrow \Lie(\vB_2)$, hence the LCS of $\vB_3$ does not stop by Lemma~\ref{lem:stationary_quotient}.

Finally, the canonical map $\Sym_3 \rightarrow \vB_3$ sending $\tau_i$ to $\tau_i$ ($i\in\{1,2\}$) splits, is thus injective, and $\Sym_3$ is not residually nilpotent. This implies that $\vB_3$ is not residually nilpotent. 
\end{proof}

We can describe precisely the behaviour of the LCS of $\vB_3$:
\begin{proposition}\label{LCS_of_vB3}
The surjection $\pi \colon \vB_3 \twoheadrightarrow \vB_2$ described in the proof of Proposition~\ref{LCS_of_vBn} induces an isomorphism $\vB_3/\LCS_\infty \cong \vB_2$. In particular, $\Lie(\pi)$ is an isomorphism between the associated Lie rings.
\end{proposition}

\begin{proof}
We first show that  $\Lie(\pi)$ is an isomorphism. The surjection $\pi$ splits, a splitting $\iota$ being given by $\sigma_1 \mapsto \sigma_1$ and  $\tau_1 \mapsto \tau_1$ and corresponds to adding a third strand. This splitting induces a surjection between abelianisations. Since the Lie rings associated to LCS are generated by their degree one, this implies that $\Lie(\iota)$ is surjective. Since $\Lie(\pi) \circ \Lie(\iota) = \Lie(\pi \circ \iota) = id$, we infer that $\Lie(\pi)$ and $\Lie(\iota)$ are isomorphisms $\Lie(\vB_3) \cong \Lie(\vB_2)$. Now, since $\vB_2 \cong \Z * \Z/2$ is residually nilpotent (Proposition~\ref{Lie(Z*Z/2)}), the result follows directly from Lemma~\ref{Quotient_by_residue}.
\end{proof}

\begin{figure}[ht]
\centering
\includegraphics{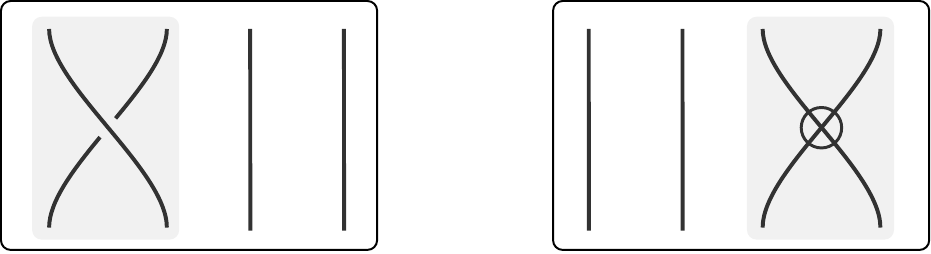}
\caption{Disjoint-support representatives for the two generators of the abelianisation of the virtual braid group $\vB_n$ for $n\geq 4$. The supports are shaded in light grey.}
\label{fig:virtual-braids}
\end{figure}

\section{Welded braids}

The definition of $\wB_n$ as a quotient of $\vB_n$ (Definition~\ref{def:presentation_welded}) allows us to deduce some results about the LCS of $\wB_n$ from their counterparts for $\vB_n$. This begins with noticing that $\wB_n^{\ab} \cong \vB_n^{\ab} \cong\Z\oplus(\Z/2)$.

\begin{proposition}\label{LCS_of_wBn}
The LCS of the welded braid group $\wB_n$:
\begin{itemize}
\item \emph{stops at $\LCS_2$} if $n \geq 4$. 
\item \emph{does not stop} if $n = 2,3$.
\end{itemize}
Moreover, $\wB_{2}$ is residually nilpotent, whereas $\wB_{3}$ is not.
\end{proposition}

\begin{proof}
\ul{Case 1: $n \geq 4$.} The fact that $\LCS_2(\vB_n) = \LCS_3(\vB_n)$ directly implies that $\LCS_2(\wB_n) = \LCS_3(\wB_n)$. In fact, one can easily see that our disjoint support argument for $\vB_n$ can be used in $\wB_n$ too. See Figure~\ref{fig:welded-braids} for an illustration from the point of view of ``loop braids''; alternatively, see Figure~\ref{fig:virtual-braids} again, interpreted now as a diagram of welded braids instead of virtual braids.

\ul{Case 2: $n = 2$.} There is no additional relation: $\wB_{2} \cong \vB_2 \cong \Z * \Z/2$.

\ul{Case 3: $n = 3$.} We treat this tricky case as a separate result. Namely, Proposition~\ref{LCS_of_wB3} below gives a surjection from $\wB_3$ onto $\Z/2 * \Z/2$, whose LCS does not stop by Proposition~\ref{LCS_of_Z/2*Z/2}. The result thus follows from Lemma~\ref{lem:stationary_quotient}.
\end{proof}

\begin{figure}[ht]
\centering
\includegraphics{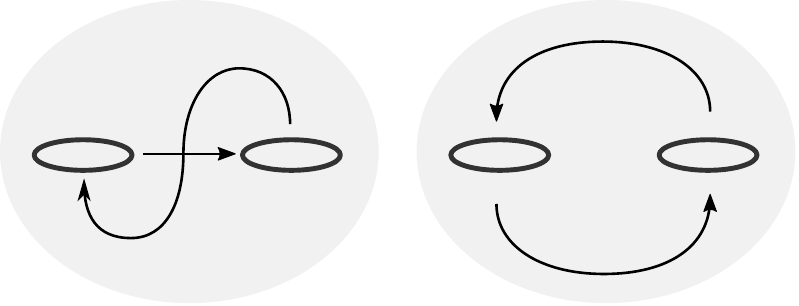}
\caption{Disjoint-support representatives for the two generators of the abelianisation of the welded braid group $\wB_n$ for $n\geq 4$. The supports are shaded in light grey.}
\label{fig:welded-braids}
\end{figure}

The group $\wB_3$ is a non-trivial quotient of $\vB_3$, and our result on $\vB_3$ cannot be transferred to $\wB_3$, nor can we directly adapt the reasoning. Indeed, the surjection $\vB_3 \twoheadrightarrow \vB_2 \cong \wB_2$ does not factor through the quotient by our additional relation, leaving us with no obvious surjection from $\wB_3$ onto $\wB_2$. However, we are still able to compute $\wB_3/\LCS_\infty$, and the situation is very similar to the case of virtual braids. In particular, the LCS does not stop, but $\wB_3$ is not residually nilpotent.
\begin{proposition}\label{LCS_of_wB3}
The residue $\LCS_\infty(\wB_3)$ is normally generated by $\sigma_1\sigma_2^{-1}$ and $\tau_1\tau_2$. The quotient $G = \wB_3/\LCS_\infty$ has the following presentation:
\[G = \left\langle \sigma, \tau\ \middle| \ \tau^2 = 1,\ \sigma^2\tau = \tau\sigma^2 \right\rangle.\]
Its centre $\mathcal ZG$ is infinite cyclic, generated by $\sigma^2$, and:
\[G/\mathcal ZG \cong \Z/2 * \Z/2.\]
\end{proposition}

\begin{proof}
The canonical map $\B_3 \rightarrow \wB_3$ sends $\LCS_\infty(\B_3)$ to $\LCS_\infty(\wB_3)$. Since $\sigma_1\sigma_2^{-1} \in \LCS_2(\B_3) = \LCS_\infty(\B_3)$, we deduce that $\sigma_1\sigma_2^{-1} \in \LCS_\infty(\wB_3)$. The same holds with $\tau_1\tau_2$, which is the image of $\tau_1\tau_2 \in \LCS_2(\Sym_3) = \LCS_\infty(\Sym_3)$ by the canonical map $\Sym_3 \rightarrow \wB_3$. Thus, $\sigma_1\sigma_2^{-1}, \tau_1\tau_2 \in \LCS_\infty(\wB_3)$, implying that $\wB_3/\LCS_\infty$ is a quotient of the group $G$ obtained from $\wB_3$ by adding the relations $\sigma_1 = \sigma_2$ and $\tau_1 = \tau_2$. We deduce from the usual presentation of $\wB_3$ that $G$ is indeed described by the presentation of the statement. We are going to show that $G$ is residually nilpotent, which implies that $G$ is indeed equal to $\wB_3/\LCS_\infty$ (it is the biggest residually nilpotent quotient of $\wB_3$).
We first deduce from its presentation that the abelianisation of $G$ is given by:
\[G^{\ab} \cong \Z\{\bar\sigma, \bar\tau\}/(2\bar\tau) \cong \Z \oplus \Z/2.\]
In particular, $\bar\sigma$ is not a torsion element of $G^{\ab}$, which means that no non-trivial power of $\sigma$ belongs to $\LCS_2(G)$.

Now, consider the element $\sigma^2  \in G$. It commutes with $\sigma$ and $\tau$, thus it is central. In particular, $\langle \sigma^2 \rangle$ is a normal subgroup. Moreover, the quotient $G/\langle \sigma^2 \rangle$ is described by:
\[G/\langle \sigma^2 \rangle\ \cong\
\left\langle \sigma,\ \tau\ \middle| \ \sigma^2 = \tau^2 = 1,\ \sigma^2\tau = \tau\sigma^2 \right\rangle\ \cong\
\left\langle \sigma, \tau\ \middle| \ \sigma^2 = \tau^2 = 1\right\rangle,\]
which is exactly $\Z/2 * \Z/2$. Proposition~\ref{LCS_of_Z/2*Z/2} then implies that $G$ is residually nilpotent. Indeed, if $x \in \LCS_\infty(G)$, then its image in $G/\langle \sigma^2 \rangle \cong \Z/2 * \Z/2$ is in $\LCS_\infty(\Z/2 * \Z/2) = \{1\}$. Thus $x \in \langle \sigma^2 \rangle$. Since $\bar \sigma$ is not a torsion element in $G^{\ab}$, we have $\langle \sigma^2 \rangle \cap \LCS_2G = \{1\}$, whence $x = 1$.

Finally, we remark that $\Z/2 * \Z/2$ is centreless, implying that $\mathcal ZG$ consists only of $\langle \sigma^2 \rangle$. As $\bar \sigma$ is not a torsion element in $G^{\ab}$, $\langle \sigma^2 \rangle$ has to be infinite cyclic, and all our claims are proved.
\end{proof}

\begin{remark}
Both $\vB_3$ and $\wB_3$ have non-stopping LCS, without being residually nilpotent. In this situation, one may continue the LCS to infinite ordinals by defining $\LCS_{\alpha+1}(G) = [G,\LCS_\alpha(G)]$ and $\LCS_\lambda(G) = \bigcap_{\alpha<\lambda} \LCS_\alpha(G)$ for limit ordinals $\lambda$. In particular, the first infinite term $\LCS_\omega(G)$ is what we denote by $\LCS_\infty(G)$. A result of Mal'cev \cite{Malcev1949} says that groups may have LCS of any length: for any ordinal $\alpha$ there is a group $G_\alpha$ such that $\LCS_*(G_\alpha)$ stops precisely at $\alpha$. Such examples are in general ``artificial'', but there are also naturally occurring groups $G$ whose LCS are \emph{long}, meaning that $\LCS_{\omega+1}(G) \neq \LCS_\omega(G)$. For example, Levine \cite{Levine1991} gave an example of a finitely presented group whose LCS is long, and Cochran and Orr \cite{CochranOrr1998} gave many examples of $3$-manifold groups whose LCS are long. In light of these results, one may wonder whether the LCS of $\vB_3$ and of $\wB_3$ are long.

In fact, in both cases, we have $[G, \LCS_\omega(G)] = \LCS_\omega(G)$, so their LCS are not long (they stop at $\omega$). Indeed, in both cases, $\LCS_\omega(G)$ is normally generated by $\tau_1\tau_2^{-1}$ and $\sigma_1\sigma_2^{-1}$, belonging respectively to $\LCS_2(\Sym_3) = \LCS_\omega(\Sym_3)$ and $\LCS_2(\B_3) = \LCS_\omega(\B_3)$. Then:
\[\tau_1\tau_2^{-1} \in \LCS_3(\Sym_3) = [\Sym_3, \LCS_2(\Sym_3)] = [\Sym_3, \LCS_\omega(\Sym_3)] \subseteq [G, \LCS_\omega(G)].\]
The same calculation implies that $[G, \LCS_\omega(G)]$ contains $\sigma_1\sigma_2^{-1}$, so it must contain $\LCS_\omega(G)$, as claimed. The same argument will lead to the same conclusion for $G = \exwB_3$; see Proposition~\ref{LCS_of_exwBn} and Remark~\ref{Gen_for_residue_exwB3}.
\end{remark}

\section{Extended welded braids}

Let us now turn our attention to the LCS of $\exwB_n$. We first compute its abelianisation: 
\begin{lemma}\label{exwB^ab}
For $n\geq 2$, the abelianisation $\exwB_n^{\ab}$ is isomorphic to $(\Z/2)^3$. A basis over $\Z/2$ is given by $\sigma, \tau$ and $\rho$, which are respectively the common class of the $\sigma_i$, the $\tau_i$ and the $\rho_i$.
\end{lemma}

\begin{proof}
On can get this directly from the presentation of Definition~\ref{def:presentation_welded}. One can also see this as a particular case of Proposition~\ref{partitioned_exwB^ab} below.
\end{proof}

\begin{proposition}\label{LCS_of_exwBn}
The LCS of the extended welded braid group $\exwB_n$:
\begin{itemize}
\item \emph{stops at $\LCS_2$} if $n \geq 4$. 
\item \emph{does not stop} if $n = 2,3$.
\end{itemize} 
\end{proposition}

\begin{proof}
\ul{Case 1: $n \geq 4$.} As in the proof of Proposition~\ref{LCS_of_vBn}, we show that each pair of generators of $\exwB_n^{\ab}$ may be represented by elements of $\exwB_n$ with disjoint support. For the first two generators in the list, this is possible since each of them may be supported inside a $3$-ball in $\D^3$ containing two of the circles; see Figure~\ref{fig:welded-braids}. For one of the first two generators together with the third generator, this is possible even for $n \geq 3$, since the third generator may be supported in a $3$-ball containing a single circle; see Figure~\ref{fig:extended-welded-braids}.

\ul{Case 2: $n = 2$.} For short, we denote $\sigma_1$ (resp.~$\tau_1$) by $\sigma$ (resp.~$\tau$). From the presentation of $\exwB_2$, we can see that the quotient by the relations $\rho_1 = \rho_2 = 1$ admits the presentation:
\[\exwB_2/(\rho_1,\rho_2) = \langle \sigma, \tau \ |\ \tau^2=1,\ \sigma = \tau \sigma^{-1} \tau  \rangle.\]
But this is a presentation of $\Z \rtimes (\Z/2)$, whose LCS, computed in Corollary~\ref{Lie_Klein_A/center}, does not stop.  

\ul{Case 3: $n = 3$.} It is known that $\exwB_3$ contains $\B_3$ as a subgroup (see for instance Remark~\ref{wBn_in_exwBn} and Corollary~\ref{wBn_in_wBn+1}), hence it cannot be residually nilpotent. In fact, $\LCS_\infty(\exwB_3)$ contains $\LCS_\infty(\Sym_3) = \LCS_2(\Sym_3)$ and $\LCS_\infty(\B_3) = \LCS_2(\B_3)$, which contain respectively $\tau_2 \tau_1$ and $\sigma_2 \sigma_1^{-1}$. Let us consider the quotient $G$ of $\exwB_3$ by these two relations, and let us denote by $\tau$ (resp.~$\sigma$) the common image of $\tau_1$ and $\tau_2$ (resp.~$\sigma_1$ and $\sigma_2$) in $G$. We remark that in $G$, $\rho_2 = \tau_1 \rho_1 \tau_1$ and $\rho_3 = \tau_2 \rho_2 \tau_2$ are both identified with $\rho_1 = \tau_2 \rho_1 \tau_2 = \tau_1 \rho_2 \tau_1$, so we can also speak of their common class $\rho$. From the presentation of $\exwB_3$, we get a presentation of $G$:
\[G = \langle \sigma, \tau, \rho \ |\ \rho^2 = \tau^2=1;\ \sigma, \tau \rightleftarrows \rho;\ \sigma \rho = \rho \tau \sigma^{-1} \tau  \rangle.\]
The second group of relations, saying $\rho$ is central, may be seen as \emph{disjoint support} relations. Since $\rho$ is central, the last relation is equivalent to $\sigma = \tau \sigma^{-1} \tau$. As a consequence, this is a presentation of $\Z/2 \times ( \Z \rtimes (\Z/2))$, whose LCS does not stop, again thanks to Corollary~\ref{Lie_Klein_A/center}.
\end{proof}

\begin{remark}\label{Gen_for_residue_exwB3}
In the third part of the proof, the fact that the quotient $G$ is residually nilpotent implies that it is the quotient by the whole of $\LCS_\infty(\exwB_3)$, which is thus normally generated by $\tau_2 \tau_1$ and $\sigma_2 \sigma_1^{-1}$. This also means that $\Lie(\exwB_3) \cong \Lie(G)$, which is easy to compute completely; see Corollary~\ref{Lie_Klein_A/center}.
\end{remark}

\begin{figure}[ht]
\centering
\includegraphics{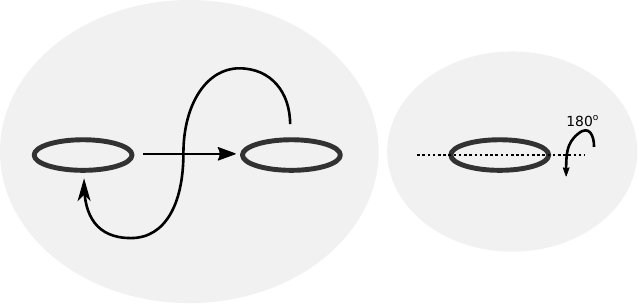}
\caption{Disjoint-support representatives for two of the generators of the abelianisation of the extended welded braid group $\exwB_n$ for $n\geq 3$. The supports are shaded in light grey.}
\label{fig:extended-welded-braids}
\end{figure}

\section{Partitioned virtual, welded and extended welded braids}

We now turn our attention to partitioned virtual and welded braids, which are defined exactly in the same way as partitioned braids have been in Definition~\ref{def_partitioned_braids}. 

\subsection{Basic theory}\label{subsec:basic_theory}

Let us consider the groups $\vB_n$, $\wB_n$ and $\exwB_n$. Each one is endowed with a canonical projection $\pi$ onto $\Sym_n$ (we use the same notation in the three cases). One can think of it as the replacement of every crossing in a diagram by virtual ones (and forgetting the mark points for extended welded braids); in terms of group presentations, it sends both $\sigma_i$ and $\tau_i$ to the transposition $(i, i+1)$, for each $i$ (and it kills the $\rho_i$). 

\begin{definition}
Let $n \geq 1$ be an integer and let $\lambda = (n_1, \ldots , n_l)$ be a partition of $n$. The corresponding \emph{partitioned virtual braid group} is:
\[\vB_\lambda := \pi^{-1}(\Sym_{\lambda}) = \pi^{-1}\left(\Sym_{n_1} \times \cdots \times \Sym_{n_l}\right)\ \subseteq \vB_n.\]
We define similarly the \emph{partitioned welded braid group} $\wB_\lambda := \pi^{-1}(\Sym_{\lambda})$ and the group of \emph{$\lambda$-partitioned extended welded braids} $\exwB_\lambda = \pi^{-1}(\Sym_{\lambda})$.
\end{definition}

Recalling that the elements $\chi_{\alpha \beta}$ introduced in~\Spar\ref{par_pstation_vB} (see Figure \ref{fig:pure-virtual-welded-braid-generators})  generate $\vP_n$ (resp.~$\wP_n$), one can easily adapt the proof of Lemma~\ref{generators_of_partitioned_braids} to the present context, giving a list of generators for $\vB_\lambda$ (resp.~$\wB_\lambda$ and $\exwB_\lambda$):
\begin{lemma}\label{generators_partitioned_v(w)B}\label{generators_partitioned_exwB}
Let $\lambda = (n_1, \ldots , n_l)$ be a partition of an integer $n \geq 1$. Then $\vB_\lambda$ (resp.~$\wB_\lambda$) is the subgroup of $\vB_n$ (resp.~$\wB_n$) generated by:
\begin{itemize}
\item The $\sigma_\alpha$ and the $\tau_\alpha$ for $1 \leq \alpha \leq n$ such that $\alpha$ and $\alpha + 1$ are in the same block of $\lambda$.
\item The $\chi_{\alpha\beta}$, for $1 \leq \alpha \neq \beta \leq n$ such that $\alpha$ and $\beta$ do not belong to the same block of $\lambda$.
\end{itemize}
Also, $\exwB_\lambda$ is the subgroup of $\exwB_n$ generated by the same set of generators, together with the $\rho_\alpha$ for all $\alpha \leq n$.
\end{lemma}

In fact, the situation is even easier than the case of partitioned braids: for virtual (or welded) braids, the canonical projection onto $\Sym_n$ splits, so that $\vB_n \cong \vP_n \rtimes \Sym_n$ (resp.~$\wB_n \cong \wP_n \rtimes \Sym_n$); see~\Spar\ref{par_pstation_vB} and~\Spar\ref{par_pstation_wB}. And this works for partitioned virtual/welded braids too, since this splitting restricts obviously to a splitting of $\mathbf{vB}_\lambda \twoheadrightarrow \Sym_\lambda$ (resp.~of $\mathbf{wB}_\lambda \twoheadrightarrow \Sym_\lambda$), so that:
\begin{equation}\label{dec_of_vB}
\vB_\lambda \cong \vP_n \rtimes \Sym_\lambda,
\end{equation}
and likewise for welded braids. Extended welded braids behave much in the same way: the decomposition $\exwB_n \cong \wP_n \rtimes W_n$ from\Spar\ref{par_pstation_exwB} restricts to a decomposition: 
\begin{equation}\label{dec_of_exwB}
\exwB_\lambda \cong \wP_n \rtimes W_\lambda,
\end{equation}
where $W_\lambda$ denotes the subgroup $(\Z/2) \wr \Sym_\lambda$ of $W_n = (\Z/2) \wr \Sym_n$.

We can use this, together with the classical calculation of $\vP_n^{\ab} \cong \wP_n^{\ab}$ from Corollary~\ref{v(w)Pn^ab}, to compute the abelianisations of $\vB_\lambda$, $\wB_\lambda$ and $\exwB_\lambda$:
\begin{proposition}\label{partitioned_v(w)B^ab}\label{partitioned_exwB^ab}
Let $\lambda = (n_1, \ldots , n_l)$ be a partition of an integer $n \geq 1$ and let $l'$ be the number of $i$ such that $n_i \geq 2$. Then:
\[\vB_\lambda^{\ab} \cong \wB_\lambda^{\ab} \cong (\Z/2)^{l'} \times \Z^{l'}  \times \Z^{l(l-1)}\ \text{ and }\ \exwB_\lambda^{\ab}
\cong (\Z/2)^{l+2l'+l(l-1)}.\]
Precisely, the abelian group $\vB_\lambda^{\ab}$ ($\cong \wB_\lambda^{\ab}$) is generated by:
\begin{itemize}
\item For each $i \in \{1, \ldots , l\}$ such that $n_i \geq 2$, one generator $s_i$ (resp.~$t_i$), which is the common class of the  $\sigma_\alpha$ (resp.~the $\tau_\alpha$) for $\alpha$ and $\alpha +1$ in the $i$-th block of $\lambda$. 
\item For each $i, j \in \{1, \ldots , l\}$ such that $i \neq j$, one generator $x_{ij}$, which is the common class of the  $\chi_{\alpha\beta}$ for $\alpha$ in the $i$-th block of $\lambda$ and $\beta$ in the $j$-th one.
\end{itemize}
The $t_i$ are a $\Z/2$-basis of the first factor, the $s_i$ are a $\Z$-basis of the second one, and the $x_{ij}$ are a $\Z$-basis of the third one.  

A $(\Z/2)$-basis of $\exwB_\lambda^{\ab} \cong (\wB_\lambda^{\ab} \otimes \Z/2) \times (\Z/2)^l$ is given by the same list of generators (which are of $2$-torsion in $\exwB_\lambda^{\ab}$), together with:
\begin{itemize}
\item For each $i \in \{1, \ldots , l\}$, one generator $r_i$, which is the common class of the  $\rho_\alpha$ for $\alpha$ in the $i$-th block of $\lambda$. 
\end{itemize}
\end{proposition}

\begin{proof}
We show our first statement for virtual braids; the exact same proof works for $\wB_\lambda$. 

Applying Lemma~\ref{lem:abelianization semidirect} to the decomposition \eqref{dec_of_vB}, we deduce that $\vB_\lambda^{\ab} \cong (\vP_n^{\ab})_{\Sym_\lambda} \times \Sym_\lambda^{\ab}$.
On the one hand, $\Sym_{n_i}^{\ab} \cong \Z/2$ (generated by the class of any transposition) for all $i$ such that $n_i \geq 2$, so that $\Sym_\lambda^{\ab} \cong (\Z/2)^l$ admits the $t_i$ as a $\Z/2$-basis. On the other hand, $\vP_n^{\ab}$ is freely generated by the set $X$ of the $\overline \chi_{\alpha \beta}$ with $1 \leq \alpha \neq \beta \leq n$, on which $\Sym_\lambda \subseteq \Sym_n$ acts by permuting the indices. Thus $(\vP_n^{\ab})_{\Sym_\lambda}$ is free on the set $X/\Sym_\lambda$. The latter is exactly the set of generators $x_{ij}$ described in the statement, plus one additional element for each $i$ such that $n_i \geq 2$, which is the common class $x_i$ of the $\chi_{\alpha \beta}$ for $\alpha$ and $\beta$ in the $i$-th block. Thus we get the decomposition of the statement, with a slightly different basis (the $x_i$ instead of the $s_i$). The relation $\chi_{\alpha, \alpha +1} = \tau_\alpha \sigma_\alpha$ implies that $s_i =  x_i - t_i$, which is the change of basis we need to finish proving our first statement.

Let us now describe $\exwB_\lambda^{\ab}$, using a similar computation. We deduce from the decomposition \eqref{dec_of_exwB} that $\exwB_\lambda^{\ab} \cong (\wP_n^{\ab})_{W_\lambda} \times W_\lambda^{\ab}$.
On the one hand, the group $W_\lambda = (\Z/2) \wr \Sym_\lambda$ is isomorphic to the product of the $(\Z/2) \wr \Sym_{n_i}$, whose abelianisation is $\Z/2 \times \Z/2$ (generated by $r_i$ and $s_i$) if $n_i \geq 2$, and $\Z/2$ (generated by $r_i$) if $n_i = 1$.
Thus $W_\lambda^{\ab} \cong (\Z/2)^{l + l'}$ gives a first factor in our formula. On the other hand the action of the generator $\rho_\gamma$ of $(\Z/2)^n \subset W_\lambda$ on $\wP_n^{\ab}$ is by fixing all the $\overline \chi_{\alpha \beta}$ for $\beta \neq \gamma$ and sending all $\overline \chi_{\alpha \gamma}$ to $-\overline \chi_{\alpha \gamma}$. One can check this from the presentation of Proposition~\ref{vP_presentation}: for instance, $\chi_{12} = \tau_1 \sigma_1$, and $\rho_1 \chi_{12} \rho_1^{-1} = \sigma_1^{-1} \tau_1 = \chi_{12}^{-1}$. As a consequence, $(\vP_n^{\ab})_{(\Z/2)^n} \cong \vP_n^{\ab} \otimes \Z/2$ is a free $\Z/2$-vector space on the $\overline \chi_{\alpha \beta}$, on which $\Sym_n$ acts by permutation of the indices. Since $(\wP_n^{\ab})_{W_\lambda} = (\wP_n^{\ab})_{(\Z/2) \wr \Sym_\lambda} = ((\vP_n^{\ab})_{(\Z/2)^n})_{\Sym_\lambda}$, the rest of the proof works exactly as for the previous case, except that we work with coefficients in $\Z/2$. 
\end{proof}

\subsection{The LCS}

The remainder of the section is devoted to the proof of the following theorem. Its proof splits into several cases, which we examine in the next few propositions; Theorem \ref{LCS_of_partitioned_v(w)B} is then deduced from these at the very end of this section.
\begin{theorem}\label{LCS_of_partitioned_v(w)B}\label{LCS_of_partitioned_exwBn}
Let $n \geq 1$ be an integer, and $\lambda = (n_1, \ldots , n_l)$ be a partition of $n$.  The LCS of the partitioned virtual (resp.~welded) braid group $\vB_\lambda$ (resp.~$\wB_\lambda$):
\begin{itemize}
\item \emph{stops at $\LCS_2$} if $n_j = 1$ for at most one $j$, and $n_i \geq 4$ for all $i \neq j$.
\item \emph{does not stop} in all the other cases.
\end{itemize}
The LCS of the partitioned extended welded braid group $\exwB_\lambda$:
\begin{itemize}
\item \emph{stops at $\LCS_2$} if $n_i \geq 4$ for all $i$. 
\item \emph{does not stop} in all the other cases, except if $\lambda = (1)$, for $\exwB_1 \cong \Z/2$.
\end{itemize}
\end{theorem}

\subsubsection{The stable case} The simplest case in the case when the size of every block is at least $4$:

\begin{proposition}\label{LCS_wB_stable}
Let $\lambda$ be a partition of an integer $n$ into blocks of size at least $4$. Then the LCS of $\vB_\lambda$ (resp.~of $\wB_\lambda^{\ab}$, resp.~of $\exwB_\lambda^{\ab}$) stops at $\LCS_2$.
\end{proposition}

\begin{proof}
As in the proofs of Propositions~\ref{LCS_of_vBn} and \ref{LCS_of_wBn}, we show that any pair of generators of $\vB_\lambda^{\ab}$ (resp.~of $\wB_\lambda^{\ab}$, resp.~of $\exwB_\lambda^{\ab}$) may be represented by a pair of elements of the group having \emph{disjoint support} (in the sense of Definition~\ref{d:support-diagram} for virtual braids, and of the equivalent sense of Definitions~\ref{d:support-autFn} and \ref{d:support-diagram} for welded braids), which therefore commute by Facts \ref{fact_support_1} and \ref{fact_support_2}. The generators we need to consider are listed in Proposition~\ref{partitioned_v(w)B^ab}: they are the $s_i$, the $t_i$, the $x_{ij}$ and (in the case of extended welded braids) the $r_i$. In each case, it is clear that the hypothesis $n_i \geq 4$ allows us to choose lifts with disjoint support among the generators of $\vB_\lambda$ (resp.~of $\wB_\lambda$, resp.~of $\exwB_\lambda$) from Lemma~\ref{generators_partitioned_v(w)B}. Indeed, the class of any such generator depends only on the blocks to which the indices involved belong, and the support of each generator contains at most two indices. For instance, if $\alpha$ denotes the minimum of the $i$-th block, $s_i$ and $t_i$ can be lifted respectively to $\sigma_\alpha$ (whose support is $\{\alpha, \alpha + 1\}$) and to $\tau_{\alpha+2}$ (whose support is $\{\alpha + 2, \alpha + 3\}$).
\end{proof}

\begin{remark}[Blocks of size $3$]
If $\lambda$ has blocks of size $3$ (say, $n_1 = 3$), the LCS of $G_\lambda$ (with $G = \vB$, $\wB$ or $\exwB$) does not stop: the disjoint support argument fails, since $t_1$ and $s_1$ do not have lifts with disjoint support. However, this failure is limited, and the argument can still be applied to understand the LCS of $G_{3, \mu}$ from the one of $G_\mu$ (where $\mu$ is any partition). Precisely, the generators $t_1$ and $s_1$ of the Lie ring do not commute with each other, but they do commute with every other generator.  Using the canonical split injection of $G_3$ into $G_{3, \mu}$, we see that $t_1$ and $s_1$ (together with $r_1$ for extended welded braids) generate a copy of $\Lie(G_3)$, which is then a direct factor of $\Lie(G_{3, \mu})$. If all the blocks of $\mu$ have size at least $3$, we can be even more precise: then the $2(l-1)$ elements $x_{1i}$ and $x_{i1}$ must be central. In this case, we get an isomorphism of Lie rings:
\[\Lie(G_{3, \mu}) \cong \Lie(G_3) \times \Z^{2(l-1)} \times \Lie(G_\mu).\]
Moreover, $\Lie(G_3)$ is well-understood (see Propositions~\ref{LCS_of_vB3} and \ref{LCS_of_wB3}, and the proof of Proposition~\ref{LCS_of_exwBn}), so this allows us to compute $\Lie(G_\lambda)$ completely if all the blocks of $\lambda$ have size at least~$3$.
\end{remark}

\subsubsection{Adding one block of size $1$} This is the tricky case, where the behaviour of $\exwB_\lambda$ will differ from the other ones. For $\wB_\lambda$ and $\vB_\lambda$, we will show that the behaviour does not really differ from the stable case, using the following analogue of Lemma~\ref{P3^ab}:
\begin{lemma}\label{vB111^ab}
Let $G$ be the subgroup of $\vP_3$ generated by the elements $\chi_{12}$, $\chi_{13}$ and $\chi_{23}$. Then any two of the three relations
\[
[\chi_{12},\chi_{23}]=1 \qquad [\chi_{13},\chi_{23}]=1 \qquad [\chi_{12},\chi_{13}]=1
\]
imply the third. In other words, the quotient of $G$ by any two of these relations is abelian.
\end{lemma}
\begin{proof}
The elements $\chi_{12}$, $\chi_{13}$ and $\chi_{23}$ satisfy the relation
\begin{equation}
\label{eq:pure-virtual-relation}
\chi_{12} \chi_{13} \chi_{23} \chi_{12}^{-1} \chi_{13}^{-1} \chi_{23}^{-1} = 1
\end{equation}
in $\vP_3$; see \cite[Th.~1]{Bardakov2004}. If we impose any two of the three relations above, then one of the three generators becomes central, and \eqref{eq:pure-virtual-relation} becomes the relation stating that the other two generators commute.
\end{proof}

\begin{proposition}\label{LCS_vB_1mu}
Let $\mu$ be a partition of an integer $n$ into blocks of size at least $4$. Then the LCS of $\vB_{1, \mu}$ stops at $\LCS_2$, and so does the LCS of $\wB_{1, \mu}$.
\end{proposition}

\begin{proof}
Notice first that the second statement follows directly from the first one, using Lemma~\ref{lem:stationary_quotient}, since $\wB_{1, \mu}$ is a quotient of $\vB_{1, \mu}$. Let us prove the first statement. The proof is similar to that of Proposition~\ref{LCS_B1mu}, with Lemma~\ref{vB111^ab} replacing Lemma~\ref{P3^ab}.  

Let us denote by $\lambda = (1, n_2, \ldots , n_l)$ the partition $(1, \mu)$ of $n+1$. The canonical morphism $\vB_\mu \hookrightarrow \vB_{1, \mu}$ induces a morphism:
\[\vB_\mu^{\ab} = \vB_\mu/\LCS_\infty \rightarrow \vB_{1, \mu}/\LCS_\infty,\]
where the equality on the left comes from the fact that all the blocks of $\mu$ have size at least $4$ (Proposition~\ref{LCS_wB_stable}). This implies that among the generators of $\vB_{1, \mu}$ from Proposition~\ref{generators_partitioned_v(w)B}, the ones coming from $\vB_\mu$ pairwise commute modulo $\LCS_\infty$. Moreover, this also implies that $\chi_{\alpha\beta}$ and $\chi_{\alpha'\beta'}$ have the same image $x_{ij}$ in $\vB_{1, \mu}/\LCS_\infty$ if $\alpha$ and $\alpha'$ are in the $i$-th block of $\lambda$ and $\beta$ and $\beta'$ are in the $j$-th one, if $i,j \geq 2$. Similarly, $\sigma_\alpha$ and  $\sigma_{\alpha'}$ (resp.~$\tau_\alpha$ and  $\tau_{\alpha'}$) have the same image $s_i$ (resp.~$t_i$) in this quotient if $\alpha$ and $\alpha'$ are in the $i$-th block of $\lambda$, with $i \geq 2$. Finally, a generator of the form $\chi_{1\alpha}$ (resp.~$\chi_{\alpha 1}$) is also sent to a class $x_{1i}$ (resp.~$x_{i1}$) modulo $\Gamma_\infty$  depending only on the index $i \geq 2$ of the block containing $\alpha$. Indeed, let us suppose that $\alpha$ and $\alpha+1$ are in the $i$-th block of $\lambda$. This block has size at least $4$, so $t_i$ has a representative $\tau_\beta$ in $\vB_{1, \mu}$ whose support does not contain $\alpha$, which then commutes with $\chi_{1\alpha}$. But $t_i$ is also the class of $\tau_\alpha$, and $\tau_\alpha \chi_{1\alpha}  \tau_\alpha = \chi_{1,\alpha+1}$. Thus $\overline \chi_{1\alpha} = t_i \overline \chi_{1\alpha} t_i = \overline \chi_{1, \alpha+1}$. The same argument shows that $\overline \chi_{\alpha 1}  = \overline \chi_{\alpha+1, 1}$.

In order to prove our statement, we need to show that the $x_{1i}$ and the $x_{j1}$ commute with each other, and with all the other generators of $\vB_{1, \mu}/\LCS_\infty$. The latter assertion is deduced from the fact that the pairs of generators involved have lifts to $\vB_{1, \mu}$ with disjoint support. The former one is a bit trickier, and requires the use of Lemma~\ref{vB111^ab}. Precisely, let $G$ denote the subgroup of the pure virtual braid group $\vP_3$ generated by the elements $\chi_{12}$, $\chi_{13}$ and $\chi_{23}$. In each case, we define a homomorphism $\theta \colon G \rightarrow \vB_{1,\mu}/\LCS_\infty$
such that two of the elements $\chi_{12}$, $\chi_{13}$ and $\chi_{23}$ are sent to the elements of $\vB_{1,\mu}/\LCS_\infty$ that we wish to show commute, and the third is sent to an element that commutes with both of them. It then follows from Lemma~\ref{vB111^ab} that the first two elements commute. In detail, the four cases are as follows, for $i,j \geq 2$ with $i \neq j$: 
\begin{itemize}
\item We wish to show that $x_{i1}$ and $x_{j1}$ commute. Let $\theta$ be induced by $1 \mapsto \alpha$, $2 \mapsto \beta$ and $3 \mapsto 1$, with $\alpha$ in the $i$-th block and $\beta$ in the $j$-th block of $\lambda$. The image of $\chi_{12}$ is $x_{ij}$, which commutes with $x_{i1} = \theta(\chi_{13})$ and $x_{j1} = \theta(\chi_{23})$, as we have deduced above from a disjoint support argument.
\item The case of $x_{1j}$ and $x_{1i}$ is similar to the previous one, the roles of $1$ and $3$ being exchanged.
\item Both $x_{i1}$ and $x_{1j}$ commute with $x_{ij}$, and again the same argument implies that they commute.
\item Finally, $x_{i1}$ and $x_{1i}$ are the classes of $\chi_{\alpha 1}$ and $\chi_{1, \alpha + 1}$ respectively, with $\alpha$ and $\alpha + 1$ both in the $i$-th block of $\lambda$. Since $x_{i1}$ and $x_{1i}$ both commute with $\overline \chi_{\alpha, \alpha + 1} = t_i s_i$, we can apply once again the same argument, $\theta$ being induced by $1 \mapsto 1$, $2 \mapsto \alpha$ and $3 \mapsto \alpha + 1$.
\end{itemize}
We have therefore shown that all generators in a generating set for $\vB_{1,\mu}/\LCS_\infty$ pairwise commute. Hence this quotient is abelian, in other words, $\LCS_\infty = \LCS_2$ for $\vB_{1,\mu}$.
\end{proof}

On the contrary, for extended welded braids, adding one block of size $1$ already prevents the LCS from stopping:

\begin{proposition}\label{LCS_of_exwB_1mu}
For any partition $\mu$ of an integer $n \geq 1$, the LCS of the partitioned extended welded braid group $\exwB_{1, \mu}$ does not stop.
\end{proposition}

\begin{proof}
Forgetting every block of $\mu$ save one induces a surjection from $\exwB_{1, \mu}$ onto $\exwB_{1, m}$ for some $m \geq 1$, so by Lemma~\ref{lem:stationary_quotient}, it is enough to show that the LCS of $\exwB_{1, m}$ does not stop. Let us consider $\exwB_{1, m}$ as a subgroup of $\Aut(\F_{1+m})$. One can check on the generators of $\exwB_{1, m}$ that the automorphisms in $\exwB_{1, m}$ preserve the normal subgroup $N$ of $\F_{1+m}$ generated by all the elements $x_\alpha x_{\alpha+1}^{-1}$ for $2 \leq \alpha \leq m$, together with the $x_\alpha^2$ for $2 \leq \alpha \leq m+1$. For instance, $\tau_{3}(x_2 x_3^{-1}) = x_2 x_4^{-1} = (x_2 x_3^{-1}) (x_3 x_4^{-1}) \in N$, $\rho_2(x_2 x_3^{-1}) = x_2^{-1} x_3^{-1} = (x_2^2)^{-1} (x_2 x_3^{-1}) \in N$, etc. As a consequence, the quotient of $\F_{1+m} \cong \Z *\F_m$ onto $\Z * \Z/2$ induces a well-defined morphism $\exwB_{1, m} \rightarrow \Aut(\Z * \Z/2)$. Let us denote again by $x_1$ and $x_2$ the generators of $\Z * \Z/2$, which are the images of $x_1, x_2 \in\F_{1+m}$. The image $G$ of $\exwB_{1, m}$ in $\Aut(\Z * \Z/2)$ is generated by $\rho_1$, $\chi_{12}$ and $\chi_{21}$, which are defined by the same formulas as the corresponding elements of $\Aut(\F_2)$, only with $x_2$ having square $1$. We note that $\chi_{21}$ and $\chi_{12}$ are conjugation by $x_1$ and by $x_2$ on $\Z * \Z/2$ respectively, so they generate $\Inn(\Z * \Z/2) \cong \Z * \Z/2$, which is normal in $\Aut(\Z * \Z/2)$, whence also in $G$. Moreover, $\rho_1$ is not an inner automorphism and, since $\rho_1^2 = 1$, we have:
\[G = \langle \chi_{21}, \chi_{12} \rangle \rtimes \langle \rho_1 \rangle \cong (\Z * \Z/2) \rtimes (\Z/2).\]
Conjugation by $\rho_1$ fixes $\chi_{12}$ and sends $\chi_{21}$ to its inverse. Hence the quotient of $G$ by $\chi_{12}$ is isomorphic to $\Z \rtimes (\Z/2)$, whose LCS does not stop by Proposition~\ref{LCS_of_Z/2*Z/2}.
\end{proof}

\begin{remark}\label{exwB11}
If $m = 1$, then the same kind of calculation can be applied directly in $\Aut(\F_2)$ to see that $\exwB_{1,1} = \exwP_2 \cong\F_2 \rtimes (\Z/2)^2$, where each generator of $(\Z/2)^2$ fixes one of the generators of $\F_2$, and sends the other one to its inverse. This group surjects onto $(\Z \rtimes (\Z/2)) \times (\Z/2)$ (by killing one of the generators of $\F_2$), so its LCS does not stop.
\end{remark}

Having dealt with all the tricky cases, we can finally prove the main result of this section:
\begin{proof}[Proof of Theorem \ref{LCS_of_partitioned_v(w)B}]
The stable case, where we have $n_i \geq 4$ for all $i$, is dealt with in Proposition~\ref{LCS_wB_stable}. 

If $\lambda$ has one block of size $2$ or $3$, then there is a surjection (defined by forgetting all strands except for those corresponding to this block) from $\vB_\lambda$ onto $\vB_2$ or $\vB_3$, thus we can use Lemma~\ref{lem:stationary_quotient} to deduce from Proposition~\ref{LCS_of_vBn} that its LCS does not stop. The same argument applies to welded braids (resp.~to extended welded braids), using Proposition~\ref{LCS_of_wBn} (resp.~Proposition~\ref{LCS_of_exwBn}) instead of Proposition~\ref{LCS_of_vBn}.

Similarly, if there are two blocks of size $1$, then both $\vB_\lambda$ and $\wB_\lambda$ surject onto $\vB_{1,1} = \wB_{1,1} = \vP_2 \cong\F_2$, whose LCS does not stop. Under this hypothesis, $\exwB_\lambda$ surjects onto $\exwB_{1,1} \cong\F_2 \rtimes (\Z/2)^2$, whose LCS does not stop either; see Remark~\ref{exwB11}.

Finally, the tricky cases where there is exactly one block of size $1$ and no blocks of size $2$ or $3$ are dealt with in Propositions~\ref{LCS_vB_1mu} and~\ref{LCS_of_exwB_1mu}.
\end{proof}

\chapter{Variants on partitioned welded braids}

This chapter is devoted to generalising the results of the previous chapter about (extended) welded braids. Precisely, we consider the configuration space of points, oriented circles and unoriented circles in $3$-dimensional space; its fundamental group then consists of welded braids with three kinds of strands that interact with each other. We call them tripartite welded braids. Notice that if there is only one kind of strand, we recover the group of welded braids, or of extended welded braids (or the symmetric group if we keep only strands corresponding to motions of points, which is less interesting). We can consider partitioned versions of these groups, defined in the obvious way (see Definition~\ref{def_partitioned_tripartite} for details). Our goal is to understand when the LCS of a partitioned tripartite welded braid group stops, a goal that is fully reached in Theorem~\ref{LCS_of_partitioned_exwBn_2}. This result contains the results of the previous chapter about the LCS of $\wB_\lambda$ and $\exwB_\lambda$; we have kept these separate only for the sake of clarity. Indeed, considering this larger family of groups requires heavier notations (mostly to distinguish between the three different types of strands), and the cases to consider are more numerous. However, if one looks beyond the complexity of notations and the increased number of cases, one will observe that this generalisation does not fundamentally increase the complexity of the problem; in particular, all the methods that we use to solve it have already been used in the previous chapters.

Our first aim is to show that every tripartite welded braid group can in fact be identified with a subgroup of the group of extended welded braids. We first prove this identification for a finite-index subgroup consisting, in a certain sense, of pure braids (\Spar\ref{par_pure_bipartite}). We then extend it to the whole group (\Spar\ref{par_tripartite}). It then applies to the subgroups of partitioned tripartite welded braids, which are introduced in~\Spar\ref{par_partitioned_tripartite}. This identification allows us to treat tripartite welded braids as automorphisms of the free group, an interpretation that we put to good use in our study of the LCS of partitioned tripartite welded braid groups (\Spar\ref{s:LCS_partitinoned_welded}).

\section{Pure bipartite welded braids}\label{par_pure_bipartite}

We begin with introducing the pure version of the group of tripartite welded braids, obtained from configurations of points and oriented circles. Namely, let $k$ and $m$ be integers. The \emph{pure bipartite} welded braid group $\wP(k,m)$ is defined as the fundamental group of the configuration space $F_{k \sqcup m \S^1}(\D^3)$ of $k$ points and $m$ oriented circles in the $3$-ball $\D^3$ (notations being as in~\Spar\ref{par_geom_wB}). In order to study it, we are going to identify this group with a subgroup of $\wP_{k+m}$ ($\subset \Aut(\F_{k+m})$). This is done by considering maps between configuration spaces. Let us first consider the map forgetting the $k$ points of each configuration.

\begin{lemma}\label{dec_bipartite_wP}
The canonical map $F_{k \sqcup m \S^1}(\D^3) \rightarrow F_{m \S^1}(\D^3)$ induces a split short exact sequence:
\[\begin{tikzcd} 
\PB_k(\bD^3_m) \ar[r, hook] 
&\wP(k,m) \ar[r, two heads] 
&\wP_m, \ar[l, bend right, dashed]
\end{tikzcd}\]
where $\bD^3_n$ is the complement of an $n$-component unlink in the interior of the unit $3$-disc. The kernel $\PB_k(\bD^3_m) = \pi_1(F_k(\bD^3_m))$ identifies with $k$ copies of the free group $\pi_1(\D^3_m) \cong\F_m$.
\end{lemma}

\begin{proof}
By \cite[Th.~C]{Palais}, the canonical map $p \colon F_{k \sqcup m \S^1}(\D^3) \rightarrow F_{m \S^1}(\D^3)$ is a locally trivial fibration. Its fibre identifies with $F_k(\D^3_m)$, where $\D^3_m$ denotes the (open) $3$-ball with $n$ unlinked, unknotted circles removed. Moreover, this fibration splits up to homotopy, that is, there exists $s \colon F_{m \S^1}(\D^3) \rightarrow F_{k \sqcup m \S^1}(\D^3)$ such that $p \circ s \simeq id$. Namely, one can choose a (smooth) isotopy equivalence $g$ between $\D^3$ and a proper subspace $D$ of $\D^3$ (explicitly, one can take $g \colon x \mapsto \frac12 x$ and $D = \frac12 \D^3$). This is an inverse up to isotopy of the inclusion $i_D \colon D \hookrightarrow \D^3$. Then one can fix a configuration $c_0$ of $k$ points outside of $D$, and let $s$ send a configuration $c$ of circles to $g(c) \sqcup c_0$. The isotopy $i_D \circ g \simeq id$ then induces a homotopy $p \circ s \simeq id$, as required. As a consequence, the long exact sequence in homotopy breaks into split short exact sequences, and the one between fundamental groups is exactly the one of the lemma. Moreover, the identification of the kernel comes directly from the classical triviality of braids on manifolds of dimension at least $3$: a pure braid with $k$ strands is just a collection of $k$ (homotopy classes of) loops; see for instance Proposition~\ref{Bn(M)} below.
\end{proof}

Let us now consider the map from the configuration space $F_{k \sqcup m \S^1}(\D^3)$ to $F_{(k+m)\S^1}(\D^3)$ obtained by replacing each point in a configuration by a small (oriented) circle.Such a map can be defined explicitly: one can replace each point $P$ of a given configuration by the horizontal circle of centre $P$ whose diameter equals half the distance between $P$ and the rest of the points and circles of the configuration.

\begin{proposition}\label{bipartite_wP_as_subgroup_of_wP}
The above map identifies $\wP(k,m)$ with the subgroup of $\wP_{k+m}$ generated by the $\chi_{\alpha\beta}$, for $1 \leq \alpha \neq \beta \leq n$ such that  $\beta > k$.
\end{proposition}

\begin{proof}
Using Lemma~\ref{dec_bipartite_wP}, we get a commutative diagram with (split) short exact rows:
\[\begin{tikzcd} 
\PB_k(\bD^3_m) \ar[r, hook] \ar[d] \ar[rd, dashed]
&\wP(k,m) \ar[r, two heads] \ar[d]
&\wP_m \ar[d, -, double line with arrow={-,-}] \\
\bullet \ar[r, hook]
&\wP_{k+m} \ar[r, two heads]
&\wP_m,
\end{tikzcd}\]
where the right-hand square is induced by the obvious commutative square at the level of configuration spaces. We want to show that the middle vertical map is injective. Using the Five Lemma, we see that we only need to show that the left vertical map is injective or, equivalently, that the dashed one is. Recall that $\PB_k(\D^3_m) \cong (\F_m)^k$ via $\pi_{\D^3_m}$; see for instance Proposition~\ref{Bn(M)} below. Moreover, by unravelling the definitions, we see that our map $\PB_k(\bD^3_m) \rightarrow \wP_{k+m}$ sends the generator $x_j$ of the $i$-th copy of $\F_n$ (corresponding to the $i$-th point passing through the $j$-th circle, with $i \leq k$ and $k < j \leq k+m$) to $\chi_{ij} \in  \wP_{k+m}$. It follows directly from the interpretation of $\wP_{k+m}$ as a group of automorphism of the free group $\F_{k+m}$ (see Theorem~\ref{exwB_in_Aut(Fn)}) that for each fixed $i \leq k$, the $\chi_{ij}$ for all $k < j \leq k+m$ generate a free group, and that these copies of $\F_m$ commute in $\wP_{k+m}$, so that our map must be an isomorphism onto its image. As a consequence the map $\wP(k, m) \rightarrow \wP_{k+m}$ under scrutiny is injective.

Now, let $G$ be the subgroup of $\wP_{k+m}$ generated by the $\chi_{\alpha\beta}$, for $1 \leq \alpha \neq \beta \leq n$ such that  $\beta > k$. All these elements are seen to be images of elements of $\wP(k, m)$, so that $G \subseteq \wP(k, m)$. Moreover, $G$ contains the image of $\PB_k(\bD^3_m)$ (described explicitly in the above reasoning), and all of $\wP_m$, so that $G$ is equal to the whole of $\wP(k, m) = \PB_k(\bD^3_m) \rtimes \wP_m$, and the proposition is proved. 
\end{proof}

\begin{remark}
In the semi-direct product decomposition $\wP(k, m) = (\F_m)^k \rtimes \wP_m$ from Lemma~\ref{dec_bipartite_wP}, the action of $\wP_m$ on $(\F_m)^k$ is diagonal, the action on each copy of $\F_m$ being the canonical one (corresponding to the injection $\wP_m \subset \Aut(F_m)$. One may see this by unravelling the geometric definitions of these objects; however, it is much easier to deduce this from Proposition~\ref{bipartite_wP_as_subgroup_of_wP}, which allows us to see this action as a restriction of conjugation in $\wP_{k+m} \subset \Aut(\F_{k+m})$, and to compute it by writing explicit automorphisms of the free group.
\end{remark}

\begin{corollary}\label{bipartite_wP_in_Aut(Fn)}
The Artin representation $\wP_{k+m} \hookrightarrow \Aut(\F_{m+k})$ identifies $\wP(k,m)$ with the group of automorphisms of $\F_{k+m}$ of the form $x_i \mapsto w_i x_i w_i^{-1}$, with $w_i \in \langle x_{k+1}, \ldots , x_{k+m} \rangle$ for all $i$. 
\end{corollary}

\begin{proof}
Let $G$ be the subgroup of $\Aut(\F_{k+m})$ described by the condition of the statement. If we look only at its action of $x_{k+1}, \ldots , x_{k+m}$, we see that every element of $G$ restricts to a basis-conjugating automorphism of $\langle x_{k+1}, \ldots , x_{k+m} \rangle \cong \F_m$. This defines a morphism from $G$ to $\wP_m$, which admits a section, given by extending automorphisms by $x_i \mapsto x_i$ if $i \leq k$. The kernel of this split projection is easily seen to be the copy of $(\F_m)^k$ generated by the $\chi_{ij}$ with $i \leq k < j$ already described in the proof of Proposition~\ref{bipartite_wP_as_subgroup_of_wP}. Moreover, since $\wP_m$ is (classically) generated by the $\chi_{ij}$ (for $1 \leq i \neq j \leq m$), the image of this section is generated by the $\chi_{ij}$, for $k < i \neq j \leq k+m$. Thus $G$ is the semi-direct product of $\langle \chi_{ij} \rangle_{i \leq k < j} \cong (\F_m)^k$ and $\langle \chi_{ij} \rangle_{k < i,j} \cong \wP_m$, so Proposition~\ref{bipartite_wP_as_subgroup_of_wP} implies that it is equal to $\wP(k,m)$.
\end{proof}

\begin{corollary}\label{bipartite_wP^ab}
The abelianisation $\wP(k,m)^{\ab}$ is free abelian on the classes of the $\chi_{\alpha\beta}$, for $1 \leq \alpha \neq \beta \leq n$ such that  $\beta > k$.
\end{corollary}

\begin{proof}
By Proposition~\ref{bipartite_wP_as_subgroup_of_wP}, these classes generate $\wP(k,m)^{\ab}$. Moreover, the canonical map $\wP(k,m)^{\ab} \rightarrow \wP_{k+m}^{\ab}$ sends these generators to linearly independent elements of $\wP_{k+m}^{\ab}$ (see Proposition~\ref{McCool_presentation}), so that they are a basis of $\wP(k,m)^{\ab}$.
\end{proof}

\section{Tripartite welded braids}\label{par_tripartite}

Given three integers $n_P$, $n_{S_+}$ and $n_S$, we define the group of \emph{tripartite welded braids} by:
\[\wB(n_P, n_{S_+}, n_S) := \pi_1\left(
F_{n_P \sqcup (n_{S_+} + n_S)\S^1}(\D^3)/(\Sym_{n_P} \times \Sym_{n_{S_+}}\!\times W_{n_S}) 
\right),\]
where the action by which we quotient is defined similarly to the ones in \Spar\ref{par_geom_wB} and \Spar\ref{par_geom_exwB}. In other words, we consider the fundamental group of the configuration space of $n_P$ points (unordered), $n_{S_+}$ oriented circles (unordered) and $n_S$ unoriented ones (unordered too), where all the circles are supposed unlinked and unknotted.

In order to study it, we identify this group with a subgroup of $\exwB_n$, where $n =  n_P + n_{S_+} + n_S$. Since $\exwB_n$ identifies with a subgroup of $\Aut(\F_n)$ (see Theorem~\ref{exwB_in_Aut(Fn)}), we will then be able to identify tripartite welded braids with automorphisms of $\F_n$, and this point of view will come in handy for doing explicit computations.   

Let us consider the map from $F_{n_P \sqcup (n_{S_+} + n_S) \S^1}(\D^3)$ to $F_{n\S^1}(\D^3)$ defined, as above, by replacing each point in a configuration by a small oriented circle. By factorisation through quotients by the appropriate group actions, we get a map between the configuration spaces of which $\wB(n_P, n_{S_+}, n_S)$ and $\exwB_n$ are the fundamental groups (this boils down to replacing points in configurations by small circles, and forgetting orientations of circles).
\begin{proposition}\label{tripartite_wB_as_subgroup_of_exwB}
The above map identifies $\wB(n_P, n_{S_+}, n_S)$ with the subgroup of $\exwB_n$ with $n =  n_P + n_{S_+} + n_S$ generated by:
\begin{itemize}
\item the $\tau_\alpha$ for $\alpha \neq n_P, n_P + n_{S_+}$ and $1 \leq \alpha < n$;
\item the $\rho_\alpha$ for $n_P + n_{S_+} < \alpha \leq n$;
\item the $\chi_{\alpha\beta}$, for $1 \leq \alpha \neq \beta \leq n$ such that  $\beta > n_P$.
\end{itemize}
\end{proposition}

\begin{proof}
Consider the following commutative square of maps between configuration spaces:
\[\begin{tikzcd} 
F_{n_P, (n_{S_+} + n_S)\S^1}(\D^3) \ar[r, two heads] \ar[d]
&F_{n_P, (n_{S_+} + n_S)\S^1}(\D^3)/(\Sym_{n_P} \times \Sym_{n_{S_+}}\!\times W_{n_S})  \ar[d] \\
F_{n \S^1}(\D^3) \ar[r, two heads]
&F_{n \S^1}(\D^3)/W_n,
\end{tikzcd}\]
where the horizontal maps are regular coverings. These induce the left square of the following commutative diagram, whose rows are exact (thanks to the usual theory of coverings):
\[\begin{tikzcd} 
\wP(n_P, n_{S_+} + n_S) \ar[r, hook] \ar[d]
&\wB(n_P, n_{S_+}, n_S)  \ar[d] \ar[r, two heads]
&\Sym_{n_P} \times \Sym_{n_{S_+}}\!\times W_{n_S} \ar[d] \\
\wP_n \ar[r, hook]
&\exwB_n \ar[r, two heads]
&W_n.
\end{tikzcd}\]
The vertical map on the right is obviously injective. By Lemma~\ref{bipartite_wP_as_subgroup_of_wP}, the one on the left is also injective. So, by the Five Lemma, the middle vertical map  $\wB(n_P, n_{S_+}, n_S) \rightarrow \exwB_n$ is too.  

Now, let $G$ be the subgroup of $\exwB_n$ generated by the elements listed in the statement. First, one easily sees that each element in the list is the image of some element of $\wB(n_P, n_{S_+}, n_S)$, so that $G \subseteq \wB(n_P, n_{S_+}, n_S)$. Moreover, $G$ contains $\wP(n_P, n_{S_+} + n_S)$ by Proposition~\ref{bipartite_wP_as_subgroup_of_wP}, and it also contains $\Sym_{n_P} \times \Sym_{n_{S_+}} \times W_{n_S}$, which is generated by the $\tau_\alpha$ and the $\rho_\alpha$ contained in $G$. Thus, $G$ is all of their semi-direct product $\wB(n_P, n_{S_+}, n_S)$, and our statement is proved.
\end{proof}

\begin{corollary}\label{tripartite_wB_in_Aut(Fn)}
The Artin representation $\exwB_n \hookrightarrow \Aut(\F_n)$ identifies the subgroup $\wB(n_P, n_{S_+}, n_S)$ with the group of automorphisms of $\F_n$ of the form 
\[x_i \longmapsto w_i x_{\sigma(i)}^{\epsilon_i} w_i^{-1},\] 
with $\sigma \in \Sym_{n_P} \times \Sym_{n_{S_+}}\!\times \Sym_{n_S} \subseteq \Sym_n$, $w_i \in \langle x_{n_P+1}, \ldots , x_n \rangle$ and $\epsilon_i \in \{\pm 1\}$ for all $i$, such that $\epsilon_i = 1$ if $i \leq n_P + n_{S_+}$.
\end{corollary}

\begin{proof}
Let $G$ be the subgroup of $\Aut(F_n)$ defined by the conditions of the statement. The canonical map from $\Aut(\F_n)$ to $\Aut(\F_n^{\ab}) \cong GL_n(\Z)$ restricts to a projection of $G$ onto a group of signed permutation matrices isomorphic to $\Sym_{n_P} \times \Sym_{n_{S_+}}\!\times W_{n_S}$. This projection is split, and the image of the obvious splitting is generated by the $\tau_\alpha$ and the $\rho_\alpha$ listed in the statement of Proposition~\ref{tripartite_wB_as_subgroup_of_exwB}. Moreover, Corollary~\ref{bipartite_wP_in_Aut(Fn)} identifies the kernel of this projection as $\wP(n_P,  n_{S_+} + n_S)$. Proposition~\ref{bipartite_wP_as_subgroup_of_wP} says that the latter is generated by the $\chi_{\alpha\beta}$ listed in Proposition~\ref{tripartite_wB_as_subgroup_of_exwB}. Thus, $G$ is generated by the elements listed in the statement of Proposition~\ref{tripartite_wB_as_subgroup_of_exwB}, so it is equal to $\wB(n_P, n_{S_+}, n_S)$.
\end{proof}

\section{Partitioned tripartite welded braids}\label{par_partitioned_tripartite}

Let us now introduce a partitioned version of this group. In order to do this, we use the canonical projection $\pi$ from $\wB(n_P, n_{S_+}, n_S)$ onto $\Sym_{n_P} \times \Sym_{n_{S_+}} \times \Sym_{n_S}$, which identifies (using Proposition~\ref{tripartite_wB_as_subgroup_of_exwB}) with a restriction of $\pi \colon \exwB_n \twoheadrightarrow \Sym_n$.
\begin{definition}\label{def_partitioned_tripartite}
Let $\lambda_P$, $\lambda_{S_+}$ and $\lambda_S$ be partitions of integers $n_P$, $n_{S_+}$ and $n_S$ respectively. Let $\lambda =\lambda_P \lambda_{S_+} \lambda_S$ be their concatenation, which is a partition of $n =  n_P + n_{S_+} + n_S$. The blocks of $\lambda$ will often be identified with blocks of $\lambda_P$, $\lambda_{S_+}$ and $\lambda_S$. The associated group of \emph{tripartite welded braids} $\wB(\lambda_P, \lambda_{S_+}, \lambda_S)$ is defined by:
\begin{equation*}
\resizebox{\hsize}{!}{$\wB(\lambda_P, \lambda_{S_+}, \lambda_S) = \pi^{-1}(\Sym_\lambda) =  \pi^{-1}(\Sym_{\lambda_P} \times \Sym_{\lambda_{S_+}} \times \Sym_{\lambda_S}) \subseteq \wB(n_P, n_{S_+}, n_S) \subset \exwB_n.
$}
\end{equation*}
\end{definition}

A direct adaptation of the proof of Lemma~\ref{generators_of_partitioned_braids} to this context (using Proposition~\ref{bipartite_wP_as_subgroup_of_wP} to supply generators of the subgroup of pure braids) gives:
\begin{lemma}\label{generators_partitioned_tripartite}
As a subgroup of $\exwB_n$, $\wB(\lambda_P, \lambda_{S_+}, \lambda_S)$ is generated by:
\begin{itemize}
\item the $\tau_\alpha$ for $1 \leq \alpha < n$ such that $\alpha$ and $\alpha + 1$ are in the same block of $\lambda$;
\item the $\sigma_\alpha$ for $n_P < \alpha < n$ such that $\alpha$ and $\alpha + 1$ are in the same block of $\lambda$;
\item the $\rho_\alpha$ for $n_P + n_{S_+} < \alpha \leq n$;
\item the $\chi_{\alpha\beta}$, for $1 \leq \alpha \neq \beta \leq n$ with  $\beta > n_P$, such that $\alpha$ and $\beta$ are not in the same block of $\lambda$.
\end{itemize}
\end{lemma}

We can compute the abelianisation of the group $\wB(\lambda_P, \lambda_{S_+}, \lambda_S)$, like we did for partitioned groups before:
\begin{proposition}\label{tripartite_wB^ab}
Let $\lambda_P$, $\lambda_{S_+}$ and $\lambda_S$ be partitions of integers $n_P$, $n_{S_+}$ and $n_S$ respectively, of respective length $l_P$, $l_{S_+}$ and $l_S$. Let us denote by $l'_P$, $l'_{S_+}$ and $l'_S$ the number of blocks of size at least two in each of these partitions. Let also $\lambda_P \lambda_{S_+} \lambda_S = \lambda = (n_1, \ldots , n_l)$ be their concatenation, of length $l = l_P + l_{S_+} + l_S$, and with $l' = l'_P + l'_{S_+} + l'_S$ blocks of size at least two. Then:
\[\wB(\lambda_P, \lambda_{S_+}, \lambda_S)^{\ab} \cong \Z^N \times (\Z/2)^M,\]
where $N = l'_{S_+} + l_{S_+}(l - 1)$ and $M = l' + l'_S + l_S l$.

Let us denote by $I_P$ the set $\{1, \ldots , l_P\}$ of indices corresponding to blocks of $\lambda_P$, by $I_{S_+}$ the set $\{l_P + 1, \ldots , l_P + l_{S_+}\}$ and by $I_S$ the set $\{l_P + l_{S_+} + 1, \ldots , l\}$. With these notations, a basis of the first factor is given by:
\begin{itemize}
\item for each $i \in I_{S+}$ such that $n_i \geq 2$, one generator $s_i$, 
\item for each $j \in I_{S+}$ and each $i \in \{1, \ldots , l\}$ such that $i \neq j$, one generator~$x_{ij}$.
\end{itemize}
and a $\Z/2$-basis of the second factor by:
\begin{itemize}
\item for each $i \in I_S$ such that $n_i \geq 2$, one generator $s_i$, 
\item for each $i \in \{1, \ldots , l\}$ such that $n_i \geq 2$, one generator $t_i$, 
\item for each $i \in I_S$, one generator $r_i$, 
\item for each $j \in I_S$ and each $i \in \{1, \ldots , l\}$ such that $i \neq j$, one generator $x_{ij}$.
\end{itemize}
The generators are obtained as follows:
\begin{itemize}
\item $s_i$ (resp.~$t_i$) is the common class of the  $\sigma_\alpha$ (resp.~$\tau_\alpha$) for $\alpha$ and $\alpha +1$ in the $i$-th block of $\lambda$,
\item $r_i$ is the common class of the  $\rho_\alpha$ for $\alpha$ in the $i$-th block of $\lambda$, 
\item $x_{ij}$ is the common class of the  $\chi_{\alpha\beta}$ for $\alpha$ in the $i$-th block of $\lambda$ and $\beta$ in the $j$-th one.
\end{itemize}
\end{proposition}

\begin{proof}
The proof is very similar to that of Proposition~\ref{partitioned_exwB^ab}, which is a particular case of the present general statement. Namely, we apply Lemma~\ref{lem:abelianization semidirect} to the decomposition:
\[\wB(\lambda_P, \lambda_{S_+}, \lambda_S) \cong \wP(n_P, n_{S_+} + n_S) \rtimes (\Sym_{\lambda_P} \times \Sym_{\lambda_{S_+}} \times W_{\lambda_S}).\]
The abelianisation of the second factor identifies with $(\Z/2)^{l' + l_S}$, generated by the generators $t_i$ and $r_i$ of the statement. The abelianisation of $\wP(n_P, n_{S_+} + n_S)$ is free on the classes of the $\chi_{\alpha \beta}$ with $1 \leq \alpha \neq \beta \leq n$ such that  $\beta > n_P$ by Corollary~\ref{bipartite_wP^ab}, and the action of $\Sym_{\lambda_P} \times \Sym_{\lambda_{S_+}} \times W_{\lambda_S}$ is by permutation of the indices. We note that this latter group identifies with $(\Z/2)^{n_S} \rtimes \Sym_\lambda$, where $\Sym_\lambda = \Sym_{\lambda_P} \times \Sym_{\lambda_{S_+}} \times \Sym_{\lambda_S}$ acts on $(\Z/2)^{n_S}$ through the projection on its third factor (whose action permutes the factors). Thus, the coinvariants we need to compute are:
\[\left(\wP(n_P, n_{S_+} + n_S)^{\ab}\right)_{(\Z/2)^{n_S} \rtimes \Sym_\lambda} 
\cong \left(\left(\wP(n_P, n_{S_+} + n_S)^{\ab}\right)_{(\Z/2)^{n_S}}\right)_{\Sym_\lambda}.\]
The computation then continues exactly as in the proof of Proposition~\ref{partitioned_v(w)B^ab}. In particular, $\left(\wP(n_P, n_{S_+} + n_S)^{\ab}\right)_{(\Z/2)^{n_S}}$ is a product of factors $\Z$ and $\Z/2$ generated by the $\overline \chi_{\alpha \beta}$, respectively with $n_P < \beta \leq n_P + n_{S_+}$ and with  $\beta > n_P + n_{S_+}$. The classes of these generators in the coinvariants give the generators $x_{ij}$ of the statement, together with the generators $s_i$, after the same change of basis as in the proof of Proposition~\ref{partitioned_exwB^ab}.
\end{proof}

From a geometrical point of view, the generators described in Proposition~\ref{tripartite_wB^ab} correspond to the motions of points and circles in the $3$-ball represented in Figure~\ref{fig:tripartite-braids}.

\begin{figure}[ht]
\centering
\includegraphics[scale=0.45]{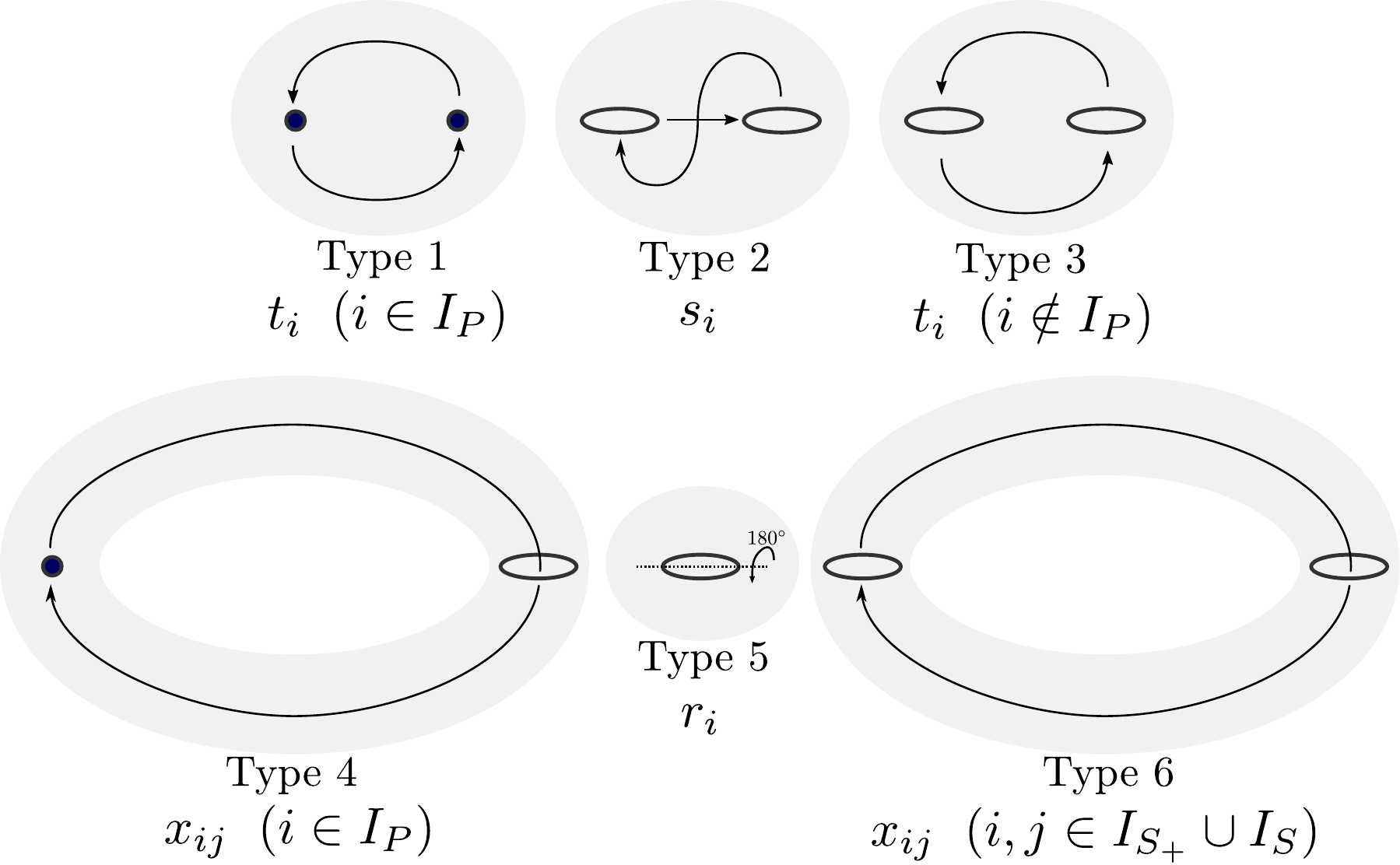}
\caption{The six types of generators for $\wB(\lambda_P, \lambda_{S_+}, \lambda_S)^{\ab}$ from Proposition \ref{tripartite_wB^ab}. Types $2$--$4$ each have two sub-types depending on whether the circles are (both) oriented or (both) unoriented. The circle in Type $5$ must be unoriented. Type $6$ has four sub-types, depending on which of the circles are oriented. Type $2$ generates a $\Z$ summand if both circles are oriented. Types $4$ and $6$ generate a $\Z$ summand if the non-moving circle is oriented. All other types generate a $\Z/2$ summand.}
\label{fig:tripartite-braids}
\end{figure}

\section{The lower central series}\label{s:LCS_partitinoned_welded}

We now study the LCS of $\wB(\lambda_P, \lambda_{S_+}, \lambda_S)$. In order to do so, we need to apply disjoint support arguments. These arguments are best understood by thinking of motions having disjoint support in the $3$-ball, and the reader is advised to keep this point of view in mind. However, writing down precise arguments and explicit calculations is much easier when dealing with automorphisms of free groups, so we mainly identify elements of $\wB(\lambda_P, \lambda_{S_+}, \lambda_S)$ with automorphisms of free groups in our proofs, using Proposition~\ref{tripartite_wB_as_subgroup_of_exwB} and Theorem~\ref{exwB_in_Aut(Fn)}.

The remainder of this section is devoted to the proof of the following theorem. Sections \Spar\ref{subsubsec:tripartite_stable_LCS}--\ref{subsubsec:tripartite_block_two_points_LCS} study each particular case of its statement; its proof is then deduced from these at the very end of \Spar\ref{s:LCS_partitinoned_welded}.

\begin{theorem}\label{LCS_of_partitioned_exwBn_2}
Let $\lambda_P$, $\lambda_{S_+}$ and $\lambda_S$ denote partitions of integers $n_P$, $n_{S_+}$ and $n_S$ respectively, of respective lengths $l_P$, $l_{S_+}$ and $l_S$. Let also $\lambda_P \lambda_{S_+} \lambda_S = \lambda = (n_1,\ldots, n_l)$ be their concatenation, of length $l = l_P + l_{S_+} + l_S$. We denote by $I_P$ the set $\{1,\ldots, l_P\}$ of indices corresponding to blocks of $\lambda_P$, by $I_{S_+}$ the set $\{l_P + 1,\ldots, l_P + l_{S_+}\}$ and by $I_S$ the set $\{l_P + l_{S_+} + 1,\ldots, l\}$. 

Let us suppose that $\wB(\lambda_P, \lambda_{S_+}, \lambda_S)$ is not trivial, i.e.~$(\lambda_P, \lambda_{S_+}, \lambda_S)$ is not among $(\varnothing,\varnothing,\varnothing)$, $((1, 1, \ldots , 1), \varnothing, \varnothing)$ and $(\varnothing,1,\varnothing)$. Then the LCS of the group $\wB(\lambda_P, \lambda_{S_+}, \lambda_S)$:
\begin{itemize}
\item \emph{stops at $\LCS_2$} if $n_i \neq 2$ for all $i \in I_P$ and $n_i \geq 4$ for all $i \in I_{S_+}\cup I_S$, save for at most one $i \in I_{S_+}$ such that $n_i = 1$;
\item \emph{stops at $\LCS_3$} if $n_i = 2$ for at least one $i \in I_P$, $I_{S_+} = \varnothing$, $I_{S} \neq \varnothing$ and $n_i \geq 4$ for all $i \in I_S$;
\item \emph{does not stop} in all the other cases, except for $\wB(\varnothing, \varnothing, 1) = \exwB_1 \cong \Z/2$.
\end{itemize}
\end{theorem}

\subsection{The stable case}\label{subsubsec:tripartite_stable_LCS}

We begin with the cases where all the blocks of $\lambda$ are big enough to apply a disjoint support argument:

\begin{proposition}\label{stable_LCS_of_partitioned_exwBn_2}
The notations being as in Theorem~\ref{LCS_of_partitioned_exwBn_2}, suppose that $n_i \geq 3$ for all $i \in I_P$ and $n_i \geq 4$ for all $i \in I_{S_+}\cup I_S$. Then the LCS of the group $\wB(\lambda_P, \lambda_{S_+}, \lambda_S)$ stops at $\LCS_2$.
\end{proposition}

\begin{proof}
The proof is an extension of that of Theorem~\ref{LCS_of_partitioned_v(w)B}. By Corollary~\ref{commuting_representatives}, we just have to show that any pair of the generators of $\wB(\lambda_P, \lambda_{S_+}, \lambda_S)^{\ab}$ listed in Proposition~\ref{tripartite_wB^ab} (and illustrated in Figure~\ref{fig:tripartite-braids}) has a pair of lifts in $\wB(\lambda_P, \lambda_{S_+}, \lambda_S)$ with disjoint support. The assumption that $n_i \geq 3$ for all $i \in I_P$ and $n_i \geq 4$ for all $i \in I_{S_+}\cup I_S$ is exactly what we need for that, as one sees by direct inspection. For instance, a generator $x_{ij}$ has a lift  whose support is $\{\alpha, \beta\}$ (namely, $\chi_{\alpha \beta}$), for any choice of $\alpha$ in the $i$-th block and $\beta$ in the $j$-th block of $\lambda$. Thus any $x_{kl}$ has a lift whose support is disjoint from some lift of $x_{ij}$, provided that the $i$-th and $j$-th blocks of $\lambda$ both have size at least $2$.
\end{proof}

\subsection{Blocks of points and oriented circles of size 1} The result in the stable case can be extended to the cases where we add blocks of size $1$ to $\lambda_P$, or one block of size $1$ to $\lambda_{S_+}$, or both. These three cases correspond respectively to Propositions~\ref{tripartite_wB_isolated_P}, \ref{tripartite_wB_isolated_S+} and \ref{tripartite_wB_isolated_P_and_S+} below. In each case, the strategy of proof is very similar to the proof of Proposition~\ref{LCS_vB_1mu}: most of the pairs of generators are already known to commute, using the previous cases and disjoint support arguments; the remaining ones are dealt with using Lemma~\ref{vB111^ab} which, for welded braids, takes the simpler form:
\begin{lemma}\label{wB111^ab}
Let $G$ be the subgroup of $\wP_3$ generated by the elements $\chi_{12}$, $\chi_{13}$ and $\chi_{23}$. Then $\LCS_2(G)$ is normally generated by $[\chi_{12},\chi_{23}]$ (resp.~by $[\chi_{12},\chi_{13}]$). In other words, the quotient of $H$ by the single relation $[\chi_{12},\chi_{23}] = 1$ (resp.~$[\chi_{12},\chi_{13}] = 1$) is abelian. 
\end{lemma}

\begin{proof}
This is a direct corollary of Lemma~\ref{vB111^ab}, noticing that $G$ is a quotient of the subgroup $\langle \chi_{12}, \chi_{13}, \chi_{23} \rangle$ of $\vP_3$ where $[\chi_{13}, \chi_{23}]$ is already killed. 
\end{proof}

Let us first consider the case where blocks of size $1$ are added to $\lambda_P$:

\begin{proposition}\label{tripartite_wB_isolated_P}
The notations being those of Theorem~\ref{LCS_of_partitioned_exwBn_2}, if $n_i \geq 4$ for all $i \in I_{S_+}\cup I_S$ and $\lambda_P$ has no blocks of size $2$, then the LCS of $\wB(\lambda_P, \lambda_{S_+}, \lambda_S)$ stops at $\LCS_2$. 
\end{proposition}

\begin{proof}
We argue like in the proof of Proposition~\ref{LCS_vB_1mu}.

Let us write $\lambda_P = (1,\ldots,1,\mu_P)$, where there are $k$ ones and each block of $\mu_P$ has size at least~$3$. There is an obvious homomorphism $\wB(\mu_P,\lambda_{S_+},\lambda_S) \to \wB(\lambda_P,\lambda_{S_+},\lambda_S)$, which induces a homomorphism
\[\wB(\mu_P,\lambda_{S_+},\lambda_S)^{\ab} = \wB(\mu_P,\lambda_{S_+},\lambda_S) / \LCS_\infty \longrightarrow \wB(\lambda_P,\lambda_{S_+},\lambda_S) / \LCS_\infty .
\]
The equality on the left-hand side comes from Proposition~\ref{stable_LCS_of_partitioned_exwBn_2}. Similarly, for each block $i$ of the partition $\lambda_{S_+} \cup \lambda_S$, there is an obvious homomorphism from $\wB((1,\ldots,1),n_i,\varnothing)$ or $\wB((1,\ldots,1),\varnothing,n_i)$ to $\wB(\lambda_P,\lambda_{S_+},\lambda_S)$, inducing a homomorphism
\begin{align*}
\begin{split}
\label{eq:wB_induced_on_LCSinfty-2}
\wB((1,\ldots,1),n_i,\varnothing)^{\ab} = \wB((1,\ldots,1),n_i,\varnothing) / \LCS_\infty &\longrightarrow \wB(\lambda_P,\lambda_{S_+},\lambda_S) / \LCS_\infty \\
\text{or} \ \ \ \wB((1,\ldots,1),\varnothing,n_i)^{\ab} = \wB((1,\ldots,1),\varnothing,n_i) / \LCS_\infty &\longrightarrow \wB(\lambda_P,\lambda_{S_+},\lambda_S) / \LCS_\infty .
\end{split}
\end{align*}
The equality on the left-hand side comes from the fact that the disjoint support argument used in the proof of Proposition~\ref{stable_LCS_of_partitioned_exwBn_2} also works for showing that $\LCS_2 = \LCS_\infty$ in this particular case where there is only one block of circles.

Let us consider the generators $\tau_\alpha , \sigma_\alpha , \rho_\alpha , \chi_{\alpha\beta}$ of $\wB(\lambda_P,\lambda_{S_+},\lambda_S)$ introduced in Lemma~\ref{generators_partitioned_tripartite}. Using Proposition~\ref{tripartite_wB^ab}, we deduce from the existence of the morphisms obtained above that their classes modulo $\LCS_\infty$ depend only on the blocks to which $\alpha$ and $\beta$ belong. In other words, if $\alpha$ and $\alpha'$ belong to the same block of the partition, then $\sigma_\alpha = \sigma_{\alpha'}$ in $\wB(\lambda_P,\lambda_{S_+},\lambda_S) / \LCS_\infty$, etc. Once we know this, it is easy to use a disjoint support argument to see that modulo $\LCS_\infty$, all of these generators commute pairwise, with the possible exception of $\chi_{\alpha\beta}$ and  $\chi_{\alpha\beta'}$, where $1\leq \alpha \leq k$ (so it corresponds to a block of points of size $1$) and where $\beta$ and $\beta'$ lie in distinct blocks of $\lambda_{S_+} \cup \lambda_S$. We deal with these as follows. Let us suppose that $\beta$ and $\beta'$ both lie in $\lambda_{S_+}$ (the argument in the other cases is similar). Consider the ``adding strands'' homomorphism
\begin{equation}\label{eq:wB_induced_on_LCSinfty-3}
\wB(1,(1,1),\varnothing) \longrightarrow \wB(\lambda_P,\lambda_{S_+},\lambda_S) \twoheadrightarrow \wB(\lambda_P,\lambda_{S_+},\lambda_S) / \LCS_\infty
\end{equation}
induced by $1 \mapsto \alpha$, $2 \mapsto \beta$ and $3 \mapsto \beta'$. We know that the images of the elements $\chi_{12}$ and $\chi_{23}$ of the domain commute in $\wB(\lambda_P,\lambda_{S_+},\lambda_S) / \LCS_\infty$. Thus the homomorphism \eqref{eq:wB_induced_on_LCSinfty-3} factors through the quotient of the left-hand side by $[\chi_{12},\chi_{23}]$. Then Lemma~\ref{wB111^ab} implies that the image $\chi_{\alpha \beta}$ of $\chi_{12}$ commutes with the image $\chi_{\alpha \beta'}$ of $\chi_{13}$ in $\wB(\lambda_P,\lambda_{S_+},\lambda_S) / \LCS_\infty$. We have now shown that all the generators of $\wB(\lambda_P,\lambda_{S_+},\lambda_S) / \LCS_\infty$ commute pairwise, so it is abelian. Thus $\LCS_2 = \LCS_\infty$ for $\wB(\lambda_P,\lambda_{S_+},\lambda_S)$.
\end{proof}

The result in the stable case (Proposition~\ref{stable_LCS_of_partitioned_exwBn_2}) can also be extended to the case where $\lambda_{S_+}$ has one  block of size $1$, thus generalising the case of $\wB_{1, \mu}$ from Proposition~\ref{LCS_vB_1mu}:

\begin{proposition}\label{tripartite_wB_isolated_S+}
The notations being as in Theorem~\ref{LCS_of_partitioned_exwBn_2}, suppose that $n_i \geq 3$ for all $i \in I_P$ and $n_i \geq 4$ for all $i \in I_{S_+}\cup I_S$, save for one $i \in I_{S_+}$ such that $n_i = 1$. Then the LCS of the group $\wB(\lambda_P, \lambda_{S_+}, \lambda_S)$ stops at $\LCS_2$.
\end{proposition}

\begin{proof}
Up to some permutation of the blocks, we can suppose that $\lambda_{S_+} = (1, \mu)$ for some partition $\mu$ whose blocks have size at least $4$. Then, as in the proof of Proposition~\ref{LCS_vB_1mu}, we have a morphism:
\[\wB(\lambda_P, \mu, \lambda_S)^{\ab} = \wB(\lambda_P, \mu, \lambda_S)/\LCS_\infty \rightarrow \wB(\lambda_P, \lambda_{S_+}, \lambda_S)/\LCS_\infty,\] 
where the first equality follows from Proposition~\ref{stable_LCS_of_partitioned_exwBn_2}. Starting from this, the generalisation of the proof of Proposition~\ref{LCS_vB_1mu} to this context is straightforward.
\end{proof}

In fact, the same result holds when there are both some blocks of points of size $1$ and one block of oriented circles of size $1$, the size of the other blocks being in the stable range:

\begin{proposition}\label{tripartite_wB_isolated_P_and_S+}
The notations being as in Theorem~\ref{LCS_of_partitioned_exwBn_2}, suppose that $n_i \neq 2$ for all $i \in I_P$, and $n_i \geq 4$ for all $i \in I_{S_+}\cup I_S$, save for one $i \in I_{S_+}$ such that $n_i = 1$. Then the LCS of the group $\wB(\lambda_P, \lambda_{S_+}, \lambda_S)$ stops at $\LCS_2$.
\end{proposition}

\begin{proof}
Let us denote by $\mu_P$ the partition obtained from $\lambda_P$ by removing blocks of size $1$, and by $\mu_{S_+}$ the one obtained from $\lambda_{S_+}$ by removing the block of size $1$. Then we have two morphisms:
\begin{align*}
\wB(\mu_P, \lambda_{S_+}, \lambda_S)^{\ab} = \wB(\mu_P, \lambda_{S_+}, \lambda_S)/\LCS_\infty \longrightarrow \wB(\lambda_P, \lambda_{S_+}, \lambda_S)/\LCS_\infty,
\\ \wB(\lambda_P, \mu_{S_+}, \lambda_S)^{\ab} = \wB(\lambda_P, \mu_{S_+}, \lambda_S)/\LCS_\infty \longrightarrow \wB(\lambda_P, \lambda_{S_+}, \lambda_S)/\LCS_\infty,
\end{align*}
where the equalities on the left come from Propositions~\ref{tripartite_wB_isolated_P} and \ref{tripartite_wB_isolated_S+}. As a consequence, the classes of the $\chi_{\alpha \beta}$, the $\tau_\alpha$, the $\sigma_\alpha$ and the $\rho_\alpha$ generating $\wB(\lambda_P, \lambda_{S_+}, \lambda_S)$ in the quotient by $\LCS_\infty$ depend only on the blocks containing $\alpha$ and $\beta$. Indeed, this is obvious when $\alpha$ and $\beta$ both belong to blocks of size one, which is the only case where it cannot directly be deduced from the above description of the abelianisations (Proposition~\ref{tripartite_wB^ab}). As usual, when $\alpha$ is in the $i$-th block and $\beta$ in the $j$-th one, we denote these classes by $x_{ij}$, $t_i$, $s_i$ and $r_i$ respectively.

Now, let us consider two generators among these. Most of the time, both of them belong to the image of one of the morphisms above, whence they commute. When this does not hold, a disjoint support argument shows that they still commute, except when one of them is $x_{ij}$ for $i \in I_P$ and $j \in I_{S_+}$ such that $n_i = n_j = 1$, and the other one is either $x_{ik}$, $x_{jk}$ or $x_{kj}$, for some $k \in I_{S_+} \cup I_S$ different from $j$. We now show that these commute too, reasoning as in the proof of Proposition~\ref{LCS_vB_1mu}. Let us denote by $\alpha$ (resp.~$\beta$) the only element of the $i$-th block (resp.~the $j$-th one), and by $\gamma, \gamma'$ two distinct elements of the $k$-th block (which must be of size at least $4$). 
\begin{itemize}[leftmargin=20pt]
\item $x_{ij}$ and $x_{kj}$ always commute, since $\chi_{\alpha \beta}$ commutes with $\chi_{\gamma \beta}$.
\item  $x_{ij}$ and $x_{ik}$ are the images of $\chi_{13}$ and $\chi_{12}$ by a morphism from $\langle \chi_{12}, \chi_{23}, \chi_{13} \rangle$ sending $\chi_{23}$ to $x_{kj}$. Since $x_{ik}$ and $x_{kj}$ are the classes of $\chi_{\alpha \gamma}$ and  $\chi_{\gamma' \beta}$, whose support is disjoint, they commute. Hence we can apply Lemma~\ref{wB111^ab} to conclude that $x_{ij}$ and $x_{ik}$ commute.
\item $x_{ij}$ and $x_{jk}$ are also in such a homomorphic image, and we know from the previous case that $x_{ij}$ and $x_{ik}$ commute, so we can apply Lemma~\ref{wB111^ab} again to deduce that  $x_{ij}$ and $x_{jk}$ commute too.
\end{itemize}
Finally, we have shown that all the generators of $\wB(\lambda_P, \lambda_{S_+}, \lambda_S)/\LCS_\infty$ commute pairwise. Thus this group is abelian, which means exactly that $\LCS_\infty = \LCS_2$ for $\wB(\lambda_P, \lambda_{S_+}, \lambda_S)$.
\end{proof}

\subsection{Small blocks of unoriented circles}\label{tripartite_small_S}

We already know that if $\lambda_{S}$ has at least one block of size $2$ or $3$, then $\wB(\lambda_P, \lambda_{S_+}, \lambda_S)$ surjects onto $\exwB_2$ or $\exwB_3$, whose LCS does not stop. The next proposition, which generalises Proposition~\ref{LCS_of_exwB_1mu}, deals with the case where $\lambda_{S}$ has at least one block of size $1$:
\begin{proposition}\label{tripartite_wB_isolated_S}
Suppose that $\lambda_{S}$ has at least one block of size $1$, and there is at least another block in $\lambda = \lambda_P \lambda_{S_+} \lambda_S$. Then the LCS of $\wB(\lambda_P, \lambda_{S_+}, \lambda_S)$ does not stop.
\end{proposition}

\begin{proof}
The group $\wB(\lambda_P, \lambda_{S_+}, \lambda_S)$ surjects either onto $\wB(m, 0, 1)$, onto $\wB(0, m, 1)$ or onto $\wB(0, 0, (m,1))$, for some $m \geq 1$. We need to show that the LCS of each of these three groups does not stop. The case of $\wB(0, 0, (m,1)) = \exwB_{1,m}$ has already been dealt with; see Proposition~\ref{LCS_of_exwB_1mu}. The method used there can be adapted to the other two cases. Indeed, both $\wB(m, 0, 1)$ and $\wB(0, m, 1)$ identify with a subgroup of $\Aut(\F_{m + 1})$ consisting of automorphisms preserving the normal closure $N$ of the $x_\alpha x_{\alpha+1}^{-1}$ for $1 \leq \alpha \leq m-1$, so they both project onto a subgroup $G$ of $\Aut(\F_{m + 1}/N) = \Aut(\F_2)$. In the first case, $G$ is generated by $\rho_2$ and $\chi_{12}$, and is isomorphic to $\Z \rtimes (\Z/2)$, whose LCS does not stop by Proposition~\ref{LCS_of_Z/2*Z/2}. In the second case, $G$ is generated by $\rho_2$, $\chi_{12}$ and $\chi_{21}$. Then $\langle \chi_{12}, \chi_{21} \rangle = \Inn(\F_2) \cong\F_2$ is a normal subgroup of $G$, and:
\[G = \langle \chi_{12}, \chi_{21} \rangle \rtimes \langle \rho_2 \rangle \cong\F_2 \rtimes (\Z/2),\]
where the action of $\rho_2$ fixes $\chi_{21}$ and sends $\chi_{12}$ to its inverse. Then the quotient of $G$ by $\chi_{21}$ is isomorphic to $\Z \rtimes (\Z/2)$, whose LCS does not stop by Proposition~\ref{LCS_of_Z/2*Z/2}.
\end{proof}

\subsection{Blocks of points of size 2 with oriented circles}

We now turn to the cases with blocks of points of size $2$ which have not already been covered. The following result takes care of a large part of them:

\begin{proposition}\label{tripartite_wB_P2_and_S+}
Suppose that $\lambda_P$ has at least one block of size $2$, and that $\lambda_{S_+} \neq \varnothing$. Then the LCS of $\wB(\lambda_P, \lambda_{S_+}, \lambda_S)$ does not stop.
\end{proposition}

\begin{proof}
Since $\lambda_{S_+}$ is non-empty, it contains some block of size $m \geq 1$. Then $\wB(\lambda_P, \lambda_{S_+}, \lambda_S)$ surjects onto $\wB(2, m, 0)$, so it it enough to show that the LCS of $\wB(2, m, 0)$ does not stop. The latter is the subgroup of $\Aut(\F_{2+m})$ generated by $\tau_\alpha$ and $\sigma_\alpha$ for $3 \leq \alpha \leq m+1$, $\chi_{1\alpha}$ and  $\chi_{2\alpha}$ for $3 \leq \alpha \leq m+2$, and $\tau_1$. All these automorphisms preserve the subgroup $N$ normally generated by the $x_\alpha x_{\alpha+1}^{-1}$ for $3 \leq \alpha \leq m+1$, as one can check by direct inspection: for instance, $\sigma_\alpha(x_{\alpha-1} x_\alpha^{-1}) = x_{\alpha-1} (x_\alpha x_{\alpha + 1} x_\alpha^{-1})^{-1}
= (x_{\alpha-1}x_\alpha^{-1}) \cdot x_{\alpha} (x_{\alpha} x_{\alpha + 1}^{-1})x_\alpha^{-1}$. As a consequence, they induce automorphisms of the quotient $\F_{2+m}/N \cong\F_3$. This defines a morphism from $\wB(2, m, 0)$ to $\Aut(\F_3)$ sending $\tau_\alpha$ and $\sigma_\alpha$ to the identity, all the $\chi_{1\alpha}$ to $\chi_{13}$, all the $\chi_{2\alpha}$ to $\chi_{23}$, and $\tau_1$ to $\tau_1$. The image of this projection is thus $\langle \tau_1, \chi_{13}, \chi_{23} \rangle = \wB(2, 1, 0)$. Moreover, since $\chi_{23} = \tau_2 \sigma_2$ and thus $\tau_1\chi_{23}\tau_1=\chi_{13}$ (see \Spar\ref{section_Aut(Fn)}), one easily sees that 
\[\wB(2, 1, 0) = \langle \chi_{13}, \chi_{23} \rangle \rtimes \langle \tau_1 \rangle \cong \Z^2 \rtimes (\Z/2),\]
where the action of $\Sym_2$, which is conjugation by $\tau_1$, exchanges the two generators $\chi_{13}$ and $\chi_{23}$ of $\Z^2$. Thus, this semi-direct product identifies with $\Z  \wr \Sym_2$, whose LCS does not stop (by Corollary~\ref{LCS_G_wr_S2}), so neither does the LCS of $\wB(2, m, 0)$, as claimed.
\end{proof}

\begin{remark}
One can give a more conceptual argument for the preservation of $N$ by $\wB(2, m, 0)$. Namely, these automorphisms act on $\F_{2+m} \cong\F_2 *\F_m$ in a block triangular fashion (they stabilise $\F_m$), and the automorphisms they induce on $\F_m$ preserve the characteristic subgroup $\LCS_2(\F_m)$. Then their action on the quotient $\F_m^{\ab} \cong \Z^m$ factors through the canonical action of $\Sym_m$, so the projection onto $\Z$ identifying all the generators of $\F_m$ is equivariant. Whence a well-defined induced action of $\wB(2, m, 0)$ on $\F_2 * \Z \cong\F_3$.
\end{remark}

\begin{remark}
Concretely, the non-stopping of the LCS in this case is witnessed by iterated commutators of elements of Type 4 and Type 1 from Figure~\ref{fig:tripartite-braids}. Here, we are viewing the pictures in Figure~\ref{fig:tripartite-braids} as elements of the group $G = \wB(\lambda_P,\lambda_{S_+},\lambda_S)$ itself, rather than as elements of its abelianisation $G^{\ab} = \Lie_1$.
\end{remark}

\subsection{Blocks of points of size 2 without oriented circles}\label{subsubsec:tripartite_block_two_points_LCS}

We now consider the LCS of $\wB(\lambda_P, \varnothing, \lambda_S)$. If there are no blocks in $\lambda_S$, then we are considering $\wB(\lambda_P, \varnothing, \varnothing) \cong \Sym_{\lambda_P}$, whose LCS stops at $\LCS_2$ (except if $\lambda_P = (1, \ldots , 1)$, for which $\wB(\lambda_P, \varnothing, \varnothing) = \{1\}$). On the contrary, if there are no blocks in $\lambda_P$, then we are considering $\wB(\varnothing, \varnothing, \lambda_S) \cong \exwB_{\lambda_P}$, whose LCS has been studied above (in Proposition~\ref{LCS_of_partitioned_exwBn}). 

Suppose now that there is at least one block in both $\lambda_P$ and $\lambda_S$. If $\lambda_S$ has a block of size $1$, $2$ or $3$, we already know that the LCS of $\wB(\lambda_P, \varnothing, \lambda_S)$ does not stop (see~\Spar\ref{tripartite_small_S}); we can thus assume that the blocks of $\lambda_S$ have size at least $4$. Under this hypothesis, Proposition~\ref{tripartite_wB_isolated_P} takes care of the case where $\lambda_P$ has no block of size $2$ (then the LCS stops). The remaining case is covered by the following.

\begin{proposition}\label{prop:welded_points_block_size_two}
Let $\lambda_P$ be any partition having at least one block of size $2$ and let $\lambda_S$ be a non-empty partition whose blocks have size at least $4$. Then the LCS of $\wB(\lambda_P, \varnothing, \lambda_S)$ stops at $\LCS_3$. 

Moreover, $\Lie_2(\wB(\lambda_P, \varnothing, \lambda_S))$ has a $(\Z/2)$-basis consisting of the $[t_i, x_{ij}]$, for $i \in I_P$ with $n_i = 2$, and $j \in I_S$, and all the other brackets of two of the generators of $\wB(\lambda_P, \varnothing, \lambda_S)^{\ab}$ described in Proposition~\ref{tripartite_wB^ab} are trivial in the Lie ring.
\end{proposition}

\begin{proof}
Generators of $\Lie_1 = \wB(\lambda_P,\varnothing, \lambda_S)^{\ab}$ are described in Proposition~\ref{tripartite_wB^ab} and Figure~\ref{fig:tripartite-braids}. Most of the pairs of such generators can be lifted to elements of $\wB(\lambda_P, \varnothing, \lambda_S)$ with disjoint support. The only ones for which this is not possible are:
\begin{itemize}
\item  $t_i$ and $x_{ij}$ for $j \in I_S$ and $i \in I_P$ such that $n_i = 2$. 
\item $x_{ij}$ and $x_{ik}$ for $j,k \in I_S$ and $i \in I_P$ such that $n_i = 1$. 
\end{itemize} 

We first use Proposition~\ref{tripartite_wB_isolated_P} to show that $[x_{ij}, x_{ik}] = 0$ in the Lie ring. In order to do this, let us denote by $\mu_P$ the partition obtained from $\lambda_P$ by removing the blocks of size $2$. Proposition~\ref{tripartite_wB_isolated_P} implies that $\Lie(\wB(\mu_P, \varnothing, \lambda_S))$ is abelian. Since $x_{ij}$ and $x_{ik}$ are in the image of this Lie ring inside $\Lie(\wB(\lambda_P, \varnothing, \lambda_S))$, their bracket must be trivial.

We now know that $\Lie_2$ is generated by the commutator $z_{ij} := [t_i, x_{ij}]$, for $j \in I_S$ and $i \in I_P$ such that $n_i = 2$ (note that this also works when $\lambda_P$ does not have blocks of size $2$; then there are no $z_{ij}$, hence $\Lie_2 = 0$, and we recover the case $I_{S_+} = \varnothing$ of Proposition~\ref{tripartite_wB_isolated_P}). Before going further in our analysis of the Lie ring, let us examine the element $z_{ij}$ (with $i,j$ as above). If the $i$-th block is $\{\alpha, \alpha + 1\}$ and $\beta$ is in the $j$-th block, $z_{ij}$ is the class of the commutator $\zeta_{\alpha \beta} := [\tau_\alpha, \chi_{\alpha \beta}]$, which is the automorphism:
\[\zeta_{\alpha \beta} \colon
\left\{
\setlength\arraycolsep{3pt}
\begin{array}{clclclclc}
x_\alpha 
&\longmapsto &x_\beta^{-1}x_\alpha x_\beta
&\longmapsto &x_\beta^{-1}x_{\alpha+1}x_\beta
&\longmapsto &x_\beta^{-1}x_{\alpha+1}x_\beta
&\longmapsto &x_\beta^{-1}x_\alpha x_\beta,  \\
x_{\alpha+1}
&\longmapsto &x_{\alpha+1}
&\longmapsto &x_\alpha
&\longmapsto &x_\beta x_\alpha x_\beta^{-1}
&\longmapsto &x_\beta x_{\alpha+1} x_\beta^{-1}
\end{array}  \right.\]
where all the $x_\gamma$ with $\gamma \notin \{\alpha, \alpha +1\}$ are fixed. The corresponding motion is drawn in Figure~\ref{fig:tripartite-braids-L2}. We note that $\zeta_{\alpha \beta} = \chi_{\alpha+1, \beta} \chi_{\alpha \beta}^{-1}$ (which can be seen directly, or deduced from $\tau_\alpha \chi_{\alpha \beta} \tau_\alpha^{-1} = \chi_{\alpha+1, \beta}$). 

Let us turn our attention to $\Lie_3$. It is generated by the commutators of the generators $z_{ij} = \overline {\zeta_{\alpha \beta}}$ of $\Lie_2$ with the generators of $\Lie_1$. Since $\Supp(\zeta_{\alpha \beta}) = \{\alpha, \alpha + 1, \beta\}$ (where $\beta$ can be chosen anywhere in the $j$-th block for $\zeta_{\alpha \beta}$ to represent $z_{ij}$), one can choose a lift having disjoint support with $\zeta_{\alpha \beta}$ for each generator of $\Lie_1$, save $x_{ik}$ (for $k \in I_S$) and $t_i$. If $k = j$, we remark that $x_{ij}$ does have a lift commuting with $\zeta_{\alpha \beta}$, albeit not for reasons of support. Indeed, it is the class of $\chi_{\alpha \beta}$, which clearly commutes with $\zeta_{\alpha \beta}$, since both are automorphisms conjugating $x_\alpha$ and $x_{\alpha + 1}$ by powers of the same element $x_\beta$. Thus the only generators of $\Lie_3$ which are possibly non-trivial are the $[t_i, z_{ij}]$ and the $[x_{ik}, z_{ij}]$ for $k \neq j$.

We first show that $[x_{ik}, z_{ij}]$ vanishes, when $k \neq j$. Let $\gamma$ be in the $k$-th block of $\lambda$. Then $[x_{ik}, z_{ij}]$ is the class of the commutator $[\chi_{\alpha \gamma}, \zeta_{\alpha \beta}]$. Since $\zeta_{\alpha \beta} = \chi_{\alpha+1, \beta} \chi_{\alpha \beta}^{-1}$ and $\chi_{\alpha \gamma}$ commutes with $\chi_{\alpha+1, \beta}$, this commutator is equal to $[\chi_{\alpha \gamma}, \chi_{\alpha \beta}^{-1}]$. We have seen above that $\chi_{\alpha \gamma}$ commutes with $\chi_{\alpha \beta}$ modulo $\LCS_\infty$. Thus, the class of $[\chi_{\alpha \gamma}, \chi_{\alpha \beta}^{-1}]$ is trivial in the Lie ring, which means that $[x_{ik}, z_{ij}] = 0$.

We now show that $[t_i, z_{ij}] = 0$, using explicit computations. We have:
\[\tau_\alpha \zeta_{\alpha \beta} \tau_\alpha^{-1} \colon
\left\{\begin{array}{clclclc}
x_\alpha 
&\longmapsto &x_{\alpha+1}
&\longmapsto &x_\beta x_{\alpha+1} x_\beta^{-1}
&\longmapsto &x_\beta x_\alpha x_\beta^{-1},\\
x_{\alpha+1}
&\longmapsto &x_\alpha
&\longmapsto &x_\beta^{-1}x_\alpha x_\beta
&\longmapsto &x_\beta^{-1}x_{\alpha+1} x_\beta
\end{array}  \right.\]
which means that $\tau_\alpha \zeta_{\alpha \beta} \tau_\alpha^{-1} = \zeta_{\alpha \beta}^{-1}$. We also have:
\[\rho_\beta \zeta_{\alpha \beta} \rho_\beta^{-1} \colon
\left\{\begin{array}{clclclc}
x_\alpha 
&\longmapsto &x_\alpha 
&\longmapsto &x_\beta^{-1} x_\alpha x_\beta
&\longmapsto &x_\beta x_\alpha x_\beta^{-1},\\
x_{\alpha+1}
&\longmapsto &x_{\alpha+1}
&\longmapsto &x_\beta x_{\alpha+1} x_\beta^{-1}
&\longmapsto &x_\beta^{-1} x_{\alpha+1} x_\beta \\
x_\beta
&\longmapsto &x_\beta^{-1}
&\longmapsto &x_\beta^{-1}
&\longmapsto &x_\beta
\end{array}  \right.\]
whence $\rho_\beta \zeta_{\alpha \beta} \rho_\beta^{-1} = \zeta_{\alpha \beta}^{-1} = \tau_\alpha \zeta_{\alpha \beta} \tau_\alpha^{-1}$. Finally:
\[[t_i, z_{ij}] = \overline{[\tau_\alpha, \zeta_{\alpha \beta}]} = \overline{[\rho_\beta, \zeta_{\alpha \beta}]} = [r_j, z_{ij}] = 0,\]
where the last equality comes from the above disjoint support argument. This finishes the proof that $\Lie_3 = 0$.
Notice that the above calculation also implies $\LCS_3 \ni [\tau_\alpha, \zeta_{\alpha \beta}] = \zeta_{\alpha \beta}^{-2}$, so that $z_{ij} = \overline{\zeta_{\alpha \beta}}$ is of order $2$ in $\Lie_2$. Thus, in order to prove our last claim, we need to show that the $z_{ij}$ are linearly independant over $\Z/2$. For each fixed choice of $j \in I_S$ and $i \in I_P$ such that $n_i = 2$, we have a projection from $\wB(\lambda_P, \varnothing, \lambda_S)$ onto $\wB(2, \varnothing, n_j)$ which induces a morphism between the $\Lie_2$ killing all the $z_{kl}$ save $z_{ij}$, which is sent to $z_{12}$. Thus, we only need to show that for $n \geq 4$, the element $z_{12}$ is non-trivial in $\Lie_2(\wB(2, \varnothing, n))$, which is then isomorphic to $\Z/2$. We do this by constructing one further projection. Elements of $\wB(2, \varnothing, n)$ are automorphisms of $\F_{2 + n}$ preserving the normal subgroup generated by the $x_\alpha x_{\alpha + 1}$ together with the $x_\alpha^2$ ($\alpha \geq 3$). Whence a well-defined morphism $\wB(2, \varnothing, n) \rightarrow \Aut(\F_2 * (\Z/2))$. The image of this morphism is easily seen to be isomorphic to $W_2 = (\Z/2) \wr \Sym_2$ (where the images of $\chi_{13}$ and $\chi_{23}$ generate $(\Z/2)^2$ and the image of $\tau_1$ generates $\Sym_2$), and the element $\zeta_{13} = \chi_{23}\chi_{13}^{-1}$ is sent to the generator $(1,1,\mathrm{id})$ of $\LCS_2(W_2)$ (see Proposition~\ref{LCS_Klein_tau} for the calculation of $\LCS_*(W_2)$). Finally, the morphism induced between the Lie rings sends $z_{12}$ to the generator of $\Lie_2(W_2) \cong \Z/2$, whence the result.
\end{proof}

\begin{figure}[ht]
\centering
\includegraphics[scale=0.7]{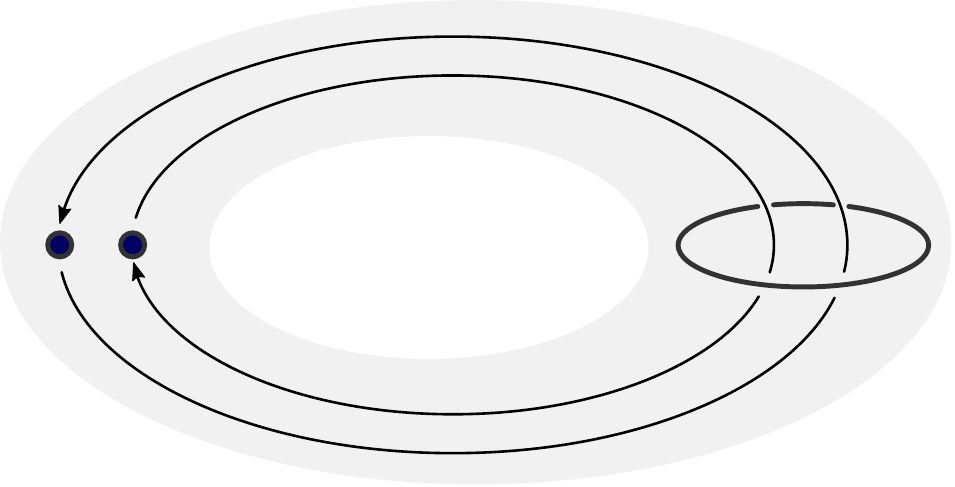}
\caption{Generators of $\Lie_2 = \LCS_2 / \LCS_3$ for $G = \wB(\lambda_P,\varnothing,m)$, where each block of $\lambda_P$ has size $\geq 2$ and $m\geq 4$. There is one generator corresponding to the commutator $\zeta_{\alpha,\beta}$, represented by a loop of the form pictured, for each block of $\lambda_P$ of size exactly $2$. This generalises to $G = \wB(\lambda_P,\varnothing,\lambda_S)$ where each block of $\lambda_S$ has size $\geq 4$, and the generators are indexed by $(b,x)$ where $b$ is a block of points of size exactly $2$ and $x$ is a block of circles.}
\label{fig:tripartite-braids-L2}
\end{figure}

\begin{proof}[Proof of Theorem \ref{LCS_of_partitioned_exwBn_2}]
The first statement is Proposition~\ref{tripartite_wB_isolated_P_and_S+}, which is a generalisation of Propositions~\ref{tripartite_wB_isolated_S+} and \ref{tripartite_wB_isolated_P}, which each in turn generalise Proposition~\ref{stable_LCS_of_partitioned_exwBn_2}.

Among the remaining cases, here are the ones for which we can deduce our conclusion directly from the calculations above (Propositions~\ref{LCS_of_wBn} and \ref{LCS_of_exwBn}), by exhibiting a quotient of $\wB(\lambda_P, \lambda_{S_+}, \lambda_S)$ whose LCS does not stop:
\begin{itemize}
\item If $\lambda_{S}$ has one block of size $2$ or $3$, then it surjects onto $\exwB_2$ or $\exwB_3$.
\item If $\lambda_{S}$ has at least two blocks of size $1$, then it surjects onto $\exwB_{1,1} = \exwP_2 \cong\F_2 \rtimes (\Z/2)^2$.
\item If $\lambda_{S_+}$ has one block of size $2$ or $3$, then it surjects onto $\wB_2$ or $\wB_3$.
\item If $\lambda_{S_+}$ has at least two blocks of size $1$, then it surjects onto $\wB_{1,1} = \wP_2$.
\end{itemize}
In all of these cases, Lemma~\ref{lem:stationary_quotient} allows us to conclude that the LCS of the partitioned tripartite welded braid group $\wB(\lambda_P, \lambda_{S_+}, \lambda_S)$ does not stop.

Moreover, if $\lambda_S$ has one block of size $1$ which is not the only block of $\lambda$, then Proposition~\ref{tripartite_wB_isolated_S} ensures that the LCS of $\wB(\lambda_P, \lambda_{S_+}, \lambda_S)$ does not stop. Then the only remaining case is the one where $\lambda_P$ has at least one block of size $2$ and the blocks of $\lambda_S$ are of size at least $4$. This last case is covered by Proposition~\ref{tripartite_wB_P2_and_S+} when $I_{S_+}$ is not empty and by Proposition~\ref{prop:welded_points_block_size_two} when $I_{S_+}$ is empty.
\end{proof}

\chapter{Braids on surfaces}
\label{sec:braids_on_surfaces}

In this chapter, we study the LCS of surface braid groups and their partitioned versions. This may be seen as a generalisation of results from \Spar\ref{sec:partitioned_braids}, where we studied classical Artin braids, that is, braids on the disc. Contrary to what is usually done in the literature, we choose to consider braids on any surface. In particular, our surfaces may be non-compact, and they may have (countably) infinite genus or boundary components. For braids, which are always compactly supported and can be pushed away from the boundary, this level of generality does not really complicate things. We also do not assume that our surfaces are orientable, because the techniques that we use work very similarly for orientable surfaces and for non-orientable ones. 

We first recall what we need from Richards' classification of surfaces (\Spar\ref{sec:surfaces}), then we introduce the tools that we need from the general theory of braids on surfaces (\Spar\ref{sec_braids_general}) and we review presentations of braid groups on compact surfaces (\Spar\ref{sec_pstations_B(S)}). We then find ourselves ready to tackle the study of LCS. We do this first for the whole braid group $\B_n(S)$: we show that the LCS stops if $n \geq 3$, and we completely compute the Lie ring in this case (\Spar\ref{sec:LCS_undiscrete_partition}). We then generalise these results to partitioned surface braid groups $\B_\lambda(S)$ in \Spar\ref{sec:partitioned_braids_surfaces_stable}, whose LCS stops if the blocks of the partition have size at least $3$, a stability hypothesis under which we can compute the associated Lie ring. Finally, we study the unstable cases in \Spar\ref{sec:surface_braid_unstable}. There, four cases stand out from the crowd: the braid groups on the sphere $\S^2$, on the torus $\Tor$, on the  Möbius strip $\Moeb$ and on the projective plane $ \Proj$.

\section{Surfaces}\label{sec:surfaces}

Recall that a \emph{surface} is a (second countable) $2$-manifold, which we will \emph{not} suppose compact or orientable in general. Such manifolds are well-understood: the classification of surfaces without boundary has been achieved by Richards \cite{Richards}, and the classification of surfaces with boundary, which is more complicated, was completed more than fifteen years later by Brown and Messer \cite{Brown-Messer}. 

\begin{remark}
By a \emph{manifold}, we mean a \emph{locally Euclidean, Hausdorff} space. Moreover, all our manifolds are assumed \emph{second countable}. For connected manifolds, this is equivalent to either metrisability or paracompactness, and it implies triangulability (see \cite{Rado1925} for the latter implication). 
\end{remark}

For studying braids, we will in fact only need to consider Richards' classification. Indeed, let $S$ be any connected surface (possibly with boundary). If we glue a copy of $\partial S \times [0,\infty)$ to $S$ by identifying $\partial S \times 0$ with $\partial S$, we obtain a surface $S'$ without boundary, and it is easy to show that the canonical injection $S \hookrightarrow S'$ is an isotopy equivalence. Moreover, one can show that an isotopy equivalence induces homotopy equivalences between configuration spaces and, in turn, isomorphisms between braid groups. In the sequel, we thus identify braids on $S$ with braids on $S'$ (or on $S - \partial S$). For instance, braids on the closed disc are identified with braids on the plane. This holds in particular for braids on one strand, that is, for fundamental groups, whose computation is recalled below in Proposition~\ref{pi_1(S)}.

Let us briefly recall Richards' construction of all surfaces, up to homeomorphism; for a detailed account, the reader is referred to \cite{Richards}, in particular to \Spar 5 and \Spar 6, especially Theorem 3, therein. 

\begin{proposition}\label{Richards_classification}
Let $S$ be a connected surface. Then $S$ is homeomorphic to a surface constructed in the following way:
\begin{itemize}
\item Consider the Cantor set $K$ embedded in the sphere $\S^2$ in the usual manner. Choose some closed subset $X$ of $K$, and remove it from $\S^2$.
\item Choose a finite or countably infinite sequence of pairwise disjoint closed $2$-discs in $\S^2 - X$, which has no accumulation point outside of $X$.
\item Along each of these discs, perform a connected sum operation with either $\Tor$ or $\Proj$.
\end{itemize}
\end{proposition}

\begin{remark}
At the second step, one can in fact choose an explicit sequence of discs depending only on $X$ together with the subset $X_{np} \subseteq X$ of accumulation points of the sequence of discs; see \cite{Richards}.
\end{remark}

As a direct corollary of Richards' classification, one can compute fundamental groups of surfaces:
\begin{proposition}\label{pi_1(S)}
Let $S$ be a connected surface without boundary. Then $\pi_1(S)$ is a free group, except when $S$ is closed. Moreover, it is of finite type if and only if $S$ is obtained (up to homeomorphism) from a closed surface by removing a finite number of points.
\end{proposition}

\begin{remark}
We see that closed surfaces (that is, compact surfaces without boundary) are singled out here, as will also be the case later in our study of braid groups; see for instance Proposition~\ref{Fadell-Neuwirth} and Theorem~\ref{Lie_ring_partitioned_B(S)}.
\end{remark}

Fundamental groups are in fact almost all homotopy groups of surfaces:
\begin{corollary}\label{pi_2(S)}
Let $S$ be a connected surface without boundary. Then $\pi_{>2}(S)$ is trivial, except when $S$ is the sphere or the projective plane.
\end{corollary}

\begin{proof}
The universal covering $\widetilde S$ of $S$ is a simply connected surface without boundary. Proposition~\ref{pi_1(S)} implies that such a surface must be of finite type, and the explicit computation of fundamental groups of surfaces of finite type shows that $\widetilde S$ must be homeomorphic either to the sphere or to the plane. If $\widetilde S \cong \S^2$, then $S$ is compact, and the fibers in the covering must be finite, so $\pi_1(S)$ is finite, and $S$ is either the sphere itself, or the projective plane. In all the other cases, $\pi_{>2}(S) \cong \pi_{>2}(\widetilde S) \cong \pi_{>2}(\R^2) = 0$.
\end{proof}

Another immediate corollary of Richards' classification is the following dichotomy:

\begin{proposition}\label{dichotomy}
Let $S$ be a connected surface without boundary. Then either $S$ can be embedded into the sphere $\S^2$, or it contains the $1$-punctured torus or the Möbius strip as an embedded subsurface.
\end{proposition}
In our study of braids, this appears as a trichotomy, between the following cases:
\begin{itemize}
    \item $S$ is \emph{planar}, i.e.~it embeds into the plane;
    \item $S$ is the sphere $\S^{2}$;
    \item $S$ contains an embedded $1$-punctured torus (a \emph{handle}) or an embedded Möbius strip (a \emph{crosscap}).
\end{itemize}
In this regard, see in particular Proposition~\ref{trichotomy}.

\begin{convention}\label{def_crosscaps}
    In the usual way, an embedded Möbius strip on a surface will be indicated in our figures by a \emph{crosscap}, that is, a circle drawn on the surface, bounding a disc on which is drawn a cross (see for instance Figure~\ref{Aij_as_bracket_crosscap}). In order to obtain the surface that is meant from the surface on which the crosscaps are drawn, one must, for each crosscap, remove the interior of the corresponding disc, and then glue together opposite points of the remaining circle. 
\end{convention}
The reader not familiar with this classical representation should check, as a good exercise, that a disc with one crosscap is a Möbius strip, that a sphere with one crosscap is a projective plane (so that adding a crosscap on a surface is the same as taking the connected sum with a projective plane), and that a sphere with two crosscaps is a Klein bottle. They should also keep in mind Dyck's theorem: three crosscaps on a surface is the same as a handle and a crosscap (see for instance \cite{Francis-Weeks}).

\section{Braids on surfaces: general theory}\label{sec_braids_general}

We gather here some fundamental results in the theory of braids on surfaces. The main tools that we need are Goldberg's theorem (\Spar\ref{subsec_gen_B(S)}), the Fadell-Neuwirth exact sequences (\Spar\ref{subsec_Fadell-Neuwirth}) and a little calculation showing that the usual pure braid generators become commutators on non-planar surfaces (\Spar\ref{subsec_A_as_com}). Goldberg's theorem \cite[Th.~1]{Goldberg} says that surface braid groups are generated by braids on the disc, together with braids obtained from loops on the surface; we give a new simple proof of this result, incidentally extending it to possibly non-compact surfaces. As for the Fadell-Neuwirth exact sequences, which are traditionally stated for pure braid groups, since they involve forgetting strands, we generalise them easily to partitioned braid groups on possibly non-compact surfaces, where the projections forget blocks of strands.

\subsection{Definitions, notations and conventions}

Let $S$ be a connected surface. Let us consider the configuration space \[F_n(S) = \{ (x_1, \ldots , x_n) \in S^n \ |\ \forall i \neq j,\ x_i \neq x_j  \} \subset S^n.\] 
The \emph{braid group on the surface $S$} on $n$ strings is the fundamental group $\B_n(S)$ of the unordered configuration space $C_n(S) = F_n(S)/\Sym_n$. When $S$ is the $2$-disc $\D$, this group is exactly Artin's braid group, that is $\B_n = \B_n(\D)$.

\medskip

Let us fix an embedded closed disc $\D \subset S$, together with a base configuration $c = (c_1, \ldots , c_n) \in F_n(S)$ of points $c_i \in \D$. Since the assignment $S \mapsto \B_n(S)$ is functorial with respect to embeddings of surfaces, we have a (not necessarily injective) group morphism:
\[\varphi \colon \B_n = \B_n(\D) \rightarrow \B_n(S).\]
In the sequel, we omit most mentions of $\varphi$: if $\beta \in \B_n$, we still denote by $\beta$ its image in $\B_n(S)$ (although we should denote it by $\varphi(\beta)$).

\medskip

We can also construct surface braids from curves on the surface. Precisely, for any $i \leq n$, let us define $\iota_i \colon S - \D \hookrightarrow F_n(S)$ by sending $x$ to the configuration $(c_1, \ldots , c_{i-1}, x, c_{i+1}, \ldots , c_n)$. This induces a morphism between fundamental groups:
\[\psi_i \colon \pi_1(S - \D) \rightarrow \PB_n(S).\]

\begin{remark}
The map $\iota_i$ cannot preserve basepoints, since the basepoint of $F_n(S)$ is not in its image. However, the induced map between fundamental groups can easily be defined using a chosen fixed path between $c_i$ and a point not in $\D$. The choice of such a path (for each $i$) is implicit in the sequel, and should be made as simply as possible. For instance, one can fix a segment from $c_i$ to a point on $\partial \D$, and extend it slightly outside $\D$, such that these paths are disjoint for different values of $i$.
\end{remark}

The canonical projection $\pi \colon \B_n(S) \twoheadrightarrow \Sym_n$, corresponding to the covering of $F_n(S)/\Sym_n$ by $F_n(S)$, can be enhanced to a projection:
\[\pi_S \colon \B_n(S) \twoheadrightarrow \pi_1(S) \wr \Sym_n\]
as follows. Given a braid $\beta \in \B_n(S)$, let us lift it to a path $ \gamma = (\gamma_1, \ldots , \gamma_n)$ in $F_n(S)$ from $(c_i)_i$ to $(c_{\sigma^{-1}(i)})_i$ where $\sigma = \pi(\beta)$. Then send $\beta$ to $((\overline \gamma_1, \ldots , \overline \gamma_n), \sigma)$, where $\overline \gamma_i$ is the image of $\gamma_i$ in $\pi_1(S/\D) \cong \pi_1(S)$. We note that $\pi_S$ is clearly surjective, since its image contains $\Sym_n$ (which is the image of $\varphi(\B_n)$ by $\pi_S$), and all the factors $\pi_1(S)$ (which are the images of the $\psi_i(\pi_1(S))$).

The kernel of $\pi_S$, which is contained in $\PB_n(S)$, obviously contains the group $\PB_n$. We denote it by $\PB_n^\circ(S)$ and we call its elements \emph{geometrically pure braids}.

\subsection{Generators of surface braid groups}\label{subsec_gen_B(S)}

The following result generalises one of Goldberg \cite[Th.~1]{Goldberg} to any connected surface. 
\begin{proposition}\label{generators_and_N(Pn)} The following statements hold for any connected surface $S$ and any integer $n \geq 1$:
\begin{itemize}
\item For any $i \leq n$, the group $\B_n(S)$ is generated by the images of $\varphi$ and $\psi_i$. 
\item Its subgroup $\PB_n(S)$ is generated by (the image of) $\PB_n$ and the images of $\psi_1, \ldots , \psi_n$.
\item The subgroup $\PB_n^\circ(S)$ of $\PB_n(S)$ is the normal closure of $\PB_n$. Since $\PB_n^\circ(S)$ is normal in $\B_n(S)$, it is also the normal closure of $\PB_n$ in $\B_n(S)$.
\end{itemize}
\end{proposition}

\begin{proof}
Let us first remark that the $\psi_i$ are conjugate to each other by elements of the image of $\varphi$. Hence, \ul{the second statement implies the first one}.

We prove both the second and the third statement by induction on $n$. Both proofs use the tools that we introduce now. Consider the following commutative diagram of spaces:
\[\begin{tikzcd}
S-Q_n \ar[r, hook, "\iota"] \ar[d, hook] 
&F_{n+1}(S) \ar[r, two heads, "p"] \ar[d, hook] 
&F_n(S) \ar[d, hook]  \\
S \ar[r, hook, "\iota"]
&S^{n+1} \ar[r, two heads, "p"]
&S^n,
\end{tikzcd}\]
where $Q_n = \{c_1, \ldots , c_n\}$, $\iota$ sends $x$ to $(c_1, \ldots , c_n, x)$ and $p$ send $(x_1, \ldots , x_{n+1})$ to $(x_1, \ldots , x_n)$. It induces a commutative diagram of morphisms between fundamental groups:
\begin{equation}\label{diag_of_pi_1}
\begin{tikzcd}
\pi_1(S-Q_n) \ar[r, "\iota_*"] \ar[d, two heads, "u"] 
&\PB_{n+1}(S) \ar[r, two heads, "p_*"] \ar[d, two heads, "v"] 
&\PB_n(S) \ar[d, two heads, "w"]  \\
\pi_1(S) \ar[r, hook]
&\pi_1(S)^{n+1} \ar[r, two heads]
&\pi_1(S)^n,
\end{tikzcd}
\end{equation}
whose bottom line is obviously exact. The map $p$ is a (locally trivial) fibration by \cite[Th.~3]{Fadell-Neuwirth} (see also~\cite[Th.~1.2]{Birman}), and the first line is part of its exact sequence in homotopy. Thus it is exact (but $\iota_*$ need not be injective in general). 

\ul{Let us prove our second statement.} For $n=1$, $\psi_1$ is the canonical isomorphism $\pi_1(S) \cong \PB_1(S) = \B_1(S)$, and there is nothing to prove. 

Let us now suppose that the conclusion holds for some $n \geq 1$. By applying the Seifert-van Kampen theorem, we see that $\pi_1(S-Q_n) \cong \F_n*_{\Z} \pi_1(S - \D)$, where $\F_n= \pi_1(\D - Q_n)$ is free on $n$ generators. Then $\iota_* \colon \F_n*_{\Z} \pi_1(S - \D) \rightarrow \PB_n(S)$ identifies with the map induced by $\F_n\hookrightarrow \PB_n$ (the kernel of $\PB_n \twoheadrightarrow \PB_{n-1}$) and $\psi_{n+1} \colon \pi_1(S - \D) \rightarrow \PB_n(S)$. 

Now let $G$ be the subgroup of $\PB_{n+1}(S)$ generated by $\PB_{n+1}$ and the images of the $\psi_i$ for $i \leq n+1$. It contains the image of $\iota_*$, which is the kernel of $p_*$. Moreover, its image by $p_*$ contains the images of $\psi_1, \ldots , \psi_n$, and $\PB_n$, hence all of $\PB_n(S)$, by the induction hypothesis. As a consequence, $G = \PB_{n+1}(S)$, which was the desired conclusion.

\ul{Let us prove our third statement.} For $n=1$, $\pi_S$ is the canonical isomorphism $\B_1(S) = \PB_1(S) \cong \pi_1(S)$ (inverse to $\psi_1$). Then $\PB_n^\circ(S)$ and $\PB_n$ are trivial, and there is nothing to prove.

Let us now suppose that the conclusion holds for some $n \geq 1$. Consider the induced maps between kernels of the vertical morphisms in \eqref{diag_of_pi_1}. By definition, the kernels of $v$ and $w$ are respectively $\PB_{n+1}^\circ(S)$ and $\PB_n^\circ(S)$. Let us denote by $K$ the kernel of $u$. We get induced maps:
\[\begin{tikzcd}
K \ar[r, "\iota_\#"] 
&\PB_{n+1}^\circ(S) \ar[r, two heads, "p_\#"] 
&\PB_n^\circ(S)
\end{tikzcd}\] 
such that $p_\# \circ \iota_\# = 1$. An easy chase in the diagram (or an application of the Snake Lemma) shows that we can lift any element in the kernel of $p_\#$ to an element of $K$: the above sequence is exact. 

The morphism $u$ identifies with the projection $\F_n*_{\Z} \pi_1(S - \D) \twoheadrightarrow \{1\} *_{\Z} \pi_1(S - \D) \cong \pi_1(S)$ killing the first factor, hence $K$ identifies with the normal closure of $\F_n$ in $\F_n *_{\Z} \pi_1(S)$. Moreover, $\iota_*$ sends $\F_n = \pi_1(\D - Q_n)$ to a subgroup of $\PB_{n+1}$, so the image of $K$ in $\PB_{n+1}(S)$ is contained in the normal closure of $\PB_{n+1}$. 

Now, let $N$ be the normal closure of $\PB_{n+1}$ in $\PB_{n+1}(S)$. Since $v(\PB_{n+1}) = \{1\}$, we have $v(N) = \{1\}$, which means that $N \subseteq \PB_{n+1}^\circ(S)$. By the induction hypothesis, $p_\#(N) = \PB_n^\circ(S)$. Moreover, $N$ contains the image of $i_\#$, which is the kernel of $p_\#$. Thus $N = \PB_{n+1}^\circ(S)$, which was the desired conclusion.
\end{proof}

\begin{remark}
For a closed, oriented surface, the group $K$ appearing in the proof is exactly the group $F_{n+1}$ from~\cite[page~227]{Gonzalez-Meneses-Paris2004}, where they give a precise description of it in this particular case. However, in their paper, they were using the very result that we are recovering and generalising here, quoting \cite{Birman} for it~\cite[page~225]{Gonzalez-Meneses-Paris2004}. 
\end{remark}

The proof of Proposition~\ref{generators_and_N(Pn)} also works for manifolds in higher dimension, allowing us to recover the classical \cite[Th.~1]{Birman1969}:
\begin{proposition}\label{Bn(M)}
For any manifold $M$ of dimension at least $3$, the morphism $\pi_M \colon \B_n(M) \twoheadrightarrow \pi_1(M) \wr \Sym_n$ is an isomorphism.
\end{proposition}

\begin{proof}
One can directly check that the proof of Proposition~\ref{generators_and_N(Pn)} works if we replace the surface $S$ with a connected manifold $M$ of any dimension $d \geq 2$ and the disc $\D$ with a $d$-disc $\D^d$. Then $\PB_n$ gets replaced with $\PB_n(\D^d)$, which is trivial whenever $d \geq 3$ (the configuration space $F_n(\D^d) \cong F_n(\R^d)$ is obtained from $\R^{nd}$ by removing subspaces of codimension $d \geq 3$, so it is simply connected). Thus, the normal closure $\PB_n^\circ(M)$ of $\PB_n(\D^d)$ is trivial too, and the latter is exactly the kernel of $\pi_M$. 
\end{proof}

\begin{remark}
This means that \emph{braid groups on manifolds of dimension at least $3$ are exactly wreath products}, whose LCS is studied in Appendix~\ref{appendix:wreath-products}; see in particular Corollaries~\ref{LCS_wr_sym_stable_partitioned} and~\ref{LCS_G_wr_S2}.
\end{remark}

\subsection{The Fadell-Neuwirth exact sequences}\label{subsec_Fadell-Neuwirth}

The locally trivial fibrations used in the proof of Proposition~\ref{generators_and_N(Pn)} are particular instances of the Fadell-Neuwirth fibrations. These induce exact sequences between pure braid groups, that are in fact exact sequences between partitioned braid groups. We now recall how these work, and when these exact sequences split.

\begin{definition}\label{def_partitioned_surface_braids}
Let $S$ be a surface, let $n \geq 1$ be an integer, and let $\lambda = (n_1, \ldots , n_l)$ be a partition of $n$. The corresponding \emph{partitioned surface braid group} is:
\[\B_\lambda(S) := \pi^{-1}(\Sym_\lambda) = \pi^{-1}\left(\Sym_{n_1} \times \cdots \times \Sym_{n_l}\right)\ \subseteq \B_n(S).\]
\end{definition}

There are canonical surjections between partitioned braid groups, obtained by forgetting blocks. For most surfaces, these projections behave exactly as they do for the disc. However, their behaviour for closed surfaces is somewhat trickier, especially when it comes to the sphere and the projective plane. The latter are in fact the only ones whose braid groups admit non-trivial torsion elements, a fact that can be seen as a consequence of their second homotopy group being non-trivial.
\begin{proposition}[Fadell-Neuwirth exact sequences]\label{Fadell-Neuwirth}
Let $S$ be a connected surface, let $\mu$ be a partition of an integer $m \geq 1$, $\nu$ be a partition of an integer $n \geq 1$, and let us denote by $\mu \nu$ their concatenation, which is a partition of $m+n$. The following sequence of canonical maps is exact:
\[
\B_\mu(S - \{n\ \text{pts}\}) \longrightarrow \B_{\mu \nu}(S) \longrightarrow \B_\nu(S) \longrightarrow 1.
\]
Moreover, except when $S = \S^2$ and $n = 1,2$ or $S = \Proj$ and $n = 1$, this is in fact a short exact sequence:
\[
1 \longrightarrow \B_\mu(S - \{n\ \text{pts}\}) \longrightarrow \B_{\mu \nu}(S) \longrightarrow \B_\nu(S) \longrightarrow 1.
\]
Furthermore, if $S$ is not closed, the surjection $\B_{\mu\nu}(S) \twoheadrightarrow \B_\nu(S)$ splits.
\end{proposition}

\begin{remark}
Recall that base configurations, hence also the $n$ points removed from $S$, must not be on the boundary of $S$, if $S$ has a non-trivial boundary.
\end{remark}

\begin{proof}[Proof of Proposition~\ref{Fadell-Neuwirth}]
Let us first recall that in considering braid groups, we consider surfaces up to isotopy equivalence, so we can remove the boundary of $S$ if it is non-trivial, and assume that $\partial S = \varnothing$. Recall that if $\lambda$ is a partition of $N$, we denote by $C_\lambda(S)$ the configuration space $F_N(S)/\Sym_\lambda$. Forgetting the first $n$ points induces a map of configuration spaces $C_{\mu \nu}(S) \to C_\nu(S)$, which is a locally trivial fibration with fibres homeomorphic to $C_\mu(S - \{n\ \text{pts}\})$, by a slight adaptation of \cite[Th.~3]{Fadell-Neuwirth} (or~\cite[Th.~1.2]{Birman}), which works for any manifold (without boundary). Since its fibres are path-connected, part of its long exact sequence of homotopy groups is:
\[\pi_2(C_\nu(S)) \rightarrow \B_\mu(S - \{n\ pts\}) \longrightarrow \B_{\mu \nu}(S) \longrightarrow \B_\nu(S) \longrightarrow 1.\]
The map $F_n(S) \twoheadrightarrow C_\nu(S)$ is a covering, so that $\pi_2(C_\nu(S)) \cong \pi_2(F_n(S))$, which is trivial except when $S = \S^2$ and $n = 1,2$ or $S = \Proj$ and $n = 1$. When $S$ is not the sphere or the projective plane, this follows from~\cite[Prop.~1.3]{Birman}, using the fact that higher homotopy groups of surfaces different from $\S^2$ and $\Proj$ are trivial (see Corollary~\ref{pi_2(S)}). When $S=\S^2$ this is \cite[Cor.~p.~244]{Fadell-vanBuskirk1962} and when $S = \Proj$ it is \cite[Cor.~p.~82]{vanBuskirk1966}.

If $S$ is not closed, then there is an isotopy equivalence between $S$ and a proper subsurface $S'$ of $S$. Then one can choose a configuration of $m$ points in $S - S'$ and add them to each configuration of $n$ points of $S'$, getting a map $F_n(S') \rightarrow F_{m+n}(S)$. The induced map $\pi_1(C_\nu(S)) \cong \pi_1(C_\nu(S')) \rightarrow \pi_1(C_{\mu\nu}(S))$ is the required section.
\end{proof}

\begin{remark}
A weaker form of the asphericity statement for $\pi_2(F_n(S))$ used in the above proof may be found~\cite[Cor.~2.2]{Fadell-Neuwirth}. They ask that the surface be compact, but this is only in order to be able to use the classification of compact surfaces, which we easily replaced by the classification of all surfaces in our proof.
\end{remark}

\begin{remark}
When $S$ is not closed, the construction of the splitting in the proof of Proposition~\ref{Fadell-Neuwirth} can also be used to get a morphism $\iota$ from $\B_n(S)$ to $\B_{n+1}(S)$, corresponding to adding an $(n+1)$-st strand near a boundary or an end of $S$. The restriction of this morphism to pure braid groups is a split injection (it is the section in the particular case $\mu = (1)$ and $\nu = (1, \ldots , 1)$ in the above proof). Since it also induces an injection from $\B_n(S)/\PB_n(S) \cong \Sym_n$ into $\B_{n+1}(S)/\PB_{n+1}(S) \cong \Sym_{n+1}$, the morphism $\iota$ itself is injective (compare for instance with the proof of Corollary~\ref{vBn_in_vBn+1}). Notice that, thanks to the connectedness of $S$, $\iota$ depends on the choices made in its construction only up to conjugation. By contrast, if $S$ is closed, there is no obvious construction of such a map, and in fact it does not exist in general~\cite{Guaschidaciberg_Pn(S), GuaschidacibergFNses}.
\end{remark}

\subsection{Pure braid generators and commutators.}\label{subsec_A_as_com}

Some of the results below will hold for all surfaces $S$. However, in order to get more precise results, we need to get more specific and use the classification of surfaces recalled in \Spar\ref{sec:surfaces}. Recall that all the generators $\sigma_i$ of $\B_n$ are identified in $\B_n^{\ab} \cong \Z$ (see Example~\ref{eg:braids}), hence also in $\B_n(S)^{\ab}$. The next proposition deals notably with the order of their common class~$\sigma$. The trichotomy that appears here, which comes from Proposition~\ref{dichotomy}, will play an important role in all that follows.
\begin{proposition}\label{trichotomy}
Let $n \geq 2$. Let us consider the generator $A_{ij}$ of $\PB_n$ as an element of $\B_n(S)$.
\begin{itemize}
\item If $S$ is planar, then the class $\overline A_{ij} \in \B_n(S)^{\ab}$ has infinite order.
\item If $S \cong \S^2$, then the class $\overline A_{ij} \in \B_n(S)^{\ab}$ has order $n-1$. However, its class in $\PB_n(S)^{\ab}$ has infinite order.
\item In all the other cases, $A_{ij}$ is the commutator of two elements of $\PB_n(S)$.
\end{itemize}
\end{proposition}

\begin{proof}
\ul{If $S$ is planar:} then $S$ can be embedded in a disc. Such an embedding induces a morphism $\B_n(S) \rightarrow \B_n(\D) = \B_n$, which in turn induces a morphism from $\B_n(S)^{\ab}$ to $\B_n^{\ab}$. The latter is infinite cyclic, generated by $\sigma$. Our element $\overline A_{ij}$ is sent to $\sigma^2$, hence it cannot be of finite order.

\ul{If $S$ is the sphere:} then from the usual presentation of $\B_n(\S^2)$ (see for instance Corollary~\ref{Braids_on_closed_surfaces} below), we get that $\B_n(\S^2)^{\ab} \cong \Z/(2(n-1))$, generated by $\sigma$. Again, $\overline A_{ij} = \sigma^2$, whose order is $n-1$.

\ul{If $S$ cannot be embedded in the sphere:} then $S$ contains a handle or a crosscap; see Proposition~\ref{dichotomy} and the remark following it. We can then use the explicit isotopies drawn in Figures~\ref{Aij_as_bracket_crosscap} and \ref{Aij_as_bracket_handle} to show that $A_{ij}$ is a bracket of two pure braids (which are respectively in the image of $\psi_j$ and in the image of $\psi_i$).
\end{proof}

\begin{figure}[ht]
  \centering
\begin{subfigure}[b]{\textwidth}
\centering
  \includegraphics[scale=0.3]{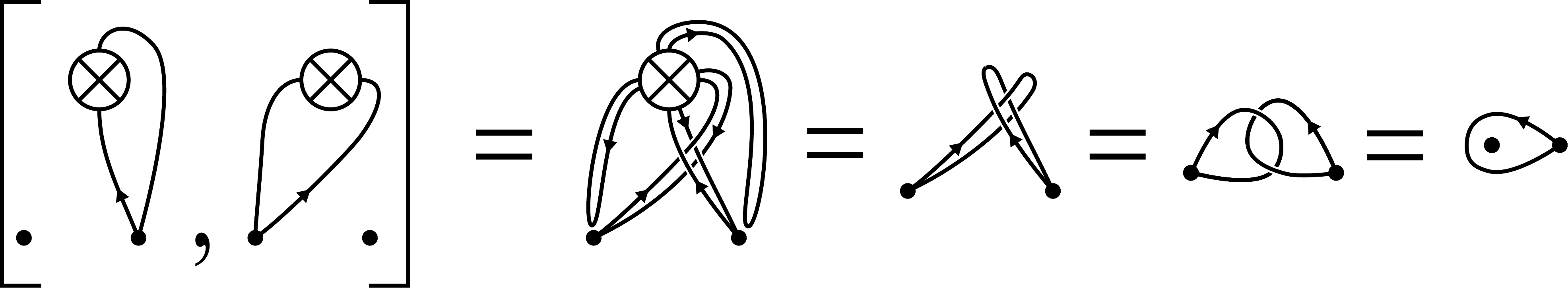}
  \caption{Pure braid generator as a commutator on a surface with a crosscap.}
  \label{Aij_as_bracket_crosscap}
\end{subfigure}
\\[2ex]
\begin{subfigure}[b]{\textwidth}
  \includegraphics[scale=0.27]{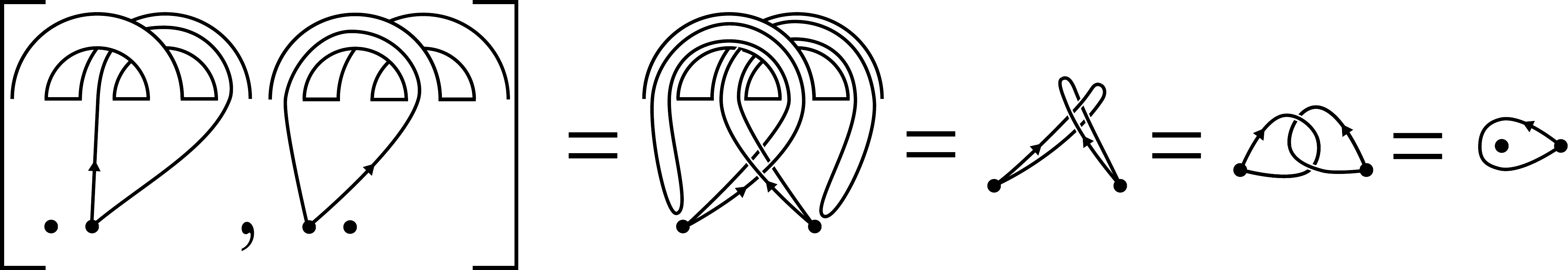}
  \caption{Pure braid generator as a commutator on a surface with a handle.}
  \label{Aij_as_bracket_handle}
\end{subfigure}
\caption{Explicit isotopies expressing a pure braid generator as a commutator. We use them in many different guises throughout the chapter, each time with respect to a different embedding of the crosscap or the handle in our surface, whose image contains exactly two points of the base configuration. For instance, with the notations from Figure~\ref{gen_of_Bn(S)} below, for any choice of $1 \leq i < j \leq n$ and of $1 \leq r \leq g$, the first one gives $A_{ij} = [c_j^{(r)}, (c_i^{(r)})^{-1}]$ in $\B_n(\mathcal N_g)$.}
\label{Aij_as_bracket}
\end{figure}

\section{Presentations of surface braid groups}\label{ss:appendix_braids_surfaces}\label{sec_pstations_B(S)}

In order to prove some of the results below, in particular to determine completely the Lie rings of partitioned braid groups in the stable case, we will need to use presentations of braids groups on compact surfaces. The main tool for determining presentations of surface braid groups (including braids on the disc, which are usual Artin braids) are the Fadell-Neuwirth exact sequences; see Proposition~\ref{Fadell-Neuwirth}. These were already used in the course of the proof of Proposition~\ref{generators_and_N(Pn)} to obtain generators of these groups. Let us now briefly explain how they may be used in order to determine defining relations on these generators. For $S$ a non-closed surface, these exact sequences give a decomposition of $\B_{1,n-1}(S)$ as a semi-direct product of $\B_{n-1}(S)$ with a free group $F$. Then suppose that one has a set of relations satisfied in $\B_n(S)$, defining a group $G_n$ and a well-defined surjection $\pi \colon G_n \twoheadrightarrow \B_n(S)$. One can consider the subgroup $G_{1,n-1} = \pi^{-1}(\B_{1,n-1}(S))$, use the relations to show that it is of index at most $n$ in $G_n$ (which implies that the induced surjection of $G_n/G_{1,n-1}$ onto $\B_n(S)/\B_{1,n-1}(S)$ is a bijection -- but of course not a group morphism, since these quotients do not bear a group structure), and determine a presentation of $G_{1,n-1}$ by using the Reidemeister-Schreier method. Then one shows that $G_{1,n-1}$ decomposes as a semi-direct product of a quotient isomorphic to $G_{n-1}$ with a kernel $K$ generated by a family of elements sent by $\pi$ to a basis of the free group $F$. The latter fact implies that this family must be a free basis of $K$, which means that $\pi \colon K \rightarrow F$ is an isomorphism. By induction, $\pi \colon G_{n-1} \rightarrow \B_{n-1}(S)$ is an isomorphism. Then $\pi \colon G_{1,n-1} \rightarrow \B_{1,n-1}(S)$ must be an isomorphism too. And since the induced surjection of $G_n/G_{1,n-1}$ onto $\B_n(S)/\B_{1,n-1}(S)$ is a bijection, $\pi \colon G_n \rightarrow \B_n(S)$ is an isomorphism (the reader can easily convince themselves that the latter implication works regardless of the existence of a group structure on the quotients).

This method can be used to get presentations of the braid groups of every non-closed surface of finite type; see \cite{Bellingeripresentations} for instance. It can also be adapted to the case of closed surfaces, replacing semi-direct product decompositions by non-split extensions, with some care for the exceptional cases where this is not even an extension. However, we prefer to deduce the case of closed surfaces from the non-closed one: we give here a direct general argument presenting $\B_n(S)$ as the quotient of $\B_n(S - \mathrm{pt})$ by one explicit relation, by applying the Seifert-van Kampen theorem to configuration spaces; see Proposition~\ref{Bn(S)_from_Bn(S-pt)}.

\subsection{Surfaces with one boundary component}

Presentations of braid groups of compact surfaces with one boundary component can be found in \cite[\Spar4]{LambropoulouOldenburg} and in \cite[Th.~1.1 and A.2]{Bellingeripresentations}. We re-write them with our own conventions, which we now explain. 

Let us denote by $\Sigma_{g,1}$ the orientable connected compact surface of genus $g$ with one boundary component, and by $\mathscr N_{g,1}$ the non-orientable connected compact surface of genus $g$ with one boundary component. We draw $\Sigma_{g,1}$ as a rectangle with $2g$ handles attached to it, and $\mathscr N_{g,1}$ as a rectangle with $g$ crosscaps (see Convention~\ref{def_crosscaps}). Our notations for braid generators are detailed in Figure~\ref{gen_of_Bn(S)}. Our drawings of braids are to be thought of as seen from above, and the left-to-right direction in products corresponds to the foreground-to-background direction in our drawings. For instance, with the notations of Figure~\ref{gen_of_Bn(S)}, we have that $\sigma_1 a_k^{(i)} \sigma_1^{-1} =  a_k^{(i+1)}$. As an illustration of these conventions, we draw different representations of the classical pure braid generator $A_{i, i+1}$ in Figure~\ref{pure_braid_generator}.

\begin{figure}[ht]
\centering
\begin{subfigure}[b]{.5\textwidth}
  \centering
  \includegraphics[width=.9\linewidth]{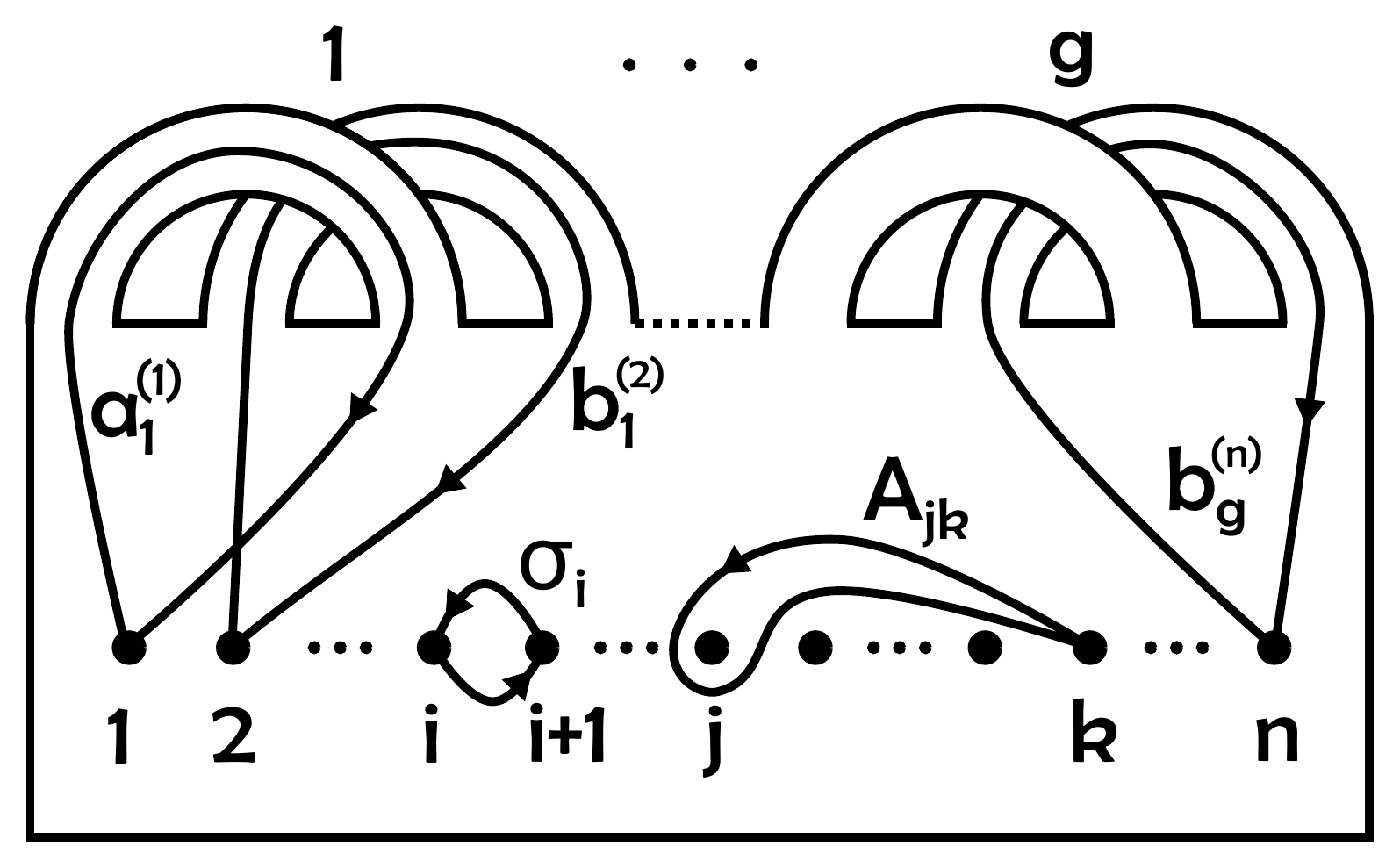}
  \caption{Generators of $\B_n(\Sigma_{g,1})$}
  \label{gen_of_Bn(Sg1)}
\end{subfigure}%
\begin{subfigure}[b]{.5\textwidth}
  \centering
  \includegraphics[width=.9\linewidth]{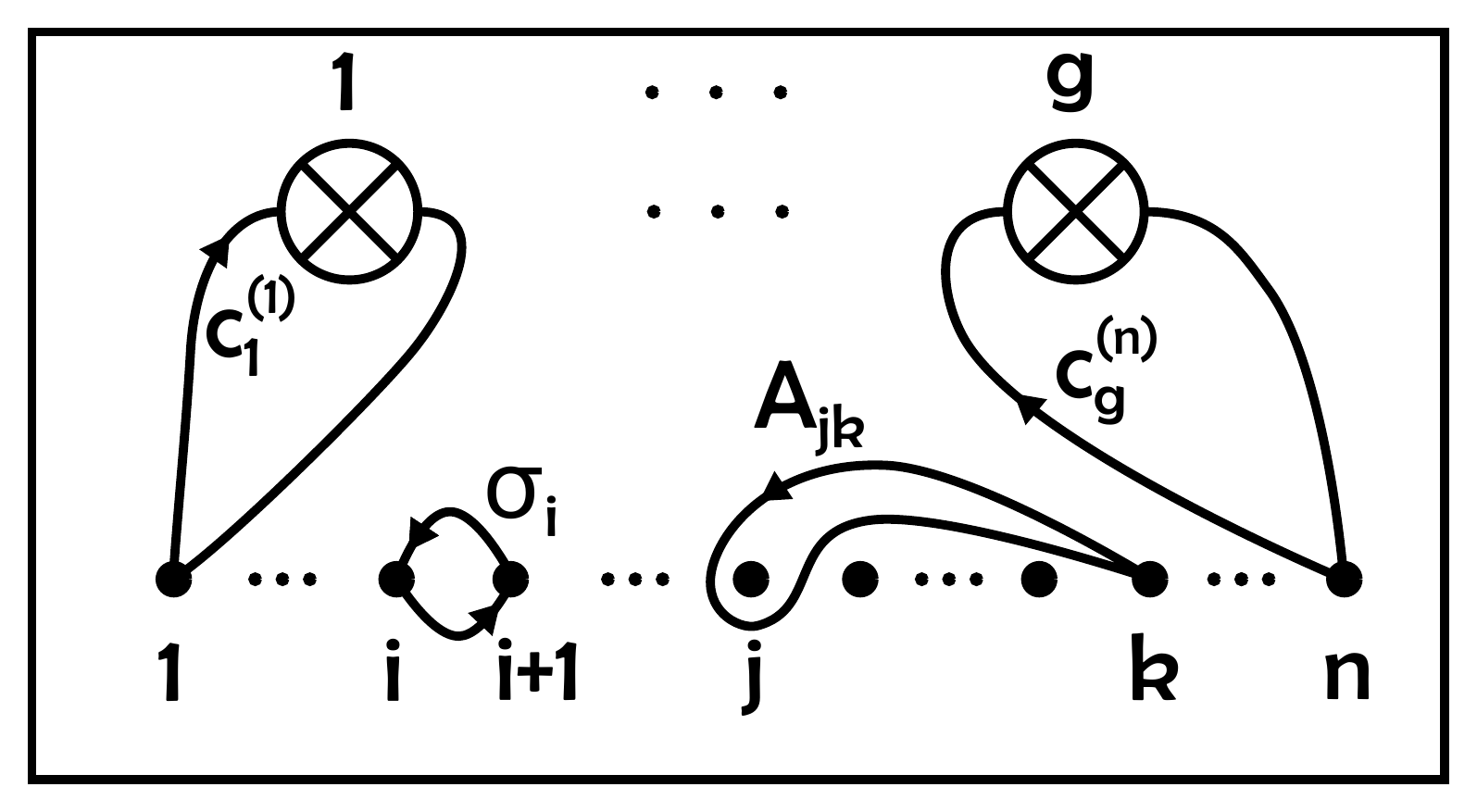}
  \caption{Generators of $\B_n(\mathscr N_{g,1})$}
  \label{gen_of_Bn(Ng1)}
\end{subfigure}
\caption{\textbf{Generators of surface braid groups.} For each generator (except for $\sigma_i$, where the points $i$ and $i+1$ move), only one point of the configuration moves, and the others stay put. We often denote $x_k^{(1)}$ by $x_k$, for $x = a,b,c$.}
\label{gen_of_Bn(S)}
\end{figure}

\begin{figure}[ht]
\centering
\includegraphics[scale=0.5]{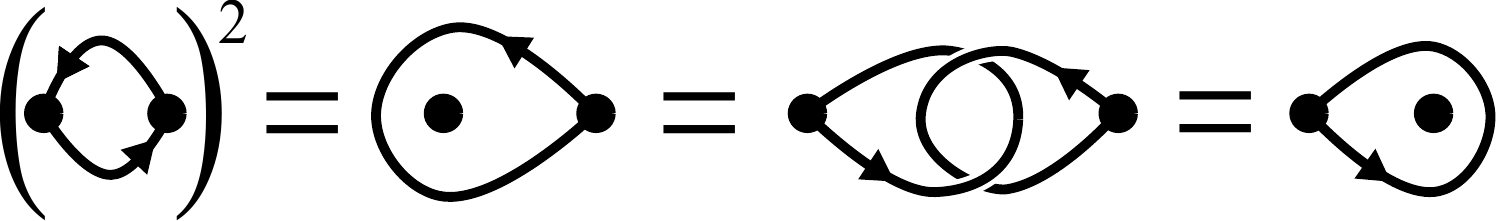}
\caption{The classical pure braid generator $A_{i, i+1} = \sigma_i^2$.}
\label{pure_braid_generator}
\end{figure}

\begin{proposition}[{\cite[Th.~1.1]{Bellingeripresentations}}]\label{Bn(Sg1)}
Let $g \geq 0$. A presentation of the braid group on $\Sigma_{g,1}$, generated by $\sigma_1, \ldots , \sigma_{n-1}$, $a_1, \ldots , a_g$, $b_1, \ldots , b_g$, is given by the braid relations for $\sigma_1, \ldots , \sigma_{n-1}$, to which are added the following four families of relations (where $x$ and $y$ denote either $a$ or $b$, and $1 \leq r, s \leq g$):
\begin{equation}\label{relations_braids_on_orientable_surface}
\begin{cases}
(BS1)\ \textrm{ }\sigma_i\ \rightleftarrows\ x_r & \textrm{for all $r$ and all $i \geq 2$}\ ;\\
(BS2)\ \textrm{ } x_r\ \rightleftarrows\ \sigma_1 y_s \sigma_1^{-1} & \textrm{for  $r < s$}\ ;\\
(BS3)\ \textrm{ }(\sigma_1 x_r)^2 = (x_r \sigma_1)^2 & \textrm{for all $r$}\ ;\\
(BS4)\ \textrm{ }[\sigma_1 b_r \sigma_1^{-1}, a_r^{-1}] = \sigma_1^2 & \textrm{for all $r$}.
\end{cases}
\end{equation}
Recall that $g\ \rightleftarrows\ h$ means that $g$ commutes with $h$ (Notation~\ref{notation_commutation}).
\end{proposition}

Notice that it is easy to check that these relations hold in $\B_n(\Sigma_{g,1})$ by drawing explicit isotopies. See for instance Figure~\ref{Aij_as_bracket_handle} for a drawing of $(BS4)$ (which generalises to $[b_s^{(j)}, (a_r^{(j)})^{-1}] = A_{r,s}$ if $r < s$). The translation between Bellingeri's conventions and ours is as follows:
\begin{itemize}
    \item Our statement is the case $p = 1$ of \cite[Th.~1.1]{Bellingeripresentations}, whence the absence of the $z_k$, and of the relations involving them.
    \item Our $\sigma_i$ is his $\sigma_i^{-1}$, our $a_r$ is his $b_r^{-1}$, and our $b_r$ is his $a_r^{-1}$. 
    \item Our $(BS1)-(BS4)$ are his $(R1)-(R4)$, with $2$ and $3$ exchanged.
\end{itemize}

\begin{remark}\label{rk_genus_0}
Although \cite[Th.~1.1]{Bellingeripresentations} is stated for $g \geq 1$, the proof works equally well if $g = 0$. In fact, the case $g = 0$ of our statement is just the usual presentation of braid groups on the disc. 
\end{remark}

\begin{proposition}[{\cite[Th.~A.2]{Bellingeripresentations}}]\label{Bn(Ng1)}
Let $g \geq 1$. A presentation of the braid group on $\mathscr N_{g,1}$, generated by $\sigma_1, \ldots , \sigma_{n-1}$, $c_1, \ldots , c_g$ is given by the braid relations for $\sigma_1, \ldots , \sigma_{n-1}$, to which are added the following three families of relations (where $1 \leq r, s \leq g$):
\begin{equation}\label{relations_braids_on_non_orientable_surface}
\begin{cases}
(BN1)\ \textrm{ }\sigma_i\ \rightleftarrows\ c_r & \textrm{for all $r$ and all $i \geq 2$}\ ;\\
(BN2)\ \textrm{ } c_r\ \rightleftarrows\ \sigma_1 c_s \sigma_1^{-1} & \textrm{for  $r < s$}\ ;\\
(BN3)\ \textrm{ }[\sigma_1 c_r \sigma_1^{-1}, c_r^{-1}] = \sigma_1^2 & \textrm{for all $r$}.
\end{cases}
\end{equation}
\end{proposition}

Here again, it is easy to check these relations explicitly. See for instance Figure~\ref{Aij_as_bracket_crosscap} for a drawing of $(BN3)$ (which generalises to $[c_s^{(j)}, (c_r^{(j)})^{-1}] = A_{rs}$ if $r < s$). Also, it is the case $p = 1$ of Bellingeri's statement so, again, the $z_k$ and the corresponding relations are irrelevant. Moreover, our $\sigma_i$ is his $\sigma_i^{-1}$, our $c_i$ are his $a_i^{-1}$, and the indexation of our relations is the same as his.

\begin{remark}
The statement of \cite[Th.~A.2]{Bellingeripresentations} is for $g \geq 2$, but the proof works equally well if $g = 1$. 
\end{remark}

\medskip

In the case of non-orientable surfaces, we will also need the case $p > 1$ of \cite[Th.~A.2]{Bellingeripresentations}, which will be used twice, in the proofs of Proposition~\ref{Lie(Bmn(Ng1))} and Proposition~\ref{B2m(Proj)}.
Let us denote by $\mathscr N_{g,n+1}$ the non-orientable connected compact surface of genus $g$ with either $n+1$ punctures or $n+1$ boundary components (recall that, up to isotopy, removing a point and removing an open disc are equivalent). The braid group $\B_m(\mathscr N_{g,n+1})$ may be seen as a subgroup of $\B_{m+n}(\mathscr N_{g,1})$; see Proposition~\ref{Fadell-Neuwirth}. Namely, this subgroup is generated by $\sigma_1, \ldots , \sigma_{n-1}$, $c_1, \ldots , c_g$ and $z_j := A_{1, m+j}$ for all $1\leq j \leq n$.

\begin{proposition}[{\cite[Th.~A.2]{Bellingeripresentations}}]\label{Bm(Ngn)}
Let $g \geq 0$ and $m \geq 1$. A presentation of the braid group $\B_m(\mathscr N_{g,n+1})$, generated by $\sigma_1, \ldots , \sigma_{m-1}, c_1, \ldots , c_g, z_1, \ldots , z_n$ is given by the relations from Proposition~\ref{Bn(Ng1)}, together with the following four families of relations:
\begin{equation*}
\begin{cases}
(BN4)\ \textrm{ }z_j\ \rightleftarrows\ \sigma_i & \textrm{for all $j \leq n$ and all $i \in \{2, \ldots , m-1 \}$}\ ;\\
(BN5)\ \textrm{ }c_r\ \rightleftarrows\ \sigma_1 z_j \sigma_1^{-1} & \textrm{for all $j \leq n$ and all $r \leq g$}\ ;\\
(BN6)\ \textrm{ }z_i\ \rightleftarrows\ \sigma_1 z_j \sigma_1^{-1} & \textrm{for  $i > j$}\ ;\\
(BN7)\ \textrm{ }(\sigma_1 z_j)^2 = (z_j \sigma_1)^2 & \textrm{for all $j \leq n$}.
\end{cases}
\end{equation*}
\end{proposition}

Once more, these relations are easy to check explicitly. The translation between \cite[Th.~A.2]{Bellingeripresentations} and our statement is the same as above, our $z_j$ being the same as Bellingeri's.

\subsection{Closed surfaces}\label{par_pstation_closed_surfaces}

When $S$ is a closed surface, one needs to add a single relation to a presentation of $\B_n(S - \textrm{pt})$ to get a presentation of $\B_n(S)$. In fact, this is a very general fact, which does not require any hypothesis on the surface.

\begin{proposition}\label{Bn(S)_from_Bn(S-pt)}
Let $S$ be a connected surface and $x \in S$ any point in its interior. The inclusion of $S-x$ into $S$ induces a surjective homomorphism
\[
\B_n(S-x) \twoheadrightarrow \B_n(S)
\]
whose kernel is normally generated by a single element $\beta$. Explicitly, $\beta$ is a braid with $n-1$ trivial strands, whose remaining strand loops once around the puncture~$x$.
\end{proposition}

\begin{proof}
Choose a subdisc $D \subset S$ containing $x$ in its interior, and a metric on $D$. Write $\cU_n(S,x)$ for the subspace of $C_n(S) = F_n(S)/\Sym_n$ of (unordered) configurations that have a unique closest point in $D$ to $x$ (which may be $x$ itself). Together with $C_n(S-x)$, this forms an open cover
\[
\left\lbrace \cU_n(S,x) , C_n(S-x) \right\rbrace
\]
of $C_n(S)$, with intersection $\cU_n(S,x) \cap C_n(S-x) = \cU_n(S-x,x)$ the space of $n$-point configurations in $S-x$ that have a unique closest point in $D-x$ to $x$. Note that these subspaces of $C_n(S)$ are all path-connected. Let us choose a basepoint for $C_n(S)$ that lies in $\cU_n(S-x,x)$. The Seifert-van Kampen theorem then gives us a pushout square of groups:
\begin{equation}
\label{eq:svk-square}
\begin{tikzcd}
\pi_1(\cU_n(S-x,x)) \ar[d] \ar[r]
&\B_n(S-x) \ar[d] \\
\pi_1(\cU_n(S,x)) \ar[r]
&\B_n(S).
\end{tikzcd}
\end{equation}

There is a well-defined projection $\cU_n(S,x) \twoheadrightarrow D$ given by remembering just the unique closest point in $D$ to $x$, which restricts to a projection $\cU_n(S-x,x) \twoheadrightarrow D-x$. These are both locally trivial fibrations with fibres canonically homeomorphic to $C_{n-1}(S-x)$. (Over a point $p \in D$, the homeomorphism is induced by the evident homeomorphism between $S-x$ and $S-\bar{B}_{d(p,x)}(x)$, where $\bar{B}_r(x)$ denotes the closed ball of radius $r$ in $D$ centred at $x$ and $d(-,-)$ is the metric that we chose on $D$.) The inclusion of the latter into the former is therefore a map of locally trivial fibrations which is the identity on fibres. Considering the induced long exact sequences of homotopy groups, we obtain a map of exact sequences:
\begin{equation}
\label{eq:map-of-les}
\begin{tikzcd}
1 \ar[r] &\B_{n-1}(S-x) \ar[d, swap, "\cong"] \ar[r] 
&\pi_1(\cU_n(S-x,x)) \ar[r] \ar[d] \ar[l, dashed, bend left=60, looseness=1, swap, "r"]
&\pi_1(D-x) \cong \Z \ar[r] \ar[d] &1 \\
1 \ar[r] &\B_{n-1}(S-x) \ar[r, swap, "\cong"] 
&\pi_1(\cU_n(S,x)) \ar[r]
&\pi_1(D) = 1.
\end{tikzcd}
\end{equation}

The map $r$ obtained from this diagram is a retraction for the upper short exact sequence, whose existence implies that $\pi_1(\cU_n(S-x,x))$ is the direct product of $\pi_1(\cU_n(S,x))$ and $\Z$. Then, the left-hand vertical map in \eqref{eq:svk-square} identifies with the projection that forgets the $\Z$ factor. Together with the fact that \eqref{eq:svk-square} is a pushout square, this implies that the right-hand vertical map in \eqref{eq:svk-square} is the quotient of $\B_n(S-x)$ by (the subgroup normally generated by) the image of the $\Z$ factor of $\pi_1(\cU_n(S-x,x))$ in $\B_n(S-x)$. We may choose for a generator of this $\Z$ factor any element of $\pi_1(\cU_n(S-x,x))$ that projects to a generator of $\pi_1(D-x)$, for example the braid described in the proposition.
\end{proof}

Let us make this explicit:
\begin{corollary}\label{Braids_on_closed_surfaces}
For all $g \geq 0$, the braid group $\B_n(\Sigma_g)$ is the quotient of $\B_n(\Sigma_{g,1})$ by the relation:
\[\sigma_1 \cdots \sigma_{n-2} \sigma_{n-1}^2 \sigma_{n-2} \cdots \sigma_1\ =\ \prod\limits_{r=1}^g [a_r, b_r^{-1}].\]
Similarly, the braid group $\B_n(\mathscr N_g)$ is the quotient of $\B_n(\mathscr N_{g,1})$ by the relation:
\[\sigma_1 \cdots \sigma_{n-2} \sigma_{n-1}^2 \sigma_{n-2} \cdots \sigma_1\ =\ c_1^2 \cdots c_g^2.\]
\end{corollary}

We note that we recover as a particular case the usual presentations of the braid group on the sphere (see \cite{Fadell-vanBuskirk1962} or \cite[Th.~1.11]{Birman}) and of the projective plane (see \cite[\Spar III, page 83]{vanBuskirk1966}). Not having to treat these as exceptional cases is one of the great advantages of the present method.

\begin{proof}[Proof of Corollary~\ref{Braids_on_closed_surfaces}]
This is a direct application of Proposition~\ref{Bn(S)_from_Bn(S-pt)}, using the fact that $\Sigma_g - \mathrm{pt}$ (resp.~$\mathscr N_g - \mathrm{pt}$) is isotopy equivalent to $\Sigma_{g,1}$ (resp.~to $\mathscr N_{g,1}$). We note that $\sigma_1 \cdots \sigma_{n-2} \sigma_{n-1}^2 \sigma_{n-2} \cdots \sigma_1 = A_{12} \cdots A_{1n}$ is the (pure) braid obtained by making the first strand turn once around all the other ones; see Figure~\ref{boundary_relations} for the relevant drawings.
\end{proof}

\begin{figure}[ht]
\centering
\begin{subfigure}[b]{.5\textwidth}
  \centering
  \includegraphics[width=.9\linewidth]{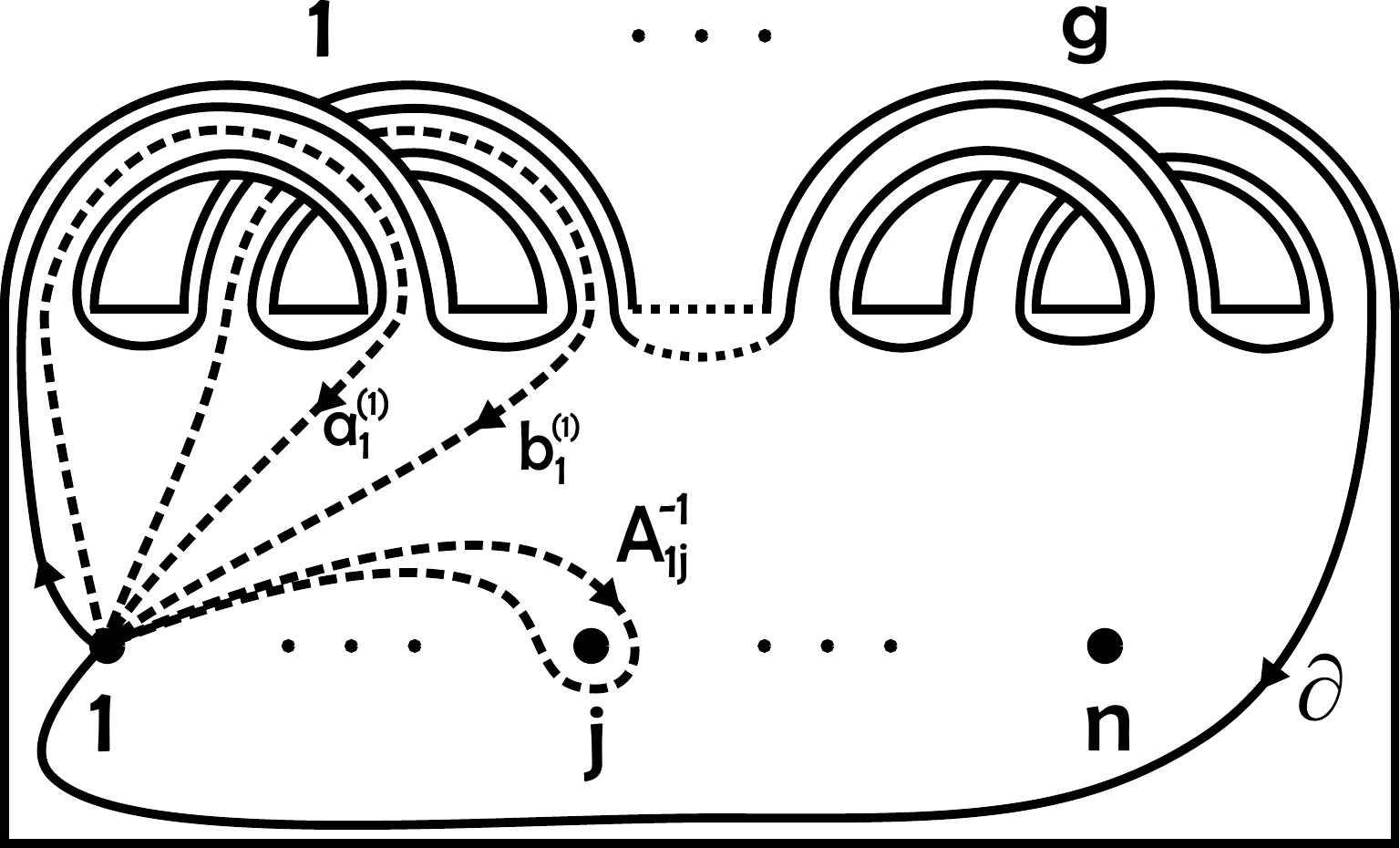}
  \caption{$\partial = [a_1, b_1^{-1}] \cdots [a_g, b_g^{-1}] A_{1n}^{-1} \cdots  A_{12}^{-1}$}
  \label{boundary_relation_on_Sg}
\end{subfigure}%
\begin{subfigure}[b]{.5\textwidth}
  \centering
  \includegraphics[width=.9\linewidth]{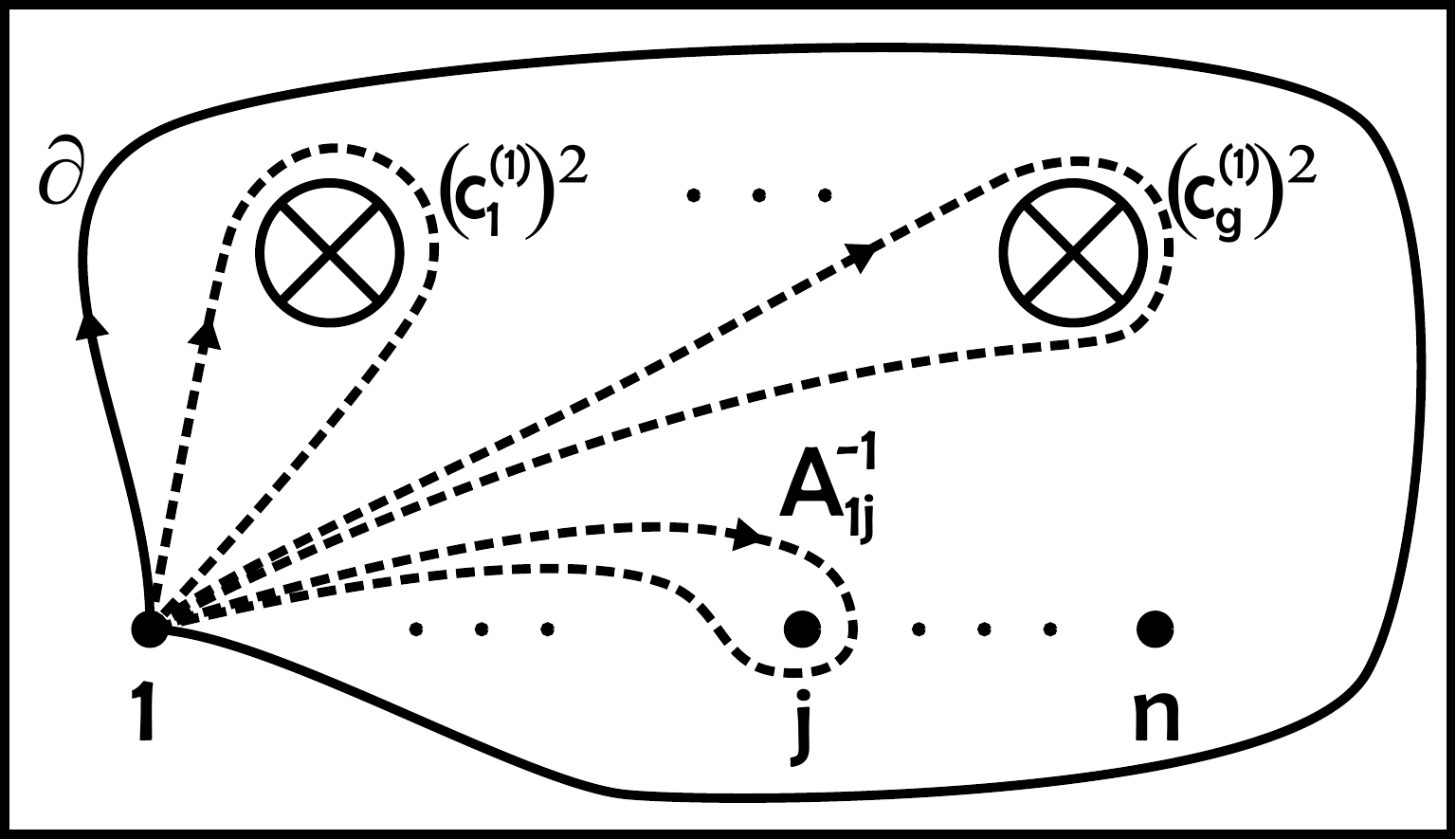}
  \caption{$\partial = c_1^2 \cdots c_g^2 A_{1n}^{-1} \cdots  A_{12}^{-1}$}
  \label{boundary_relation_on_Ng}
\end{subfigure}
\caption{The boundary elements in $\B_n(\Sigma_{g,1})$ and $\B_n(\mathscr N_{g,1})$.}
\label{boundary_relations}
\end{figure}

\subsection{Partitioned braids on closed surfaces}

The above proof of Proposition~\ref{Bn(S)_from_Bn(S-pt)} generalises to partitioned braid groups without much difficulty; see Remark~\ref{adaptation_of_proof_B(S)_from_B(S-pt)}. However, we prefer to deduce these generalisations directly from Proposition~\ref{Bn(S)_from_Bn(S-pt)} itself. We begin with the case of pure braids by describing a direct equivalence between Proposition~\ref{Bn(S)_from_Bn(S-pt)} and the following statement:
\begin{proposition}\label{Pn(S)_from_Pn(S-pt)}
Let $S$ be a connected surface and $x \in S$ any point in its interior. The inclusion of $S-x$ into $S$ induces a surjective homomorphism
\[
\PB_n(S-x) \twoheadrightarrow \PB_n(S)
\]
whose kernel is normally generated by $n$ elements $\beta_1, \ldots , \beta_n$. Explicitly, $\beta_i$ is a braid whose $i$-th strand loops once around the puncture $x$, the other strands being trivial.
\end{proposition}

\begin{proof}[Equivalence between Proposition~\ref{Bn(S)_from_Bn(S-pt)} and Proposition~\ref{Pn(S)_from_Pn(S-pt)}]

Let $N$ be the subgroup of $\B_n(S-x)$ normally generated by the braid $\beta$ from Proposition~\ref{Bn(S)_from_Bn(S-pt)}. Note that $\beta$ is a pure braid, hence we have $N \subseteq \PB_n(S-x)$. Now, let $N'$ be the subgroup of $\PB_n(S-x)$ normally generated by the $\beta_i$. We will show that $N = N'$. This implies the equivalence of Propositions \ref{Bn(S)_from_Bn(S-pt)} and \ref{Pn(S)_from_Pn(S-pt)} by considering the diagram
\[
\begin{tikzcd}
N' \ar[r,hook] & \mathrm{ker}(\pi_{\PB}) \ar[r,hook] \ar[d,equal] & \PB_n(S-x) \ar[r,two heads,"\pi_{\PB}"]  \ar[d,hook] & \PB_n(S) \ar[d,hook] \\
N \ar[r,hook] & \mathrm{ker}(\pi_{\B}) \ar[r,hook] & \B_n(S-x) \ar[r,two heads,"\pi_{\B}"] \ar[d,two heads] & \B_n(S) \ar[d,two heads] \\
&& \mathfrak{S}_n \ar[r,equal] & \mathfrak{S}_n
\end{tikzcd}
\]
and noting that the two propositions are equivalent, respectively, to the statements $N=\mathrm{ker}(\pi_{\B})$ and $N'=\mathrm{ker}(\pi_{\PB})$.

Our definitions of $\beta$ and the $\beta_i$ are up to some choices, but all these choices give elements conjugate to each other (in $\B_n(S-x)$ or in $\PB_n(S-x)$, respectively) or each other's inverses, which does not affect the definition of $N$ and  $N'$. We can make these choices so that:
\begin{itemize}
\item the only moving strand of $\beta$ is the first one,  
\item $\beta$ commutes with every element of the subgroup $\B_{1, n-1} = \langle \sigma_2, \ldots , \sigma_{n-1} \rangle$ of $\B_n(S-x)$ (which consists of braids in a fixed disc $\D \subset S-x$ involving only the strands $2$ to $n$),
\item for each $i$, $\beta_i = (\sigma_1 \cdots \sigma_{i-1})^{-1} \beta (\sigma_1 \cdots \sigma_{i-1})$.
\end{itemize}
See Figure \ref{fig:filling-in-a-puncture} for an example of such choices. The latter relations imply $N' \subseteq N$.

We now show that $N'$ contains all the conjugates of $\beta$ by elements of $\B_n(S-x)$, which implies $N' \supseteq N$. In order to do this, we need only show that it contains $t^{-1} \beta t$ for $t$ in a set of representatives of classes modulo $\PB_n(S-x)$: then every element of $\B_n(S-x)$ is of the form $t \alpha$ for some such $t$ and some $\alpha \in \PB_n(S-x)$, and $(t \alpha)^{-1} \beta (t \alpha) = \alpha^{-1} (t^{-1} \beta t) \alpha$ must be in $N'$.

Every element $\tau \in \Sym_n \cong \B_n(S-x)/\PB_n(S-x)$ is the product of an element $\tau'$ fixing $1$ with some cycle $\tau_1 \cdots \tau_{i-1}$ (precisely, $i = \tau(1)$). Since $\beta$ commutes with every element of the subgroup $\langle\sigma_2, \ldots , \sigma_{n-1}\rangle \subset \B_n(S-x)$, and since this subgroup surjects onto permutations fixing $1$, we can choose a lift $t'$ of $\tau'$ commuting with $\beta$, so that the lift $t = t' \sigma_1 \cdots \sigma_{i-1}$ of $\tau$ to $\B_n(S-x)$ satisfies:
\begin{align*}
t^{-1} \beta t 
&= (\sigma_1 \cdots \sigma_{i-1})^{-1} t'^{-1} \beta t' (\sigma_1 \cdots \sigma_{i-1}) \\
&= (\sigma_1 \cdots \sigma_{i-1})^{-1} \beta (\sigma_1 \cdots \sigma_{i-1}) = \beta_i \in N',
\end{align*}
whence our result.
\end{proof}

\begin{figure}[ht]
	\centering
	\includegraphics[scale=0.8]{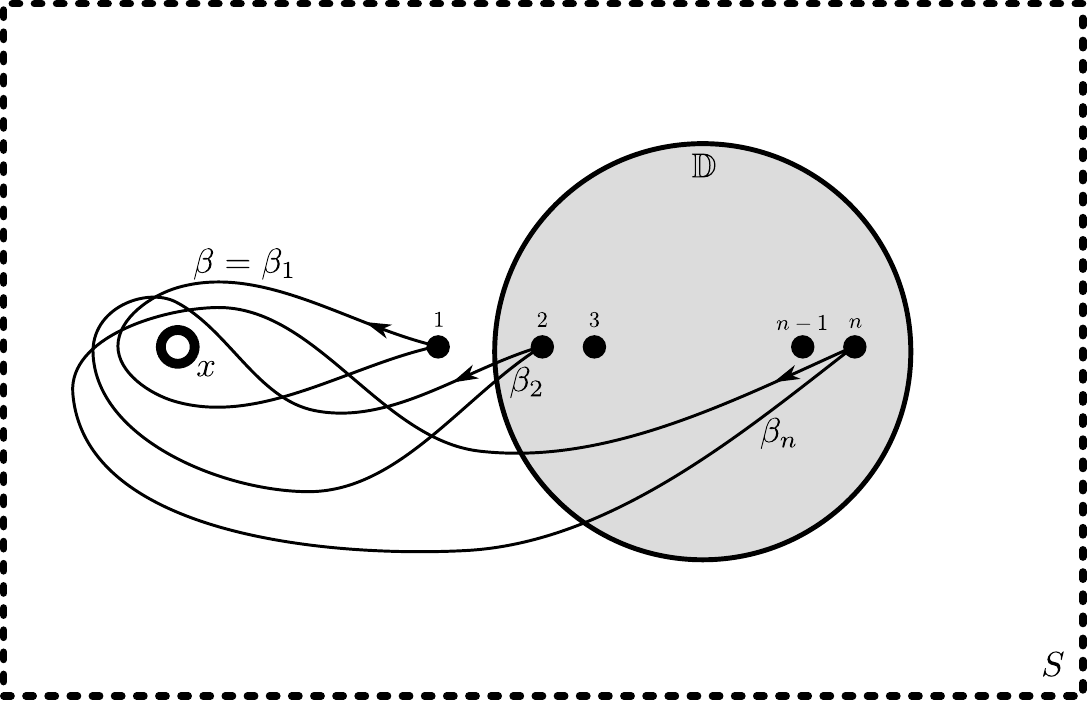}
	\caption{Braids $\beta$ and $\beta_1,\ldots,\beta_n$ from the proof of the equivalence of Propositions \ref{Bn(S)_from_Bn(S-pt)} and \ref{Pn(S)_from_Pn(S-pt)}.}
	\label{fig:filling-in-a-puncture}
\end{figure}

It is not difficult to generalise this to any partitioned braid group:
\begin{corollary}\label{partitioned_B(S)_from_B(S-pt)}
Let $S$ be a connected surface, $x \in S$ any point in its interior and $\lambda = (n_1, \ldots, n_l)$ be a partition of $n$ of length $l$. The inclusion of $S-x$ into $S$ induces a surjective homomorphism
\[
\B_\lambda(S-x) \twoheadrightarrow \B_\lambda(S)
\]
whose kernel is normally generated by $l$ elements $\beta_1,\ldots,\beta_l$. Explicitly, $\beta_i$ is a braid with $n-1$ trivial strands, except one block in the $i$-th block which loops once around the puncture $x$. 
\end{corollary}

\begin{proof}
We have $\PB_n(S-x) \subset \B_\lambda(S-x) \subset \B_n(S-x)$. As a consequence, if $N$ is a normal subgroup of $\PB_n(S-x)$ which is normal in $\B_n(S-x)$, we have $\PB_n(S-x)/N \subset \B_\lambda(S-x)/N \subset \B_n(S-x)/N$. Moreover, normal generators for $N$ in $\PB_n(S-x)$ are also normal generators for $N$ in $\B_\lambda(S-x)$ and, in fact, fewer generators are needed to normally generate $N$ in $\B_\lambda(S-x)$. Namely, using the notations from the previous proof, we have $\beta_{\alpha+1} = \sigma_\alpha^{-1}\beta_\alpha \sigma_\alpha$, so that $\beta_{\alpha+1}$ and $\beta_\alpha$ are conjugate in $\B_\lambda(S-x)$ whenever $\sigma_\alpha \in \B_\lambda(S-x)$, which happens when $\alpha$ and $\alpha + 1$ are in the same block of $\lambda$. As a consequence, we need to pick only one index $\alpha$ in each block of $\lambda$ in order for the $\beta_\alpha$ to normally generate $N$ in $\B_\lambda(S-x)$.
\end{proof}

\begin{remark}
This boils down to considering representatives modulo $\B_\lambda(S-x)$ instead of modulo $\PB_n(S-x)$ (that is, elements of $\Sym_n/\Sym_\lambda$ instead of $\Sym_n$) in the previous reasoning. In fact, one can see that there are straightforward equivalences between all these statements for the different partitions of $n$.
\end{remark}

Let us make these statements explicit. We use the usual convention $A_{ji} = A_{ij}$ and $A_{ii} = 1$:

\begin{corollary}\label{partitioned_Braids_on_closed_surfaces}
For all $g \geq 0$ and any partition $\lambda = (n_1, \ldots , n_l)$, the braid group $\B_\lambda(\Sigma_g)$ is the quotient of $\B_\lambda(\Sigma_{g,1})$ by the relations:
\[A_{\alpha 1} \cdots A_{\alpha n}\  =\ \prod\limits_{r=1}^g [a_r^{(\alpha)}, (b_r^{(\alpha)})^{-1}].\]
Similarly, the braid group $\B_\lambda(\mathscr N_g)$ is the quotient of $\B_\lambda(\mathscr N_{g,1})$ by the relations:
\[A_{\alpha 1} \cdots A_{\alpha n}\  =\ (c_1^{(\alpha)})^2 \cdots (c_g^{(\alpha)})^2.\]
In both cases, $\alpha$ runs through any set of representatives of the blocks of $\lambda$.
\end{corollary}

\begin{remark}\label{adaptation_of_proof_B(S)_from_B(S-pt)}
Instead of the reasoning above, one could adapt the proof of Proposition~\ref{Bn(S)_from_Bn(S-pt)} to partitioned configuration spaces. Precisely, we get an open cover $\{ \cU_\lambda(S,x) , C_\lambda(S-x) \}$ of $C_\lambda(S) = F_n(S)/\Sym_\lambda$ by an obvious adaptation of the definition of $\cU_n(S,x)$. However, this time $\cU_\lambda(S,x)$ is disconnected, with one path-component for each block of the partition $\lambda$. As a consequence, one needs to apply the Seifert-van Kampen theorem once for each path-component of $\cU_\lambda(S,x)$, resulting in taking the quotient of the fundamental group by one additional relation for each application of the theorem. In doing so, one needs to be careful about basepoints, since one obviously cannot choose a common basepoint in the different path-components of $\cU_\lambda(S,x)$.
\end{remark}

\section{The lower central series of the whole group}\label{sec:LCS_undiscrete_partition}

We now turn to the study of the LCS of $\B_n(S)$, which we completely determine for any $n \geq 3$ and any surface $S$. We begin by computing the abelianisation of this group; see Proposition~\ref{Bn(S)^ab}. Then we study $\B_n(S)/\LCS_\infty$, and we show that when $n \geq 3$, it is nilpotent of class at most $2$, which means that the LCS of $\B_n(S)$ stops at most at $\LCS_3$; see Theorem~\ref{Bn(S)/residue} and Corollary~\ref{LCS_Bn(S)_stable}. Finally, we compute the Lie ring, generalising a result of \cite{BellingeriGervaisGuaschi}; see Theorem~\ref{Lie_ring_B(S)}.

\subsection{The abelianisation}

We first compute the abelianisation of $\B_n(S)$, for any $n$ and any surface $S$. In order to do this, recall that the group morphism $\varphi\colon\B_n \rightarrow \B_n(S)$ induces a map $\B_n^{\ab} \rightarrow \B_n(S)^{\ab}$, and that, since all the $\sigma_i$ are identified in $\B_n^{\ab} \cong \Z$, they are so also in $\B_n(S)^{\ab}$. We denote by $\sigma$ their common image in $\B_n(S)^{\ab}$.

\begin{proposition}\label{Bn(S)^ab}
In general, for all $n \geq 2$, we have:
\[\B_n(S)^{\ab} \cong \pi_1(S)^{\ab} \times \langle \sigma \rangle.\]
Moreover, $\sigma$ is:
\begin{itemize}
\item of infinite order if $S$ is planar,
\item of order $2(n-1)$ if  $S \cong \S^2$,
\item of order $2$ in all the other cases.
\end{itemize}
\end{proposition}

\begin{proof}
Consider the short exact sequence $\PB_{n}^{\circ}(S)\hookrightarrow\B_{n}(S)\overset{\pi_{S}}{\twoheadrightarrow}\pi_{1}(S)\wr\Sym_{n}$. By Lemma~\ref{abelianization_from_coinv}, it induces an exact sequence of abelian groups:
\[\begin{tikzcd}
(\PB_n^\circ(S)^{\ab})_{\B_n(S)} \ar[r] 
&\B_n(S)^{\ab} \ar[r, two heads] 
&(\pi_1(S) \wr \Sym_n)^{\ab} \ar[r]
&0.
\end{tikzcd}\]
On the one hand, the quotient $(\pi_1(S) \wr \Sym_n)^{\ab}$ is isomorphic to $\pi_1(S)^{\ab} \times \Z/2$ (Lemma~\ref{wr_ab}). On the other hand, it follows from Proposition~\ref{generators_and_N(Pn)} that $\PB_n^\circ(S)$ is generated by $\PB_n$ under the action of $\B_n(S)$. As a consequence, the map $\PB_n^{\ab} \rightarrow (\PB_n^\circ(S)^{\ab})_{\B_n(S)}$ induced by $\varphi$ is surjective. Moreover, it factors through $(\PB_n^{\ab})_{\B_n} = (\PB_n^{\ab})_{\Sym_n} \cong \Z$. Thus $(\PB_n^\circ(S)^{\ab})_{\B_n(S)}$ is cyclic, and its image in $\B_n(S)^{\ab}$ is generated by $\sigma^2$, which is the image of any pure braid generator.

All this implies that $\B_n(S)^{\ab}/\langle \sigma \rangle \cong \pi_1(S)^{\ab}$. Moreover, the corresponding projection $\B_n(S)^{\ab} \twoheadrightarrow \pi_1(S)^{\ab}$ splits, a splitting being induced by any of the $\psi_i$. As a consequence, $\B_n(S)^{\ab}$ identifies with $\pi_1(S)^{\ab} \times \langle \sigma \rangle$, and we can use Proposition~\ref{trichotomy} to get a complete calculation.
\end{proof}

\subsection{The lower central series}

Let us now turn to the study of the LCS of $\B_n(S)$. Our main tool for studying it is a decomposition theorem (Theorem~\ref{Bn(S)/residue} below), whose proof relies on the following:
\begin{lemma}\label{sigma_central}
Let $n \geq 3$. The image of $\B_n$ in $\B_n(S)/\LCS_\infty$ is cyclic, and it is central. Namely, it is generated by the common class $\sigma$ of the usual generators $\sigma_i$ of $\B_n$.
\end{lemma}

\begin{proof}
The morphism $\varphi$ sends $\LCS_\infty(\B_n)$ to $\LCS_\infty(\B_n(S))$. We know that $\sigma_i \sigma_j^{-1} \in \LCS_\infty(\B_n)$ for all $i,j < n$, so $\sigma_i \equiv \sigma_j \pmod{\LCS_\infty(\B_n(S))}$. Let us denote by $\sigma \in \B_n(S)/\LCS_\infty$ the common image of the $\sigma_i$. Since the $\sigma_i$ generate $\B_n$, the image of $\B_n$ in $\B_n(S)/\LCS_\infty$ is the cyclic subgroup generated by $\sigma$. In particular, its elements commute with $\sigma$. Moreover, for all $\gamma \in \pi_1(S - \D)$, the braids $\sigma_2$ and $\psi_1(\gamma)$ have disjoint support, hence $\sigma$ also commutes with the image of $\psi_1$. Since the images of $\varphi$ and $\psi_1$ generate $\B_n(S)$ by Proposition~\ref{generators_and_N(Pn)}, this means that $\sigma$ is a central element of $\B_n(S)/\LCS_\infty$.
\end{proof}

We can now state our main decomposition theorem:
\begin{theorem}\label{Bn(S)/residue}
For all $n \geq 3$, there is a central extension:
\[\begin{tikzcd}
\langle \sigma^2 \rangle \ar[r, hook] 
&\B_n(S)/\LCS_\infty \ar[r, two heads] 
&\pi_1(S)^{\ab} \times \Z/2.
\end{tikzcd}\] 
\end{theorem}

\begin{proof}
Since $\pi_S$ is surjective, it sends $\LCS_\infty(\B_n(S))$ onto a normal subgroup of $\pi_1(S) \wr \Sym_n$. This normal subgroup is contained in $\LCS_\infty(\pi_1(S) \wr \Sym_n)$ and contains the $\tau_i \tau_j^{-1}$, which are the images of the $\sigma_i \sigma_j^{-1} \in \LCS_\infty(\B_n(S))$. By Lemma~\ref{LCS_wr_sym_stable}, it is equal to $\LCS_2(\pi_1(S) \wr \Sym_n)$. Then, we can apply the Nine Lemma to the diagram:
\[\begin{tikzcd}
\LCS_\infty(\B_n(S)) \cap \PB_n^\circ(S) \ar[r, hook] \ar[d, hook]
&\LCS_\infty(\B_n(S)) \ar[r, two heads] \ar[d, hook]
&[15pt]\LCS_2(\pi_1(S) \wr \Sym_n). \ar[d, hook] \\
\PB_n^\circ(S) \ar[r, hook] \ar[d, two heads]
&\B_n(S) \ar[r, two heads, "\pi_S"] \ar[d, two heads]
&\pi_1(S) \wr \Sym_n \ar[d, two heads] \\
\PB_n^\circ(S)/\LCS_\infty(\B_n(S)) \cap \PB_n^\circ(S) \ar[r, dashed] 
&\B_n(S)/\LCS_\infty \ar[r, dashed, "\overline \pi_S"] 
&(\pi_1(S) \wr \Sym_n)^{\ab}.
\end{tikzcd}\] 
Then recall that $(\pi_1(S) \wr \Sym_n)^{\ab} \cong \pi_1(S)^{\ab} \times \Z/2$. Thus, we are left with analysing the kernel $\PB_n^\circ(S)/\LCS_\infty(\B_n(S)) \cap \PB_n^\circ(S)$ of $\overline \pi_S$, which is the image of $\PB_n^\circ(S)$ in $\B_n(S)/\LCS_\infty$. Since $\PB_n^\circ(S)$ is the normal closure of $\PB_n$ in $\B_n(S)$ by Proposition~\ref{generators_and_N(Pn)}, its image in $\B_n(S)/\LCS_\infty$ is the normal closure of the image of $\PB_n$. But $\PB_n$ is sent to $\langle \sigma^2 \rangle$ which is central (and, in particular, normal) in $\B_n(S)/\LCS_\infty$, whence the result.
\end{proof}

\begin{remark}
\label{rmk:central-ext}
We also have a central extension:
\[\begin{tikzcd}
\langle \sigma \rangle \ar[r, hook] 
&\B_n(S)/\LCS_\infty \ar[r, two heads] 
&\pi_1(S)^{\ab}.
\end{tikzcd}\] 
This slightly different statement tells us slightly different things: it implies that $\sigma$ is central, whereas the statement of the theorem says that $\overline \sigma$ is not trivial in $\B_n(S)^{\ab}$.
\end{remark}

\begin{corollary}\label{LCS_Bn(S)_stable}
For any $n \geq 3$, we have $\LCS_3(\B_n(S)) = \LCS_4(\B_n(S))$.
\end{corollary}

\begin{proof}
Proposition~\ref{Bn(S)/residue} implies that $\B_n(S)/\LCS_\infty$ is $2$-nilpotent, which means exactly that its $\LCS_3$ is trivial. In other words, $\LCS_3 \subseteq \LCS_\infty$ for $\B_n(S)$.
\end{proof}

The remaining cases consist in the cases when $n=1$ and when $n=2$.

\begin{proposition}\label{prop:LCS_non_stop_pi_1_surfaces}
For any connected surface $S$, either $\B_1(S) = \pi_1(S)$ is abelian, which occurs precisely when $S \in \{\D- pt, \D, \S^2, \Tor, \Proj, \Moeb\}$ up to isotopy equivalence, or its LCS does not stop.
\end{proposition}

\begin{proof}
If $S$ is closed and not in $\{\S^2,\Proj,\Tor\}$ then the LCS of $\pi_1(S)$ does not stop. To see this, note that it admits a presentation of the form $\langle a_1,b_1,\ldots,a_g,b_g \mid [a_1,b_1] \cdots [a_g,b_g] = 1 \rangle$ or $\langle c_1,\ldots,c_g \mid c_1^2 \cdots c_g^2 = 1 \rangle$ for $g\geq 2$, and these both project onto $\Z/2 * \Z/2 = \langle x,y \mid x^2 = y^2 = 1 \rangle$ by sending $a_1,b_1,c_1$ to $x$ and all other generators to $y$; the LCS of $\Z/2 * \Z/2$ does not stop by Proposition \ref{LCS_of_Z/2*Z/2}. If $S$ is not closed, its fundamental group is free by Proposition~\ref{pi_1(S)}. If it is also not in $\{\D - pt,\D,\Moeb\}$ then its fundamental group is moreover \emph{non-abelian} free, and hence its LCS does not stop by~\cite{Magnus1935} (see also~\cite[Chap.~5]{MagnusKarrassSolitar}).
\end{proof}

\begin{remark}\label{pi1(S)_res_nilp}
The fundamental group of a surface is known to always be residually nilpotent, and almost always residually torsion-free nilpotent (the only exceptions to the latter being the projective plane and the Klein bottle). This follows from~\cite[Chap.~5]{MagnusKarrassSolitar} for non-closed surfaces, whose fundamental group is free. For closed surfaces, it follows from~\cite{Baumslag} if $S$ contains a handle: then the Seifert-van Kampen theorem applied to a decomposition $S = \Tor \# S'$ gives a presentation of the form of Theorem~1 therein. The remaining surfaces are the sphere, the projective plane and the Klein bottle, whose LCS are easily computed explicitly (see Proposition~\ref{LCS_Klein} for the last one).
\end{remark}

\begin{proposition}\label{prop:LCS_non_stop_unpartitioned_braid_groups_surfaces}
When $S$ is not $\D$, $\S^2$ or $\Proj$, up to isotopy equivalence, the LCS of $\B_2(S)$ does not stop.
\end{proposition}

\begin{proof}
The group $\B_2(S)$ surjects onto $\pi_1(S) \wr \Sym_2$. Since $\pi_1(S)^{\ab}$ surjects onto~$\Z$ except in the three excluded cases of the statement, we can apply Corollary~\ref{LCS_G_wr_S2} to see that the LCS of $\pi_1(S) \wr \Sym_2$ does not stop. Then the one of $\B_n(S)$ does not either by Lemma~\ref{lem:stationary_quotient}.
\end{proof}

\begin{remark}
For orientable surfaces, $\B_2(S)$ is residually nilpotent \cite[Cor.~10]{BellingeriBardakov2}. This result can in fact be extended to non-orientable surfaces, with the same method, using the results quoted in Remark~\ref{pi1(S)_res_nilp}.
\end{remark}

\begin{remark}\label{rmk:ref_exceptional_block_size_2} 
Clearly, the LCS of $\B_2(\D) \cong \Z$ and of $\B_2(\S^2) \cong \Z/2$ both stop at $\LCS_2$. The fact that $\B_2(\S^2) \cong \Z/2$ is the case $n = 2$ and $g = 0$ of Corollary~\ref{Braids_on_closed_surfaces}. On the other hand, the group $\B_2(\Proj)$ is the dicyclic group of order 16 (Corollary~\ref{B2(Proj)}), which is $3$-nilpotent, so its LCS stops at $\LCS_4$.
\end{remark}

\subsection{The Lie ring}\label{subsec:Lie_ring_undiscrete_partition_braids_surfaces}

We can be more precise about our description of the LCS of surface braid groups. In particular, Corollary~\ref{LCS_Bn(S)_stable} says that, for $n \geq 3$, the LCS stops at most at $\LCS_3$, but it does not say when it stops at $\LCS_2$ (then the associated Lie ring consists only of the abelianisation, which has already been computed in Proposition~\ref{Bn(S)^ab}), or when it stops at $\LCS_3$ (in which case the Lie ring is $2$-nilpotent but not abelian). We now show that the latter holds only for non-planar orientable surfaces, and we compute precisely $\Lie(\B_n(S))$, generalising results of \cite{BellingeriGervaisGuaschi}.

\begin{theorem}\label{Lie_ring_B(S)}
Let $n \geq 3$ be an integer and $S$ be a connected surface. The LCS of $\B_n(S)$:
\begin{itemize}
\item \emph{stops at $\LCS_2$} if $S$ is either planar or non-orientable, or if $S \cong \S^2$. 
\item \emph{stops at $\LCS_3$} in the other cases. Then $\Lie_2(\B_n(S))$ is cyclic, generated by the common class $\sigma^2$ of the pure braid generators $A_{ij}$.
\end{itemize}
Moreover, in the second case, $\sigma^2$ is of finite order if and only if $S$ is closed, in which case its order is $n+g-1$, where $g$ is the genus of $S$.
\end{theorem}

\begin{proof}
Since the LCS of $\B_n(S)$ stops at most at $\LCS_3$, $\Lie_2(\B_n(S))$ identifies with $\LCS_2 \left(\B_n(S)/\LCS_\infty \right)$ (the latter being $\LCS_2/\LCS_\infty(\B_n(S)) = \LCS_2/\LCS_3(\B_n(S))$). Moreover, using the central extension of Theorem~\ref{Bn(S)/residue}, we see that in $\B_n(S)/\LCS_\infty$, the subgroup $\LCS_2$ must be contained in $\langle \sigma^2 \rangle$, which implies that it is cyclic, generated by a power of $\sigma^2$.

\ul{Planar surfaces:} if $S$ is planar, then the common class $\sigma$ of the $\sigma_i$ in $\B_n(S)^{\ab} = (\B_n(S)/\LCS_\infty)^{\ab}$ is of infinite order by Proposition~\ref{Bn(S)^ab}. Hence $\LCS_2(\B_n(S)/\LCS_\infty)$ does not contain any power of $\sigma = \overline{\sigma_1}$. But $\LCS_2(\B_n(S)/\LCS_\infty)$ is contained in $\langle \sigma^2 \rangle$, so it must be trivial, which means that $\LCS_2 = \LCS_\infty$ for $\B_n(S)$.

\ul{The sphere:} if $S = \S^2$, then $\B_n(\S^2)$ is a quotient of $\B_n$ (see Corollary~\ref{Braids_on_closed_surfaces}), which implies that its LCS also stops at $\LCS_2$.

\ul{Non-orientable surfaces:} the surface $S$ is non-orientable if and only if it contains an embedded Möbius band. Then $\sigma_1^2 = [\sigma_1 c \sigma_1^{-1} , c^{-1}]$ for some $c \in \B_n(S)$: precisely, $c$ is a braid whose first strand goes around the Möbius strip once, that is, through the crosscap; see Figure~\ref{Aij_as_bracket_crosscap}. Since $\sigma_1$ is sent to the central element $\sigma$ of $\B_n(S)/\LCS_\infty$, this relation implies that $\sigma^2 = [\overline c, \overline c^{-1}] = 1$. Thus $\langle \sigma^2 \rangle$ is trivial, and so is its subgroup  $\LCS_2 \left(\B_n(S)/\LCS_\infty \right) \cong \Lie_2(\B_n(S))$.

\ul{Non-planar orientable surfaces:} if $S$ has a handle, then $\sigma_1^2 = [a^{-1}, \sigma_1 b \sigma_1^{-1}]$ for some $a,b\in \B_n(S)$. Precisely, $a$ and $b$ are braids whose first strands go around a handle; see Figure~\ref{Aij_as_bracket_handle}. Hence $\sigma^2 \in \LCS_2\left(\B_n(S)/\LCS_\infty \right)$, which implies that $\Lie_2(\B_n(S))$ is generated by $\sigma^2$.

If $S$ is a \ul{non-planar compact orientable surface with at least one boundary component}, then $S$ can be embedded in some $\Sigma_{g,1}$ for an arbitrary large $g$, by attaching a disc to each boundary component save one. The induced map $\Lie_2(\B_n(S)) \rightarrow \Lie_2(\B_n(\Sigma_{g,1}))$ sends $\sigma^2$ to $\sigma^2$, and the latter is of infinite order (Proposition~\ref{Lie(Bn(Sg1))}).

If $S$ is a \ul{non-compact non-planar orientable surface}, let us suppose that $\sigma^2$ is a torsion element in $\B_n(S)/\LCS_\infty = \B_n(S)/\LCS_3$. Then for some integer $k$, $\sigma_1^{2k}$ is equal to some product of commutators of length at least $3$ in $\B_n(S)$.  Such a formula involves only a finite number of braids. Let us choose a representative of each of these isotopy classes. These involve a finite number of paths on the surface. If moreover we choose an isotopy realising the aforementioned equality of braids (using the concatenation of the chosen representatives as a representative of the right-hand side), the image of this isotopy is contained in a compact subsurface $S'$ of $S$. Thus, our formula also holds in $\B_n(S')$: $\sigma_1^{2k}$ is equal to some product of commutators of length at least $3$ in there. Then $\sigma^{2k} = 1$ in  $\B_n(S')/\LCS_3$. However, this contradicts the previous case, since $S'$ cannot be closed (nor planar). We conclude that  $\sigma^2$ cannot be a torsion element in $\B_n(S)/\LCS_\infty$; equivalently, $\Lie_2(\B_n(S)) \cong \Z$.

If $S$ is a \ul{closed orientable surface} $\Sigma_g$ of genus $g \geq 1$, then $\Lie_2(\B_n(\Sigma_g)) \cong \Z /(n+g-1)$. This is \cite[Th.~1]{BellingeriGervaisGuaschi}, which can be deduced from Proposition~\ref{Lie(Bn(Sg1))} below. We will do so in a more general context later; see Proposition~\ref{partitioned_Lie(B(Sg)}.
\end{proof}

The proof of the following result is inspired from \cite{BellingeriGervaisGuaschi}. In fact, it is equivalent to \cite[Th.~1]{BellingeriGervaisGuaschi} in a quite straightforward way; see Remark \ref{rmk:double_recoverings}.
\begin{proposition}\label{Lie(Bn(Sg1))}
Let $g \geq 1$, and let $\Sigma_{g,1}$ denote the compact surface of genus $g$ with one boundary component. For every $n \geq 1$, $\Lie_2(\B_n(\Sigma_{g,1})) \cong \Z$, generated by the common class $\sigma^2$ of the pure braid generators.
\end{proposition}

\begin{proof}
We compute completely $\B_n(\Sigma_{g,1})/\LCS_3$. Let us consider the quotient of $\B_n(\Sigma_{g,1})$ by the relations $\sigma_i = \sigma_{i+1}$, that is, by the normal closure $N$ of the $\sigma_i \sigma_{i+1}^{-1}$ for $1\leq i\leq n-1$. We already know that, modulo $\LCS_3$, the braid generators $\sigma_i$ have a common class $\sigma$. In particular, $N \subseteq \LCS_3$, and we will show that $N = \LCS_3$ by showing that $G := \B_n(\Sigma_{g,1})/N$ is $2$-nilpotent. Thanks to Proposition~\ref{Bn(Sg1)}, we can compute a presentation of this quotient. It is generated by $\sigma$, together with $a_r$ and $b_r$ for $1\leq r\leq g$. The braid relations on the $\sigma_i$ become trivial there. $(BS1)$ says that $\sigma$ is central. $(BS2)$ says that the $a_r$ and the $b_s$ commute with one another, except $a_r$ and $b_r$ (for each $r$). Since $\sigma$ is central, $(BS3)$ becomes trivial. Finally, $(BS4)$ can be written as $b_r a_r b_r^{-1} = a_r \sigma^{-2}$. From this presentation, one can see that $G \cong (\Z \times \Z^g) \rtimes \Z^g$, where the three factors are free abelian on $\sigma$, the $a_r$ and the $b_r$ respectively; the action of each $b_r$ is trivial on $\sigma$ and the $a_s$ if $s \neq r$, and $b_r \cdot a_r = a_r - 2\sigma$. Since this group is $2$-nilpotent, we have $N = \LCS_3(\B_n(\Sigma_{g,1}))$, as announced. Moreover, the Lie ring of $G$, which is the Lie ring of $\B_n(\Sigma_{g,1})$ by Corollary~\ref{LCS_Bn(S)_stable}, is easy to compute. Namely, $\Lie_1(G) = G^{\ab} \cong (\Z/2)^2 \times \Z^{2g}$, $\Lie_2(G) = \LCS_2(G) \cong \Z$ is generated by $\overline{\sigma^2}$, and the only non-trivial brackets of generators are $[\overline a_r, \overline b_r] = \overline{\sigma^2}$. 
\end{proof}

\begin{remark}\label{rmk:double_recoverings}
Here we choose to recover \cite[Th.~1]{BellingeriGervaisGuaschi} from Proposition~\ref{Lie(Bn(Sg1))}. However, one can also \emph{deduce} Proposition~\ref{Lie(Bn(Sg1))} from \cite[Th.~1]{BellingeriGervaisGuaschi}. Namely, one can embed $\Sigma_{g,1}$ into $\Sigma_{g'}$ for any $g' \geq g$, by attaching $\Sigma_{g'-g,1}$ along the boundary component. The induced map $\Lie_2(\B_n(\Sigma_{g,1})) \rightarrow \Lie_2(\B_n(\Sigma_{g'}))$ sends $\sigma^2$ to $\sigma^2$, and the latter is of order $n+g'-1$ by \cite[Th.~1]{BellingeriGervaisGuaschi}. Thus, by varying $g'$, we see that $\sigma^2$ cannot be of finite order inside $\Lie_2(\B_n(\Sigma_{g,1}))$.
\end{remark}

\begin{remark}
The group $(\Z \times \Z^g) \rtimes \Z^g$ appearing in the proof of Proposition~\ref{Lie(Bn(Sg1))}, which is then the maximal nilpotent quotient of $\B_n(\Sigma_{g,1})$, has a nice interpretation as a matrix group resembling the Heisenberg group of a symplectic vector space. Namely, it is the subgroup of $\mathrm{GL}_{g+2}(\mathbb Q)$ given by:
\[\begin{pmatrix}
1 &\Z &\Z     &\cdots &\Z     &\frac12 \Z \\
  &1  &0      &\cdots &0      &\Z         \\
  &   &\ddots &\ddots &\vdots &\vdots     \\
  &   &       &\ddots &0      &\Z         \\
  &   &       &       &1      &\Z         \\
  &   &       &       &       &1  
\end{pmatrix}\] 
\end{remark}

\begin{remark}[About the Riemann-Hurwitz formula]\label{rmk:Riemann-Hurwitz_formula}
The result \cite[Th.~1]{BellingeriGervaisGuaschi} quoted above and recovered below (Proposition~\ref{partitioned_Lie(B(Sg)}) says that when $S$ is an orientable closed surface, the information encoded in $\Lie_2(\B_n(S))$ is essentially the genus of $S$, or its Euler characteristic. In fact, one can recover the Riemann-Hurwitz formula for (unramified) coverings of closed orientable surfaces from this computation. Indeed, let $p \colon \Sigma_h \twoheadrightarrow \Sigma_g$ be a $k$-sheeted covering. It induces a continuous map $p^* \colon C_n(\Sigma_g) \rightarrow C_{kn}(\Sigma_h)$ between unordered configuration spaces sending a configuration of $n$ points to the configuration of the $kn$ preimages of these points by $p$. The induced map between the fundamental groups sends the pure braid generator $A_{12}$ to the product of the $A_{j,j+1}$ for $j \in p^{-1}(1)$ (if the correct conventions are chosen). Thus the induced map from $\Lie_2(\B_n(\Sigma_g))$ to $\Lie_2(\B_{kn}(\Sigma_h))$ sends $\sigma^2$ to $k \sigma^2$. Since $\sigma^2$ is of order $n+g-1$ in $\Lie_2(\B_n(\Sigma_g))$, we obtain $(n+g-1)k \sigma^2 = 0$ in $\Lie_2(\B_{kn}(\Sigma_h))$. However, $\sigma^2$ is of order $kn + h - 1$ there, implying that $kn + h - 1$ divides $kn + k(g-1)$. This holds for all $n$. For $n$ big enough, $kn + k(g-1) < 2(kn + h - 1)$, so the only possibility is that $kn + k(g-1) = kn + h - 1$, that is, $k(g-1) = h - 1$, which is equivalent to the Riemann-Hurwitz formula for $p$.
\end{remark}

\section{Partitioned braids on surfaces}\label{sec:partitioned_braids_surfaces_stable}

Let us now study the LCS of partitioned surface braid groups $\B_{\lambda}(S)$, generalising the results from \Spar\ref{sec:LCS_undiscrete_partition}, which can be seen as the case of the trivial partition $\lambda = (n)$ of $n$. We follow the same steps: we first compute the abelianisations of $\B_{\lambda}(S)$ in  Proposition~\ref{partitioned_B(S)^ab}, before studying $\B_{\lambda}(S)/\LCS_\infty$ and showing that the LCS of $\B_{\lambda}(S)$ stops at most at $\LCS_3$ when the partition $\lambda$ has only blocks of size at least $3$, in Theorem~\ref{B_lambda(S)/residue} and Corollary~\ref{LCS_B_lambda(S)_stable}. Finally, under the latter hypothesis, we compute the associated Lie rings in Theorem~\ref{Lie_ring_partitioned_B(S)}.

\subsection{The abelianisation}

Let $\lambda = (n_1, \ldots , n_l)$ be a partition of an integer $n$. A computation of $\B_\lambda(S)^{\ab}$ can be obtained by a quite straightforward generalisation of the computation of $\B_n(S)^{\ab}$ from Proposition~\ref{Bn(S)^ab}. 

Let us first recall that the morphism $\varphi$ from \Spar\ref{sec_braids_general} induces a map $\B_\lambda^{\ab} \rightarrow \B_\lambda(S)^{\ab}$. Then, from Proposition~\ref{partitioned_B^ab}, we get that:
\begin{itemize}
\item For each $i \leq l$ such that $n_i \geq 2$, all the $\sigma_\alpha$ with $\alpha$ and $\alpha + 1$ in the $i$-th block of $\lambda$ have a common image in  $\B_\lambda(S)^{\ab}$, called $s_i$.
\item For each $i < j \leq l$, all the $A_{\alpha \beta}$ with $\alpha$ (resp.~$\beta$) in the $i$-th (resp.~the $j$-th) block of $\lambda$ have a common image in $\B_\lambda(S)^{\ab}$, called $a_{ij}$ (or $a_{ji}$).
\end{itemize}

Let us now consider the short exact sequence:
\[\begin{tikzcd}
\PB_n^\circ(S) \ar[r, hook] 
&\B_\lambda(S) \ar[r, two heads, "\pi_S"] 
&\pi_1(S) \wr \Sym_\lambda.
\end{tikzcd}\] 
We can apply Lemma~\ref{abelianization_from_coinv} to it, and we get an exact sequence of abelian groups:
\[\begin{tikzcd}
(\PB_n^\circ(S)^{\ab})_{\B_\lambda(S)} \ar[r] 
&\B_\lambda(S)^{\ab} \ar[r] 
&(\pi_1(S) \wr \Sym_\lambda)^{\ab} \ar[r]
&0.
\end{tikzcd}\] 
On the one hand, the quotient $(\pi_1(S) \wr \Sym_\lambda)^{\ab}$ is isomorphic to the product of the $(\pi_1(S) \wr \Sym_{n_i})^{\ab}$, which is $(\pi_1(S)^{\ab})^l \times (\Z/2)^{l'}$, where $l'$ is the number of indices $i \leq l$ such that $n_i \geq 2$ (see Corollary~\ref{wr_ab_partitioned}). On the other hand, it follows from Proposition~\ref{generators_and_N(Pn)} that $\PB_n^\circ(S)$ is generated by $\PB_n$ under the action of $\PB_n(S)$, which is a subgroup of $\B_\lambda(S)$. As a consequence, the map $\PB_n^{\ab} \rightarrow (\PB_n^\circ(S)^{\ab})_{\B_\lambda(S)}$ induced by $\varphi$ is surjective. Moreover, it factors through the quotient $(\PB_n^{\ab})_{\B_\lambda} = (\PB_n^{\ab})_{\Sym_\lambda} \cong \Z^{l(l-1)/2} \times \Z^{l'}$. Thus the image of $(\PB_n^\circ(S)^{\ab})_{\B_\lambda(S)}$ in $\B_\lambda(S)^{\ab}$ is generated by the images of the pure braid generators, which are the elements $2 s_i$ for $i \leq l$ such that $n_i \geq 2$, and $a_{ij}$ for $i < j \leq l$.

Now, let $H$ be the subgroup of $\B_\lambda(S)^{\ab}$ generated by the $s_i$ and the $a_{ij}$. From the above, we get an isomorphism  $\B_\lambda(S)^{\ab}/H \cong (\pi_1(S)^{\ab})^l$. Moreover, the corresponding projection $\B_\lambda(S)^{\ab} \twoheadrightarrow (\pi_1(S)^{\ab})^l$ splits: a splitting is defined by sending $(g_i)_i$ to $\sum \psi_i(g_i)$ where, for each $i$, $\psi_i$ is induced by $\psi_\alpha$ for any $\alpha$ in the $i$-th block of $\lambda$. As a consequence, $\B_\lambda(S)^{\ab}$ identifies with $(\pi_1(S)^{\ab})^l \times H$, and we can use Proposition~\ref{trichotomy} to get a complete calculation:

\begin{proposition}\label{partitioned_B(S)^ab}
Let $n \geq 1$ and $\lambda = (n_1, \ldots , n_l)$ be a partition of $n$. Then:
\[\B_\lambda(S)^{\ab}\ \cong\ \left(\pi_1(S)^{\ab}\right)^l \times \left(\B_\lambda^{\ab}/R\right),\]
where $\B_\lambda^{\ab}$ is free abelian on the $s_i$ and the $a_{ij}$ from Proposition~\ref{partitioned_B^ab}, and $R$~is:
\begin{itemize}
\item trivial if $S$ is planar,
\item generated by the relations $2(n_i-1) s_i + \sum_{j \neq i} n_j a_{ij}$ ($i = 1, \ldots , n$) if~$S \cong \S^2$,
\item generated by the $2 s_i$ and the $a_{ij}$ in all the other cases.
\end{itemize}
Explicitly, if $l'$ denotes the number of indices $i \leq l$ such that $n_i \geq 2$, we have, for $S \ncong \S^2$:
\[\B_\lambda^{\ab}/R\ \cong\
\begin{cases}
\Z^{l'} \times \Z^{l(l-1)/2} &\text{ if $S$ is planar,} \\
(\Z/2)^{l'} &\text{ if $S$ does not embed into the sphere.}
\end{cases}\]
\end{proposition}

\begin{proof}
If $S \neq \S^2$, all the relations are direct consequences of Proposition~\ref{trichotomy}, and it remains to prove that they are the only ones. In the first case, we can directly use the result of Proposition~\ref{partitioned_B^ab}: an embedding of $S$ into $\R^2 \cong \S^2 - pt$ induces a retraction $\B_\lambda(S)^{\ab} \twoheadrightarrow \B_\lambda^{\ab}$ of the morphism $\B_\lambda^{\ab} \rightarrow \B_\lambda(S)^{\ab}$ induced by $\varphi$, whence the result in this case. In the third one, the same proof as the proof of Proposition~\ref{partitioned_B^ab} works: projections onto the factors are given by projections onto the $\Sym_{n_i}^{\ab} \cong \Z/2$ for $n_i \geq 2$.

If $S = \S^2$, Corollary~\ref{partitioned_Braids_on_closed_surfaces} describes $\B_\lambda(\S^2)$ as the quotient of $\B_\lambda$ by the relations $A_{\alpha 1} A_{\alpha 2} \cdots A_{\alpha n} = 1$ for all $1\leq \alpha\leq n$. The abelianisation $\B_\lambda(\S^2)^{\ab}$ is then the quotient of $\B_\lambda^{\ab}$ (described in Proposition~\ref{partitioned_B^ab}) by the classes of the above relations. If $\alpha$ is in the $i$-th block of $\lambda$, then the class of $A_{\alpha\beta}$ is either $2 s_i$ if $\beta$ is also in the $i$-th block (which holds for $n_i - 1$ values of $\beta$), or $a_{ij}$ if $\beta$ is in the $j$-th block for some $j \neq i$ (which happens for $n_j$ values of $\beta$, for each $j \neq i$). Thus the classes of the $A_{\alpha 1} A_{\alpha 2} \cdots A_{\alpha n}$ are indeed the relations of the statement.
\end{proof}

\begin{remark}
The relation $2(n_i-1) s_i + \sum_{j \neq i} n_j a_{ij}$ makes sense even if $n_i = 1$: in this case, with our conventions, there is no  generator $s_i$, but this does not matter since its coefficient is $2(n_i - 1) = 0$ (alternatively, one could add a generator $s_i$ and ask that $s_i = 0$).
\end{remark}

\subsection{The lower central series}\label{sec:LCS_stable_partitioned_surface_braids}

The work done above to show that the LCS of $\B_n(S)$ stops if $n \geq 3$ (see Corollary~\ref{LCS_Bn(S)_stable}) generalises to the partitioned braid group when all the blocks of the partition are of size at least $3$. First, here follows the generalisation of Lemma~\ref{sigma_central}.
\begin{lemma}\label{central_image}
Let $\lambda = (n_1, \ldots , n_l)$ be a partition of $n$, with $n_i \geq 3$ for all $i$. The image of $\B_\lambda$ in $\B_\lambda(S)/\LCS_\infty$ is central. In particular, it is a quotient of $\B_\lambda^{\ab}$.
\end{lemma}

\begin{proof}
It follows from Theorem~\ref{thm:partitioned-braids} that $\LCS_\infty(\B_\lambda) = \LCS_2(\B_\lambda)$. As a consequence, the morphism $\varphi \colon \B_\lambda \rightarrow \B_\lambda(S)/\LCS_\infty$ factors through $\B_\lambda/\LCS_2 = \B_\lambda^{\ab}$. Hence its image is abelian, generated by the images of the $s_i$ and the $a_{ij}$ from Proposition~\ref{partitioned_B^ab}. 

In order to show that it is central, we need to show that these elements commute with generators of $\B_\lambda(S)/\LCS_\infty$. We deduce from Proposition~\ref{generators_and_N(Pn)} that $\B_\lambda(S)$ is generated by the images of $\varphi$ and of the $\psi_\alpha$. In fact, we can restrict to taking one $\alpha$ in each block of $\lambda$, since $\psi_\alpha$ and $\psi_\beta$ are conjugated by elements of $\B_\lambda$ if $\alpha$ and $\beta$ are in the same block. The $s_i$ and the $a_{ij}$ already commute with each other, so we only need to show that they commute with the images of the selected $\psi_\alpha$. Since the blocks of~$\lambda$ have size at least~$3$, we can find representatives of the $s_i$ and the $a_{ij}$ having disjoint support with elements of all the chosen $\ima(\psi_\alpha)$. Thus, the $s_i$ and the $a_{ij}$ commute with a family of generators of $\B_\lambda(S)/\LCS_\infty$, which proves our claim.
\end{proof}

We can now generalise our main decomposition  theorem (Theorem~\ref{Bn(S)/residue}) to partitioned braids:
\begin{theorem}\label{B_lambda(S)/residue}
Let $\lambda = (n_1, \ldots , n_l)$ be a partition of $n$, with $n_i \geq 3$ for all $i$. There is a central extension:
\[\begin{tikzcd}
\langle s_i^2, a_{ij} \rangle_{i,j \leq l} \ar[r, hook] 
&\B_\lambda(S)/\LCS_\infty \ar[r, two heads] 
&(\pi_1(S)^{\ab} \times \Z/2)^l.
\end{tikzcd}\] 
\end{theorem}

\begin{proof}
The proof is essentially the same as the proof of Theorem~\ref{Bn(S)/residue}, so we only stress what changes. The element $\sigma_\alpha \sigma_\beta^{-1}$ is in $\LCS_\infty(\B_n(S))$ only when $\alpha$ and $\beta$ are in the same block of $\lambda$. However, their images $\tau_\alpha \tau_\beta^{-1}$ in $\pi_1(S) \wr \Sym_\lambda$ still normally generate $\LCS_2(\pi_1(S) \wr \Sym_\lambda)$, because $\pi_1(S) \wr \Sym_\lambda$ is the product of the $\pi_1(S) \wr \Sym_{n_i}$, whose $\LCS_2$ is normally generated by the $\tau_\alpha \tau_\beta^{-1}$ for $\alpha$ and $\beta$ in the $i$-th block of~$\lambda$ (Lemma~\ref{LCS_wr_sym_stable}). Thus, arguing exactly as in the proof of Theorem~\ref{Bn(S)/residue}, we get a short exact sequence:
\[\begin{tikzcd}
\PB_n^\circ(S)/\LCS_\infty(\B_\lambda(S)) \cap \PB_n^\circ(S) \ar[r, hook] 
&\B_\lambda(S)/\LCS_\infty \ar[r, two heads, "\pi_S"] 
&(\pi_1(S) \wr \Sym_\lambda)^{\ab}.
\end{tikzcd}\] 
Recall that $(\pi_1(S) \wr \Sym_\lambda)^{\ab} \cong (\pi_1(S)^{\ab} \times \Z/2)^l$. Moreover, the kernel is the normal closure of the image of $\PB_n$, but this image is already normal, and even central by Lemma~\ref{central_image}, generated by the $a_{ij}$ and the squares of the $s_i$.
\end{proof}

Finally, Corollary~\ref{LCS_Bn(S)_stable} adapts readily to this context:
\begin{corollary}\label{LCS_B_lambda(S)_stable}
If all the blocks of $\lambda$ have size at least $3$, then:
\[\LCS_3(\B_\lambda(S)) = \LCS_4(\B_\lambda(S)).\]
\end{corollary}

\subsection{The Lie ring}\label{sec:Lie_ring_partitioned_braid}

Here again, in the stable case (that is, when all the blocks of the partition $\lambda$ have size at least $3$), we can be more precise about the description of the LCS. Namely, our next goal is to show the following generalisation of Theorem~\ref{Lie_ring_B(S)}, whose proof occupies the rest of the present section. 

\begin{theorem}\label{Lie_ring_partitioned_B(S)}
Let $\lambda = (n_1, \ldots , n_l)$ be a partition of an integer $n$ with $n_i \geq 3$ for all $i$, and let $S$ be a connected surface. The LCS of $\B_\lambda(S)$:
\begin{itemize}
\item \emph{stops at $\LCS_2$} if $S$ is planar, if $S \cong \S^2$ or if $l = 1$ and $S$ is non-orientable. 
\item \emph{stops at $\LCS_3$} in the other cases. 
\end{itemize}
\end{theorem}

\begin{remark}
In both cases, the Lie ring can be computed completely. Namely, $\Lie_1(\B_\lambda(S)) = \B_\lambda(S)^{\ab}$ has already been computed in Proposition~\ref{partitioned_B(S)^ab}. In the first case, no further computation is required. In the second case, $\Lie_2(\B_\lambda(S))$ is described completely in Proposition~\ref{partitioned_Lie(B(Sg1)} and Proposition~\ref{partitioned_Lie(B(Sg)} for orientable surfaces, and in Corollary~\ref{partitioned_Lie(B(Ng1))} and Corollary~\ref{partitioned_Lie(B(Ng))} for non-orientable ones. Moreover, one can easily describe the Lie bracket from the computations given there. Precisely, the only non-trivial brackets come from the computations depicted in Figure~\ref{Aij_as_bracket}. 
\end{remark}

The proof of Theorem~\ref{Lie_ring_partitioned_B(S)} begins with the following observation, which will be of essence in our study of the Lie ring of $\B_\lambda(S)$:
\begin{fact}\label{Lie2(B(S))}
Let $\lambda = (n_1, \ldots , n_l)$ be a partition of $n$ with all blocks of size at least $3$. Then $\Lie_2(\B_\lambda(S))$ identifies with $\LCS_2 \left(\B_\lambda(S)/\LCS_\infty \right)$, which is included in the subgroup $\langle s_i^2, a_{ij} \rangle_{i,j \leq l}$ of $\B_\lambda(S)$ from Theorem~\ref{B_lambda(S)/residue}. Furthermore, if $S$ has a handle or a crosscap, then $\Lie_2(\B_\lambda(S)) \cong \langle s_i^2, a_{ij} \rangle_{i,j \leq l}$.
\end{fact} 

\begin{proof}
 We know that $\LCS_3 = \LCS_\infty$ for $\B_\lambda(S)$ by Corollary~\ref{LCS_B_lambda(S)_stable}. Then, since $\Lie_2 = \LCS_2/\LCS_3$, $\Lie_2(\B_\lambda(S))$ identifies with $\LCS_2 \left(\B_\lambda(S)/\LCS_\infty \right)$. Moreover, in $\B_\lambda(S)/\LCS_\infty$, the subgroup $\LCS_2$ must be contained in the kernel $\langle s_i^2, a_{ij} \rangle_{i,j \leq l}$ of the central extension of Theorem~\ref{B_lambda(S)/residue}. Now, if $S$ has a handle or a crosscap, then the quotient $(\pi_1(S)^{\ab} \times \Z/2)^l$ of $\B_\lambda(S)/\LCS_\infty$ identifies with its abelianisation (see Proposition~\ref{partitioned_B(S)^ab}). As a consequence, the subgroup $\LCS_2 \left(\B_\lambda(S)/\LCS_\infty \right)$ is the whole of $\langle s_i^2, a_{ij} \rangle_{i,j \leq l}$.
\end{proof}

\subsubsection{Non-planar orientable surfaces.} Let us turn our attention to the case where the surface is not closed.

\begin{proposition}\label{partitioned_Lie(B(Sg1)}
If $S$ is a non-planar orientable surface which is not closed, then $\Lie_2(\B_\lambda(S))$ is free abelian on the $a_{ij}$ and the $s_i^2$.
\end{proposition}

\begin{proof}
We already know that these elements generate $\Lie_2(\B_\lambda(S))$ by Fact~\ref{Lie2(B(S))}, so we need to show that they are linearly independent. The argument from the proof of Proposition~\ref{partitioned_B^ab} works here, using the maps from $\Lie_2(\B_\lambda(S))$ to $\Lie_2(\B_{n_i}(S)) \cong \Z$ and $\Lie_2(\B_{n_i + n_j}(S)) \cong \Z$ instead of the maps from $\B_\lambda^{\ab}$ to $\B_{n_i}^{\ab}$ and $\B_{n_i + n_j}^{\ab}$. This uses the fact that $\Lie_2(\B_n(S)) \cong \Z$ if $n \geq 3$, from Theorem~\ref{Lie_ring_B(S)}.
\end{proof}

The case of closed orientable surfaces is a bit trickier. In fact, we first need to generalise \cite[Th.~1]{BellingeriGervaisGuaschi}:

\begin{proposition}\label{partitioned_Lie(B(Sg)}
Let $g \geq 1$. Then $\Lie_2(\B_\lambda(\Sigma_g))$ is the quotient of the free abelian group on the $a_{ij}$ and the $s_i^2$ by the relations:
\[(n_i + g -1)s_i^2 + \sum\limits_{j \neq i} n_j a_{ij} = 0, \ \ \ \ \text{for all}\  1\leq i\leq n.\]
\end{proposition}

\begin{proof}
Let us consider the subsurface $\Sigma_{g,1}$ of $\Sigma_g$ obtained by removing an open disc.  The corresponding embedding of $\Sigma_{g,1}$ into $\Sigma_g$ induces the quotient map described in Corollary~\ref{partitioned_Braids_on_closed_surfaces}. In particular, $\B_\lambda(\Sigma_g)/\LCS_3$ is the quotient of $\B_\lambda(\Sigma_{g,1})/\LCS_3$ by the classes of the relations from Corollary~\ref{partitioned_Braids_on_closed_surfaces}. Since the $A_{\alpha\beta}$ are in $\LCS_2(\B_\lambda(\Sigma_{g,1}))$, these relations are between elements of $\LCS_2$, that is, elements of $\langle s_i^2, a_{ij} \rangle_{i,j \leq l}$. In order to write them as relations between the $s_i^2$ and the $a_{ij}$, we need to recall that if $\alpha$ is in the $i$-th block of $\lambda$, for any $r \leq g$, we have, in $\Lie_2(\B_\lambda(S))$:
\[\overline{\left[a_r^{(\alpha)}, (b_r^{(\alpha)})^{-1}\right]} =
- \bigg[\overline{a_r^{(\alpha)}}, \overline{b_r^{(\alpha)}}\bigg] = 
- \overline{\left[\sigma_\alpha b_r^{(\alpha)}\sigma_\alpha ^{-1}, (a_r^{(\alpha)})^{-1}\right]} =
- \overline{\sigma_\alpha^2} =
-s_i^2,
\]
where the third equality comes from Figure~\ref{Aij_as_bracket_handle}.
Moreover, the class of $A_{\alpha\beta}$ modulo $\LCS_3(\B_\lambda(S))$ is either $s_i^2$ if $\alpha$ and $\beta$ are both in the $i$-th block of $\lambda$, or $a_{ij}$ if $\alpha$ is in the $i$-th block and $\beta$ is in the $j$-th block for some $j \neq i$. Finally, using additive notations in the central subgroup $\langle s_i^2, a_{ij} \rangle$,  the classes of the relations from Corollary~\ref{partitioned_Braids_on_closed_surfaces} are:
\[(n_i-1)s_i^2 + \sum\limits_{j \neq i} n_j a_{ij} = - g s_i^2 \ \ \ \ \text{for all}\  1\leq i\leq n.\]
Moreover, these relators are central in $\B_\lambda(\Sigma_g)/\LCS_\infty$, so their normal closure is only the subgroup they generate, which implies that $\LCS_2(\B_\lambda(\Sigma_g)/\LCS_3)$ is the quotient of $\LCS_2(\B_\lambda(\Sigma_{g,1})/\LCS_3)$ by these relations. By Proposition~\ref{partitioned_Lie(B(Sg1)}, the latter is free on the $s_i^2$ and the $a_{ij}$. Whence our claim.
\end{proof}

\begin{remark}
Let us point out that this computation of $\Lie_2(\B_\lambda(\Sigma_g))$ is very similar to the computation of $\Lie_1(\B_\lambda(\S^2))$ in the proof of Proposition~\ref{partitioned_B(S)^ab} (note that $s_i^2$ was equal to $2 s_i$ there, but here $s_i$ and $s_i^2$ do not live in the same part of the Lie ring). The only difference lies in the degree of our relations with respect to the LCS: if the surface has a handle, then the pure braid generators $A_{\alpha \beta}$ belong to the derived subgroup (see Figure~\ref{Aij_as_bracket_handle}).
\end{remark}

\subsubsection{Non-orientable surfaces.} For non-orientable surfaces, we already know that the $s_i^2$ vanish in the quotient by $\LCS_\infty$, as in the proof of Theorem~\ref{Lie_ring_B(S)}. As a consequence, $\LCS_2 \left(\B_n(S)/\LCS_\infty \right)$ is generated by the $a_{ij}$.
The following proposition is an analogue of \cite[Prop.~3.7]{BellingeriGodelleGuaschi} in the non-orientable case:
\begin{proposition}\label{Lie(Bmn(Ng1))}
Let $g \geq 0$ and $m,n \geq 3$. We have:
\[\B_{m,n}(\mathscr{N}_{g,1})/\LCS_\infty \cong (\Z/2)^2 \times (\Z \times \Z^g) \rtimes \Z^g,\]
where the factors are respectively generated by $s_1$ and $s_2$, $a_{12}$, $c_r'$ and $c_r$. The action is given by $c_r c_r' c_r^{-1} = c_r' a_{12}$ (for all $r \leq g$) and all the other pairs of generators commuting. In particular, this group is $2$-nilpotent, and $\Lie_2(\B_{m,n}(\mathscr{N}_{g,1}))$ is infinite cyclic, generated by $a_{12}$.
\end{proposition}

\begin{proof}
Recall that we have a split extension (from Proposition~\ref{Fadell-Neuwirth}):
\[\B_m(\mathscr{N}_{g,1+n}) \hookrightarrow \B_{m,n}(\mathscr{N}_{g,1}) \twoheadrightarrow \B_n(\mathscr{N}_{g,1}).\]
Thus, we can get a presentation of $\B_{m,n}(\mathscr{N}_{g,1})$ from the presentation of the quotient described in Proposition~\ref{Bn(Ng1)} and the presentation of the kernel from Proposition~\ref{Bm(Ngn)}, using the method of Appendix~\ref{section_pstation_of_ext}. The generators of this presentation are the $\sigma_i$ for $i \neq m$, the $c_r := c_r^{(1)}$, the $c_r' :=  c_r^{(m+1)}$ and the $z_j = A_{1, m+j}$ (where our conventions are those from \Spar\ref{sec_pstations_B(S)}).

We use this to get a presentation of the quotient of $\B_{m,n}(\mathscr{N}_{g,1})$ by the normal closure $N$ of the $\sigma_i\sigma_{i+1}^{-1}$ for $i < m+n$ and $i \notin \{m-1,m \}$, together with the $z_j z_{j+1}^{-1}$ for $j < n$. This quotient is generated by the common class $s_1$ (resp.~$s_2$, resp.~$a_{12}$) of $\sigma_1, \ldots , \sigma_{m-1}$, the common class $s_2$ of $\sigma_{m+1}, \ldots , \sigma_{m-1}$ and the common class $a_{12}$ of $z_1, \ldots , z_n$, together with the $c_r$ and the $c_r'$ for $r \leq g$. They are subject to the following relations:
\begin{itemize}
\item Relations coming from those of $\B_n(\mathscr{N}_{g,1})$: $s_2$ commutes with the $c_r'$ $(BN1)$, the $c_r'$ commute with one another $(BN2)$ and $s_2^2 = 1$ $(BN3)$.
\item Relations coming from those of $\B_m(\mathscr{N}_{g,1+n})$: $s_1$ commutes with the $c_r$ $(BN1)$, the $c_r$ commute with one another $(BN2)$, $s_1^2 = 1$ $(BN3)$, $a_{12}$ commutes with $s_1$ $(BN4)$, and with the $c_r$ $(BN5)$. $(BN6)$ and $(BN7)$ become trivial.  
\item Relations describing the action by conjugation of $s_2$ and the $c_r'$ on $s_1$, $a_{12}$ and the $c_r$. This action is easily seen to be trivial in most cases, since most of the pairs of elements involved come from elements having disjoint support in $\B_{m,n}(\mathscr{N}_{g,1})$, hence they commute. Namely, this holds for $c_s$ and $c_r'$ when $r < s$: using the fact that $a_{12}$ commutes with $c_s$, we see that $c_s$ is the class of $z_1 c_r^{(1)} z_1^{-1}$, whose support is disjoint from the support of $c_r^{(m+1)}$ (up to isotopy). This is again true for $a_{12}$ and the $c_r'$, since $s_2$ commutes with the $c_r'$: $c_r'$ is the class of $c_r^{(m+2)} = \sigma_{m+1} c_r^{(m+1)} \sigma_{m+1}^{-1}$, whose support is disjoint from that of $z_1$. Finally, the only pair of generators under scrutiny for which this does not hold are $c_r$ and $c_r'$ (for $r \leq g$). But for them, the situation is the one from Figure~\ref{Aij_as_bracket_crosscap}: 
$[(c_r')^{-1}, c_r] = \overline A_{1, m+1} = a_{12}.$
\end{itemize} 
These relations can be summed up as:
\begin{itemize}
\item All the generators commute pairwise, except $c_r$ and $c_r'$ for $r \leq g$,
\item $s_1^2 = s_2^2 = 1$,
\item $c_r c_r' c_r^{-1} = c_r' a_{12}$ (for $r = 1, \ldots , g$).
\end{itemize} 
This is a presentation of the group described in the statement. Its commutator subgroup is the factor $\Z$ of the decomposition, which is infinite cyclic, generated by $a_{12}$. It is also central, hence the group is $2$-nilpotent. Since this group is $\B_{m,n}(\mathscr{N}_{g,1})/N$, $N$ contains $\LCS_3(\B_{m,n}(\mathscr{N}_{g,1}))$. But the elements normally generating $N$ are in $\LCS_\infty(\B_{m,n}(\mathscr{N}_{g,1}))$; see Lemma~\ref{central_image} and its proof. Thus $N = \LCS_\infty(\B_{m,n}(\mathscr{N}_{g,1}))$, and the Lie ring of the quotient $\B_{m,n}(\mathscr{N}_{g,1})/N$ identifies with the Lie ring of $\B_{m,n}(\mathscr{N}_{g,1})$. In particular, since its $\LCS_3$ is trivial, its $\Lie_2$ coincides with its $\LCS_2$, which is infinite cyclic, generated by $a_{12}$.
\end{proof}

\begin{corollary}\label{partitioned_Lie(B(Ng1))}
Let $S$ be a non-orientable surface that is not closed. Let $\lambda = (n_1, \ldots , n_l)$ be a partition whose blocks are of size at least $3$. Then  $\Lie_2(\B_\lambda(S))$ is free abelian on the $a_{ij}$ for $1 \leq i < j \leq l$.
\end{corollary}

\begin{proof}
We first show that if $m,n \geq 3$, $\Lie_2(\B_{m,n}(S))$ is infinite cyclic, generated by $a_{12}$. This is true for $S = \mathscr N_{g,1}$ by Proposition~\ref{Lie(Bmn(Ng1))}. Then we can follow the same method as in the proof of Theorem~\ref{Lie_ring_B(S)}, to which the reader is referred for more details. Namely, we already know that $\Lie_2(\B_{m,n}(S))$ is generated by $a_{12}$, and we need to show that it has infinite order. If $S$ is compact, we can embed $S$ into some $\mathscr N_{g,1}$, and the image of $a_{12}$ by the corresponding morphism from $\Lie_2(\B_{m,n}(S))$ to $\Lie_2(\B_{m,n}(\mathscr N_{g,1}))$ has infinite order, whence our claim in this case. Then, if $S$ is not compact, a relation saying that $a_{12}$ has finite order in $\B_{m,n}(S)/ \LCS_3$ would hold in a compact subsurface, which is impossible by the previous case. 

From this, we can deduce the result for $\lambda$ having more than two blocks, reasoning as in the proofs of Propositions~\ref{partitioned_B^ab} and~\ref{partitioned_Lie(B(Sg1)}. Indeed, each canonical map $\Lie_2(\B_\lambda(S)) \rightarrow \Lie_2(\B_{n_i, n_j}(S)) \cong \Z$ kills all the $a_{kl}$, except $a_{ij}$, which is sent to a generator of the target. Thus the $a_{ij}$ must be linearly independent. 
\end{proof}

\begin{proposition}\label{Lie(Bmn(Ng))}
Let $g \geq 0$ and $m,n \geq 3$. Then $\Lie_2(\B_{m,n}(\mathscr N_g))$ is cyclic of order $2$, generated by~$a_{12}$.
\end{proposition}

\begin{proof}
Let us consider $G := \B_{m,n}(\mathscr N_g)/\LCS_3$. This group is the quotient of $\B_{m,n}(\mathscr N_{g,1})/\LCS_3 = \B_{m,n}(\mathscr N_g)/\LCS_\infty \cong (\Z/2)^2 \times (\Z \times \Z^g) \rtimes \Z^g$ described in Proposition~\ref{Lie(Bmn(Ng1))} by the images of the boundary relations from Corollary~\ref{partitioned_Braids_on_closed_surfaces}. Before considering these relations, we can already remark that the commutator subgroup of $\B_{m,n}(\mathscr N_{g,1})/\LCS_3$ is cyclic generated by $a_{12}$, so the same holds for $G$. Since $\LCS_2(G)$ identifies with $\Lie_2(\B_{m,n}(\mathscr N_g))$, we only need to show that $a_{12}$ has order $2$ in $G$ to prove our statement.

Recall that $A_{\alpha \beta}$ is sent to $s_1^2 = s_2^2 = 1$ if $\alpha$ and $\beta$ are in the same block of the partition $(m,n)$, and to $a_{12}$ if they are not. As a consequence, these relations are:
\[c_1^2 \cdots c_g^2 = a_{12}^m\ \ \ \text{and}\ \ \ (c_1')^2 \cdots (c_g')^2 = a_{12}^n.\]
Since $c_1$ commutes with all the other generators except $c_1'$, and $[c_1, c_1'] = a_{12}$, by applying the commutator with $c_1$ to the second relation, we get $a_{12}^2 = 1$. We can thus consider the intermediate quotient by the central element $a_{12}^2$, and see $G$ as a quotient of $(\Z/2)^2 \times ((\Z/2) \times \Z^g) \rtimes \Z^g$. Modulo $a_{12}^2$, the above relations (hence the structure of $G = \B_{m,n}(\mathscr N_g)/\LCS_3$) depend only on the parity of $m$ and $n$. Moreover, we now know that $a_{12}^2 = 1$ in $G$, so we are left with showing that $a_{12}$ is not trivial in $G$.

Suppose first that $m$ and $n$ are both even. Then the relations become $c_1^2 \cdots c_g^2 = (c_1')^2 \cdots (c_g')^2 = 1$. Notice that after the quotient by $a_{12}^2$, all the $c_r^2$ and the $(c_r')^2$ have become central, so both $c_1^2 \cdots c_g^2$ and $(c_1')^2 \cdots (c_g')^2$ are central in $(\Z/2)^2 \times ((\Z/2) \times \Z^g) \rtimes \Z^g$. Thus, if $A$ denotes the abelian group $\Z^g/(2,2, \ldots , 2)$ (which is also $\pi_1(\mathscr N_g)^{\ab}$), we see that $G \cong (\Z/2)^2 \times ((\Z/2) \times A) \rtimes A$, whose commutator subgroup is cyclic of order $2$, generated by $a_{12}$. 

If $m$ and $n$ have different parities, we can assume (by symmetry) that $m$ is even and $n$ is odd. Then the relations become $c_1^2 \cdots c_g^2 = 1$ and $(c_1')^2 \cdots (c_g')^2 = a_{12}$. In $(\Z/2)^2 \times ((\Z/2) \times \Z^g) \rtimes \Z^g$, both $c_1^2 \cdots c_g^2$ and $(c_1')^2 \cdots (c_g')^2 a_{12}^{-1}$ are central, so they general cyclic normal subgroups. Thus, if we denote by $\tilde A$ the quotient $\Z^g/(4,4, \ldots , 4)$ of $(\Z/2) \times \Z^g$ by $(1, 2, \ldots , 2)$, we have $G \cong (\Z/2)^2 \times \tilde A \rtimes A$, whose commutator subgroup is cyclic of order $2$, generated by $a_{12}$ (the latter identifies with the class of $(2,2, \ldots , 2)$ in $\tilde A$). 

If $m$ and $n$ are both odd, then the relations become $c_1^2 \cdots c_g^2 = (c_1')^2 \cdots (c_g')^2 = a_{12}$. In this case, there is no obvious semi-direct product decomposition of $G$ where $a_{12}$ is clearly non-trivial, so we need another argument to show that $a_{12} \neq 1$. If $g = 1$, one can see that $G \cong (\Z/2)^2 \times Q_8$, where $c_1$ is sent to $i$, $c_1'$ is sent to $j$, and $a_{12}$ identifies with the central element $-1$ of the quaternion group $Q_8$. Indeed, one can easily check that the above correspondence defines a morphism from $G$ to $(\Z/2)^2 \times Q_8$, and use the presentation $Q_8 = \langle i,j \ |\ i^2 = i^2 = (ij)^2 \rangle$ to construct its inverse.
For $g \geq 1$, we can find a similar quotient of $G$, by considering the quotient $(Q_8)^g/H$, where $H$ is the hyperplane of $\mathcal Z((Q_8)^g) \cong (\Z/2)^g$ defined by the vanishing of the sum of the coordinates. Since $H$ is central in  $(Q_8)^g$, it is normal. Moreover, $(Q_8)^g$ decomposes as a central extension of $(\Z/2)^{2g}$ by $(\Z/2)^g$, which induces a central extension:
\[\begin{tikzcd}
(\Z/2)^g/H \cong \Z/2 \ar[r, hook] &(Q_8)^g/H \ar[r, two heads] &(\Z/2)^{2g}.
\end{tikzcd}\]
Using the presentation of $G$, we can see that there is a well-defined projection from $G$ onto $(Q_8)^g/H$ sending $c_r$ to $(1, \ldots , 1, i, 1, \ldots , 1)$, $c_r'$ to $(1, \ldots , 1, j, 1, \ldots , 1)$ (where the non-trivial coordinate is the $r$-th one in both cases), and $s_1$ and $s_2$ to $1$. This projection sends $a_{12} = [c_1, c_1']$ to the generator of the centre $(\Z/2)^g/H \cong \Z/2$. Thus $a_{12}$ is again not trivial in $G$, whence our claim.
\end{proof}

\begin{corollary}\label{partitioned_Lie(B(Ng))}
Let $g \geq 0$ and let $\lambda = (n_1, \ldots , n_l)$ be a partition whose blocks are of size at least $3$. Then $\Lie_2(\B_\lambda(\mathscr N_g)) \cong (\Z/2)^{l(l-1)/2}$ is the free elementary abelian $2$-group on the $a_{ij}$ for $1 \leq i < j \leq l$.
\end{corollary}

\begin{proof}
The proof is again the same as the proofs of Propositions~\ref{partitioned_B^ab} and~\ref{partitioned_Lie(B(Sg1)}, using the canonical maps $\Lie_2(\B_\lambda(\mathscr N_g)) \rightarrow \Lie_2(\B_{n_i, n_j}(\mathscr N_g)) \cong \Z/2$, where the latter is generated by the image of $a_{ij}$, by Proposition~\ref{Lie(Bmn(Ng))}.
\end{proof}

We can now finish the proof of the main result of this section.
\begin{proof}[Proof of Theorem~\ref{Lie_ring_partitioned_B(S)}]
Under our hypotheses, $\Lie_2(\B_\lambda(S))$ identifies with $\LCS_2 \left(\B_\lambda(S)/\LCS_\infty \right) \subseteq \langle s_i^2, a_{ij} \rangle_{i,j \leq l}$ (see Fact~\ref{Lie2(B(S))}).

\ul{Planar surfaces:} if $S$ is planar, then the canonical projection from $\B_\lambda(S)/\LCS_\infty$ to $\B_\lambda(S)^{\ab}$ sends the $s_i$ and the $a_{ij}$ to their counterparts in $\B_\lambda(S)^{\ab}$. The latter is a linearly independent family by Proposition~\ref{partitioned_B(S)^ab}, so the restriction  $\langle s_i^2, a_{ij} \rangle_{i,j \leq l} \rightarrow \B_\lambda(S)^{\ab}$ must be injective. But its kernel is $\LCS_2(\B_\lambda(S) / \LCS_\infty) = \Lie_2(\B_\lambda(S))$, which must then be trivial. 

\ul{The sphere:} if $S = \S^2$, then $\B_\lambda(\S^2)$ is a quotient of $\B_\lambda$ (see Corollary~\ref{partitioned_Braids_on_closed_surfaces}), which implies that its LCS also stops at $\LCS_2$.

\ul{Non-planar orientable surfaces:} Proposition~\ref{partitioned_Lie(B(Sg1)} deals with most of them; the remaining ones are closed orientable surfaces, for which Proposition~\ref{partitioned_Lie(B(Sg)} is the relevant statement.

\ul{Non-orientable surfaces:} these are dealt with in Corollary~\ref{partitioned_Lie(B(Ng1))}, except for the closed ones, whose Lie ring is studied separately, in Corollary~\ref{partitioned_Lie(B(Ng))}.
\end{proof}

\section{Partitions with small blocks}\label{sec:surface_braid_unstable}

Now that we have a complete description of the LCS of $\B_\lambda(S)$ in the stable case, that is, when the blocks of the partition $\lambda$ have size at least $3$ (\Spar\ref{sec:partitioned_braids_surfaces_stable}), we turn our attention to the cases where $\lambda$ does have blocks of size $1$ or $2$. Then we ask ourselves: under this assumption, when does the LCS of $\B_\lambda(S)$ stop? For most surfaces, it does not; see Proposition~\ref{LCS_unstable_partitioned_surface_braids}. In fact, there are only six surfaces to which this result does not apply. One of them is the disc, for which an answer has already been given in Chapter~\ref{sec:partitioned_braids}. Another one is the cylinder, whose case can easily be deduced from the case of the disc. Four surfaces remain: the torus $\Tor$ (\Spar\ref{sec:partitioned_braids_torus}), the  Möbius strip $\Moeb$ (\Spar\ref{sec:partitioned_braids_Moebius}), the sphere $\S^2$ (\Spar\ref{sec:partitioned_braids_sphere}) and the projective plane $\Proj$ (\Spar\ref{sec:partitioned_braids_projective_plane}).

\subsection{The generic cases}

As a direct corollary of Propositions~\ref{prop:LCS_non_stop_pi_1_surfaces} and \ref{prop:LCS_non_stop_unpartitioned_braid_groups_surfaces}, we get the following result:
\begin{proposition}\label{LCS_unstable_partitioned_surface_braids}
Let $\lambda = (n_1, \ldots , n_l)$ be a partition of an integer $n \geq 1$.
\begin{itemize}
\item If $\lambda$ has at least one block of size $1$ and $\pi_1(S)$ is not abelian (that is, if we suppose that $S \notin \{\D-pt, \D, \S^2, \Tor, \Proj, \Moeb\}$ up to isotopy equivalence), then the LCS of $\B_{\lambda}(S)$ does not stop.
\item If $\lambda$ has at least one block of size $2$ and $\pi_1(S)^{\ab}$ is not finite (that is, if we suppose that $S \notin \{\D, \S^2, \Proj\}$ up to isotopy equivalence), then the LCS of $\B_{\lambda}(S)$ does not stop.
\end{itemize}
\end{proposition}

\begin{proof}
In the first case, there is a surjection $\B_\lambda(S) \twoheadrightarrow \B_1(S) \cong \pi_1(S)$, and in the second one, a surjection $\B_\lambda(S) \twoheadrightarrow \B_2(S)$. Propositions~\ref{prop:LCS_non_stop_pi_1_surfaces} and \ref{prop:LCS_non_stop_unpartitioned_braid_groups_surfaces} say that, under our hypotheses, the LCS of the quotient does not stop in either case. The result then follows from Lemma~\ref{lem:stationary_quotient}.
\end{proof}

Thus the question of whether the LCS of $\B_\lambda(S)$ stops has been answered for every partition, except for the six surfaces $\D- pt$, $\D$, $\S^2$, $\Tor$, $\Proj$ and $\Moeb$. In fact, $\B_\lambda(\D) = \B_\lambda$ has already been considered in Chapter~\ref{sec:partitioned_braids}; see Theorem~\ref{thm:partitioned-braids}. Also, since $\B_1(\D) = \{1\}$, we have an isomorphism $\B_\lambda(\D - pt) \cong \B_{1,\lambda}(\D)$ for every partition $\lambda$ by Proposition~\ref{Fadell-Neuwirth}, so we can deduce the remaining answer for $\D - pt$ from the answer for the disc. Namely, Lemma~\ref{LCS_B111} and Proposition~\ref{LCS_B11mu} imply:

\begin{lemma}\label{LCS_B1nu(D-pt)}
If $\lambda$ has at least two blocks of size $1$, then the LCS of $\B_\lambda(\D - pt)$ does not stop. If $\lambda = (1,n_2,\ldots,n_l)$ where every $n_i \geq 3$, then its LCS stops at $\LCS_2$.
\end{lemma}

Therefore, we are left with four remaining cases: the torus, the Möbius strip, the sphere and the projective plane.

\subsection{Partitioned braids on the torus}\label{sec:partitioned_braids_torus}

We know that the LCS of $\B_\lambda(\Tor)$ stops if $\lambda$ has only blocks of size at least $3$, and that it does not if there is at least one block of size $2$. The remaining cases are dealt with using the following generalisation of \cite[Lem.~17]{BellingeriGervaisGuaschi}:

\begin{proposition}\label{B_1mu(T)}
There is an isomorphism $\B_{1, \mu}(\Tor) \cong \B_\mu(\Tor - pt) \times \Z^2$, for any partition $\mu$.
\end{proposition}

\begin{proof}
Proposition~\ref{Fadell-Neuwirth} gives a short exact sequence:
\[\begin{tikzcd}
\B_\mu(\Tor - pt) \ar[r, hook]
&\B_{1, \mu}(\Tor) \ar[r, two heads] 
&\pi_1(\Tor) \cong \Z^2. 
\end{tikzcd}\]
Let $n\geq 2$ such that $\mu$ is a partition of $n-1$. The centre of $\B_n(\Tor)$ is generated by two braids $\alpha$ and $\beta$ corresponding to rotating all the punctures along each factor of $\S^1 \times \S^1$ \cite[Prop.~4.2]{Paris-Rolfsen}. These are pure braids, so they are in the subgroup $\B_{1, \mu}(\Tor)$ of $\B_n(S)$.  The above projection (forgetting all strands but one) maps these two elements to a basis of $\pi_1(\Tor) \cong \Z^2$. As a consequence, it restricts to an isomorphism $\langle \alpha, \beta \rangle \cong \Z^2$. Thus the above short exact sequence splits, and the corresponding action of $\Z^2$ is trivial ($\alpha$ and $\beta$ being central), hence it is in fact a direct product.
\end{proof}

It is then an easy task to finish proving the following summary of our results for partitioned torus braids:
\begin{theorem}\label{thm:braid_torus}
Let $\lambda$ be a partition of an integer $n \geq 1$. The LCS of $\B_\lambda(\Tor)$:
\begin{itemize}
\item \emph{does not stop} if $\lambda$ has at least two blocks of size $1$ or one block of size $2$. 
\item \emph{stops} at $\LCS_3$ in all the other cases, except for $\B_1(\Tor) \cong \Z^2$.
\end{itemize}
\end{theorem}

\begin{proof}
If $\lambda$ has \ul{at least two blocks of size $1$}, then $\B_\lambda(\Tor)$ surjects onto $\B_{1,1}(\Tor)$, which is isomorphic to $\B_1(\Tor - pt) \times \Z^2 \cong\F_2 \times \Z^2$ by Proposition~\ref{B_1mu(T)}. The LCS of $\F_2$ does not stop, whence the result in this case, by Lemma~\ref{lem:stationary_quotient}. 

If $\lambda$ has \ul{exactly one block of size $1$ and no block of size $2$}, then Proposition~\ref{B_1mu(T)} gives an isomorphism $\B_\lambda(\Tor) \cong \B_\mu(\Tor - pt) \times \Z^2$ where $\mu$ has only blocks of size at least $3$. Then Corollary~\ref{LCS_B_lambda(S)_stable} implies that the LCS stops at most at $\LCS_3$ in this case. In fact, if $\mu$ is non-trivial, then the LCS of $\B_\mu(\Tor - pt)$ stops exactly at $\LCS_3$; see Theorem~\ref{Lie_ring_partitioned_B(S)}.

\ul{The other cases} have already been treated as part of Corollary~\ref{LCS_B_lambda(S)_stable}, Theorem~\ref{Lie_ring_partitioned_B(S)} and Proposition~\ref{LCS_unstable_partitioned_surface_braids}.
\end{proof}

\subsection{Partitioned braids on the M\"obius strip}\label{sec:partitioned_braids_Moebius}

As in the case of the torus, we know that the LCS of $\B_\lambda(\Moeb)$ stops if $\lambda$ has only blocks of size at least~$3$, and that it does not if there is at least one block of size $2$. The only remaining cases are the ones when $\lambda$ has some blocks of size $1$, the other ones being of size at least $3$. We begin by showing that the LCS does not stop when there are at least two blocks of size $1$. In this case, $\B_\lambda(\Moeb)$ surjects onto $\B_{1,1}(\Moeb) = \PB_2(\Moeb)$ so, by Lemma~\ref{lem:stationary_quotient}, this case follows from the following study of $\PB_2(\Moeb)$:

\begin{lemma}\label{P2(M)}
A presentation of the pure braid group $\PB_2(\Moeb)$ is given by generators $\gamma_1, \gamma_2$ and $A$ subject to the relations
\[\begin{cases}
\gamma_1 A \gamma_1^{-1} = \gamma_2^{-1} A^{-1} \gamma_2 \\
\gamma_1 \gamma_2 \gamma_1^{-1} = \gamma_2^{-1} A^{-1} \gamma_2^2.
\end{cases}\]
\end{lemma}

Recall that $\Moeb$ is the surface $\mathcal N_{1,1}$ from  \Spar\ref{sec_pstations_B(S)}. Here, we have changed the notations from \Spar\ref{sec_pstations_B(S)} to lighter ones, more suited to the case with only one crosscap  ($\gamma_i := c_1^{(i)}$) and only two strands ($A := A_{12}$).

\begin{proof}
Proposition~\ref{Fadell-Neuwirth} gives a decomposition $\PB_2(\Moeb) \cong \pi_1(\Moeb- pt) \rtimes \pi_1(\Moeb) \cong\F_2 \rtimes \Z$, where the projection onto $\pi_1(\Moeb) \cong \Z$ is given by forgetting one strand (say, the first one). Then the factor $\Z$ is generated by $\gamma_1$, and a free basis of the factor $\F_2$ is given by $\gamma_2$ and $A$ (with the notations introduced just before the proof). Moreover, the action of $\gamma_1$ by conjugation on $\langle A, \gamma_2 \rangle$ is not difficult to compute. Namely, we have $\gamma_1 A \gamma_1^{-1} = \gamma_2^{-1} A^{-1} \gamma_2$; see Figure~\ref{Moebius_relation}. Then, from Figure~\ref{Aij_as_bracket_crosscap}, we get that $[\gamma_2, \gamma_1^{-1}] = A$, which is equivalent to $\gamma_1 \gamma_2 \gamma_1^{-1} = \gamma_1 A \gamma_1^{-1} \gamma_2$. Using the previous relation, the latter equals $\gamma_2^{-1} A^{-1} \gamma_2^2$. These relations determine $\PB_2(\Moeb)$, since the group $G$ that they define decomposes as $\langle A, \gamma_2 \rangle \rtimes \langle \gamma_1 \rangle$, and the obvious projection of $G$ onto $\PB_2(\Moeb)$ must be an isomorphism on both the kernel and the quotient, hence it must be an isomorphism.
\end{proof}

\begin{figure}[ht]
  \centering
  \includegraphics[scale=0.3]{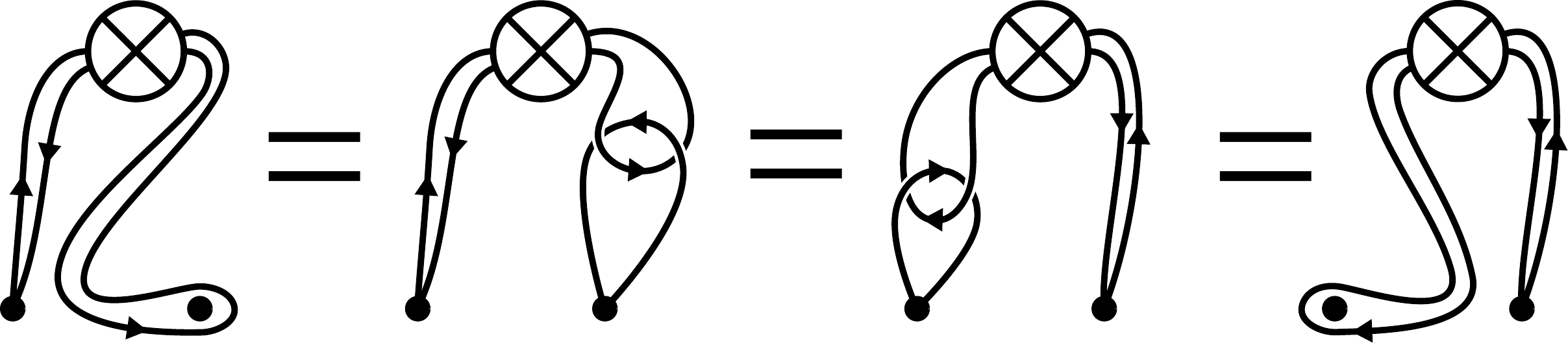}
  \caption{Relation in $\PB_2(\Moeb)$: $\gamma_1 A \gamma_1^{-1} = \gamma_2^{-1} A^{-1} \gamma_2$.}
  \label{Moebius_relation}
\end{figure}

\begin{corollary}\label{LCS_B11(M)}
The LCS of $\PB_2(\Moeb)$ does not stop.
\end{corollary}

\begin{proof}
Let us consider the quotient of $\PB_2(\Moeb) =\F_2 \rtimes \Z$ by $\LCS_2(\F_2)$. It is $\Z^2 \rtimes \Z$, where the generator $\overline \gamma_1$ of $\Z$ acts via the involution $\tau$ sending $\overline A$ to $-\overline A$ and $\overline \gamma_2$ to $\overline \gamma_2 - \overline A$. Then $V := \ima(\tau - 1)$ is $\Z \cdot \overline A$. By Proposition~\ref{LCS_Klein_tau}, for $i \geq 2$, we have $\LCS_i(\Z^2 \rtimes \Z) = 2^{i-2} V$, so the LCS of $\Z^2 \rtimes \Z$ does not stop. Thus the LCS of $\PB_2(\Moeb)$ does not either.
\end{proof}

The answer for the remaining cases are consequences of the following result:

\begin{proposition}\label{LCS_B1m(Moeb)}
For every $m \geq 3$, the LCS of $\B_{1,m}(\Moeb)$ does not stop.
\end{proposition}

\begin{proof}
We can use Proposition~\ref{Fadell-Neuwirth} to get a decomposition:
\[\B_{1, m}(\Moeb) \cong \B_m(\Moeb - pt) \rtimes \pi_1(\Moeb),\] 
where $\pi_1(\Moeb) \cong \Z$ is identified with the subgroup of $\B_{1, m}(\Moeb)$ generated by $\gamma_1$ (as above, we denote $c_1^{(i)}$ by $\gamma_i$). We know that $\LCS_\infty(\B_{1, m}(\Moeb))$ contains $\LCS_\infty(\B_ m(\Moeb - pt))$. The latter is fully invariant in $\B_m(\Moeb - pt)$, hence normal in $\B_{1, m}(\Moeb)$, and we can consider the quotient 
\[G := \B_{1, m}(\Moeb)/\LCS_\infty(\B_ m(\Moeb - pt)) \cong (\B_m(\Moeb - pt)/\LCS_\infty)\rtimes \Z.\]
Moreover, Lemma~\ref{sigma_central} and Remark~\ref{rmk:central-ext} give a central extension:
\[\begin{tikzcd}
\langle \sigma \rangle \ar[r, hook] 
&\B_m(\Moeb - pt)/\LCS_\infty \ar[r, two heads] 
&\pi_1(\Moeb - pt)^{\ab}.
\end{tikzcd}\]
The element $\sigma$ is the common class of the usual generators $\sigma_\alpha$ of $\B_m$. Since these commute with $\gamma_1$ in $\B_{1, m}(\Moeb)$, $\sigma$ is central not only in $\B_m(\Moeb - pt)/\LCS_\infty$, but in fact in $G$. In particular, $\langle \sigma \rangle$  is normal in $G$, and the quotient decomposes as:
\[G/\sigma \cong \left[(\B_\mu(\Moeb - pt)/\LCS_\infty)/\sigma\right] \rtimes \pi_1(\Moeb) \cong \pi_1(\Moeb - pt)^{\ab} \rtimes \pi_1(\Moeb) \cong \Z^2 \rtimes \Z.\]
A basis of the $\Z^2$ factor is given by $\overline{\gamma_2}$ and $A = \overline{A_{12}}$. Consider the morphism $\PB_2(\Moeb) \rightarrow \B_{1, m}(\Moeb)$ corresponding to adding $m-1$ trivial strands to the second block (constructed as in the proof of Proposition~\ref{Fadell-Neuwirth}). Composing with the projection, we get a morphism from $\PB_2(\Moeb) =\F_2 \rtimes \Z$ to $G/\sigma = \Z^2 \rtimes \Z$. From the explicit description of both of these groups, one easily sees that it has to induce an isomorphism $\PB_2(\Moeb)/\LCS_2(\F_2) \cong G/\sigma$. But the LCS of $\PB_2(\Moeb)/\LCS_2(\F_2)$ does not stop (see the proof of Corollary~\ref{LCS_B11(M)}), so we have found a quotient of $\B_{1, m}(\Moeb)$ whose LCS does not stop, whence the result.
\end{proof}

Let us sum up our results about the LCS of $\B_\lambda(\Moeb)$:

\begin{theorem}\label{thm:braid_moebius}
Let $\lambda = (n_1, \ldots , n_l)$ be a partition of an integer $n \geq 1$. The LCS of $\B_\lambda(\Moeb)$:
\begin{itemize}
\item \emph{does not stop} if $\lambda$ has at least one block of size $1$ or $2$, with the exception of~$\B_1(\Moeb) \cong \Z$.
\item \emph{stops} in all the other cases, at $\LCS_3$ if $l \geq 2$ and at $\LCS_2$ if $l = 1$.
\end{itemize}
\end{theorem}

\begin{proof}
The first statement follows from Proposition~\ref{LCS_unstable_partitioned_surface_braids} (for blocks of size $2$) and from Corollary \ref{LCS_B11(M)} and Proposition \ref{LCS_B1m(Moeb)} (for blocks of size $1$). The second one is part of the general results of Corollary~\ref{LCS_B_lambda(S)_stable} and Theorem~\ref{Lie_ring_partitioned_B(S)}.
\end{proof}

\subsection{Partitioned braids on the sphere}\label{sec:partitioned_braids_sphere}

For a partition $\lambda = (n_1, \ldots , n_l)$, the inclusion of the disc into the sphere induces a surjection $\B_\lambda \twoheadrightarrow \B_\lambda(\S^2)$ by Corollary~\ref{partitioned_Braids_on_closed_surfaces}. We can thus apply Lemma~\ref{lem:stationary_quotient} to deduce that the LCS of $\B_\lambda(\S^2)$ stops whenever the one of $\B_\lambda$ does. Namely, by Theorem~\ref{thm:partitioned-braids}, it stops at $\LCS_2$ if $n_i \geq 3$ for all $i$, save at most two indices for which $n_i = 1$. Since the abelianisation of $\B_\lambda(\S^2)$ has already been computed in Proposition~\ref{partitioned_B(S)^ab}, we have a complete description of the LCS in these cases. 

The above argument can be extended somewhat if we remark that, because of Proposition~\ref{Fadell-Neuwirth}, $\B_1(\S^2) = \pi_1(\S^2) = \{1\}$ is the cokernel of the canonical morphism $\B_\mu(\D) \rightarrow \B_{1,\mu}(\S^2)$ (for any partition~$\mu$), which is thus surjective. As a consequence, the LCS of $\B_\lambda(\S^2)$ also stops at $\LCS_2$ when $\lambda$ has three blocks of size $1$, the other ones being of size at least $3$. This proves the first half of the following theorem. The second half is proven in the remainder of this subsection in several different cases, which are synthesised into a proof of the theorem at the end of the subsection.

\begin{theorem}
\label{LCS_B_lambda(S2)}
Let $n \geq 1$ be an integer, let $\lambda = (n_1, \ldots , n_l)$ be a partition of $n$. The LCS of $\B_{\lambda}(\S^2)$:
\begin{itemize}
\item \emph{stops at $\LCS_2$} if $n_i \geq 3$ for all $i$, save at most three indices for which $n_i = 1$. 
\item \emph{does not stop} in all the other cases, except for $\B_2(\S^2) \cong \Z/2$, $\B_{2,1}(\S^2) \cong \Z/4$ and the $\B_{2,m}(\S^2)$ with $m \geq 3$, whose LCS stop either at $\LCS_{v_2(m)+1}$ or at $\LCS_{v_2(m)+2}$ when $m$ is even, where $v_2$ is the $2$-adic valuation, and either at $\LCS_2$ or at $\LCS_3$ when $m$ is odd.
\end{itemize}
\end{theorem}

In the case of $\B_{2,m}(\S^2)$ with $m \geq 3$, where the answer given above is ambiguous between two consecutive possibilities, we conjecture that its LCS stops at $\LCS_{v_2(m)+2}$ for all $m\geq 3$; see Remark \ref{rmk:experimental}.

\subsubsection{Blocks of size $1$.} We need to consider the case of the pure braid group on four strands:
\begin{lemma}\label{LCS_P4(S2)}
The LCS of $\PB_4(\S^2)$ does not stop.
\end{lemma}

\begin{proof}
We sketch a proof of the decomposition of $\PB_4(\S^2)$ from \cite[Th.~4]{GoncalvesGuaschiroots}. Recall that we have a short exact sequence by Proposition~\ref{Fadell-Neuwirth}:
\[\PB_1(\S^2 - \{3\ \text{pts}\}) \hookrightarrow \PB_4(\S^2) \twoheadrightarrow \PB_3(\S^2).\]
It is known that $\PB_3(\S^2) \cong \Z/2$. In fact, an isomorphism between $\pi_1(SO_3(\R))$ and $\PB_3(\S^2)$ is induced by $\varphi \mapsto (\varphi(e_1), \varphi(e_2), \varphi(e_3))$ from $SO_3(\R)$ to $F_3(\S^2)$. Moreover, a splitting of the above short exact sequence is given by sending the generator of $\PB_3(\S^2)$ to the full twist. Since the latter is central, $\PB_4(\S^2)$ is the direct product of  $\Z/2$ with  $\pi_1(\S^2 - \{3\ \text{pts} \}) \cong\F_2$. Thus $\PB_4(\S^2)$ is residually nilpotent (but not nilpotent), whence the result.
\end{proof}

\subsubsection{Blocks of size $2$.} Let us begin with the case where the partition has at least three blocks:
\begin{proposition}\label{LCS_B2mu(S2)}
Let $\lambda$ be a partition of $n$ with at least three blocks, and at least one block of size $2$. Then the LCS of $\B_\lambda(\S^2)$ does not stop.
\end{proposition}

\begin{proof}
Let $\mu$ be any partition. Let us consider the quotient of $\B_{2,\mu}(\S^2)$ by $\LCS_2(\B_{1,1,\mu}(\S^2))$, which is (clearly) an extension:
 \begin{equation}\label{ext1'}
\begin{tikzcd}
\B_{1,1, \mu}(\S^2)^{\ab} \ar[r, hook]
&G = \B_{2, \mu}(\S^2)/\LCS_2(\B_{1,1, \mu}(\S^2))\ar[r, two heads] 
&\Sym_2.
\end{tikzcd}
\end{equation}
Let us fix $(n_1, \ldots , n_l) := (1,1,\mu)$. From a presentation of the kernel (from Proposition~\ref{Bn(S)^ab}) and a presentation of the quotient, using the method from Appendix~\ref{section_pstation_of_ext}, we can write down a presentation of $G$. Precisely, $G$ admits the presentation with generators $s, s_i$ (for $1 \leq i \leq l$), and $a_{ij}$ (for $1 \leq i < j \leq l$), subject to the following relations:
\begin{equation*}
\begin{cases}
(1)\quad s^2 = a_{12}, \\
(2)\quad s_i = 1                                                          &\text{if } n_i = 1,\\ 
(3)\quad [s_i, s_j] = [s_i, a_{pq}] = [a_{pq}, a_{uv}] = 1                &\forall i,j,p,q,u,v,\\
(4)\quad [s, s_i] = 1                                                     &\forall i \geq 1, \\
(5)\quad [s, a_{ij}] = 1                                                  &\forall j > i \geq 3, \\
(6)\quad s a_{1j} s^{-1} =  a_{2j} \text{ and } s a_{2j} s^{-1} =  a_{1j} &\forall j \geq 3,  \\
(7)\quad 2(n_i-1) s_i + \sum\limits_{j \neq i} n_j a_{ij} = 0             &\forall i \geq 1, \\
\end{cases}
\end{equation*}
where the last relation uses additive notation in the abelian subgroup generated by the $s_i$ and the~$a_{ij}$.

Consider the subgroup $H$ of $G$ generated by $s^2$ together with the $s_i$ and the $a_{ij}$ for $i,j \geq 3$. One can easily see that it is central in $G$. Moreover, we can deduce from the presentation of $G$ a presentation of $G/H$, which turns out to be very simple. Indeed, it is generated by $s$, the $a_{1i}$ and the $a_{2i}$ for $i \geq 3$. Moreover, since $n_1 = n_2 = 1$, for $i \geq 3$ the last relation becomes $a_{1i} + a_{2i} = 0$, which means that $a_{1i}$ and $a_{2i}$ are inverse to each other.
For $i\in\{1,2\}$, we find the relations $\sum n_j a_{1j} = 0$ and $\sum n_j a_{2j} = 0$, which are equivalent to each other modulo the previous relations.
Thus, if we denote by $a_i$ the element $a_{1i} = a_{2i}^{-1}$ of $G/H$, we get that $G/H$ is generated by $a_3, a_4, \ldots , a_l$ and $s$, subject to the relations:
\begin{equation*}
\begin{cases}
(1)\quad s^2 = 1 \\                
(3)\quad [a_i, a_j] = 1            &\forall i,j,\\   
(6)\quad s a_j s^{-1} =  a_j^{-1}  &\forall j \geq 3,\\
(7)\quad \sum n_j a_j = 0.
\end{cases}
\end{equation*}
Finally, $G/H \cong A \rtimes (\Z/2)$, where $A$ is the quotient of the free abelian group on $a_3, a_4, \ldots , a_l$ by the single relation $\sum n_j a_j = 0$, and $\Z/2$ acts on $A$ via $-id$. We can use Corollary~\ref{Lie_Klein_A/center} to compute the LCS of $G/H$. Namely, since $A \cong \Z^{l-3} \times \Z/\gcd(n_i)$, it does not stop whenever $l > 3$. Thus, we have found a quotient of $\B_{2,\mu}(\S^2)$ having a non-stopping LCS whenever $\mu$ has at least two blocks. 
\end{proof}

If $\lambda$ has only two blocks, we first assume that the other block is large enough: 
\begin{proposition}\label{LCS_B2m(S2)}
Let $m \geq 3$ be an integer. The LCS of $\B_{2,m}(\S^2)$ stops at $\LCS_{v_2(m)+1}$ or at $\LCS_{v_2(m)+2}$ when $m$ is even, where $v_2$ is the $2$-adic valuation, and at $\LCS_2$ or at $\LCS_3$ when $m$ is odd.
\end{proposition}

\begin{proof}
We have seen above that the LCS of $\B_{1,1, m}(\S^2)$ (which is a quotient of $\B_{1,1, m}$) stops at $\LCS_2$ if $m \geq 3$. Then, as in the proof of Proposition~\ref{LCS_B2mu}, we have that $\LCS_\infty(\B_{2,m}(\S^2))$ contains $\LCS_\infty(\B_{1,1,m}(\S^2)) = \LCS_2(\B_{1,1,m}(\S^2))$. Thus it is enough to understand when the LCS of $G := \B_{2,m}(\S^2)/\LCS_2(\B_{1,1,m}(\S^2))$ stops. In order to do this, we use the calculations already done in the course of the proof of Proposition~\ref{LCS_B2mu(S2)}: the group $G$ is a central extension of $(\Z/m) \rtimes (\Z/2)$, where $\Z/2$ acts via $-id$. Thanks to Corollary~\ref{Lie_Klein_A/center}, we know that the LCS of the latter is given by $\LCS_i = 2^{i-1}\Z/m$ for $i \geq 2$. Hence it stops at $\LCS_2$ if $m$ is odd and at $\LCS_{v_2(m)+1}$ if $m$ is even. Thus, the result will follow from Corollary~\ref{Quotient_by_central_subgroup}, if we show that the kernel $H$ of the central extension $H \hookrightarrow G \twoheadrightarrow (\Z/m) \rtimes (\Z/2)$ injects into $G^{\ab}$. We can compute $G^{\ab}$ using the presentation of $G$ from the proof of Proposition~\ref{LCS_B2mu(S2)}, or using Proposition~\ref{partitioned_B(S)^ab} and the fact that $G^{\ab} \cong \B_{2, m}(\S^2)^{\ab}$ (by definition of $G$). Either way, we find that it is generated by three elements $s$~$(=s_1)$, $s_3$ and $a$ (the latter being the common image of $a_{13}$ and $a_{23} = s a_{13} s^{-1}$ in $G^{\ab}$), subject to the relations:
\[\begin{cases}
2s + ma = 0; \\
2(m-1) s_3 + 2a = 0.
\end{cases}\]
In other words, $G^{\ab} = \Z^3/R$ where, if $(s, s_3, a)$ denotes a basis of $\Z^3$, $R$ is generated by $2s + ma$ and $2(m-1) s_3 + 2a$.  The image $Q$ of $H$ in $G^{\ab}$ is generated by $2s$ and $s_3$. An easy calculation shows that, in $\Z^3$, $\langle 2s, s_3 \rangle \cap R$ is generated by $r_m := 2s + m(m-1)s_3$ if $m$ is even, and $r_m :=  4s + 2m(m-1)s_3$ if $m$ is odd. This means that, as an abelian group, $Q$ is presented as the quotient of the free abelian group $\Z^2$ over $2s$ and $s_3$ by the relation $r_m = 0$. But this relation already holds in $H$. Indeed, in $\B_{1,1, m}(\S^2)^{\ab} \subset G$, we have $s^2 = a_{12} -m a_{13} = -m a_{23}$ (see the proof of Proposition~\ref{LCS_B2mu(S2)}), so if we multiply the relation $2(m-1)s_3 + a_{13} + a_{23} = 0$ by $m/2$ if $m$ is even (resp.~by $m$ if $m$ is odd), we obtain $m(m-1)s_3 = s^2$ (resp.~$2m(m-1)s_3 = 2s^2$), hence $r_m = 0$ in $H \subset G$. Finally, the projection $H \twoheadrightarrow Q$ must be an isomorphism (one can construct an inverse to it using the presentation of $Q$), whence our conclusion.
\end{proof}

\begin{remark}
\label{rmk:experimental}
It seems difficult to decide theoretically between the two possibilities, although our experimental calculations~\cite{GAPcode} using GAP~\cite{GAP4} and the package NQ~\cite{Nickel1996} suggest that the LCS of $G$ (and hence also the one of $\B_{2,m}(\S^2)$) always stops at $\LCS_{v_2(m)+2}$ (for both even and odd $m$); we conjecture that this is the case.
\end{remark}

Finally, let us show that the LCS of $\B_{2,2}(\S^2)$ does not stop, using Lemma~\ref{lem:stationary_quotient}. Namely, we are looking for a quotient of $\B_{2,2}(\S^2)$ whose LCS does not stop. As a manageable non-abelian quotient, one can think of $\B_{2,2}(\S^2)/\LCS_2(\PB_4(\S^2))$, which is an extension of $\Sym_2 \times \Sym_2$ by $\PB_4(\S^2)^{\ab}$. In fact, in order to make it even more manageable, we take a further quotient, turning it into a split extension to which the methods of Appendix~\ref{sec:appendix_2} apply. 

\begin{proposition}\label{LCS_B22(S2)}
The LCS of $\B_{2,2}(\S^2)$ does not stop.
\end{proposition}

\begin{proof}
Recall that $\LCS_2(\PB_4(\S^2))$ is fully invariant in $\PB_4(\S^2)$, hence normal in $\B_{2,2}(\S^2)$, so the quotient $G = \B_{2,2}(\S^2)/\LCS_2(\PB_4(\S^2))$ is a well-defined extension of $\Sym_2 \times \Sym_2$ by $\PB_4(\S^2)^{\ab}$. The latter it is the abelian group generated by $a_{ij}$ for $1 \leq i < j \leq 4$, subject to the four relations $\sum_{j \neq i} a_{ij} = 0$ (for each $i \leq 4$); this classical computation is part of Proposition~\ref{Bn(S)^ab}.

The action of $\B_{2,2}(\S^2)$ on $\PB_4(\S^2)^{\ab}$ induced by conjugation factors through $\B_{2,2}(\S^2)/\PB_4(\S^2) \cong \Sym_2 \times \Sym_2$. This action is by permutation of the indices of the generators $a_{ij}$ of $\PB_4(\S^2)^{\ab}$. In particular, it fixes $a_{12}$ and $a_{34}$, which then generate a central subgroup $H$ of $G$. Let us consider the quotient by this central subgroup, that is, the quotient of $\B_{2,2}(\S^2)$ by its subgroup $\tilde H = \langle A_{12}, A_{34}, \LCS_2(\PB_4(\S^2)) \rangle$. This quotient is an extension:
\[\begin{tikzcd}
\PB_4(\S^2)^{\ab}/H \ar[r, hook] &G/H \ar[r, two heads] &\Sym_2 \times \Sym_2.
\end{tikzcd}\]
By definition of $\tilde H$, we have $\overline \sigma_1^2 = \overline \sigma_3^2 = 1$ in $\B_{2,2}(\S^2)/\tilde H = G/H$, so this short exact sequence splits. Moreover, $\PB_4(\S^2)^{\ab}/H$ is the quotient of $\PB_4(\S^2)^{\ab}$ by the relations $a_{12} = a_{34} = 0$, modulo which the relations defining $\PB_4(\S^2)^{\ab}$ become $a_{13} = - a_{14} = a_{13} = - a_{23}$. Thus $\PB_4(\S^2)^{\ab}/H \cong \Z$, and $G/H \cong \Z \rtimes (\Sym_2)^2$, where both transpositions act via a sign. This action factors through the signature $\varepsilon \colon \Sym_2 \times \Sym_2 \twoheadrightarrow \Z/2$, so that $\LCS_*^{\Sym_2 \times \Sym_2}(\Z) = \LCS_*^{\Z/2}(\Z)$; see \Spar\ref{ss:relative_LCS} for the definition of relative LCS. From Lemma~\ref{LCS_Klein} (whose proof could alternatively be repeated in this situation), we get that $\LCS_i(\Z \rtimes (\Sym_2)^2) = 2^{i-1}\Z$ for $i\geq 2$, so the LCS of $\B_{2,2}/\tilde H = G/H$ does not stop.
\end{proof}

\begin{proof}[Proof of Theorem \ref{LCS_B_lambda(S2)}]
The first point of the theorem was proven just before its statement. We explain how to deduce the second point from the results above.

If $\lambda$ has at least four blocks of size $1$, then $\B_{\lambda}(\S^2)$ surjects onto the pure braid group $\PB_4(\S^2)$ (by forgetting the other blocks). Using Lemma~\ref{lem:stationary_quotient}, the  result in this case follows from Lemma~\ref{LCS_P4(S2)}.

The remaining cases are the ones where there is at least one block of size $2$. If the partition has at least three blocks, the result follows from Proposition~\ref{LCS_B2mu(S2)}. If the partition has exactly two blocks, and if the size of the other block is at least $2$, then we can apply either Proposition~\ref{LCS_B2m(S2)} or Proposition~\ref{LCS_B22(S2)}. 

Finally, let us compute $\B_{2,1}(\S^2)$ and $\B_2(\S^2)$. As recalled in the proof of Lemma~\ref{LCS_P4(S2)}, $\PB_3(\S^2) \cong \pi_1(SO_3(\R)) \cong \Z/2$, so that $\B_{2,1}(\S^2)$ is an extension of $\Sym_2 \cong \Z/2$ by $\Z/2$. Thus, it must be isomorphic to $(\Z/2)^2$ or to $\Z/4$. In order to decide between the two, one can use the computation of $\B_{2,1}^{\ab}$ from Proposition~\ref{Bn(S)^ab}, and find that $\B_{2,1}(\S^2) \cong \Z/4$. As for $\B_2(\S^2)$, we use Proposition~\ref{Fadell-Neuwirth} to get an exact sequence $\pi_1(\S^2-\{pt\}) \rightarrow \PB_2(\S^2) \rightarrow \pi_1(\S^2)$, which implies that $\PB_2(\S^2) =1$, whence $\B_2(\S^2) \cong \Sym_2$. 
\end{proof}

\subsection{Partitioned braids on the projective plane}\label{sec:partitioned_braids_projective_plane}

As in \Spar\ref{par_pstation_closed_surfaces}, we see the projective plane $\Proj = \mathcal N_1$ as a sphere with one crosscap, which is the quotient of the Möbius strip $\Moeb = \mathcal N_{1,1}$ by its boundary. This allows us to use the conventions of Figure~\ref{gen_of_Bn(Ng1)} (with $g = 1$) to describe braids on the projective plane. We only modify slightly our notation, as we did before for braids on the Möbius strip:
\begin{notation}\label{notation_gamma}
We use the notation of \Spar\ref{sec_pstations_B(S)} for braids on non-orientable surfaces, but since there is only one crosscap, we denote $c_1^{(\alpha)}$ by $\gamma_\alpha$. 
\end{notation}

We now prove the following theorem, which describes when the LCS of $\B_{\lambda}(\Proj)$ stops, except for $\B_{2,m}(\Proj)$ with $m \geq 3$ (for this case, see Conjecture \ref{conj:B2m}):

\begin{theorem}\label{LCS_B_lambda(Proj)}
Let $n \geq 1$ be an integer, let $\lambda = (n_1, \ldots , n_l)$ be a partition of $n$. The LCS of $\B_{\lambda}(\Proj)$:
\begin{itemize}
\item \emph{stops at $\LCS_2$} if $l = 1$, except for $\B_2(\Proj)$, which is (strictly) $3$-nilpotent.
\item \emph{stops at $\LCS_3$} if $l \geq 2$ and $n_i \geq 3$ for all $i$.
\item \emph{does not stop} in all the other cases where $l \geq 3$. 
\end{itemize}
Moreover, the LCS of $\B_{2,2}(\Proj)$ and of $\B_{1,2}(\Proj)$ do not stop, the LCS of $\B_{1,1}(\Proj)$ stops at $\LCS_3$ and, for $m\geq 3$, the LCS of $\B_{1,m}(\Proj)$ stops at $\LCS_{v_2(m)+2}$ or at $\LCS_{v_2(m)+3}$ when $m$ is even, where $v_2$ is the $2$-adic valuation, and at $\LCS_3$ or $\LCS_4$ when $m$ is odd.
\end{theorem}

In the case of $\B_{1,m}(\Proj)$ with $m \geq 3$, where the answer given above is ambiguous between two consecutive possibilities, we conjecture that its LCS stops at $\LCS_{v_2(m)+3}$ for all $m\geq 3$; see Remark \ref{rmk:experimental2}.

The proof splits into many different cases, which we consider individually below; they are synthesised into a proof of Theorem \ref{LCS_B_lambda(Proj)} immediately after Remark \ref{rk_LCS_B22(P2)}.

\subsubsection{Blocks of size $1$.}
In order to study partitioned braids with blocks of size $1$ on the projective plane, we need to know what the Fadell-Neuwirth exact sequence becomes in the exceptional case $n = 1$ of Proposition~\ref{Fadell-Neuwirth}.

\begin{proposition}\label{P(Proj)_from_P(Moeb)}
For any partition $\mu$ of an integer $m$, we have a short exact sequence:
\[\begin{tikzcd}
\B_\mu(\Moeb)/\xi^2 \ar[r, hook] 
&\B_{1,\mu}(\Proj) \ar[r, two heads] 
&\pi_1(\Proj) \cong \Z/2,
\end{tikzcd}\]
where $\xi$ is the central element of $\B_\mu(\Moeb)$ given by all the strands going once along $\partial\Moeb$.
\end{proposition}

\begin{proof}[Sketch of proof]
One needs to show that $\xi^2$ is the image of a generator of $\pi_2(\Proj)$ by the connecting morphism in the long exact sequence of the proof of Proposition~\ref{Fadell-Neuwirth}. A generator of $\pi_2(\Proj)$ is given by the canonical projection $\S^2 \twoheadrightarrow \Proj$, which can be lifted to a map from the disc $\D$ to $F_{m+1}(\Proj)$ as follows. Let $\pi$ denote the projection of $\D$ onto $\S^2 = \D/\partial \D$ sending $\partial \D$ to the south pole $P = -N$. Let $\rho \colon \D \rightarrow SO_3(\R)$ be the unique continuous function sending each $x \in \Int(\D)$ to the rotation of axis $N \times \pi(x)$ and angle $(N, \pi(x))$ (so that $\rho(x)(N) = \pi(x)$). Let $\ev \colon SO_3(\R) \rightarrow F_{m+1}(\Proj)$ be evaluation at a base configuration whose first element is $\pm N$. Then $\ev \circ \rho$ is the required lift. Moreover, when $x$ goes once around $\partial \D$, then $\rho(x)$ goes twice around the circle of rotations of angle $\pi$ with axis orthogonal to $N$. Since $\Proj$ is to be thought of as the quotient of the Möbius strip by its boundary (which goes to $\pm N$), this circle of rotations evaluates to the element $\xi$ described in the statement.
\end{proof}

\begin{proposition}\label{LCS_B1mu(P2)}
Let $\mu$ be a partition having at least two blocks. Then the LCS of $\B_{1, \mu}(\Proj)$ does not stop.
\end{proposition}

\begin{proof}
We use the extension from Proposition~\ref{P(Proj)_from_P(Moeb)}:
\[\begin{tikzcd}
\B_\mu(\Moeb)/\xi^2 \ar[r, hook] 
&\B_{1,\mu}(\Proj) \ar[r, two heads] 
&\pi_1(\Proj) \cong \Z/2.
\end{tikzcd}\]
We denote the partition $\lambda := (1, \mu)$ of the integer $n$ by $(n_1, \ldots , n_l)$, and we denote by $l'$ the number of indices $i$ such that $n_i \geq 2$. Using notations from the \Spar\ref{sec_pstations_B(S)} (changed to $\gamma_i := c_1^{(i)}$), we have that the quotient in the previous extension is generated by the class of $\gamma_1$, and we can write $\xi$ as the following product of commuting braids:
\[\xi = \prod\limits_{i = 2}^{n}\gamma_i^2 (A_{i, i+1} \cdots A_{i,n})^{-1}.\]
Let us consider the quotient $G$ of $\B_{1,\mu}(\Proj)$ by $\LCS_2(\B_\mu(\Moeb)/\xi^2)$. It is an extension of $\Z/2$ by $\B_\mu(\Moeb)^{\ab}/2 \overline \xi$. Recall from Proposition~\ref{partitioned_B(S)^ab} that:
\[\B_\mu(\Moeb)^{\ab} \cong \Z^{l-1} \times (\Z/2)^{l'}.\]
A basis of the first factor is given by $c_2, \ldots , c_l$, where each $c_i$ is the common class of the $\gamma_\alpha$ with $\alpha$ in the $i$-th block of $\lambda$. A $(\Z/2)$-basis of the second factor is given by the elements $s_i$ described in Proposition~\ref{partitioned_B^ab}. The images of the $s_i$ in $G$ commute with the class $c_1$ of $\gamma_1$, since they have lifts with disjoint support. As a consequence, they are central not only in $\B_\mu(\Moeb)^{\ab}/2\overline \xi$, but in $G$. Let them generate the subgroup $A$ of $G$, and consider $G/A$. There is an extension:
\[\begin{tikzcd}
\Z^{l-1}/2 \overline\xi \ar[r, hook] 
&G/A \ar[r, two heads] 
&\Z/2,
\end{tikzcd}\]
where $\overline\xi = 2 \sum_{i \geq 2} n_ic_i$. This extension is not split, but we can quotient further to get a split extension of abelian groups. Namely, we need to kill the element $c_1^2$ of $G$. In order to understand it, we use the relations from Corollary~\ref{partitioned_Braids_on_closed_surfaces}: $\gamma_\alpha^2 = A_{\alpha1} \cdots A_{\alpha n}$ in $\B_{1, \mu}(\Proj)$. In $G/A$, where the $A_{\alpha \beta}$ are killed if $\alpha, \beta \geq 2$, these relations give $\overline{A_{1 \alpha}} = \overline {\gamma_\alpha}^2$ if $\alpha \geq 2$ and
\[c_1^2 \ =\ \overline {\gamma_1}^2 \ =\ \overline{A_{12}} \cdots \overline{A_{1n}} \ =\ \overline {\gamma_2}^2 \cdots \overline {\gamma_n}^2 = 2 \sum_{i \geq 2} n_i c_i \ =\ \overline \xi.\]
In particular, $\overline \xi$ commutes with $c_1$, which implies that it is central in $G/A$. The quotient $H$ of $G/A$ by $\overline \xi$ is thus a semi-direct product:
\[\begin{tikzcd}
H \cong (\Z^{l-1}/\overline \xi) \rtimes \Z/2.
\end{tikzcd}\]
Finally, we need to compute the action of $c_1$ by conjugation on the other $c_i$. Recall that $[\gamma_\alpha, \gamma_1^{-1}] = A_{1, \alpha}$ in $\B_{1, \mu}(\Proj)$ (see Figure~\ref{Aij_as_bracket_crosscap}), and $\overline{A_{1, \alpha}} = \overline {\gamma_\alpha}^2$ in $H$, so that $[c_i, c_1^{-1}] = c_i^2$ in $H$, whence $c_1^{-1} c_i^{-1} c_1 = c_i$, which implies that $c_1$ acts via $-id$ on $\langle c_2, \ldots , c_l \rangle = \Z^{l-1}/\overline \xi \cong \Z^{l-2} \times (\Z/2)$. Since we have assumed that $l \geq 3$, the LCS of $H = (\Z^{l-2} \times (\Z/2)) \rtimes (\Z/2)$ does not stop; see Corollary~\ref{LCS_Klein_A}.
\end{proof}

\begin{proposition}\label{LCS_B1m(P2)}
For $m \geq 3$, the LCS of $\B_{1, m}(\Proj)$ stops either at $\LCS_{v_2(m)+2}$ or at $\LCS_{v_2(m)+3}$ when $m$ is even, where $v_2$ is the $2$-adic valuation, and either at $\LCS_3$ or at $\LCS_4$ when $m$ is odd.
\end{proposition}

\begin{proof}
Again, we use the extension from Proposition~\ref{P(Proj)_from_P(Moeb)}:
\[\begin{tikzcd}
\B_m(\Moeb)/\xi^2 \ar[r, hook] 
&\B_{1,m}(\Proj) \ar[r, two heads] 
&\pi_1(\Proj) \cong \Z/2.
\end{tikzcd}\]
Recall that, since $m \geq 3$, the LCS of $\B_m(\Moeb)$ stops at $\LCS_2$; see Theorem~\ref{Lie_ring_B(S)}. Moreover, $\B_m(\Moeb)/\LCS_\infty = \B_m(\Moeb)^{\ab} \cong \Z \times \Z/2$, where the first factor is generated by the common class $\gamma$ of $\gamma_2, \ldots , \gamma_{m+1}$ and the second factor is generated by the common class $\sigma$ of the $\sigma_i$; see Proposition~\ref{Bn(S)^ab}. The image of the central element $\xi$ of $\B_m(\Moeb)^{\ab}$ is then:
\[\overline \xi = \sum\limits_{i = 2}^{m+1}\overline{\gamma_i^2 (A_{i, i+1} \cdots A_{i,m+1})^{-1}} = 2m \gamma - 2 \cdot \frac{m(m+1)}2 \sigma = 2m \gamma,\]
so that $\B_m(\Moeb)/(\xi^2 \LCS_2) \cong \Z/4m \times \Z/2$.
Let us consider the quotient $G$ of $\B_{1,m}(\Proj)$ by the image of $\Gamma_2(\B_m(\Moeb))$. The image of $\LCS_2(\B_m(\Moeb)) = \LCS_\infty(\B_m(\Moeb))$ is inside $\LCS_\infty(\B_{1,m}(\Proj))$, so that $G$ and $\B_{1,m}(\Proj)$ have the same associated Lie ring. Now, $G$ is an extension:
\[\begin{tikzcd}
\Z/4m \times \Z/2 \ar[r, hook] 
&G \ar[r, two heads] 
&\Z/2.
\end{tikzcd}\]
We can already see that $G$ is finite, and deduce that its LCS stops, so that  $\LCS_*(\B_{1,m}(\Proj))$ stops too. In order to be more precise, let us recall that $G = \langle \gamma, \sigma, \gamma_1 \rangle$, where the class of $\gamma_1$ generates the quotient in the previous extension. Notice that $\gamma_1^2$ is in the kernel, so it commutes with $\gamma$ and $\sigma$. Since it obviously commutes with $\gamma_1$, it is central in $G$. Moreover, as in the previous proof, the boundary relations in $\B_{1,m}(\Proj)$ give:
\[\gamma = \overline{\gamma_i^2} = \overline{A_{i,1} \cdots A_{i, m+1}} = \overline{A_{i,1}}\ \ \text{for}\ i \geq 2 \ \ \text{and}\ \  \gamma_1^2 = \overline{A_{1,2} \cdots A_{1, m+1}} = \gamma^{2m} = \overline \xi.\]
Thus, the quotient of $G$ by its central subgroup $A = \langle \gamma^{2m} \rangle$ is an extension of $\Z/2$ by $(\Z/2m) \times \Z/2$. Since $\gamma_1^2 = 1$ in $G/A$, this extension splits as a semi-direct product:
\[G/A \cong ((\Z/2m) \times \Z/2) \rtimes \Z/2.\] 
The element $\gamma_1$ of $G/A$  commutes with $\sigma$, and, using once again the isotopy from Figure~\ref{Aij_as_bracket_crosscap}: 
\[[\gamma, \gamma_1^{-1}] = \overline{[\gamma_2, \gamma_1^{-1}]} = \overline{A_{12}} = \gamma^2,\]
which implies that $\gamma_1^{-1}\gamma^{-1}\gamma_1 = \gamma$. This means that $ \Z/2$ acts trivially on the $\Z/2$ factor and via $-id$ on the $\Z/2m$ factor. This explicit description allows us to compute completely the LCS of $G/A \cong \Z/2 \times (\Z/2m \rtimes \Z/2)$ using Corollary~\ref{Lie_Klein_A/center}. Precisely, $\LCS_i(G/A) = 2^{i-1}\Z/(2m)$ for $i \geq 2$. Hence it stops at $\LCS_3$ if $m$ is odd and at $\LCS_{v_2(m)+2}$ if $m$ is even. Finally, since $A$ is cyclic of order $2$, we can apply Lemma~\ref{Quotient_by_finite_subgroup} (with $l = 1$) to conclude that the LCS of $G$ (whence the one of $\B_{1,m}(\Proj)$) stops at $\LCS_{v_2(m)+2}$ or at $\LCS_{v_2(m)+3}$ when $m$ is even and at $\LCS_3$ or $\LCS_4$ when $m$ is odd.
\end{proof}

\begin{remark}
\label{rmk:experimental2}
Similarly to the situation of $\B_{2,m}(\S^2)$ (see Remark \ref{rmk:experimental}), it seems difficult to decide theoretically between the two possibilities in Proposition \ref{LCS_B1m(P2)}. However, based on experimental calculations~\cite{GAPcode} using GAP~\cite{GAP4} and the package NQ~\cite{Nickel1996}, we conjecture that the LCS of $\B_{1, m}(\Proj)$ always stops at $\LCS_{v_2(m)+3}$ (for both even and odd $m$).
\end{remark}

We are left with two cases to consider where there is a block of size $1$, namely $\B_{1,1}(\Proj)$ and $\B_{1,2}(\Proj)$. The group $\B_{1,1}(\Proj) = \PB_2(\Proj)$ is isomorphic to the quaternion group $Q_8$ (see Corollary~\ref{B2(Proj)} below), which is $2$-nilpotent, so its LCS stops at $\LCS_3$. Notice that this means that the conclusion of Proposition \ref{LCS_B1m(P2)} is correct also for $m=1$, since that proposition would assert that the LCS of $\B_{1,1}(\Proj)$ stops at $\LCS_3$ or $\LCS_4$. This fact is also compatible with our conjecture in Remark \ref{rmk:experimental2}. In contrast, we will show that the LCS of $\B_{1,2}(\Proj)$ does not stop (Proposition~\ref{LCS_B12(P2)}). The study of this latter case is postponed, and will be part of our study of $\B_{2,m}(\Proj)$, of which a presentation will be computed in Proposition~\ref{B2m(Proj)} for every $m \geq 1$.

\subsubsection{Blocks of size $2$.} We now study the case where there is at least one block with exactly two strands.

Proposition~\ref{P(Proj)_from_P(Moeb)} can be used to recover the following classical calculations from \cite[p.~87]{vanBuskirk1966}:

\begin{corollary}\label{B2(Proj)}
The pure braid group $\PB_2(\Proj)$ is isomorphic to the quaternion group $Q_8$ (which is $2$-nilpotent), and $\B_2(\Proj)$ to the dicyclic group of order $16$ (which is $3$-nilpotent). Precisely, a presentation of the latter is:
\[\B_2(\Proj) = \left\langle \sigma_1, \gamma_1\ \middle| \ \gamma_1^2 = [\sigma_1\gamma_1\sigma_1^{-1}, \gamma_1^{-1}] = \sigma_1^2 \right\rangle,\]
where $\sigma_1$ and $\gamma_1$ are the elements of $\B_2(\Proj)$ defined above (Notation~\ref{notation_gamma}).
\end{corollary}

\begin{proof}
Recall that the dicyclic group of order $16$ can be defined by the presentation:
\[Dic_{16} = \left\langle s, x, y\ \middle| \ sx = ys,\  x^2 = y^2 = (xy)^2 = s^2 \right\rangle,\]
and that it contains the subgroup $Q_8$ of quaternions as the index-$2$ subgroup generated by $x$ and $y$. Notice that, modulo the other relations, $(xy)^2 = s^2$ is equivalent to $s^2 x^{-1} y x s^2 = s^2$, which is equivalent to $[y, x^{-1}] = s^2$ by passing to the inverses. We can also use the first relation to eliminate $y = sxs^{-1}$. As a consequence, we also have:
\[Dic_{16} = \left\langle s, x\ \middle| \ x^2 = [sxs^{-1}, x^{-1}] = s^2 \right\rangle.\]
It is easy to check that the elements $\sigma_1, \gamma_1$ and $\gamma_2$ of $\B_2(\Proj)$ satisfy the above relations, so that $s \mapsto \sigma_1$ and $x \mapsto \gamma_1$ define a morphism $\varphi$ from $Dic_{16}$ to $\B_2(\Proj)$. Indeed, the relation $\gamma_1^2 = \sigma_1^2$ is one of the boundary relations from Corollary~\ref{Braids_on_closed_surfaces}, and an isotopy witnessing the last one is drawn in Figure~\ref{Aij_as_bracket_crosscap}. Note that $\varphi$ sends the element $y = sxs^{-1}$ to  $\sigma_1 \gamma_1 \sigma_1^{-1} = \gamma_2$.

By taking $m = n = 1$, the short exact sequence of Proposition~\ref{P(Proj)_from_P(Moeb)} specialises to:
\[\begin{tikzcd}
\pi_1(\Moeb)/\xi^2 \cong \Z/4 \ar[r, hook] 
&\PB_2(\Proj) \ar[r, two heads] 
&\pi_1(\Proj) \cong \Z/2.
\end{tikzcd}\]
Indeed, $\B_1(\Moeb) = \pi_1(\Moeb)$ is isomorphic to $\Z$, and the element $\xi$, which is a loop parallel to the boundary of the Möbius strip, is the square of a generator. From this, we deduce that $\PB_2(\Proj)$ has eight elements, and that it is generated by $\gamma_1$ (the image of a generator of $\pi_1(\Moeb)$) and $\gamma_2$ (a lift of the generator of $\pi_1(\Proj)$). As a consequence, $\varphi$ must induce an isomorphism between $Q_8 = \langle x,y \rangle \subset Dic_{16}$ and $\PB_2(\Proj)$. Then, we can use the usual extension
\[\begin{tikzcd}
\PB_2(\Proj) \ar[r, hook] 
&\B_2(\Proj) \ar[r, two heads] 
&\Sym_2
\end{tikzcd}\]
to deduce that $\sigma_1, \gamma_1$ and $\gamma_2$ generate $\B_2(\Proj)$, and that $\B_2(\Proj)$ has sixteen elements. Hence $\varphi$ is an isomorphism.
\end{proof}

We first deal with the case where there are at least three blocks:

\begin{proposition}\label{LCS_B2mu(P2)} 
Let $\mu$ be a partition having at least two blocks. Then the LCS of $\B_{2, \mu}(\Proj)$ does not stop.
\end{proposition}

\begin{proof}
We are looking for a quotient whose LCS can be computed, and does not stop. Let us consider the Fadell-Neuwirth extension from Proposition~\ref{Fadell-Neuwirth}:
\[\begin{tikzcd}
\B_\mu(\Moeb - \{pt\}) \ar[r, hook] 
&\B_{2,\mu}(\Proj) \ar[r, two heads] 
&\B_2(\Proj).
\end{tikzcd}\]
We use notations similar to the ones from the proof of Proposition~\ref{LCS_B1mu(P2)}: we denote the partition $\lambda := (1, 1, \mu)$ of the integer $n$ by $(n_1, \ldots , n_l)$, and we denote by $l'$ the number of indices $i$ such that $n_i \geq 2$. We also use Notation~\ref{notation_gamma} for braids on the projective plane. 

In order to get a more manageable extension, we first take the quotient $G$ of $\B_{2,\mu}(\Proj)$ by $\LCS_2(\B_\mu(\Moeb - \{pt\}))$, getting an extension:
\[\begin{tikzcd}
\B_\mu(\Moeb - \{pt\})^{\ab} \ar[r, hook] 
&G \ar[r, two heads] 
&\B_2(\Proj).
\end{tikzcd}\]
The kernel $\B_\mu(\Moeb - \{pt\})^{\ab}$ is computed in Proposition~\ref{partitioned_B(S)^ab}: it is the product $(\pi_1(\Moeb - \{pt\})^{\ab})^{l-2} \times (\Z/2)^{l'}$, where the first factor is generated by the classes $c_i$ of $\gamma_\alpha$, $a_{1i}$ of $A_{1\alpha}$ and $a_{2i}$ of $A_{2\alpha}$ (for $\alpha$ in the $i$-th block, with $i \geq 3$), subject to the relations $a_{1i} + a_{2i} = 2 c_i$ (for each $i$, we get a copy of $\pi_1(\Moeb - \{pt\})^{\ab} \cong \Z^2$), and the second one is generated by the classes $s_i$ of the $\sigma_\alpha$ (for $\alpha$ and $\alpha + 1$ in the $i$-th block of $\lambda$, which is possible only if $n_i \geq 2$). The group $G$ is generated by these elements, together with the classes of $\sigma_1$ and $\gamma_1$ (whose images generate $\B_2(\Proj)$). 

The elements $s_i$ commute with the other elements of $\B_\mu(\Moeb - \{pt\})^{\ab}$ (which is abelian), but also with $\sigma_1$ and $\gamma_1$ (for reasons of support). Hence they are central elements of $G$. Let us denote by $A \cong (\Z/2)^{l'}$ the central subgroup they generate, and let us consider the extension:
\[\begin{tikzcd}
\Z^{2(l-2)} \ar[r, hook] 
&G/A \ar[r, two heads] 
&\B_2(\Proj).
\end{tikzcd}\]

Corollary~\ref{B2(Proj)} gives a presentation of the quotient, namely:
\[\B_2(\Proj) = \left\langle \sigma_1, \gamma_1\ \middle| \ \gamma_1^2 = [\sigma_1\gamma_1\sigma_1^{-1}, \gamma_1^{-1}] = \sigma_1^2 \right\rangle.\]
We now compute a presentation of $G/A$, using the method from Appendix~\ref{section_pstation_of_ext}. Generators are $c_i$, $a_{1i}$ and $a_{2i}$ (for $3 \leq i \leq l$), together with $\sigma_1$ and $\gamma_1$. Relations defining the kernel are the ones saying that the $c_i$, the $a_{1i}$ and the $a_{2i}$ commute with each other, together with $a_{1i} + a_{2i} = 2 c_i$ (for each $i \geq 3$). The latter could be used to eliminate $a_{2i} = 2c_i - a_{1i}$. However, we will choose not to do so here, and to give a redundant, but more tractable presentation of $G/A$. Relations lifting the above presentation of the quotient are:
\begin{equation*}
\begin{cases}
\gamma_1^2  =  \sigma_1^2 \cdot a_{13}^{n_3} \cdots a_{1l}^{n_l}, \\
[\sigma_1\gamma_1\sigma_1^{-1}, \gamma_1^{-1}] = \sigma_1^2.
\end{cases}
\end{equation*}
Indeed, these hold in $\B_{2,\mu}(\Proj) \subset \B_n(\Proj)$: the first one is one of the boundary relations from Corollary~\ref{Braids_on_closed_surfaces} and the second one is the one pictured in Figure~\ref{Aij_as_bracket_crosscap}. Moreover, these are clearly lifts of the relations defining $\B_2(\Proj)$. 

We are left with understanding how $\sigma_1$ and $\gamma_1$ (whence also $\gamma_2 = \sigma_1 \gamma_1 \sigma_1^{-1}$) act by conjugation on the $c_i$, the $a_{1i}$ and the $a_{2i}$.  We claim that the following relations hold in $G/A$:
\begin{equation*}
\begin{cases}
\sigma_1 a_{1i} = a_{2i} \sigma_1,\ \ \sigma_1 a_{2i} =  a_{1i} \sigma_1\ \text{ and }\ \sigma_1 \rightleftarrows c_i \\
\gamma_1 a_{ki} \gamma_1^{-1} = (-1)^{\delta_{1k}}a_{ki} \ \text{ and }\  \gamma_1 c_i \gamma_1^{-1} = c_i - a_{1i},
\end{cases}
\end{equation*}
where we use additive notations in the (abelian) subgroup generated by the $c_i$, the $a_{1i}$ and the $a_{2i}$. These are images of relations holding in $\B_{2,\mu}(\Proj)$, which can be proved by drawing explicit isotopies. Precisely, the first one comes from $\sigma_1 A_{1i} \sigma_1^{-1} = A_{2i}$, the second one from $\sigma_1 A_{2i} \sigma_1^{-1} = A_{2i}^{-1} A_{1i} A_{2i}$ and the third one from the commutation of $\sigma_1$ with all the $\gamma_\alpha$ if $\alpha \geq 3$. The other relation comes from $\gamma_1$ commuting with $A_{2i}$, and from $[\gamma_i, \gamma_1^{-1}] = A_{1i}$ (Figure~\ref{Aij_as_bracket_crosscap}), that is, $\gamma_1^{-1} \gamma_i \gamma_1 = A_{1i}^{-1}\gamma_i$. Notice that the relation involving $\gamma_1 a_{1i} \gamma_1^{-1}$ can be deduced from the other two, using $a_{1i} = 2c_i - a_{2i}$.

We now have a presentation of $G/A$. In order to get a simpler quotient, we quotient further by $\sigma_1^2$ and $\gamma_1^2$. That is, we add the relations $\sigma_1^2 = \gamma_1^2 = 1$ to the previous presentation. The result is a split extension:
\[ \Z^{2(l-2)}/\left(\sum\limits_{i = 3}^l n_i a_{1i} = \sum \limits_{i = 3}^l n_i a_{2i} = 0 \right) \rtimes W_2.\]
If we quotient further by $\sum n_i c_i$ (which is central in the above semi-direct product, since it is fixed by the action of $W_2$), we obtain a semi-direct product $M \rtimes W_2$, where $M$ has an explicit workable description. In fact, as a $W_2$-representation, $M \cong \Lambda \otimes B$, where $\Lambda$ is the canonical representation of $W_2$ defined in~\Spar\ref{par_can_rep_of_W2} and $B$ is the quotient of  $\Z^{l-2}$ by the vector $(n_3, \ldots , n_l)$, seen as a trivial $W_2$-representation. Precisely, if $e_i$ ($i = 3,\ldots, l$) is the generating family of $B$ obtained from the canonical basis of $\Z^2$, an isomorphism $\Lambda \otimes B \cong M$ is given, with the notations from Remark~\ref{Basis_of_Lambda}, by $a \otimes e_i \mapsto a_{2i}$, $b \otimes e_i \mapsto a_{1i}$ and $c \otimes e_i \mapsto c_i$ (with notations from Remark~\ref{Basis_of_Lambda}). 

We finally use our hypothesis: since $l \geq 4$, the rank of $B$ is not $0$, so it surjects onto $\Z$. Thus, $M$ surjects onto $\Lambda \otimes \Z = \Lambda$ (as a $W_2$-representation), and $M \rtimes W_2$ surjects onto $\Lambda \rtimes W_2$, whose LCS (computed in Proposition~\ref{LCS_can_rep_of_W2}) does not stop.
\end{proof}

Now, we are left with the case where there are precisely two blocks, one of which has exactly two strands. First of all, we give an explicit presentation of the associated group:

\begin{proposition}\label{B2m(Proj)}
Let $m \geq 1$ be an integer. The group $\B_{2,m}(\Proj)$ admits the presentation with generators $\sigma_1, \sigma_3, \sigma_4, \ldots , \sigma_{m+1}$, $\gamma_1$, $\gamma_3$ and $A_{23}$, subject to the following relations:
\begin{equation*}
\begin{cases}
(\mathbb PR1)\quad \sigma_4, \ldots , \sigma_{m+1}\ \rightleftarrows\ \gamma_3, A_{23}; \\
(\mathbb PR2)\quad A_{23}\ \rightleftarrows\  \sigma_3 \gamma_3 \sigma_3^{-1};  \\
(\mathbb PR3)\quad (\sigma_3 A_{23})^2 = (A_{23}\sigma_3)^2; \\
(\mathbb PR4)\quad  [\sigma_i \gamma_i \sigma_i^{-1}, \gamma_i^{-1}] = \sigma_i^2\ \ \text{for}\ i \in \{1,3\};\\
(\mathbb PR5)\quad \gamma_1^2 = \sigma_1 \left( \prod\limits_{k = 2}^{m+1} (\sigma_k \cdots \sigma_3) A_{23} (\sigma_k \cdots \sigma_3)^{-1} \right) \sigma_1; \\
(\mathbb PR6)\quad \sigma_1\ \rightleftarrows\ \sigma_3, \sigma_4, \ldots , \sigma_{m+1}, \gamma_3; \\
(\mathbb PR7)\quad \gamma_1\ \rightleftarrows\ \sigma_3, \sigma_4, \ldots , \sigma_{m+1}, A_{23}; \\
(\mathbb PR8)\quad \gamma_3^2 = (\sigma_1^{-1} A_{23} \sigma_1) A_{23} \cdot \prod\limits_{k = 3}^{m+1} (\sigma_k \cdots \sigma_4) \sigma_3^2 (\sigma_k \cdots \sigma_4)^{-1}; \\
(\mathbb PR9)\quad [\gamma_3, \gamma_1^{-1}] = \sigma_1^{-1} A_{23} \sigma_1.
\end{cases}
\end{equation*}
As above, the choice of names for the generators is coherent with their geometric interpretation (see Notation~\ref{notation_gamma} and Figure~\ref{gen_of_Bn(Ng1)}).
\end{proposition}

\begin{remark}
This set of generators is not minimal, although it is close to being minimal. Indeed, we can eliminate $A_{23}$ using $(\mathbb PR9)$. The set of generators thus obtained is then minimal, at least for $m = 1,2$, since the classes of the generators are a $\Z/2$-basis of the abelianisation in this case. However, this elimination would render the relations much less tractable; the above seems a better compromise between the number of generators and readability of the relations.
\end{remark}

\begin{proof}[Proof of Proposition~\ref{B2m(Proj)}]
Let us begin by checking that these relations hold for the usual elements $\sigma_1, \sigma_3, \sigma_4, \ldots , \sigma_{m+1}$, $\gamma_1$, $\gamma_3$ and $A_{23}$ of $\B_{2,m}(\Proj)$. Let us first remark that we can express other usual elements in terms of these elements. Namely, $A_{13} = \sigma_1^{-1} A_{23} \sigma_1$ and, if $k \in \{3,\ldots, m+2\}$, $\gamma_k$ (resp.~$A_{2k}$, $A_{1k}$) is obtained from $\gamma_3$ (resp.~$A_{23}$, $A_{13}$) by conjugation by $\sigma_k \cdots \sigma_4$ on the left (by convention, this product equals $1$ if $k = 3$). Then one can see that $(\mathbb PR5)$ and $(\mathbb PR8)$ are the boundary relations from Corollary~\ref{Braids_on_closed_surfaces}, corresponding to the strands $1$ and $3$ (recall that $A_{12} = \sigma_1^2$ and that $\sigma_1$ commutes with $\sigma_3, \ldots , \sigma_{m+1}$). All the other relations can be checked by drawing explicit isotopies. Precisely, the relations $(\mathbb PR1)$, $(\mathbb PR2)$, $(\mathbb PR6)$ and $(\mathbb PR7)$ are commutation relations between elements having disjoint support; $(\mathbb PR4)$ and $(\mathbb PR9)$ are instances of the relation drawn in Figure~\ref{Aij_as_bracket_crosscap}; $(\mathbb PR3)$ can either be proved by drawing an explicit isotopy, or it can be deduced from the similar relation in $\B_{1,2}$ (see Lemma~\ref{LCS_B12}), by considering the morphism from $\B_{1,2}$ to $\B_{2,m}(\Proj)$ induced by a well-chosen embedding of $\D$ into $\Proj$.

In order to show that these relations describe the group, we now apply the methods of Appendix~\ref{section_pstation_of_ext} to the Fadell-Neuwirth extension from Proposition \ref{Fadell-Neuwirth}:
\[\begin{tikzcd}
\B_m(\Proj - \{2\ \text{pts}\}) \ar[r, hook] 
&\B_{2,m}(\Proj) \ar[r, two heads] 
&\B_2(\Proj).
\end{tikzcd}\]
Since $\Proj - \{2\ \text{pts}\} = \mathscr N_{1,2}$, a presentation of the kernel is given by Proposition~\ref{Bm(Ngn)}, for $n = 1$. In order for it to identify with the right subgroup of  $\B_{2,m}(\Proj)$ (corresponding to braids on the strands $3,\ldots, m+2$), indices are shifted by $2$ (so that, for instance, $c_1$ becomes $\gamma_3$), and we take $z_1 = A_{23}$; this last choice changes the presentation a little bit, but one can easily figure out the necessary modifications. Thus, $(BN1)$ and $(BN4)$ give $(\mathbb PR1)$; $(BN2)$ and $(BN6)$ are empty; $(BN3)$ is the case $i = 3$ of $(\mathbb PR4)$; $(BN5)$ is $(\mathbb PR2)$; finally, $(BN7)$ is the case $i = 3$ of $(\mathbb PR3)$.

A presentation of the quotient is given by Corollary~\ref{B2(Proj)} and its proof:
\[\B_2(\Proj) = \left\langle \sigma_1, \gamma_1, \gamma_2\ \middle| \ \gamma_1^2 = [\sigma_1\gamma_1\sigma_1^{-1}, \gamma_1^{-1}] = \sigma_1^2 \right\rangle.\]
The elements $\gamma_1, \sigma_1 \in \B_{2,m}(\Proj)$ are lifts of the elements $\gamma_1, \sigma_1 \in \B_2(\Proj)$. The second relation holds without change for these lifts, giving the case $i = 1$ of $(\mathbb PR4)$. The relation $\gamma_1^2 = \sigma_1^2$ lifts to the boundary relation associated with the first strand, which is $(\mathbb PR5)$.

We are left with finding relations describing the action of $\sigma_1$ and $\gamma_1$ (or of $\sigma_1^{-1}$ and $\gamma_1^{-1}$ -- see Remark~\ref{inverses_of_gen_in_pstation_of_ext}) by conjugation on the other generators. The commutation relations $(\mathbb PR6)$ and $(\mathbb PR7)$ describe most of this action. Only two elements still need to be expressed in terms of the generators $\sigma_3,\ldots, \sigma_{m+1}$, $\gamma_3$ and $A_{23}$ of the kernel, namely $\sigma_1^{- 1} A_{23} \sigma_1$ and $\gamma_1^{- 1} \gamma_3 \gamma_1$. The boundary relation $(\mathbb PR8)$ deals with $\sigma_1^{-1} A_{23} \sigma_1$. Finally, $(\mathbb PR9)$ deals with $\gamma_1^{-1} \gamma_3 \gamma_1$: it says that $\gamma_1^{-1} \gamma_3 \gamma_1 = \gamma_3 \sigma_1^{-1} A_{23} \sigma_1$, and the right-hand side can be expressed in terms of the generators of the kernel using the previous relation $(\mathbb PR8)$. This finishes the proof that the above relations are the ones obtained using the method of Appendix~\ref{section_pstation_of_ext}, whence the result.
\end{proof}

Let us now consider the case $m = 1$. In this case, there is no $\sigma_i$ for $i \geq 3$, so the presentation is much simpler; in fact, in the extension of the proof, the kernel is just $\pi_1(\Proj - \{2\ \text{pts}\})$, which is free on $\gamma_3$ and $A_{23}$. Thus, $(\mathbb PR1)$, $(\mathbb PR2)$ and $(\mathbb PR3)$ are empty in this case, and there is no case $i = 3$ in $(\mathbb PR4)$. The other relations reduce to:
\begin{corollary}\label{B12(Proj)}
The group $\B_{2,1}(\Proj)$ has a presentation with $4$ generators $\sigma_1$, $\gamma_1$, $\gamma_3$ and $A_{23}$ and $6$ relations (indexed as above):
\begin{equation*}
\begin{cases}
(4)\quad [\sigma_1 \gamma_1 \sigma_1^{-1}, \gamma_1^{-1}] = \sigma_1^2; \\
(5)\quad  \gamma_1^2 = \sigma_1 A_{23} \sigma_1; \\
(6)\quad  \sigma_1\ \rightleftarrows\ \gamma_3;
\end{cases}\hspace{2em}\begin{cases}
(7)\quad \gamma_1\ \rightleftarrows\ A_{23}; \\
(8)\quad \gamma_3^2 = \sigma_1^{-1} A_{23} \sigma_1 A_{23}; \\
(9)\quad [\gamma_3, \gamma_1^{-1}] = \sigma_1^{-1} A_{23} \sigma_1.
\end{cases}
\end{equation*}
\end{corollary}

\begin{remark}
This presentation is smaller than van Buskirk's presentation from \cite[Lem.~p.~84]{vanBuskirk1966}, which has $6$ generators and $13$ relations. Note that in our language, his generators are as follows, where the right-hand side of each equation uses our notation and the left-hand side uses his:
\begin{equation*}
\begin{cases}
\sigma_2 = \sigma_1, \\
a_2 = \sigma_1^{-1} A_{23} \sigma_1, \\
a_3 = A_{23},
\end{cases}\hspace{2em}\begin{cases}
\rho_1 = \gamma_3, \\
\rho_2 = \gamma_1 \sigma_1^{-1} A_{23}^{-1} \sigma_1, \\
\rho_3 = \sigma_1^{-1} \gamma_1 \sigma_1.
\end{cases}
\end{equation*}
\end{remark}

\begin{proposition}\label{LCS_B12(P2)}
The LCS of $\B_{2,1}(\Proj)$ does not stop.
\end{proposition}
\begin{proof}
Let $G$ be the quotient of $\B_{2,1}(\Proj)$ by the single relation $\sigma_1 \gamma_1 \sigma_1^{-1} = \gamma_1^{-1}$. Let us consider the presentation of $G$ given by the presentation of Corollary~\ref{B12(Proj)}, to which this relation is added. Then relation $(4)$ becomes $\sigma_1^2 = 1$, that is, $\sigma_1^{-1} = \sigma_1$. Relation $(5)$ becomes $A_{23} = \sigma_1 \gamma_1^2 \sigma_1 = \gamma_1^{-2}$, hence $(7)$ becomes redundant. Relation $(8)$ is then equivalent to $\gamma_3^2 = 1$. Relation $(9)$ becomes $\gamma_3 \gamma_1^{-1} \gamma_3^{-1} = \gamma_1$. If we add to this relation $(6)$, which says that $\sigma_1$ commutes with $\gamma_3$, we get a presentation of $\Z \rtimes (\Z/2)^2$ (where both elements of a basis of $(\Z/2)^2$ act via $-id$). Thus, $G \cong \Z \rtimes (\Z/2)^2$, whose LCS does not stop (one can either compute it with the method of Appendix~\ref{sec:appendix_2}, or quotient further by $\sigma_1 = \gamma_3$ to get $\Z \rtimes (\Z/2)$ as a quotient and apply Corollary \ref{Lie_Klein_A/center}).
\end{proof}

\begin{remark}\label{rq_LCS_B12(P2)}
Even if the imposed relation looks very much like a relation defining the infinite dihedral group $\Z \rtimes (\Z/2)$, it is not at all clear \emph{a priori} why adding this one relation should work. Much experimentation has been needed before ending up here.
\end{remark}

Next, let us consider the case $m = 2$. In this case, there is no $\sigma_i$ for $i \geq 4$. Thus, $(\mathbb PR1)$ is empty, and the boundary relations $(\mathbb PR5)$ and $(\mathbb PR8)$ become much simpler. Let us spell out the result in this case:
\begin{corollary}\label{B22(Proj)}
The group $\B_{2,2}(\Proj)$ has a presentation with $5$ generators $\sigma_1$, $\gamma_1$, $\gamma_3$ and $A_{23}$ and $8$ relations (indexed as above):
\begin{equation*}
\begin{cases}
(2)\quad A_{23}\ \rightleftarrows\  \sigma_3 \gamma_3 \sigma_3^{-1};  \\
(3)\quad (\sigma_3 A_{23})^2 = (A_{23}\sigma_3)^2; \\
(4)\quad  [\sigma_i \gamma_i \sigma_i^{-1}, \gamma_i^{-1}] = \sigma_i^2\ \ \text{for}\ i \in \{1,3\};\\
(5)\quad \gamma_1^2 = \sigma_1 A_{23} \sigma_3 A_{23} \sigma_3^{-1} \sigma_1; \\
\end{cases}\hspace{2em}\begin{cases}
(6)\quad \sigma_1\ \rightleftarrows\ \sigma_3, \gamma_3; \\
(7)\quad \gamma_1\ \rightleftarrows\ \sigma_3, A_{23}; \\
(8)\quad \gamma_3^2 = \sigma_1^{-1} A_{23} \sigma_1 A_{23} \sigma_3^2; \\
(9)\quad [\gamma_3, \gamma_1^{-1}] = \sigma_1^{-1} A_{23} \sigma_1.
\end{cases}
\end{equation*}
\end{corollary}

\begin{proposition}\label{LCS_B22(P2)}
The LCS of $\B_{2,2}(\Proj)$ does not stop.
\end{proposition}
\begin{proof}
Let us consider, as above, the projection $p \colon \B_{2,2}(\Proj) \twoheadrightarrow \B_2(\Proj) \cong Dic_{16}$ induced by forgetting the last two strands. Recall that Corollary~\ref{B2(Proj)} gives a presentation of this quotient:
\[\B_2(\Proj) = \left\langle \sigma_1, \gamma_1\ \middle| \ \gamma_1^2 = [\sigma_1\gamma_1\sigma_1^{-1}, \gamma_1^{-1}] = \sigma_1^2 \right\rangle.\]
Since the second relation is already true in $\B_{2,2}(\Proj)$ (it is the case $i = 1$ of relation $(4)$ in Corollary~\ref{B22(Proj)}), the projection $p$ becomes split if we impose the relation $\sigma_1^2 = \gamma_1^2$. We will in fact consider the quotient $G$ of $\B_{2,2}(\Proj)$ by the two relations:
\[\begin{cases}
(Q1)\quad \sigma_1^2 = \gamma_1^2; \\
(Q2)\quad \sigma_3 \gamma_3 \sigma_3^{-1} = \gamma_3^{-1}.
\end{cases}\]
Since the second relator also sits in the kernel of $p$, we get an induced split projection $\overline p \colon G \twoheadrightarrow \B_2(\Proj)$. 

Let us consider the presentation of $G$ given by the presentation of Corollary~\ref{B22(Proj)}, together with the two relations $(Q1)$ and $(Q2)$. Modulo $(Q1)$, relation $(5)$ becomes $\sigma_3 A_{23} \sigma_3^{-1}  = A_{23}^{-1}$. Modulo $(Q2)$, relation $(2)$ says that $A_{23}$ commutes with $\gamma_3$, and the case $i = 3$ of relation $(4)$ gives $\sigma_3^2 = 1$. At this point, let us remark that the relations obtained so far say that the subgroup $\langle \gamma_3, A_{23}, \sigma_3 \rangle$ is a quotient of $\Z^2 \rtimes (\Z/2)$, where the action of $\Z/2$ is by $-id$. 

Continuing our investigation, we remark that relation $(3)$ is a consequence of the previous relation (both sides of it are killed modulo these). Relation $(8)$ becomes $\sigma_1^{-1} A_{23} \sigma_1 = \gamma_3^2 A_{23}^{-1}$ and (remembering that $\sigma_3^2 = 1$) relation $(9)$ becomes $\gamma_1^{-1} \gamma_3 \gamma_1 = A_{23}\gamma_3^{-1}$. If we add relations $(6)$ and $(7)$ without change, we get a presentation of $G$, which will allow us to describe it in an explicit way. 

Let us first consider the subgroup $A := \langle \gamma_3, A_{23} \rangle$ (which is abelian by relation $(2)$). Relations $(Q2)$ and $(2)$ imply that it is stable under conjugation by $\sigma_3$ (which equals $\sigma^{-1}$, by $(4)$). Relations $(6)$ and $(8)$ imply that it is stable under conjugation by $\sigma_i^{-1}$, and also by $\sigma_i$ (as one sees by conjugating $(8)$ by $\sigma_i$, taking $(6)$ into account). In the same way, relations $(7)$ and $(9)$ imply that it is stable under conjugation by $\gamma_1^{\pm 1}$. Finally, all this implies that it is normal in $G$. Moreover, the presentation of $G/A$ that we get from the one of $G$ clearly gives $G/A \cong Dic_{16} \times (\Z/2)$ and, in fact, the relations defining $G/A$ were already true in $G$ for the generators $\sigma_1$, $\gamma_1$ and $\sigma_3$, so $G \twoheadrightarrow G/A$ splits, and finally:
\[G \cong A \rtimes (Dic_{16} \times (\Z/2)).\]
We are left with understanding $A$ (which is a quotient of $\Z^2$) and the action of $K := Dic_{16} \times (\Z/2)$ on it. In order to do this, let us consider the relations describing the action of $K$ on $A$, namely:
\begin{equation*}
\begin{cases}
(Q2)\quad \sigma_3 \gamma_3 \sigma_3^{-1}  = \gamma_3^{-1};  \\
(5)\quad \sigma_3 A_{23} \sigma_3^{-1}  = A_{23}^{-1};  \\
(6)\quad \sigma_1 \gamma_3 \sigma_1^{-1} = \gamma_3; \\
(8)\quad \sigma_1 A_{23} \sigma_1^{-1} = A_{23}^{-1} \gamma_3^2; \\ 
(9)\quad \gamma_1 \gamma_3 \gamma_1^{-1} = \gamma_3^{-1} A_{23}; \\
(7)\quad \gamma_1 A_{23} \gamma_1^{-1} = A_{23}.
\end{cases}
\end{equation*}
We remark that these already define an action of $K$ on $\Z^2$, which is exactly the action on $\Z^2 = \Gamma$ considered at the end of \Spar\ref{par_can_rep_of_W2}. Precisely, with the identifications $\sigma_1 \mapsto \sigma$, $\gamma_1 \mapsto \gamma$ and $\sigma_3 \mapsto \tau$ for generators of $K$, we get an equivariant map $\Gamma \twoheadrightarrow A$ sending $a$ to $A_{23}$ and $c$ to $\gamma_3$ (with the notations of Remark~\ref{Basis_of_Lambda}). This induces a surjective morphism from $\Gamma \rtimes K$ onto $A \rtimes K = G$. It is then easy to check that all the relations defining $G$ are in fact already true in $\Gamma \rtimes K$, which allows us to define a converse isomorphism $G \cong \Z^2 \rtimes K$.

Finally, Proposition~\ref{LCS_can_rep_of_W2}, together with the equality $\LCS_*^K(\Lambda) = \LCS_*^{W_2}(\Lambda)$ from the end of \Spar\ref{par_can_rep_of_W2}, gives us a complete description of the LCS of $G$ which, in particular, does not stop.
\end{proof}

\begin{remark}
The projection onto $Dic_{16} \times (\Z/2)$ in the proof can be seen as coming from the geometry. Precisely, it is the factorisation through $G$ of the projection
\[q \colon \B_{2,2}(\Proj) \twoheadrightarrow \B_2(\Proj) \times \Sym_2 \cong Dic_{16} \times (\Z/2)\]
whose first factor forgets the last two strands and whose second factor forgets the first two strands and then applies the usual projection $\pi \colon \B_2(\Proj) \twoheadrightarrow \Sym_2$. 
\end{remark}

\begin{remark}
\label{rk_LCS_B22(P2)}
This quotient looks very much like the one from the proof of Proposition~\ref{LCS_B12(P2)}, and the same remark applies (see Remark~\ref{rq_LCS_B12(P2)}). Namely, $(Q1)$ is a natural relation to impose (making the extension split), whereas it is much less clear why quotienting by $(Q2)$ (which is the same relation as in the aforementioned proof, up to re-indexing the strands) should work. 
\end{remark}

We may now complete the proof of Theorem \ref{LCS_B_lambda(Proj)}.

\begin{proof}[Proof of Theorem \ref{LCS_B_lambda(Proj)}]
The first two statements are part of the general results of Corollary~\ref{LCS_B_lambda(S)_stable} and Theorem~\ref{Lie_ring_partitioned_B(S)}, except for $\B_1(\Proj) \cong \Z/2$ and $\B_2(\Proj)$, which is the dicyclic group of order $16$ (Corollary \ref{B2(Proj)}). The third statement combines Propositions~\ref{LCS_B1mu(P2)} (if $\lambda$ has blocks of size $1$) and \ref{LCS_B2mu(P2)} (if $\lambda$ has blocks of size $2$). The fourth statement combines Propositions~\ref{LCS_B1m(P2)}, \ref{LCS_B12(P2)} and \ref{LCS_B22(P2)}, together with the fact that $\B_{1,1}(\Proj) = \PB_2(\Proj)$ is the quaternion group $Q_8$ (Corollary \ref{B2(Proj)}).
\end{proof}

The remaining cases to consider are $\B_{2,m}(\Proj)$ for $m\geq 3$. These are the only examples of partitioned surface braid groups for which we have not been able to answer the question of whether their LCS stop. We can still say something about these LCS, using the proof of Proposition~\ref{LCS_B2mu(P2)}. Recall that in this proof, the hypothesis on the number of blocks of $\mu$ was not used until the end. Moreover, in the case $\mu = (m)$ and $m \geq 3$, Proposition \ref{Lie_ring_B(S)} applies, implying that the first quotient $G$ is the quotient by $\LCS_\infty(\B_m(\Moeb - \{pt\}))$. The latter must be contained in $\LCS_\infty(\B_{2,m}(\Proj))$ so, in order to understand the LCS of $\B_{2,m}(\Proj)$, we only need to understand the LCS of $G$. Then, since the central subgroup $A = \langle s_3 \rangle$ (where $s_3$ is the class of $\sigma_3$) is cyclic of order $2$, we can apply Lemma~\ref{Quotient_by_finite_subgroup} to see that the LCS of $G$ stops if and only if the one of $G/A$ does (with possibly one more step). We have the same presentation of $G/A$ as in the proof of Proposition~\ref{LCS_B2mu(P2)}, from its decomposition as an extension of $\B_2(\Proj) = Dic_{16}$ by $\Z^2$. Note that the action of $Dic_{16}$ on the (abelian) kernel in this extension is exactly the one on $\Lambda$ from \Spar\ref{par_can_rep_of_W2}, which is through the quotient $Dic_{16} \twoheadrightarrow W_2$. Precisely, as in the previous proof, $a = a_{23}$ and $c = c_3$ identify with the basis of $\Lambda$ from Remark~\ref{Basis_of_Lambda}. However, this extension is not split, so computing its LCS seems tricky. We can try to make it split, by considering the quotient by the relation $\sigma_1^2 = \gamma_1^2$ (which is equivalent to $a^m = 1$), but then we also kill $c^{2m}$, getting  the finite quotient: 
\[(G/A)/\sigma_1^2\gamma_1^{-2} \cong \left(\Lambda/\langle ma, 2mc \rangle \right) \rtimes Dic_{16} \cong \left(\Z/m \times \Z/2m \right) \rtimes Dic_{16}.\]
In fact, using the notation of Proposition~\ref{LCS_can_rep_of_W2}, we have $\langle ma, 2mc \rangle = mV \subset \Lambda$. Thus we can deduce from Proposition~\ref{LCS_can_rep_of_W2} a computation of
\[\LCS_*^{Dic_{16}}(\Z/m \times \Z/2m) = \LCS_*^{W_2}(\Lambda/mV).\]
Namely, if $m = 2^\nu m'$, with $m'$ odd, we have that $2^\nu V$ contains $mV$, and that $2^{\nu+1}V$ equals $2^\nu V$ modulo $mV$, so $\LCS_*^{W_2}(\Lambda/mV)$ stops at $\LCS_{2\nu + 2}^{W_2}(\Lambda/mV) = 2^\nu V/mV$. Finally, the LCS of the above quotient stops at $\LCS_k$, where $k = 2 v_2(m) + 2$ if $m$ is even and $k = 4$ is $m$ is odd (in the latter case, note that the relative LCS stops at the second step, but $Dic_{16}$ is $3$-nilpotent).

This gives a lower bound for the step at which the LCS of $\B_{2,m}(\Proj)$ stops: it cannot stop before $\LCS_k$ for $k = \mathrm{max}\{ 4 , 2v_2(m)+2 \}$. However, this lower bound is far from optimal: our experimental calculations~\cite{GAPcode} using GAP~\cite{GAP4} and the package NQ~\cite{Nickel1996} show that, for all $m \leq 1024$, the LCS of $G/A$, and hence those of $G$ and of $\B_{2,m}(\Proj)$, do not stop before $\LCS_{100}$ (we also verified this for all $m = 2^\nu$ with $\nu \leq 23$). We thus conjecture:

\begin{conjecture}
\label{conj:B2m}
If $m \geq 3$, the LCS of $\B_{2,m}(\Proj)$ does not stop.
\end{conjecture}

\appendix

\chapter{Some calculations of lower central series}
\label{sec:appendix_2}

This appendix is devoted to computing the LCS of some combinatorially-defined groups. These include notably the Klein group $\Z \rtimes \Z$, the free products $\Z/2 * \Z/2$ and $\Z * \Z/2$, the Artin group of type $B_2$ and wreath products. Our main tool here is the decomposition of the LCS of a semi-direct product into a semi-direct product of filtrations, which we recall first. 

\section{Relative lower central series}\label{ss:relative_LCS}

In order to obtain actual computations, we need to recall some material from \cite[\Spar 3]{Darne_subgroups} about the LCS of a semi-direct product. 

\begin{definition}\label{SCD_relative}
Let $G$ be a group, of which $H$ is a normal subgroup. We define the \emph{relative lower central series} $\LCS_*^G (H)$ by:
\[\begin{cases} \LCS_1^G (H):= H, \\ \LCS_{k+1}^G (H):= [G, \LCS_k^G (H)]. \end{cases}\]
\end{definition}

If $G$ is the semi-direct product of $H$ with a group $K$, we write $\LCS_*^K(H)$ for $\LCS_*^{H \rtimes K}(H)$ (which does not cause any confusion: if $H$ is a normal subgroup of a group $G$, then $\LCS_*^G(H) = \LCS_*^{H \rtimes G}(H)$, for the semi-direct product associated to the conjugation action of $G$ on $H$). It was shown in \cite{Darne_subgroups} that in this case:
\[\LCS_*(H \rtimes K) = \LCS_*^K(H) \rtimes \LCS_*(K).\]
Moreover, the filtration $\LCS_*^K(H) = H \cap \LCS_*(H \rtimes K)$ does have the property that $[\LCS_i^K(H), \LCS_j^K(H)] \subseteq\LCS_{i+j}^K(H)$ (for all $i,j \geq 1$), which allows one to define an associated graded Lie ring $\Lie^K(H)$ (with brackets induced by commutators, as in \Spar\ref{Lie_rings}). Then, the Lie ring of $H \rtimes K$ decomposes into a semi-direct product of Lie rings:
\[\Lie(H \rtimes K) = \Lie^K(H) \rtimes \Lie(K).\]
This is in fact a generalisation of Lemma~\ref{lem:abelianization semidirect}, which is the degree-one part (one can check that $\Lie_1^K(H) = (H^{\ab})_K$).

We can devise an analogue of Lemma~\ref{commuting_representatives} in this context, which gives a criterion for the relative LCS to stop:
\begin{lemma}\label{commuting_representatives_rel}
Let a group $K$ act on a group $H$. Let the set $S_K$ generate $K^{\ab}$ and let the set $S_H$ generate $(H^{\ab})_K$. Suppose that, for each pair $(s,t) \in S_H^2$ (resp.~each pair $(s,t) \in S_H \times S_K$), we can find representatives $\tilde s, \tilde t \in H$ (resp.~$\tilde s \in H$ and $\tilde t \in K$) of $s$ and $t$ such that $\tilde s$ and $\tilde t$ commute in $H$ (resp.~in $H \rtimes K$). Then $\LCS_2^K (H) = \LCS_3^K (H)$, which means that $\LCS_*^K (H)$ stops at $\LCS_2^K (H)$, and:
\[\Lie(H \rtimes K) \cong (H^{\ab})_K \times \Lie(K),\]
where the first factor is concentrated in degree one. 
\end{lemma}

\begin{proof}
On the one hand, by definition of the relative LCS, an element of  $\Lie_2^K(H)$ is a sum of brackets in $\Lie(G \rtimes K)$, either of two elements of $\Lie_1^K(H)$, or of an element of $\Lie_1(K)$ with an element of $\Lie_1^K(H)$. On the other hand, the relation $[\tilde s,\tilde t] = 1$ in $H \rtimes K$ readily implies that $[s, t] = 0$ in $\Lie(H \rtimes K) \cong \Lie^K(H) \rtimes \Lie(K)$. Since $S_H$ linearly generates $\Lie_1^K(H) = (H^{\ab})_K$ and $S_K$ linearly generates $\Lie_1(K) = K^{\ab}$, we infer that under our hypothesis, all elements of $\Lie_2^K(H)$ are trivial, which means that $\LCS_2^K (H) = \LCS_3^K (H)$. Moreover, from the definition of $\LCS_*^K(H)$, this obviously implies that $\LCS_i^K (H) = \LCS_{i+1}^K (H)$ for all $i \geq 2$. The statement about Lie rings is then just a reformulation of the decomposition $\Lie(H \rtimes K) \cong \Lie^K(H) \rtimes \Lie(K)$ taking into account these conclusions.
\end{proof}

\section{Semi-direct products of abelian groups}

Let a group $G$ act on an abelian group $A$. Then the LCS of $A \rtimes G$ can be computed using linear algebra. Indeed, we have:
\[\LCS_{k+1}^G(A):= \left[A \rtimes G,\ \LCS_k^G (A)\right] = \left[G,\ \LCS_k^G (A)\right].\]
These are commutators in $A \rtimes G$, given, for  $g \in G$ and $a \in A$, by:
\[[g,a] = g \cdot a - a = (g - id)(a).\]  
As a consequence, $\LCS_{k+1}^G(A)$ is the subgroup of $A$ generated by the $(g-id)(\LCS_k^G(A))$ (for $g \in G$), which can be computed by studying the endomorphisms $g - id$ of $A$.

\medskip

We now study several instances of this situation. We begin by computing the LCS of the Klein group $\Z \rtimes \Z$ (which is the fundamental group of the Klein bottle). We then generalise this calculation to any semi-direct product of an abelian group by $\Z$ acting by $-id$. This can be generalised again, to the case of an action of $\Z$ by an involution. Finally, we compute the LCS of $\Lambda \rtimes W_2$, where $W_2  \cong (\Z/2) \wr \Sym_2$ is the Coxeter group of type $B_2$ (or $C_2$), acting on the lattice~$\Lambda$ generated by its root system.

\subsection{The Klein group}

There are two distinct automorphisms of $\Z$ (that is, $\pm id$), whence only one non-trivial action of $\Z$ on $\Z$. Thus the following definition makes sense:
\begin{definition}\label{def_Klein}
The \emph{Klein group} is the semi-direct product $K = \Z \rtimes \Z$.
\end{definition}

Let us denote by $x$ (resp.~by $t$) the element $(1,0)$ (resp.~$(0,1)$) of  $\Z \rtimes \Z$. A presentation of $K$ is given by 
\[K = \langle x,t\ |\ txt^{-1} = x^{-1} \rangle.\]
The LCS of $K$ decomposes as $\LCS_*(K) = \LCS_*^{\Z}(\Z) \rtimes \LCS_*(\Z)$. Thus, in order to understand it, we need to understand the filtration $\LCS_*^{\Z}(\Z)$.
\begin{proposition}\label{LCS_Klein}
The LCS of the Klein group is $\LCS_i(\Z \rtimes \Z) = (2^{i-1}\Z) \rtimes \{1\}$ for $i \geq 2$. In other words, $\LCS_i^{\Z}(\Z) = 2^{i-1}\Z$. In particular, $\Z \rtimes \Z$ is residually nilpotent.
\end{proposition}

\begin{proof}
This follows from the formula $[x^{2^j},t] = x^{2^j}(tx^{-2^j}t^{-1}) = x^{2^{j+1}}$, by induction on $i$.
\end{proof}

\begin{corollary}\label{Lie_Klein}
The (graded) Lie ring  of the Klein group identifies with $(\Z/2)[X] \rtimes \Z$, where the polynomial ring $(\Z/2)[X]$ is seen as a graded abelian Lie ring (where $X^i$ is of degree $i$), and the generator $T$ of $\Z$ (of degree $1$) acts via $[X^i, T] = X^{i+1}$.
\end{corollary}

\begin{proof}
From Proposition~\ref{LCS_Klein}, we get a decomposition
$\Lie(K) = \Lie(2^{*-1}\Z) \rtimes \Lie(\Z).$
Since $\Z$ is abelian, the two factors are abelian Lie rings. The result follows, by calling $X^i$ the class of $x^{2^{i-1}}$ and $T$ the class of $t$. The formula for brackets comes from $[x^{2^j},t] = x^{2^{j+1}}$.
\end{proof}

\subsection{Generalised Klein groups}

Let $A$ be any abelian group and let $\Z$ act on $A$ via the powers of $-id_A$. The corresponding semi-direct product $K_A = A \rtimes \Z$ is a generalisation of the Klein group $K = K_{\Z}$. We can generalise the above results to this context:

\begin{proposition}\label{LCS_Klein_A}
The LCS of $K_A$ is given by $\LCS_i(A \rtimes \Z) = (2^{i-1} A) \rtimes \{1\}$ for $i \geq 2$. In other words, for all $i \geq 2$, $\LCS_i^{\Z}(A) = 2^{i-1} A$. In particular, for any free abelian group $A$, $K_A$ is residually nilpotent.
\end{proposition}

\begin{proof}
Let $t$ denote the generator of $\Z$. For all $a$ in $A$, we have $[a,t] = a - t \cdot a = 2a$ in $A \rtimes \Z$. Hence $[2^j A, t] = 2^{j+1} A$, from which the calculation of the LCS follows. Then $K_A$ is residually nilpotent if and only if the intersection of the $2^j A$ is trivial, which is true for instance when $A$ is finitely generated or when $A$ is free abelian.
\end{proof}

\begin{corollary}\label{Lie_Klein_A}
Let us consider the graded abelian Lie ring $\Lie(2^{*-1} A) = \bigoplus 2^{i-1} A/2^i A$ (where the sum is taken over $i \geq 1$). The (graded) Lie ring  of $K_A$ identifies with $\Lie(2^* A) \rtimes \Z$, where the generator $T$ of $\Z$ acts via the degree-one map induced by $a \mapsto 2a$.
\end{corollary}

\begin{proof}
From Proposition~\ref{LCS_Klein_A}, we get a decomposition 
$\Lie(K) = \Lie(2^{*-1}A) \rtimes \Lie(\Z).$
Since $\Z$ and $A$ are abelian, the two factors are abelian Lie rings. The result follows, since brackets with $T$ come from commutators with $t$, given by $[a,t] = 2a$ in $A \rtimes \Z$.
\end{proof}

Since $t^2$ acts trivially on $A$, it is a central element of $K_A$ (in fact, one easily sees that it generates the centre of $K_A$ if $A$ is not trivial). Thus we can consider $A \rtimes (\Z/2)$ (where $\Z/2$ acts on $A$ via $-id_A$) as a quotient of $K_A$, which behaves in much the same way:
\begin{corollary}\label{Lie_Klein_A/center}
Consider the group $A \rtimes (\Z/2)$, where $\Z/2$ acts on the abelian group $A$ via $-id_A$. We have $\LCS_i(A \rtimes (\Z/2)) = (2^{i-1} A) \rtimes \{1\}$ for all $i \geq 2$. In particular, for any finitely generated abelian group $A$, $A \rtimes (\Z/2)$ is residually nilpotent. Moreover, the (graded) Lie ring of this group identifies with $\Lie(2^{*-1}A) \rtimes (\Z/2)$, where the generator $T$ of $\Z/2$ acts via the degree-one map induced by $a \mapsto 2a$.
\end{corollary}

Finally, let us spell out the particular case where $A \cong \Z^n$ is free abelian on some basis $x_1,\ldots, x_n$. We then denote $K_A$ by $K_n$, for short.
\begin{corollary}\label{Lie_Klein_n}
The (graded) Lie ring  of $K_n$ identifies with $(\Z/2)^n[X] \rtimes \Z$, where the polynomial ring $(\Z/2)^n[X]$ is seen as an abelian Lie ring (where $X^i$ is of degree $i$), and the generator $T$ of $\Z$ (of degree $1$) acts via $[v \cdot X^i, T] = v \cdot  X^{i+1}$, for $v \in (\Z/2)^n$.
\end{corollary}

\begin{proof}
This is just a matter of identifying $\Lie(2^{*-1} \Z^n)$ with $(\Z/2)^n[X]$, by calling $v \cdot X^i$ the class of the sum of the $2^{i-1} x_k$ such that $v_k = 1$.
\end{proof}

\subsection{A further generalisation}

Let $A$ be an abelian group and let $\Z$ act on $A$ via the powers of some involution $\tau$ (for instance, $\tau$ could exchange two isomorphic direct factors of $A$). Let us denote by $K_\tau$ the corresponding semi-direct product $A \rtimes \Z$. We have:
\[((\tau - 1) + 2)(\tau - 1) = \tau^2 - 1 = 0,\]
which means that $\tau - 1$ acts via multiplication by $-2$ on $V := \ima(\tau - 1)$. 

\begin{proposition}\label{LCS_Klein_tau}
The LCS of $K_\tau$ is given by $\LCS_i(A \rtimes \Z) = (2^{i-2} V) \rtimes \{1\}$ for $i \geq 2$. In other words, for all $i \geq 2$, $\LCS_i^{\Z}(A) = 2^{i-2} V$. In particular, for any free abelian group $A$, $K_\tau$ is residually nilpotent.
\end{proposition}

\begin{proof}
For all $a \in A$, we have $[a,t] = a - \tau(a) = (1- \tau)(a)$. This implies that $\LCS_2^{\Z}(A) = \ima(\tau - 1) = V$. Then for all $v$ in $V$, we have $[v,t] = v - \tau(v) = 2v$ in $A \rtimes \Z$, and the rest of the proof is similar to that of Proposition~\ref{LCS_Klein_A}.
\end{proof}

Corollary~\ref{Lie_Klein_A} generalises immediately to this context:
\begin{corollary}\label{Lie_Klein_tau}
The (graded) Lie ring of $K_\tau$ identifies with $(A/V \oplus V/2V \oplus 2V/4V \oplus \cdots) \rtimes \Z$, where $A/V$ and $\Z$ are in degree $1$, $2^{i-2}V/2^{i-1}V$ is in degree $i$ and the generator $T$ of $\Z$ acts via the degree-one map induced by $1 - \tau$ (which coincides with $v \mapsto 2v$ on $V$).
\end{corollary}

The reader can also easily write a generalisation of Corollary~\ref{Lie_Klein_A/center} to this context, by factoring the action of $\Z$ through $\Z/2$.

\subsection{More actions on abelian groups}\label{par_can_rep_of_W2}

Let us consider the group defined by the presentation
$\langle \sigma, \gamma \ |\ \sigma^2 = \gamma^2 = (\sigma\gamma)^4 = 1 \rangle$,
which is the Coxeter group of type $B_2$, also denoted by $W_2 = (\Z/2) \wr \Sym_2$ in the rest of the memoir. It acts on $\R^2$ in the usual way: $\gamma$ acts by \scalebox{0.5}{$\begin{pmatrix} -1 &0 \\ 0 &1 \end{pmatrix}$} and $\sigma$ by \scalebox{0.5}{$\begin{pmatrix} 0 &1 \\ 1 &0 \end{pmatrix}$}. This action preserves the lattice $\Lambda := \Z \cdot (0,1) \oplus \Z \cdot (\frac12, \frac12)$ (which is generated by roots). It also preserves the lattice $V = \Z^2$, which is of index $2$ in $\Lambda$.

\begin{proposition}\label{LCS_can_rep_of_W2}
The filtration $\LCS_*^{W_2}(\Lambda)$ on $\Lambda$ is given by: 
\[\Lambda \supset V \supset 2 \Lambda \supset 2V \supset 4 \Lambda \supset \cdots.\] 
In particular, $\Lambda \rtimes W_2$ is residually nilpotent, but not nilpotent.
\end{proposition}

\begin{proof}
One can easily write down explicitly the eight matrices for the actions of elements of $W_2$ (which are the invertible monomial matrices in $\mathrm{GL}_2(\Z)$). Recall that, for every $g \in W_2$, $g - id$ is the commutator by $g$ in $\Lambda \rtimes W_2$. It is then easy to check that the $g - id$ send $\Lambda$ to $V$ (resp.~$V$ to $2\Lambda$) and that the $(g-id)(\Lambda)$ (resp.~the $(g-id)(V)$) generate $V$ (resp.~$2\Lambda$), whence the result.
\end{proof}

\begin{remark}
One can look at the proof in a geometric way, by seeing each element $g \cdot v - v$ as the difference between two vertices of a square centred at $0$.
\end{remark}

\begin{remark}
One can compute completely the associated Lie ring, which is a semi-direct product of the abelian Lie ring $(\Z/2)[X]$ by the mod $2$ Heisenberg Lie ring $\Lie(W_2) \cong \mathfrak{n}_3(\Z/2)$.
\end{remark}

\begin{remark}\label{Basis_of_Lambda}
Let us use the notations $a := (0,1)$, $b := (1,0)$ and $c := (\frac12, \frac12)$. Then $\Lambda$ can be described abstractly as the abelian group generated by $a$, $b$ and $c$, modulo the relation $a+b = 2c$. Moreover, the action of $\sigma$ fixes $c$ and exchanges $a$ and $b$, and the action of $\gamma$ is via $a \mapsto a$, $b \mapsto -b$ and $c \mapsto c - b$. In particular, since $b = 2c - a$ and  $c - b = a - c$, in the basis $(a,c)$ of $\Lambda$, $\sigma$ acts by \scalebox{0.5}{$\begin{pmatrix} -1 &0 \\ 2 &1 \end{pmatrix}$} and $\gamma$ by \scalebox{0.5}{$\begin{pmatrix} 1 &1 \\ 0 &-1 \end{pmatrix}$}. Notice that, in this basis, $V$ is just the subgroup of elements whose second coordinate is even.
\end{remark}

In the course of the proof of Proposition~\ref{LCS_B22(P2)}, we encounter a slight variation on the above action of $W_2$ on $\Lambda \cong \Z^2$. Namely, we can construct an action of the group $K = Dic_{16} \times (\Z/2)$ (where $Dic_{16}$ is the dicyclic group of order $16$, cf.~Corollary~\ref{B2(Proj)}, which is an index-$2$ central extension of $W_2$) on $\Lambda$ by making $Dic_{16} = \langle \sigma, \gamma \rangle$ act through its quotient $W_2$ (which is the quotient by $\sigma^2$) and making $\Z/2 = \langle \tau \rangle$ act by $-id$. Notice that there is already an element in $W_2$ acting by $-id$, namely the central element $(\sigma\gamma)^2$ of $W_2$, so this action is in fact through the quotient $K \twoheadrightarrow W_2$ sending respectively $\sigma$, $\gamma$ and $\tau$ to $\sigma$, $\gamma$ and $(\sigma\gamma)^2$. In particular, this implies that we have:
\[\LCS_*^K(\Lambda) = \LCS_*^{W_2}(\Lambda).\]

\section{Free products}

\subsection{Two examples}\label{ex_of_free_products}

Consider the simplest free product of two groups, which is the infinite dihedral group $\Z/2 * \Z/2 = \langle x, y \ |\ x^2 = y^2 = 1\rangle$. We can determine its LCS from its description as a semi-direct product:
\begin{proposition}\label{LCS_of_Z/2*Z/2}
There is an isomorphism:
\[\Z/2 * \Z/2 \ \cong \ \Z \rtimes (\Z/2).\] 
As a consequence, this group is residually nilpotent, and its Lie ring is $(\Z/2)[X] \rtimes (\Z/2)$, where both factors are abelian Lie rings, and the generator $T$ of $\Z/2$ acts by $[X^i,T] = X^{i+1}$.
\end{proposition}

\begin{proof}
We know a presentation of each of these groups, namely $\Z/2 * \Z/2 = \langle x, y \ |\ x^2 = y^2 = 1\rangle$ and $\Z \rtimes (\Z/2) = \langle a, t \ |\ tat^{-1} = a^{-1},\  t^2 = 1\rangle$. It is then easy to check that the assignments $x \mapsto t$, $y \mapsto ta$ and $t \mapsto x$, $a \mapsto xy$ define morphisms inverse to each other. The rest is an application of Corollary~\ref{Lie_Klein_A/center} with $A = \Z$.
\end{proof}

Consider now the free product $\Z * (\Z/2) = \langle x, y \ |\ y^2 = 1\rangle$. We have a similar decomposition into a semi-direct product:
\begin{proposition}\label{Z*Z/2_as_sdp}
There is an isomorphism:
\[\Z * (\Z/2) \ \cong \ \F_2 \rtimes (\Z/2),\]
where the action of the generator $t$ of $\Z/2$ is given by exchanging the two elements $a$ and $b$ of a basis of the free group $\F_2$.
\end{proposition}

\begin{proof}
Again, we know a presentation of each these groups, namely $\Z * (\Z/2) = \langle x, y \ |\ y^2 = 1\rangle$ and $\F_2 \rtimes (\Z/2) = \langle a, b, t \ |\ tat^{-1} = b,\ tbt^{-1} = a,\  t^2 = 1\rangle$. It is then easy to check that the assignments $x \mapsto tb$, $y \mapsto t$ and $t \mapsto y$, $a \mapsto xy$, $b \mapsto yx$ define morphisms inverse to each other. One may alternatively observe that the Tietze transformation removing the generator $a$ turns the second presentation into the first.
\end{proof}

\begin{remark}
The group $\Z * (\Z/2)$ is isomorphic to $\wB_2$ (or $\vB_2$) and the above isomorphism can be identified with $\wB_2 \cong \wP_2 \rtimes \Sym_2$, together with $\wP_2 \cong\F_2$ (with basis $(\chi_{12},\chi_{21})$).
\end{remark}

The LCS of $\Z * (\Z/2)$ is much more difficult to compute than the one of $\Z/2 * \Z/2$ above. The reader can find a presentation of the associated Lie ring in \cite{Labute}, where the methods used are somewhat different from ours. Our methods could be adapted to recover this result, together with the residual nilpotence of the group, but we will not do so here. We only give a proof of the following:
\begin{proposition}\label{Lie(Z*Z/2)}
The group $\Z * (\Z/2)$ is residually nilpotent, but not nilpotent. Its Lie ring has only $2$-torsion elements, except in degree one.
\end{proposition}

\begin{proof}
Note that the group $\Z * (\Z/2)$ surjects onto $\Z/2 * \Z/2$, whose LCS does not stop, by Proposition~\ref{LCS_of_Z/2*Z/2}. Notice also that the statement about torsion has already been proven in Example~\ref{G^ab_equals_Z+Z/d}. As a consequence, we only need to prove that $\Z * (\Z/2)$ is residually nilpotent, which is the difficult part of the statement. We can prove it using a kind of Magnus expansion. Namely, let us consider the associative algebra of non-commutative formal power series $A := \mathbb\F_2 \llangle X, Y \rrangle/(Y^2 = 1)$. We get a morphism $\Phi$ from $\Z * (\Z/2)$ to the group $A^\times$ of units of $A$ by sending $x$ to $1+X$ and $y$ to $1+Y$. It is injective, by the usual argument: if 
$x^{a_1} y^{\epsilon_1} \cdots x^{a_l} y^{\epsilon_l} x^{a_{l+1}}$ 
is a non-trivial reduced expression of some non-trivial element $g \in \Z * (\Z/2)$ (with $\epsilon_i = \pm 1$,  $a_i \in \Z$, and $a_1$ and $a_{l+1}$ possibly trivial), then, by writing $a_i = 2^{b_i}(2c_i + 1)$, using $(1+X)^{2^k} = 1 + X^{2^k}$ and $(1+T)^\alpha = 1 + \alpha T + \cdots$, we see that the coefficient of the monomial $X^{2^{b_1}}Y X^{2^{b_2}}Y \cdots Y X^{2^{b_{l+1}}}$ in $\Phi(g)$ is not trivial, hence $\Phi(g) \neq 1$. 

Now let us denote by $(X,Y)$ the ideal generated by $X$ and $Y$ in $A$, and by $A^\times_k$ the subgroup $1 + (X,Y)^k$ of $A^\times$ (for $k \geq 1$). It is easy to see that $[A^\times_1, A^\times_k] \subset A^\times_{k+1}$ for all $k \geq 1$. As a consequence, for all $k \geq 1$, we have $\LCS_k(A^\times_1) \subseteq A^\times_k$. Since the intersection of the $1 +(X,Y)^k$ is obviously trivial, $A^\times_1$ is residually nilpotent, whence also $\Z * (\Z/2)$, which is isomorphic to one of its subgroups.
\end{proof}

\begin{remark}\label{LCS_Z/2*Z_weak}
As mentioned at the beginning of the proof above, if one wishes only to see that the LCS of $\Z*(\Z/2)$ does not stop, one need only apply~Lemma~\ref{lem:stationary_quotient} to the obvious projection $\Z*(\Z/2) \twoheadrightarrow (\Z/2)*(\Z/2)$ and one can then deduce the required result from the much simpler explicit computation of the LCS of the infinite dihedral group $(\Z/2)*(\Z/2) \cong \Z \rtimes (\Z/2)$ done in Proposition~\ref{LCS_of_Z/2*Z/2}.
\end{remark}

\subsection{An Artin group}\label{subsubsec:an_artin_group}

Consider the Artin group of type $B_2$ (which is also $\B_{1,2}$ -- see Lemma~\ref{LCS_B12}), that is:
\[G:= \langle \sigma, x\ |\ (\sigma x)^2 = (x \sigma)^2 \rangle.\]
Let $\delta := \sigma x$. Then $(\sigma x)^2 = (x \sigma)^2$ is equivalent to $\delta^2 = \sigma^{-1} \delta^2 \sigma$, so that:
\[G = \langle \sigma, \delta \ |\ \delta^2 \sigma = \sigma \delta^2 \rangle.\]
The element $\delta^2$ commutes with the generators $\delta$ and $\sigma$, hence it is central in $G$. Since $G^{\ab}$ is free on the classes of $\delta$ and $\sigma$ (as is obvious from the presentation), $\delta^2$ is of infinite order.  Moreover, the above relation clearly becomes trivial when $\delta^2$ is killed, so that:
\[G/\delta^2 = \langle \sigma, \delta \ |\ \delta^2 = 1 \rangle \cong \Z * (\Z/2).\]
We thus have a central extension:
\[\begin{tikzcd}
\Z \ar[r, hook] &G \ar[r, two heads] &\Z * (\Z/2).
\end{tikzcd}\]

\begin{proposition}\label{Artin_B2}
The group $G= \langle \sigma, x\ |\ (\sigma x)^2 = (x \sigma)^2 \rangle$ is residually nilpotent (but not nilpotent).
\end{proposition}

\begin{proof}
We have observed above that the central subgroup $\langle \delta^2 \rangle$ injects into $G^{\ab}$, i.e., we have $\langle \delta^2 \rangle \cap \LCS_2(G) = \{1\}$. Thus, we can deduce the (strict) residual nilpotence of $G$ from the fact that $\Z * (\Z/2)$ is  (strictly) residually nilpotent (see Proposition~\ref{Lie(Z*Z/2)}) by applying Proposition~\ref{Extensions_and_residues}.
\end{proof}

\begin{remark}\label{LCS_Artin_B2_weak}
The weaker fact that the LCS of $G$ does not stop can also be deduced more directly from the reasoning of Remark~\ref{LCS_Z/2*Z_weak}.
\end{remark} 

The decomposition of $G$ into a central extension can also be used to describe its Lie ring. Namely, it can be obtained from the Lie ring of $\Z * (\Z/2)$ described in \cite{Labute}; see \Spar\ref{ex_of_free_products}.
\begin{proposition}\label{Lie(Artin_B2)}
The Lie ring of $G= \langle \sigma, x\ |\ (\sigma x)^2 = (x \sigma)^2 \rangle$ is a central extension of $\Lie(\Z * (\Z/2))$ by $\Z$, concentrated in degree one. Precisely, $2\overline{\sigma x}$ is central in $\Lie(G)$, and $\Lie(G)/(2\overline{\sigma x}) \cong \Lie(\Z * (\Z/2))$.
\end{proposition}

\begin{proof}
We have seen that $\langle \delta^2 \rangle$ injects into $G^{\ab}$, which means that $\langle \delta^2 \rangle \cap \LCS_2(G) = \{1\}$. As a consequence, the projections $\pi \colon \LCS_k(G) \twoheadrightarrow \LCS_k(G/\delta^2)$ have trivial kernels for $k \geq 2$, which means that they are isomorphisms. Thus, the canonical morphism from $\Lie(G)$ to $\Lie(G/\delta^2)$ is an isomorphism in degree at least $2$. In degree one, it identifies with the projection of $G^{\ab} \cong \Z^2$ onto $(G/\delta^2)^{\ab} \cong \Z \times (\Z/2)$, whose kernel is generated by $\overline{\delta^2} = 2\overline \delta$. Moreover, since $\delta^2$ is central in $G$, its class in $\Lie(G)$ must be central, whence our result.
\end{proof}

\section{Wreath products}\label{appendix:wreath-products}

This section is devoted to the study of the LCS of $G \wr \Sym_\lambda = G^n \rtimes \Sym_\lambda$, where $G$ is any group, $\lambda = (n_1, \ldots, n_l)$ is a partition of the integer $n$ and $\Sym_\lambda \subseteq \Sym_n$ acts on $G^n$ by permuting the factors. By Proposition \ref{Bn(M)}, wreath products of this form are precisely the (partitioned) braid groups on manifolds of dimensions at least $3$. We first compute the abelianisation (Corollary~\ref{wr_ab_partitioned}), then we show that the LCS stops under some stability condition (Corollary~\ref{LCS_wr_sym_stable_partitioned}). Finally, we look at the unstable cases, whose LCS we compute if $G$ is abelian (\Spar\ref{par_unstable_wr}).

\smallskip

Let us remark that, in order to understand the LCS of $G \wr \Sym_\lambda$, we only need to study the LCS of $G \wr \Sym_n$. Indeed:
\[G \wr \Sym_\lambda \cong \prod\limits_{i=1}^l G \wr \Sym_{n_i}.\]

\subsection{Abelianisations}

Lemma~\ref{lem:abelianization semidirect} allows us to compute abelianisations of wreath products:
\begin{lemma}\label{wr_ab}
For any integer $n \geq 2$, we have $(G \wr \Sym_n)^{\ab} \cong G^{\ab} \times (\Z/2)$.
\end{lemma}

\begin{proof}
It is a direct consequence of Lemma~\ref{lem:abelianization semidirect}, applied to $G \wr \Sym_\lambda = G^n \rtimes \Sym_\lambda$: $(G \wr \Sym_n)^{\ab} = ((G^n)^{\ab})_{\Sym_n} \times \Sym_n^{\ab} =  ((G^{\ab})^n)_{\Sym_n} \times (\Z/2) = G^{\ab} \times (\Z/2)$.
\end{proof}

\begin{corollary}\label{wr_ab_partitioned}
For any partition $\lambda = (n_1,\ldots, n_l)$ of $n$, if $l'$ denotes the number of indices $i \leq l$ such that $n_i \geq 2$, we have $(G \wr \Sym_\lambda)^{\ab} \cong (G^{\ab})^l \times (\Z/2)^{l'}$.
\end{corollary}

\begin{proof}
$G \wr \Sym_\lambda$ decomposes as the direct product of the $G \wr \Sym_{n_i}$, and $(G \wr \Sym_{n_i})^{\ab}$ identifies with $G^{\ab}$ when $n_i = 1$, and with $G^{\ab} \times (\Z/2)$ if $n_i \geq 2$.
\end{proof}

\subsection{The stable case}

We now use a disjoint support argument to show that there is a stable behaviour for the LCS of $G \wr \Sym_\lambda$, occurring as soon as $n_i \geq 3$ for every $i \leq l$.

Recall that the usual generators $\tau_i$ of $\Sym_n$ are conjugate to each other, hence $\Sym_n^{\ab} \cong \Z/2$ is generated by their common class $\tau$, and $\LCS_2 = \LCS_\infty$ for $\Sym_n$. 
\begin{proposition}\label{LCS_wr_sym_stable}
Let $G$ be a group. If $n \geq 3$, the $\tau_i \tau_j^{-1}$ normally generate $\LCS_2(G \wr \Sym_n)$, and $(G \wr \Sym_n)^{\ab} \cong G^{\ab} \times \Z/2$. Moreover, the LCS of $G \wr \Sym_n$ stops at $\LCS_2$.
\end{proposition}

\begin{proof}
Let $N$ be the subgroup of $G \wr \Sym_n$ normally generated by the $\tau_i \tau_j^{-1}$. These are in $\LCS_2(\Sym_n)$, whence also in $\LCS_2(G \wr \Sym_n)$, hence the latter contains $N$. In order to show the converse inclusion, we need to show that $(G \wr \Sym_n)/N$ is abelian. For any $g \in G$, let us denote by $\overline g$ the class of $(g, 1, \ldots , 1) \in G^n$ modulo $N$. We now show that the $\overline g$, together with $\tau$, generate $(G \wr \Sym_n)/N$, and we use a disjoint support argument to show that they commute with one another.

First, let us remark that $\overline g$ commutes with $\tau$. This comes from the fact that $\tau_2$ acts trivially on $(g, 1, \ldots , 1)$, thus commutes with it in $G \wr \Sym_n$. From this, we deduce that $\overline g$ is also the class of $\sigma (g, 1, \ldots , 1) \sigma^{-1} = (1, \ldots , 1, g, 1, \ldots , 1)$ for any $\sigma \in \Sym_n$ (whose class modulo $N$ is a power of $\tau$). In particular, for all $g, h \in G$, $(g, 1, \ldots , 1)$ commutes with $(1, h, \ldots , 1)$, hence $\overline g$ commutes with $\overline h$. 

Now, every $(g_1, \ldots , g_n) \in G^n$ is the product of the $(1, \ldots , 1, g_i, 1, \ldots , 1)$, so that all the elements $(1, \ldots , 1, g, 1, \ldots , 1)$ (for all $g \in G$ and any choice of position), together with the $\tau_i$, generate $G \wr \Sym_n$. This implies that their classes $\overline g$, together with $\tau$, generate $(G \wr \Sym_n)/N$. Since these generators commute with one another, this ends the proof that $N = \LCS_2(G \wr \Sym_n)$.

The rest of the statement is a direct application of Lemmas~\ref{wr_ab} and~\ref{commuting_representatives_rel}. 
\end{proof}

\begin{corollary}\label{LCS_wr_sym_stable_partitioned}
Let $G$ be a group, $n \geq 3$ be an integer and $\lambda = (n_1, \ldots , n_l)$ be a partition of $n$ with $n_i \geq 3$ for all $i$. Then the $\tau_\alpha \tau_\beta^{-1}$ for $\alpha$ and $\beta$ in the same block of $\lambda$ normally generate $\LCS_2(G \wr \Sym_\lambda)$, and $(G \wr \Sym_\lambda)^{\ab} \cong (G^{\ab} \times \Z/2)^l$. Moreover, the LCS of $G \wr \Sym_\lambda$ stops at $\LCS_2$.
\end{corollary}

\begin{proof}
Apply Proposition~\ref{LCS_wr_sym_stable} to each factor of
$G \wr \Sym_\lambda \cong \prod\limits_{i=1}^l G \wr \Sym_{n_i}$.
\end{proof}

\subsection{Unstable cases}\label{par_unstable_wr}

Since $G \wr \Sym_1 = G$, the only case left in our study of $G \wr \Sym_\lambda$ is the case of $G \wr \Sym_2$, which can be quite complicated. We treat the case where $G$ is an abelian group, which we denote by $A$.

\begin{proposition}\label{A_wr_S2}
Let us denote by $\delta A$ the subgroup $\{(a,-a)\ |\ a \in A \}$ of $A^2$.  For all $i \geq 2$, we have  $\LCS_i(A \wr \Sym_2) = 2^{i-2} (\delta A)$. Moreover, the Lie algebra decomposes as:
\[\Lie(A \wr \Sym_2) \cong \left(A \oplus A/2A \oplus 2A/4A \oplus \cdots \right) \rtimes (\Z/2),\] 
where $A$ and $\Z/2$ are in degree $1$ and each factor of the form $2^{i-2}A/2^{i-1}A$ is in degree $i$. The Lie ring $A \oplus A/2A \oplus \cdots$ is abelian and the generator $T$ of $\Z/2$ acts on it via the degree-one map:
\[\begin{cases}
a \mapsto \overline a             &\text{in degree}\ 1, \\
\overline a \mapsto \overline{2a} &\text{in degree at least}\ 2.
\end{cases}\]
\end{proposition}

\begin{proof}
This is a straightforward application of Proposition~\ref{LCS_Klein_tau} and Corollary~\ref{Lie_Klein_tau} (adapted to an action of $\Sym_2 \cong \Z/2$ instead of $\Z$). Namely, $V = \delta A$ is the subspace of $A^2$ on which $\Sym_2$ acts by $- id$. Moreover, $A^2/V \cong A$ (via $\overline{(a,0)} \mapsfrom a$) and the map induced by $1 - \tau$ identifies with the one described in our statement.
\end{proof}

\begin{corollary}\label{LCS_G_wr_S2}
Let $G$ be a group and $\lambda$ be a partition with at least one block of size $2$. Suppose that the filtration $2^* G^{\ab}$ of $G^{\ab}$ does not stop. Then the LCS of $G \wr \Sym_\lambda$ does not stop.
\end{corollary}

\begin{proof}
The group $G^{\ab} \wr \Sym_\lambda$ is a quotient of $G \wr \Sym_\lambda$, whose LCS does not stop by Proposition~\ref{A_wr_S2}. Thus, our statement follows from Lemma~\ref{lem:stationary_quotient}.
\end{proof}

\begin{remark}
If $G^{\ab}$ is finitely generated, the condition in Corollary \ref{LCS_G_wr_S2} holds if and only if it is infinite. In general, the condition is equivalent to $2^i A \neq \{0\}$ for all $i\geq 1$, where $A$ is the quotient of $G^{\ab}$ by its maximal $2$-divisible subgroup.
\end{remark}

\subsection{About general wreath products} 

Let us indicate briefly how the work done above generalises to study the LCS of $G \wr K = G^X \rtimes K$, where $G$ is a group, $K$ is a group acting on a finite set $X$ and $G^X$ denotes the group of functions $X \to G$. Notice first that Lemma \ref{wr_ab} generalises easily to $(G \wr K)^{\ab} \cong (G^{\ab})^{X/K} \times K^{\ab}$.

Then, by looking closely at the proof of Proposition~\ref{LCS_wr_sym_stable}, one can devise a stability hypothesis that ensures that Lemma~\ref{commuting_representatives_rel} can be applied, so that $\Lie(G \wr K) \cong (G^{\ab})^{X/K} \times \Lie(K)$. In fact, one needs to assume that $K$ acts on $X$ without fixed points, and that we can find a set $S$ generating $K$ such that for all $s \in S$, each $K$-orbit of $X$ has a fixed point under the action of $s$. This means that each orbit must be large enough, but also large enough with respect to the supports of the generators of $K$. This is satisfied for instance when $K$ is a subgroup of $\Sym_n$ generated by transpositions (which means exactly that $K$ is conjugate to $\Sym_\lambda$ for some partition $\lambda$ of $n$), and all its orbits are of size at least $3$.

\chapter{Presentation of an extension}\label{section_pstation_of_ext}

Here we recall the classical construction of a presentation of a group extension from a presentation of the quotient and a presentation of the kernel, together with some knowledge of the structure of the extension (see also \cite[\Spar 2.4.3]{HBE}). We then apply this construction to show that $2$-nilpotent groups whose abelianisation is free are determined by their Lie ring. 

\medskip

Let $G$ be a group, which is an extension of a quotient $K$ by a normal subgroup $H$. Suppose that presentations of $H$ and $K$ are known, namely $H = \langle X | R \rangle$ and $K = \langle Y | S \rangle$, where $R$ is a subset of the free group $F[X]$ (resp.~$S \subset F[Y]$). For each $y$ in $Y$, let us fix a lift $\tilde y$ of the corresponding generator of $K$ to an element of $G$. Then a presentation of $G$ is given by 
\[G = \langle X \sqcup Y | R \cup \widetilde S \cup T \rangle,\]
where $\widetilde S$ and $T$ are obtained as follows:
\begin{itemize}
\item Each $s \in S$ is a word in the elements of $Y$ and their inverses. If we replace each $y$ in $s$ by its chosen lift $\tilde y$, we get an element $\tilde s$ of $G$, which is in fact in $H$, since its projection to $K$ is trivial by construction. Each element of $H$ is represented by a word on the elements of $X$ and their inverses, so we can choose some $w_s \in F[X]$ representing $\tilde s$. Then $\widetilde S$ is the following set of relations:
\[\widetilde S := \{s w_s^{-1}\ |\ s \in S\} \subset F[X \sqcup Y].\]
\item For each $y \in Y$ and each $x \in X$, the element $\tilde y x \tilde y^{-1}$ is an element of $H$, which can be represented by a word $w_{x,y} \in F[X]$. Then we define:
\[T := \{yxy^{-1}w_{x,y}^{-1}\ |\ x \in X,\ y \in Y\} \subset F[X \sqcup Y].\]
\end{itemize}

\begin{remark}
If the presentations of $H$ and $K$ are finite, this construction gives a finite presentation of $G$. 
\end{remark}

\begin{remark}[Split extensions]
When the extension splits (that is, when $G$ is a semi-direct product of $K$ by $H$), one usually chooses the lifts of generators of $K$ to be their images under a fixed section. Then the presentation obtained is somewhat simpler, since the relations in $S$ hold in $G$ (that is, $\widetilde S=S$).
\end{remark}

\begin{proposition}\label{pstation_of_ext}
The above presentation is indeed a presentation of the extension $G$.
\end{proposition}

\begin{proof}
Let $G_0$ be the group defined by the above presentation. By construction, the assignments $x \mapsto x \in H \subset G$ and $y \mapsto \tilde y$ induce a well-defined morphism $\pi$ from $G_0$ to $G$. Let $H_0$ be the subgroup of $G_0$ generated by $X$. The morphism $\pi$ restricts to a morphism $\pi_H \colon H_0 \rightarrow H$. Since the relations $R$ are satisfied in $H_0$, we can construct an inverse to $\pi_H$: it is an isomorphism.

The relations $T$ ensure that $H_0$ is stable under left conjugation by the $y \in G_0$. Moreover, for all $h \in H$,  $\pi_H(y h y^{-1}) = \pi(y) \pi(h) \pi(y)^{-1} =  \tilde y \pi_H(h) \tilde y^{-1}$.
Since $H$ is normal in $G$, left conjugation by $\tilde y$ is an automorphism of $H$. Since $\pi_H$ is an isomorphism, left conjugation by $y \in G_0$ must be an automorphism of $H_0$, which implies that $H_0$ is stable under left conjugation by $y^{-1}$. Finally, $H_0$ is stable under left conjugation by $y^{\pm 1}$, and also (clearly) by $x^{\pm 1}$, so it is normal in $G_0$.

As a consequence, $\pi$ induces a morphism of extensions:
\[\begin{tikzcd}
H_0 \ar[r, hook] \ar[d, "\pi_H"] 
&G_0 \ar[r, two heads] \ar[d, "\pi"] 
&G_0/H_0 \ar[d, dashed, "\overline\pi"] \\
H \ar[r, hook] &G \ar[r, two heads]  &K.
\end{tikzcd}\]
The relations $R$ and $T$ become trivial in $G_0/H_0$, and $\widetilde S$ reduces to $S$ there, so this quotient admits the presentation $\langle Y | S \rangle$. This implies that $\overline\pi$ is an isomorphism. Since we already know that $\pi_H$ is an isomorphism, the Five Lemma allows us to conclude that $\pi$ is an isomorphism.
\end{proof}

\begin{remark}\label{inverses_of_gen_in_pstation_of_ext}
We can replace some of the generators $y$ by their inverses before doing this construction, so we can choose $T$ to encode either left conjugation by $y$, or right conjugation by $y$, for each $y$. 
\end{remark}

\begin{corollary}\label{cor:two_nilpotent_presentation}
If $G$ is a $2$-nilpotent group whose abelianisation is free abelian, then $G$ is determined (up to isomorphism) by its associated Lie ring.
\end{corollary}

\begin{proof}
We construct a presentation of $G$ which depends only on the structure of $\Lie(G)$ (and on some choices not involving elements of $G$). The group $G$ is an extension of $G^{\ab} = \Lie_1(G)$ by $\LCS_2(G) = \Lie_2(G)$, to which we can apply the previous construction. Since $G^{\ab}$ is free abelian on some set $Y$, a presentation of this group is given by generators $Y$ and relations $\{[y,z]\ |\ y,z \in Y\}$. Let $\langle X | R \rangle$ be a presentation of the group $\Lie_2(G)$. Then $G$ admits the presentation with generators $X \sqcup Y$ and relations $R \cup \widetilde S \cup T$, constructed as above. We need to show that $S$ and $T$ can be recovered from calculations in $\Lie(G)$ alone.
\begin{itemize}
\item Let $s \in S$, that is, $s = [y,z]$ for some $y, z \in Y$. Then, by definition of $\Lie(G)$, $\tilde s = [\tilde y, \tilde z]$ is the element of $\LCS_2(G) = \Lie_2(G)$ given by the bracket of $y, z \in \Lie_1(G)$.
\item Since $[G, \LCS_2(G)] = \{1\}$, the set $T$ consists of the relations $[y,x]$, for $x \in R$ and $y \in S$. 
\end{itemize}
Thus the above presentation of $G$ can be obtained from the data of $\Lie(G)$, as claimed.
\end{proof}

\begin{remark}
Corollary~\ref{cor:two_nilpotent_presentation} is not true in general if the abelianisation of $G$ is not free. For example, the dihedral group $D_8$ of order $8$ and the quaternion group $Q_8$ are not isomorphic, but they are $2$-nilpotent groups whose Lie rings are isomorphic. Indeed, in both cases, we have $\Lie_1 G = (\Z/2)^2$ and $\Lie_2 G = \Z/2$ and the Lie structure is fully determined by saying that whenever $a$ and $b$ are two distinct non-trivial elements of $\Lie_1 G$, then $[a,b]$ is non-trivial in $\Lie_2 G$.
\end{remark}

\backmatter
\bibliographystyle{amsalpha}
\bibliography{biblio}
\printindex

\end{document}